%% file: iclr2027-broxGD-main.tex
\documentclass[11pt,letterpaper]{article}
\usepackage[hmargin=1in,vmargin=0.65in]{geometry}
\usepackage{times}
\usepackage[authoryear,round]{natbib}
\setcitestyle{citesep={;},aysep={,},yysep={;}}

\newenvironment{ack}{\section*{Acknowledgments and Disclosure of Funding}}{}
\PassOptionsToPackage{disable}{todonotes}

\usepackage[utf8]{inputenc} %
\usepackage[T1]{fontenc}    %
\usepackage{hyperref}       %
\usepackage{url}            %
\usepackage{longtable,array}
\usepackage{overpic}
\usepackage{booktabs}       %
\usepackage{amsfonts}       %
\usepackage{nicefrac}       %
\usepackage{microtype}      %
\usepackage{xcolor}         %
\usepackage{colortbl}
\definecolor{mainresultshade}{RGB}{232,242,252}

\usepackage[hyperpageref]{backref}
\hypersetup{ 
    colorlinks=true,
    citecolor=darkscarlet,
    linkcolor=black,
    filecolor=magenta,      
    urlcolor=black,
    menucolor=gray
}
\renewcommand*{\backref}[1]{}%
\renewcommand*{\backrefalt}[4]{%
   \ifcase #1 %
     \footnotesize{(Not cited.)}%
   \or
     \footnotesize{(Cited on page~#2)}%
   \else
     \footnotesize{(Cited on page~#2)}%
\fi }

\include{macros-common}

\include{macros-theorems-article}

\newcommand{\rbrac}[1]{\left( #1 \right)}
\newcommand{\sbrac}[1]{\left[ #1 \right]}

\allowdisplaybreaks

\title{\bfseries Broximal Gradient Descent: A Projection-Free Sister of Projected Gradient Descent}

\author{Peter Richt\'{a}rik \qquad Kaja Gruntkowska \qquad Hanmin Li\\[0.5em]
\small King Abdullah University of Science and Technology (KAUST)\\
\small Thuwal, Saudi Arabia}
\date{}

\usepackage{placeins}

\DeclareCaptionFont{tableonecaption}{\fontsize{10}{12}\selectfont}

\begin{document}

\maketitle
 
\begin{abstract}\phantomsection\label{concept:abstract}
We propose  Broximal Gradient Descent (\algnamesmall{BroxGD}), a projection-free sister method to projected gradient descent for constrained optimization. Its forward--backward construction replaces the proximal backward operation by the broximal operation of \citet{gruntkowska2025ball}. Instead of projecting, each step minimizes a linear function over the intersection of the constraint set $\cX$ and a ball $\mathbb{B}(x_k,t_k)$  centered at the current iterate $x_k$, of suitable radius $t_k>0$:
\[
x_{k+1}\in\argmin_{z\in\cX\cap\mathbb{B}(x_k,t_k)}\langle\nabla f(x_k),z\rangle.
\]
 We develop a comprehensive convergence theory spanning a wide range of optimization regimes and radius rules. We expect \algnamesmall{BroxGD} to find many applications and inspire numerous extensions, much like projected gradient descent. Our contribution is theoretical; potential applications and toy experiments illustrate the method and suggest directions for future work.
\end{abstract}

\section{Introduction}\label{concept:intro}

Constrained optimization problems play a fundamental role in optimization theory and machine learning (ML) \citep{boyd2004convex}. In this work, we revisit the classical formulation
\begin{equation} \label{eq:main}
    \min_{x\in \cX} f(x),
\end{equation}
where the set $\cX\subseteq \R^d$ represents constraints, and $f:\R^d \to \R$ represents a loss/cost/objective function. Hence, we seek to minimize the loss $f$ subject to the constraint $\cX$. We denote the set of all
minimizers by $\cX_\star$.
Such problems arise throughout ML \citep{pokutta2020deep, sangalli2021constrained}, statistics \citep{hathaway1985constrained, hall1999density, behr2013portfolio}, and signal processing \citep{mattingley2010real, zhong2026stability}, where constraints are used to encode prior structure, enforce regularization, or impose physical feasibility.
Constrained formulations play an increasingly important role in modern ML systems. As learning algorithms are deployed in safety-critical domains---including finance, healthcare, and autonomous systems---there is a growing demand for models that are safe, robust, interpretable, and fair \citep{donini2018empirical, yang2022safety}. Constrained optimization offers a framework for encoding such requirements directly into the training objective \citep{cotter2019optimization, dai2023safe}.
Naturally, the extremely rich optimization literature offers a multitude of algorithms for handling such problems. A useful way of mapping out the landscape of these methods is to group them by the oracle used to access the constraint set $\cX$. Most approaches fall into one of these three categories\footnote{These three oracle models are representation-free: they do not assume an explicit parametrization of $\cX$, but only access it through queries revealing geometric information about the set. When additional structure of $\cX$ is available, it can be exploited to design more specialized algorithms--see \Cref{sec:interior_pt}, \Cref{sec:splitting_rev}, \Cref{sec:penalty_rev} and \citet{boyd2004convex}.}:

\phantomsection\label{concept:projection}{\bf (i) Projection oracle.} Methods in this class assume access to a projection operator of the form \[ \Proj_{\cX}^{\phi}(x) \in \arg\min_{z \in \cX} D_{\phi}(z, x), \] where $D_{\phi}(z,x) \eqdef \phi(z) - \phi(x) - \inp{\nabla \phi(x)}{z - x}$ is the \emph{Bregman divergence} induced by a strictly convex, differentiable function $\phi$.
This includes Projected Gradient Descent (\algname{ProjGD}) \citep{bertsekas1999nonlinear} when $\phi(x)=\tfrac{1}{2}\|x\|^2$, mirror descent \citep{Nemirovsky1983problem} (and its special cases such as exponentiated gradient methods \citep{Kivinen1997Exponentiated}), and dual averaging \citep{nesterov2009primal}.
    
\phantomsection\label{concept:lmo}{\bf (ii) Linear minimization oracle (LMO).}
These methods, which include the Frank--Wolfe method \citep{FrankWolfe1956, levitin1966constrained} and its variants, access the constraint set only through linear optimization queries of the form
\[
\argmin_{z\in\cX}\langle g,z\rangle.
\]
 Such oracles are attractive in settings where projections may be expensive or intractable.

\phantomsection\label{concept:separation}{\bf (iii) Separation oracle.} Such oracles provide, for a given query point~$x$, either a certificate that $x \in \cX$ or a hyperplane separating $x$ and $\cX$. Classical methods based on this oracle include the ellipsoid method \citep{Grotschel1993Geometric}, cutting-plane schemes \citep{Gilmore1961Linear}, and bundle-type epigraph algorithms \citep{Lemarechal1995Variants}. Membership oracles \citep{Grotschel1993Geometric} are sometimes considered in this context as a weaker primitive, but are typically converted into separation access for algorithmic purposes.

\subsection{From proximal to broximal steps}\label{sec:bpm-motivation}

Our starting point is the Ball-Proximal (``Broximal'') Method (\algname{BPM}) of \citet{gruntkowska2025ball}. For a proper closed function $h:\R^d\to\R\cup\{+\infty\}$, a center $x\in\operatorname{dom}h$, and a radius $t>0$, its broximal operator is the mapping
\begin{equation}\label{eq:main-brox-definition}
 \brox_h^t(x)\eqdef\argmin_{\|z-x\|\le t}h(z).
\end{equation}
\algname{BPM} repeatedly applies this operator to the objective itself, replacing the proximal quadratic penalty by a ball constraint. Remarkably, for a proper closed convex objective with a minimizer, exact \algname{BPM} with any fixed positive radius enjoys linear convergence of objective values and reaches a minimizer in finitely many steps, without smoothness or strong convexity \citep[Theorem~8.1 and Corollary~8.2]{gruntkowska2025ball}. 

These striking guarantees concern exact minimization over each ball; they motivate the design of a new forward--backward (\algname{FB}) algorithm for solving \eqref{eq:main}, analogous to projected gradient descent, but {\em somehow} using the broximal operator instead of the projection (proximal) operator \citep{Moreau65}.

\phantomsection\label{concept:model} In order to design such a method, we first move the constraint to the objective, which is a standard procedure in convex optimization. In particular, we write the constrained optimization problem \eqref{eq:main} in the form \[ \min_{x\in \R^d} F(x), \qquad F\eqdef f+\iota_{\cX}, \] where $\iota_{\cX}$ is the indicator function of $\cX$. We then linearize  $f$, and form the model \[ m_k(z)\eqdef \langle\nabla f(x_k),z-x_k\rangle + \iota_{\cX}(z). \]
Adding a quadratic movement penalty to the model and minimizing leads to projected gradient descent:
\begin{equation}\label{eq:main-prox-model}
 x_{k+1}^{\algnamemath{ProjGD}}=\argmin_{z\in \R^d} \left\{m_k(z)+\frac{\|z-x_k\|^2}{2\gamma_k}\right\},
\end{equation}
where  $\{\gamma_k\}_{k\geq 0}$ is an appropriately chosen sequence of positive stepsizes.

\subsection{Broximal gradient descent}\label{concept:broxgd}
Our broximal viewpoint instead suggests replacing this penalty by a {\em hard movement constraint}\footnote{Appendix~\ref{sec:fb_brox} develops this viewpoint in detail.}.  The resulting method minimizes the same  model, but  over a ball centered at the current iterate:
\begin{equation}\label{eq:main-brox-model}
\boxed{ x_{k+1}\in\argmin_{\|z-x_k\|\le t_k}m_k(z),}\end{equation}
 where $\{t_k\}_{k\geq 0}$ is an appropriately chosen sequence of positive radii.  This is {\em Broximal Gradient Descent} (\algname{BroxGD})---the method we propose and analyze in this paper.  Letting $\mathbb{B}(x,t) \eqdef \{y \in \R^d \;:\; \norm{y-x}\leq t\}$, the \algname{BroxGD} update step \eqref{eq:main-brox-model} can be  equivalently written in the form
\begin{equation} \label{eq:method}
\boxed{x_{k+1}\in \argmin_{z\in\cX\cap\mathbb{B}(x_k,t_k)}\langle\nabla f(x_k),z\rangle,}
\end{equation}
which more clearly reveals its projection-free nature. Instead of projections onto the constraint $\cX$, \algname{BroxGD} relies on linear minimization over the intersection of $\cX$ and a local ball (``trust-region'').
Note that without constraints and when $\nabla f(x_k)\ne0$, \algname{BroxGD} reduces to normalized gradient descent with  stepsize $t_k$, or equivalently, to gradient descent with stepsize $\nicefrac{t_k}{\|\nabla f(x_k)\|}$.  We formalize the method in \Cref{alg:new}.

\begin{algorithm}[t]
\begin{algorithmic}[1]
\STATE \textbf{Input:} starting point $x_0\in \cX$
\FOR{$k=0,1,2,\ldots$}
    \STATE Compute the gradient
    $
        g_k \eqdef \nabla f(x_k)
    $
    \STATE Choose a radius $t_k\ge0$
    \STATE Compute the next iterate as the solution of the local linear minimization problem
    \[
        x_{k+1}\in \argmin_{z\in \cX\cap \mathbb{B}(x_k,t_k)} \langle g_k,z\rangle
    \]
\ENDFOR
\end{algorithmic}
\caption{Broximal gradient descent (\algname{BroxGD})}
\label{alg:new}
\end{algorithm}

\subsection{Related work}\label{sec:related_work}

Here we give an overview of closely related literature, with a broader discussion in \Cref{sec:lit_rev}.

Our method is a \emph{local projection-free first-order method}. In that sense, it is naturally positioned between Projected Gradient Descent (\algname{ProjGD}) \citep{rosen1960gradient} and Frank--Wolfe (\algname{FW}) \citep{FrankWolfe1956, levitin1966constrained}.
\algname{ProjGD} performs the update
\[ 
x_{k+1}=\Proj_{\cX}\bigl(x_{k}-\gamma_k \nabla f(x_{k})\bigr),
\]
where $\Proj_{\cX}(x)\eqdef\arg\min_{z\in\cX}\|x-z\|$, and thus relies on a \emph{projection oracle}. Under standard assumptions, it matches {the standard convergence guarantees of unaccelerated} unconstrained Gradient Descent (\algname{GD}), without dependence on global geometric properties of $\cX$, such as its diameter $\diam \cX \eqdef \sup_{x,y\in\cX} \norm{x-y}$, and naturally applies to unbounded constraints. Its main drawback, however, is that projections can be computationally expensive or even intractable for many constraints.

By contrast, classical Frank--Wolfe computes
\begin{align*} 
    s_k\in {\argmin_{z\in\cX}\langle\nabla f(x_k),z\rangle}
\end{align*}
and then updates $x_k$ by moving toward $s_k$, thus relying on a \emph{global linear minimization oracle} (see \Cref{sec:fw_rev}, the original paper of \citep{FrankWolfe1956} and the modern overview by \citep{Jaggi2013}). Thus, \algname{FW} avoids projections altogether.
Classical \algname{FW} is typically analyzed for problems of the form~\eqref{eq:main} under the assumptions that $\cX$ is nonempty, convex, and compact, that $f$ is convex and differentiable on~$\cX$. A large body of work shows that, if $f$ is additionally smooth, one obtains the classical sublinear rate $f(x_K)-f(x_\star)=\cO(\nicefrac{1}{K})$, which is generally unimprovable \citep{Hazan2008Sparse, lan2013complexity, Jaggi2013}. This smoothness assumption is often expressed via the \algname{FW} curvature constant ${C_{\algnamemath{FW}}}$ (see~\eqref{eq:curvature}), rather than a global gradient Lipschitz constant \citep{Jaggi2013}.
Linear convergence generally needs additional assumptions, such as an interior solution \citep{guelat1986some}, or strong convexity and favorable constraint geometry \citep{levitin1966constrained}. Strong convexity alone does not improve the worst-case rate \citep{lan2013complexity}. Away-step \algname{FW} \citep{wolfe1970integer}, pairwise, and fully corrective variants retain the global LMO and use correction mechanisms; their linear-rate analyses typically involve polyhedral constants such as pyramidal width \citep{lacoste2015global}. \algname{BroxGD} instead restricts the oracle to a local ball. Table~\ref{tbl:main} compares the resulting guarantees.

\subsection{Closely related work}\label{concept:prior}

To the best of our knowledge, the method we call \algname{BroxGD} as well as local linear subproblems of this form were first studied by \citet{ferris1996linear} in their successive linear programming work. \citet{garber2016linearly} subsequently developed a relaxed local linear optimization oracle (LLOO) with an implementation for polytopes. We derived \algname{BroxGD} and our initial convergence results independently of these works, from the \algname{BPM} viewpoint described in Section~\ref{sec:bpm-motivation}. We became aware of Garber and Hazan while finalizing the first version of this work after all our results were already obtained, and of Ferris and Zavriev while preparing the present revision. Garber and Hazan also appear to have been unaware of Ferris and Zavriev's work, which is not cited in their paper. We fully acknowledge these prior works here. Having said that, our results are  complementary, much more comprehensive, and have a different focus.  Appendix~\ref{sec:lloo} explains in detail how our results differ from those of Ferris and Zavriev and Garber and Hazan, including the algorithms, assumptions, and convergence guarantees. We moved the discussion to the appendix because we need substantial space for the comparison.

\section{Summary of contributions}\label{concept:summary}

We develop a comprehensive convergence theory for \algname{BroxGD} in both convex and nonconvex regimes, with and without smoothness, and with and without strong convexity. See \Cref{tbl:main} for a summary of most of our results.  We now briefly describe some of these contributions, mostly cast in contrast to what is known about \algname{FW} and/or \algname{ProjGD}.

\begin{table}[!p]
\captionsetup{font=tableonecaption}
\centering
\hypersetup{linkcolor=black}
\fontsize{8}{9.5}\selectfont
\setlength{\tabcolsep}{10pt}
\begin{adjustbox}{max width=\textwidth,max totalheight=0.97\textheight}
\begin{threeparttable}

\caption{Comparison of \algnameintable{FW} and \algnameintable{BroxGD} under different regimes. {Rates are grouped by the guarantee: practical objective convergence (A), solution-dependent convex benchmarks (B), gradient-difference residuals (C), and nonconvex stationarity or projected-PL convergence (D). Shaded rate rows are presented in the main paper; the remaining rate results are in the appendix. \Cref{sec:table-proofs} follows this order, including all three Polyak-radius guarantees in group B. \algnameintable{FW} entries are upper bounds on compact domains; compare the stated metrics, which can differ across columns.} Notation:
$x_\star\in \cX_\star$,
${\Delta_k} \eqdef f(x_k)-f(x_\star)$,
$R_k \eqdef \|x_k-x_\star\|$, {${\bar x_K}=K^{-1}\sum_{k=0}^{K-1}x_k$, ${M_0}=\|\nabla f(x_\star)\|+LR_0$,}
${Z_k} \eqdef \norm{\nabla f(x_k) - \nabla f(x_\star)}$,
$c_k \eqdef L_0 + L_1 \|\nabla f(x_k)\|$,
$c_\star \eqdef L_0 + L_1 \|\nabla f(x_\star)\|$,
${\cD_{\cX}} \eqdef \diam \cX$, $R\ge R_0$ with $R>0$, ${C_0}\ge{\Delta_0}$ with ${C_0}>0$, $q=1-\frac{\mu}{4L}$,
and
$G_\gamma(x) \eqdef \frac{1}{\gamma}\rbrac{x-\Proj_{\cX}\rbrac{x-\gamma\nabla f(x)}}$.
}\label{tbl:main}
\begin{tabular}{ccccc}
    \toprule
    \bf Property & \bf Frank--Wolfe & {\bf \algnameintable{BroxGD}} & \begin{tabular}{c}\bf Practical\\\bf radius\end{tabular}  & {\bf Result}\\
    \midrule\midrule
    Oracle  & LMO for $\cX$ & LMO for $\cX \cap \mathbb{B}(x_k,t_k)$ & ---  & ---\\
    \midrule
    Projection-free & \cmark & \cmark & ---  & ---\\
    \midrule
    
    \begin{tabular}{c}{Diameter-free guarantees}\\{on possibly unbounded sets}\end{tabular} & \xmark & \cmark & ---  & ---\\
    \midrule
    
    \begin{tabular}{c}Reduces to {\algnameintable{GD}} if $\cX=\R^d$\end{tabular} & \xmark & \cmark & --- & {Prop.~\ref{prop:affine-subspace-explicit}}\\
    \midrule

\multicolumn{5}{l}{{\bfseries A. Practical objective guarantees: fixed schedules, then backtracking}}\\
\midrule
\rowcolor{mainresultshade}
\begin{tabular}{c}{Fixed horizon}\\$G$-gradient bounded \& convex $f$\end{tabular}
    & {---}\tnote{(a)}
    & \begin{tabular}{c}
      $\red\min_{0\le k\le K}{\Delta_k}
        \overset{\eqref{eq:main-horizon-bounded}}{\le}\frac{GR}{\sqrt K}$\\
      (with radii $t_k=\frac{{\varepsilon}}{\|\nabla f(x_k)\|}$, ${\varepsilon}=\frac{GR}{\sqrt K}$)
      \end{tabular}
    & \cmark\tnote{(i)} & {Thm.~\ref{thm:main-horizon}}\\
\midrule
\rowcolor{mainresultshade}
\begin{tabular}{c}{Fixed horizon}\\$L$--smooth \& convex $f$\end{tabular}
    & ${\Delta_K}=\mathcal O(\frac{L{\cD^2_{\cX}}}{K})$
    & \begin{tabular}{c}
      $\red\min_{0\le k\le K}{\Delta_k}
        \overset{\eqref{eq:main-horizon-smooth}}{\le}\frac{LR^2}{2K}$\\
      (with radii $t_k=\frac{R}{\sqrt K}$)
      \end{tabular}
    & \cmark\tnote{(i)} & {Thm.~\ref{thm:main-horizon}}\\
\midrule
\rowcolor{mainresultshade}
\begin{tabular}{c}{Predetermined geometric}\\$L$--smooth \& $\mu$--strongly convex $f$\end{tabular}
    & ${\Delta_K}=\mathcal O(\frac{L{\cD^2_{\cX}}}{K})$
    & \begin{tabular}{c}
      $\red{\Delta_K}\overset{\eqref{eq:practical-geometric-rate}}{\le}{C_0}q^K$\\
      (with radii $t_k=\sqrt{\frac{\mu {C_0}}{2L^2}}\,q^{\frac{k}{2}}$)
      \end{tabular}
    & \cmark\tnote{(i)} & {Thm.~\ref{thm:practical-geometric}}\\
\midrule
\rowcolor{mainresultshade}
\begin{tabular}{c}{Radius backtracking}\\$L$--smooth \& convex $f$\end{tabular}
    & ${\Delta_K}=\mathcal O(\frac{L{\cD^2_{\cX}}}{K})$
    & \begin{tabular}{c}
      $\red{\Delta_K}\overset{\eqref{eq:main-bt-convex}}{\le}\frac{\max\{{\Delta_0},{\frac{4LR_{\rm lev}^2}{\xi^2}}\}}{K+1}$\\
      ({$t\leftarrow\xi t$}, \Cref{alg:radius-backtracking})
      \end{tabular}
    & \cmark\tnote{(j)} & {Thm.~\ref{thm:main-backtracking}}\\
\midrule
\rowcolor{mainresultshade}
\begin{tabular}{c}{Radius backtracking}\\$L$--smooth \& $\mu$--strongly convex $f$\end{tabular}
    & ${\Delta_K}=\mathcal O(\frac{L{\cD^2_{\cX}}}{K})$
    & \begin{tabular}{c}
      $\red{\Delta_K}\overset{\eqref{eq:main-bt-strong}}{\le}{(1+\frac{\xi^2\mu}{L})^{-K}}{\Delta_0}$\\
      ({$t\leftarrow\xi t$}, \Cref{alg:radius-backtracking})
      \end{tabular}
    & \cmark\tnote{(j)} & {Thm.~\ref{thm:main-backtracking}}\\
\midrule
\multicolumn{5}{l}{{\bfseries B. Solution-dependent convex benchmarks: objective error and distance}}\\
\midrule
\begin{tabular}{c}{Polyak radius}\\
    $G$-gradient bounded \& convex $f$
    \end{tabular}
    & {---}\tnote{(a)}
    & \begin{tabular}{c}
    $\red
    \min_{0\le k\le K-1}{\Delta_k}
    \overset{\ref{thm:bounded_gradients}}{\le}
    \frac{G R_0}{\sqrt{K}}
    $ \tnote{(h)} \\
    (with radii $t_k = \frac{{\Delta_k}}{\|\nabla f(x_k)\|}$)
    \end{tabular}
     & \xmark & {Thm.~\ref{thm:bounded_gradients}}\\

\midrule
\rowcolor{mainresultshade}
\begin{tabular}{c}{Polyak radius}\\
{$G$-subgradient bounded \& convex $f$}\end{tabular}
 & {---}
 & \begin{tabular}{c}
 $\red{f({\bar x_K})-f_\star\overset{\eqref{eq:avg_iter_G_subgradient}}{\le}\frac{GR_0}{\sqrt K}}$\\
 {(with radii $t_k=\frac{\Delta_k}{\|g_k\|}$, $g_k\in\partial f(x_k)$)}
 \end{tabular}
 & \xmark & {Thm.~\ref{thm:bounded_subgradients}}\\
\midrule
\rowcolor{mainresultshade}
\begin{tabular}{c}{Polyak radius}\\
{$L$--smooth \& convex $f$}\end{tabular}
 & ${\Delta_K=\mathcal O(\frac{L{\cD^2_{\cX}}}{K})}$
 & \begin{tabular}{c}
 $\red{f({\bar x_K})-f_\star\overset{\eqref{eq:polyak_radius_smooth_convex_average}}{\le}\frac{{M_0}R_0}{\sqrt K}}$\\
 $\red{f(\bar x_K)-f_\star\overset{\eqref{eq:polyak-smooth-stationary-average}}{\le}\frac{2LR_0^2}{K}}$\tnote{(k)}\\
 $\red{f(\bar x_K)-f_\star\overset{\eqref{eq:polyak-neighborhood-average}}{\le}\frac{2LR_0^2}{K}\max\{1,\frac{R_0}{r_U}\}}$\tnote{(l)}\\
 {(with radii $t_k=\frac{\Delta_k}{\|\nabla f(x_k)\|}$)}
 \end{tabular}
 & \xmark & {Thm.~\ref{thm:polyak_radius_smooth_convex}}\\
\midrule
\rowcolor{mainresultshade}
\begin{tabular}{c}{Distance-based radius}\\
    $L$--smooth \& $\mu$--strongly convex $f$
    \end{tabular}
    & ${\Delta_K = \mathcal{O} \parens{\frac{L {\cD^2_{\cX}}}{K}}}$ 
    & \begin{tabular}{c}
    $\red
    \|x_K-x_\star\|^2
    \overset{\ref{thm:smooth-strong}}{\le}
    \left(\frac{L-\mu}{L+\mu}\right)^{2K} R_0^2$ \tnote{(h)} \\
    (with radii $t_k = \frac{2\sqrt{\mu L}}{L+\mu} R_k$)
    \end{tabular}
     & \xmark & {Thm.~\ref{thm:smooth-strong}}\\
\midrule
\multicolumn{5}{l}{{\bfseries C. Convex gradient-difference residual guarantees}}\\
\midrule
\begin{tabular}{c}{Gradient-difference radius}\\
    $L$--smooth \& convex\tnote{(b)}\quad$f$
    \end{tabular}
    & ${\Delta_K = \mathcal{O} \parens{\frac{L {\cD^2_{\cX}}}{K}}}$
    & \begin{tabular}{c}
    $\red
    \min_{0\le k\le K-1}{Z_k}^2
    \overset{\ref{thm:L-smooth}}{\le}
    \frac{L^2 R_0^2}{K}
    $ \tnote{(h)} \\
    (with radii $t_k= \frac{{Z_k}}{L}$)
    \end{tabular}
     & \xmark & {Thm.~\ref{thm:L-smooth}}\\
\midrule
\begin{tabular}{c}{Constant radius}\\$L$--smooth \& convex $f$\end{tabular}
    & ---
    & \begin{tabular}{c}
      $\red\min_{0\le k\le K}{Z_k}^2
        \overset{\eqref{eq:practical-residual-horizon}}{<}\frac{L^2R^2}{K}$\\
      (with radii $t_k=\frac{R}{\sqrt K}$)
      \end{tabular}
    & \cmark\tnote{(i)} & \begin{tabular}{c}{Thm.~\ref{thm:practical-residual}}\\{Cor.~\ref{cor:practical-residual-horizon}}\end{tabular}\\
\midrule
\begin{tabular}{c}{Gradient-difference radius}\\
    {Globally $(L_0,L_1)$--smooth \& convex $f$}
    \end{tabular}
    & \begin{tabular}{c}
    ${\Delta_K} = \mathcal{O} \parens{\frac{ \rbr{L_0 + L_1 {G_K^{\algnameintable{FW}}}}{\cD^2_{\cX}}}{K}}$
    \tnote{(c)}
    \end{tabular}
    & \begin{tabular}{c}
    $\red
    \min_{0 \le k \le K-1} {Z_k}^2 \overset{\ref{thm:asym_L0L1_rate}}{\le} \frac{ 4c_\star^2R_0^2}{K}
    $
    \tnote{(h, d)} \\
    (with radii $t_k=\frac{{Z_k}\rbr{c_k + c_\star}}{2c_kc_\star}$)
    \end{tabular}
     & \xmark & {Thm.~\ref{thm:asym_L0L1_rate}}\\
\midrule
\multicolumn{5}{l}{{\bfseries D. Beyond convexity: stationarity and projected-PL guarantees}}\\
\midrule
\begin{tabular}{c}{Gradient-mapping radius}\\
    $L$--smooth \& non-convex $f$
    \end{tabular}
    & \begin{tabular}{c}
    $\bar{g}_K = \mathcal{O}\parens{\frac{1}{\sqrt{K}}}$ \tnote{(e)}
    \end{tabular}
    & \begin{tabular}{c}
    $\red
    \min_{0 \le k \le K-1}\norm{G_{1/L}(x_k)}^2
    \overset{\ref{thm:local_lmo_nonconvex_smooth}}{\le}
    \frac{2L{\Delta_0}}{K}
    $\tnote{(h)}  \\
    (with radii $t_k = \frac{1}{L}\norm{G_{1/L}(x_k)}$)
    \end{tabular}
     & \xmark & {Thm.~\ref{thm:local_lmo_nonconvex_smooth}}\\
\midrule
\begin{tabular}{c}{Gradient-mapping radius}\\
    $L$--smooth \& projected-P\L\tnote{(f)}\quad$f$
    \end{tabular}
    & \xmark\tnote{(g)}
    & \begin{tabular}{c}
    $\red
    {\Delta_K}
    \overset{\ref{thm:local_lmo_projected_pl}}{\le}
    \left(1-\frac{\mu}{L}\right)^K {\Delta_0}
    $ \\
    (with radii $t_k = \frac{1}{L}\norm{G_{1/L}(x_k)}$)
    \end{tabular}
     & \xmark & {Thm.~\ref{thm:local_lmo_projected_pl}}\\
\bottomrule
\end{tabular}
\begin{tablenotes}
    \fontsize{7}{8}\selectfont
    \item[(a)] {Bounded gradients need not give finite curvature, so the classical curvature-based \algnameintable{FW} rate is unavailable. This does not imply nonconvergence: uniformly continuous gradients on compact domains allow \algnameintable{FW} convergence with suitable stepsizes \citep{xu2017convergence}.}
    \item[(b)] Convexity suffices for \Cref{thm:L-smooth}; neither strict convexity nor gradient injectivity is required. A zero gradient-difference residual gives a zero-radius step.
    \item[(c)] A direct {\algnameintable{FW}} guarantee under $(L_0,L_1)$--smoothness appears in \cite{vyguzov2025frank}. {Here $G_K^{\algnameintable{FW}}:=\max_{0\le k\le K}\|\nabla f(x_k^{\algnameintable{FW}})\|$ bounds gradients along the \algnameintable{FW} trajectory,} and is therefore similar in spirit to a bounded gradient assumption.
    \item[(d)] {For this row, convexity and asymmetric smoothness hold on all of $\R^d$ (\Cref{asp:l0l1}), not merely on $\cX$. The displayed bound holds for every $K\ge1$; see \Cref{thm:asym_L0L1_rate}.}
    \item[(e)] For a compact convex $\cX$, the {\algnameintable{FW}} gap is defined by $g_{\algnameintable{FW}}(x) \eqdef \max_{s\in \cX}\inner{\nabla f(x)}{x-s}$. For $L$--smooth non-convex objectives, {\algnameintable{FW}} guarantees that the minimum {\algnameintable{FW}} gap $\bar{g}_K \eqdef \min_{0 \le k \le K-1} g_{\algnameintable{FW}}(x_k) = \mathcal{O}(\frac{1}{\sqrt{K}})$ \citep{lacoste2016convergence}.  This uses a different stationarity measure and requires boundedness of $\cX$.
    It is known that $g_{{\algnameintable{FW}}}(x) \geq \gamma \norm{G_\gamma(x)}^2$.
    {This converts the cited \algnameintable{FW} bound into an $\mathcal{O}(K^{-\frac{1}{2}})$ upper bound on the squared gradient mapping, versus our $\mathcal{O}(K^{-1})$ bound with a projection-dependent radius. It is not a lower bound on \algnameintable{FW} or a separation in \algnameintable{FW}-gap complexity.}
    \item[(f)] Here, projected-P{\L} means that $\frac{1}{2}\norm{{G_{1/L}(x)}}^2 \ge \mu \rbrac{f(x)-f(x_\star)}$ for all $x \in \cX$; see \Cref{ass:projected_pl} for more details.
    \item[(g)] We are not aware of a direct Frank--Wolfe guarantee under projected-P\L\, in the above form.
    \item[(h)] Same as {\algnameintable{ProjGD}} (see \Cref{sec:pgd_positive}).
    \item[(i)] Practical: the radius uses neither unknown solution quantities nor a projection, given certified bounds and exact \algnameintable{BroxGD} access (\Cref{sec:practical-radii}). Horizon rules use $R,K$ (and $G$ in the bounded-gradient case); geometric radii use $L,\mu,{C_0}$. The nonconvex and projected-P\L\ radii require a projection and receive crosses. Dashes mean not applicable or no corresponding bound is listed. Practicality concerns the radius, not the cost of the exact oracle or observability of the residual.
    \item[(j)] {Backtracking uses a fixed input factor $\xi\in(0,1)$, $L$, the current gradient, and exact local solves.} Here $K$ counts accepted steps; rejected trials cost additional oracle calls. The convex bound assumes $\operatorname{dist}(x_k,\cX_\star)\le R_{\rm lev}$ for every accepted center; $R_{\rm lev}$ is not the initial-distance bound $R$. The strongly convex bound needs no such condition and no knowledge of $\mu$ to run the method. See \Cref{sec:practical-backtracking}.
    \item[(k)] The sharper Polyak bound additionally assumes global convexity and $L$-smoothness, and $\nabla f(x_\star)=0$.
    \item[(l)] Neighborhood version: $\nabla f(x_\star)=0$, and $f$ is convex and $L$-smooth on an open convex neighborhood $U\supset\cX$. Here $r_U>0$ satisfies $(\cX\cap\overline{\mathbb{B}}(x_\star,R_0))+\overline{\mathbb{B}}(0,r_U)\subset U$; it is not an algorithm input.
\end{tablenotes}
\end{threeparttable}
\end{adjustbox}
\end{table}

    \phantomsection\label{concept:rates}$\triangleright$ \textbf{\algname{GD}-type guarantees without global geometric dependence.}
We match the classical \algname{ProjGD} rates: $\mathcal O(1/K)$ objective error for smooth convex objectives, $\mathcal O(1/\sqrt K)$ under bounded gradients (\Cref{thm:main-horizon}), and linear iterate convergence under smoothness and strong convexity (\Cref{thm:smooth-strong}). These bounds require neither bounded constraints nor diameter or \algname{FW}-curvature constants. The bounded-subgradient result also covers nonsmooth convex objectives (\Cref{thm:bounded_subgradients}), where direct subgradient substitution in vanilla \algname{FW} can fail \citep[Example 1]{Nesterov-2017}. Objective convergence does not itself imply iterate convergence; see \citet{FW-iterates-do-not-converge} for \algname{FW} examples.

    $\triangleright$ \textbf{A new perspective on projection-free optimization.}
\algname{BroxGD} generalizes \algname{GD}: for $\cX=\R^d$ and $\nabla f(x_k)\ne0$,
\begin{equation}\label{eq:gd}
 x_{k+1}=x_k-\gamma_k\nabla f(x_k),\qquad
 \gamma_k=\frac{t_k}{\|\nabla f(x_k)\|}.
\end{equation}
On an affine subspace it is a gradient step for the restricted objective (\Cref{sec:GD-subspace}).

    \phantomsection\label{concept:extensions}$\triangleright$ \textbf{Extensions.}
The main theory covers fixed-horizon and geometric radii, Polyak radii, backtracking, and sharp strongly convex contraction. The appendix adds standard and generalized-smoothness residual bounds (\Cref{sec:table-proof-C}), nonconvex stationarity and projected-P\L\ guarantees with projection-dependent radii (\Cref{sec:non-convex}), and stochastic extensions (\Cref{sec:stochastic}). Table~\ref{tbl:main} summarizes the deterministic results; \Cref{sec:table-proof-B} collects the Polyak proofs.

    $\triangleright$ \textbf{Analysis of the oracle.}
The local oracle over $\cX\cap\mathbb{B}(x_k,t_k)$ can be cheaper or more expensive than a global LMO, depending on the constraints. \Cref{sec:closed_form} gives explicit implementations; the experiments examine their computational costs.

\input{main-theory}

\input{main-experiment}

\section{Conclusions} 
 Our \algname{BroxGD} framework may provide new insights into recently proposed LMO-based deep learning optimizers such as \algname{Muon}, \algname{Scion} and \algname{Gluon}. In particular, in \Cref{sec:muon} we show how intersecting spectral constraints with local trust-region balls leads to a damped LMO update that may be useful for neural network training. These insights are also of a preliminary nature, and are merely meant to hint at interesting future work directions. We invite the community to contribute to further development of this promising new direction. 

\begin{ack}
    This work was supported by funding from King Abdullah University of Science and Technology (KAUST): i) KAUST Baseline Research Scheme, ii) Center of Excellence for Generative AI (award no. 5940). K.G. is supported by the Google PhD Fellowship.
\end{ack}

\clearpage
\bibliographystyle{plainnat}
\bibliography{2026-22-Local-LMO}

\clearpage
\appendix
\part*{Appendix}

\tableofcontents

\clearpage
\section{BroxGD as a forward--backward method with a broximal backward step}
\label{sec:fb_brox}

In this section we explain that {\algname{BroxGD}} can be viewed as a forward--backward scheme for the composite problem
\begin{equation}
\label{eq:fb_composite_98ufd}
    \min_{x\in\R^d} \; F(x) \eqdef f(x) + g(x),
\end{equation}
where
\begin{equation}
\label{eq:g-indicator}
    g(x) \eqdef \iota_{\cX}(x)
    \;\eqdef\;
    \begin{cases}
        0, & x\in \cX,\\
        +\infty, & x\notin \cX.
    \end{cases}
\end{equation}

Problem \eqref{eq:fb_composite_98ufd} is of course equivalent to the constrained problem
\[
    \min_{x\in \cX} f(x).
\]

The point of view developed below is useful for two reasons.
First, it clarifies the relation between {\algname{BroxGD}} and classical forward--backward / proximal-gradient methods.
Second, it shows that {\algname{BroxGD}} can be interpreted as a \emph{hard-constrained} analogue of proximal gradient descent, in exactly the same spirit in which the broximal operator introduced by \citep{gruntkowska2025ball} is a hard-constrained analogue of the classical proximity operator \citep{Moreau65}.

\subsection{Classical forward--backward splitting}

Let us first recall the usual forward--backward update for \eqref{eq:fb_composite_98ufd} in the case when \(f\) is differentiable and~\(g\) is proper, closed, and convex. Given a stepsize \(\gamma_k>0\), the method takes the form
\begin{equation}
\label{eq:classical-fb}
    x_{k+1}
    =
    \prox_{\gamma_k g}\left(x_k-\gamma_k \nabla f(x_k)\right).
\end{equation}
Equivalently, \(x_{k+1}\) is the minimizer of the regularized first-order model
\begin{equation}
\label{eq:classical_fb_model_090s}
    x_{k+1}
    \in
    \argmin_{z\in\R^d}
    \left\{
        g(z)
        +
        \inner{\nabla f(x_k)}{z-x_k}
        +
        \frac{1}{2\gamma_k}\norm{z-x_k}^2
    \right\}.
\end{equation}
Thus, the forward--backward method can be interpreted as follows:

\begin{itemize}
    \item the \emph{forward step} is the linearization of \(f\) at \(x_k\),
    \item the \emph{backward step} is the minimization of this linearized model regularized by the quadratic penalty \(\frac{1}{2\gamma_k}\norm{z-x_k}^2\).
\end{itemize}

In the special case \eqref{eq:g-indicator}, we have
\[
    \prox_{\gamma_k g}(y)=\Proj_{\cX}(y),
\]
and hence \eqref{eq:classical-fb} reduces to Projected Gradient Descent:
\[
    x_{k+1}
    =
    \Proj_{\cX}\left(x_k-\gamma_k \nabla f(x_k)\right).
\]

\subsection{From prox to brox}

The broximal operator \citep{gruntkowska2025ball} replaces the quadratic penalty by a hard ball constraint.
Given a proper function \(h:\R^d\to \R\cup\{+\infty\}\), a center \(y\in\R^d\), and a radius \(t>0\), define
\begin{equation}
\label{eq:brox-definition-section}
    \brox_h^t(y)
    \;\eqdef\;
    \argmin_{z\in \mathbb{B}(y,t)} h(z),
\end{equation}
where
\[
    \mathbb{B}(y,t)\eqdef \left\{ z\in\R^d \;:\; \norm{z-y}\le t \right\}.
\]
Hence, whereas the proximity operator solves the penalized problem
\[
    \argmin_z \brac{ h(z) + \frac{1}{2\gamma}\norm{z-y}^2 },
\]
the broximal operator solves the hard-constrained counterpart
\[
    \argmin_z \brac{ h(z) \;:\; \norm{z-y}\le t }.
\]

Thus, moving from the proximity to the broximal operator should be thought of 
as replacing a \emph{soft quadratic penalty} by a \emph{hard trust-region constraint}.

\subsection{A forward--backward interpretation of BroxGD}

We now apply this idea to the composite problem \eqref{eq:fb_composite_98ufd}. At iteration \(k\), define the first-order model of \(F=f+g\) around \(x_k\) by
\begin{equation}
\label{eq:model-mk}
    m_k(z)
    \eqdef
    g(z)+\inner{\nabla f(x_k)}{z-x_k}.
\end{equation}
The natural broximal forward--backward step associated with this model is
\begin{equation}
\label{eq:fb-brox-step}
    x_{k+1}\in \brox_{m_k}^{t_k}(x_k).
\end{equation}
Using the definition \eqref{eq:brox-definition-section}, this means
\begin{equation}
\label{eq:fb-brox-step-expanded}
    x_{k+1}
    \in
    \argmin_{z\in \mathbb{B}(x_k,t_k)}
    \brac{
        g(z)+\inner{\nabla f(x_k)}{z-x_k}
    }.
\end{equation}
Substituting \(g=\iota_{\cX}\), we obtain
\begin{equation}
\label{eq:indicator-expanded}
    x_{k+1}
    \in
    \argmin_{z\in \mathbb{B}(x_k,t_k)}
    \brac{
        \iota_{\cX}(z)+\inner{\nabla f(x_k)}{z-x_k}
    }.
\end{equation}
Since \(\iota_{\cX}(z)=+\infty\) outside \(\cX\), the feasible set in \eqref{eq:indicator-expanded} collapses to \(\cX\cap \mathbb{B}(x_k,t_k)\), and hence
\begin{equation}
\label{eq:indicator-restricted}
    x_{k+1}
    \in
    \argmin_{z\in \cX\cap \mathbb{B}(x_k,t_k)}
    \inner{\nabla f(x_k)}{z-x_k}.
\end{equation}
Finally, since \(-\inner{\nabla f(x_k)}{x_k}\) is constant with respect to \(z\), \eqref{eq:indicator-restricted} is equivalent to
\begin{equation}
\label{eq:local-lmo-from-brox}
    x_{k+1}
    \in
    \argmin_{z\in \cX\cap \mathbb{B}(x_k,t_k)}
    \inner{\nabla f(x_k)}{z}.
\end{equation}
But \eqref{eq:local-lmo-from-brox} is exactly the {\algname{BroxGD}} update.

We have therefore proved the following interpretation.

\begin{proposition}[Forward--backward--brox interpretation of {\algname{BroxGD}}]
\label{prop:fb-brox-interpretation}
Consider the constrained optimization problem
\[
    \min_{x\in \cX} f(x),
\]
and rewrite it in composite form as
\[
    \min_{x\in\R^d} f(x)+\iota_{\cX}(x).
\]
Then the {\algname{BroxGD}} iteration
\[
    x_{k+1}
    \in
    \argmin_{z\in \cX\cap \mathbb{B}(x_k,t_k)}
    \inner{\nabla f(x_k)}{z}
\]
is equivalent to the forward--backward step
\[
    x_{k+1}\in \brox_{m_k}^{t_k}(x_k),
    \qquad
    m_k(z)=\iota_{\cX}(z)+\inner{\nabla f(x_k)}{z-x_k}.
\]
In other words, {\algname{BroxGD}} is obtained from classical forward--backward splitting by replacing the proximal backward step with a broximal backward step.
\end{proposition}

\begin{proof}
The equivalence follows immediately from \eqref{eq:fb-brox-step}--\eqref{eq:local-lmo-from-brox}.
\end{proof}

\subsection{Two variants of the broximal backward step}

It is important to emphasize that the interpretation above is \emph{not} the same as writing
\[
    x_{k+1} \in \brox_g^{t_k}\left(x_k-\gamma_k\nabla f(x_k)\right).
\]
Indeed, this latter expression would minimize only \(g\) over a ball centered at the forward point \(x_k-\gamma_k\nabla f(x_k)\), and hence would correspond to a direct replacement of \(\prox_{\gamma_k g}(y)\) by  \(    \brox_g^{t_k}(y)\) inside the classical formula \eqref{eq:classical-fb}. This is \emph{not} what {\algname{BroxGD}} does. Instead, {\algname{BroxGD}} keeps the center of the ball at the current iterate \(x_k\), linearizes the smooth part \(f\), adds the non-smooth term \(g\), and then applies the broximal operator to this linearized composite model:
\[
    x_{k+1}
    \in
    \brox_{\,g+\inner{\nabla f(x_k)}{\cdot-x_k}}^{t_k}(x_k).
\]
Thus, the correct interpretation is:

\begin{quote}
{\algname{BroxGD}} is a forward--backward method in which the backward step is performed not on \(g\) itself, but on the \emph{linearized composite model} \(g+\inner{\nabla f(x_k)}{\cdot-x_k}\), using a broximal operator instead of a proximal one.
\end{quote}

\subsection{Comparison with classical ProjGD}

The distinction between \algname{ProjGD} and {\algname{BroxGD}} can now be expressed in one line.

For \(g=\iota_{\cX}\), projected gradient descent solves
\begin{equation}
\label{eq:pgd-model-section}
    x_{k+1}
    \in
    \argmin_{z\in\R^d}
    \brac{
        \iota_{\cX}(z)
        +
        \inner{\nabla f(x_k)}{z-x_k}
        +
        \frac{1}{2\gamma_k}\norm{z-x_k}^2
    }.
\end{equation}
By contrast, {\algname{BroxGD}} solves
\begin{equation}
\label{eq:lmo-model-section}
    x_{k+1}
    \in
    \argmin_{z\in\R^d}
    \brac{
        \iota_{\cX}(z)
        +
        \inner{\nabla f(x_k)}{z-x_k}
        \;:\;
        \norm{z-x_k}\le t_k
    }.
\end{equation}
Thus:

\begin{itemize}
    \item \algname{ProjGD} uses a \emph{soft quadratic regularization},
    \item {\algname{BroxGD}} uses a \emph{hard trust-region constraint}.
\end{itemize}

This is exactly analogous to the relation between the proximity operator and the broximal operator.

\subsection{A monotone inclusion viewpoint}

{Throughout this subsection, assume additionally that $f$ is convex.}
There is also a useful inclusion interpretation. The optimality condition for \eqref{eq:fb_composite_98ufd} is
\[
    0\in \nabla f(x)+\partial g(x).
\]
Since \(g=\iota_{\cX}\), we have \(\partial g(x)=N_{\cX}(x)\), and therefore the problem is equivalent to the monotone inclusion
\begin{equation}
\label{eq:monotone-inclusion}
    0\in \nabla f(x)+N_{\cX}(x).
\end{equation}

For classical forward--backward, the implicit step reads
\[
    0\in \nabla f(x_k)+\partial g(x_{k+1})+\frac{1}{\gamma_k}(x_{k+1}-x_k).
\]
For {\algname{BroxGD}}, the first-order optimality condition for \eqref{eq:local-lmo-from-brox} gives
\begin{equation}
\label{eq:lmo-kkt-brox}
    0
    \in
    \nabla f(x_k)+\partial g(x_{k+1})+N_{\mathbb{B}(x_k,t_k)}(x_{k+1}).
\end{equation}
Since \(g=\iota_{\cX}\), this becomes
\begin{equation}
\label{eq:lmo-kkt-brox-indicator}
    0
    \in
    \nabla f(x_k)+N_{\cX}(x_{k+1})+N_{\mathbb{B}(x_k,t_k)}(x_{k+1}).
\end{equation}
Hence the Euclidean penalty term
\(
    \frac{1}{\gamma_k}(x_{k+1}-x_k)
\)
from proximal forward--backward is replaced by the normal cone to the trust-region ball,
\(
    N_{\mathbb{B}(x_k,t_k)}(x_{k+1}).
\)

\clearpage
\section{Related work}\label{sec:lit_rev}

\subsection{Constrained optimization in machine learning}\label{app:constrained_ml}

Constrained optimization is a fundamental tool in machine learning, used to encode prior knowledge, enforce structure, and ensure desirable properties of learned models. Rather than relying solely on unconstrained objectives with regularization, many modern applications impose explicit constraints of the form
\begin{equation*}
    \min_{x \in \cX} f(x),
\end{equation*}
where the feasible set $\cX$ captures domain-specific requirements. In this section, we provide a brief and non-exhaustive overview of key applications, focusing on representative examples rather than a complete survey.

Constraints are commonly used to impose structure on model parameters. Classical examples include sparsity constraints for feature selection \citep{Tibshirani1996Regression}, low-rank constraints in matrix completion and representation learning \citep{recht2010guaranteed}, and simplex constraints in probabilistic models such as topic models \citep{Blei2003Latent}. In deep learning, structured constraints are used to enforce sparsity, low-rank structure, or orthogonality in neural networks, improving training stability, generalization, and interpretability \citep{vorontsov2017orthogonality,bansal2018can,grontas2025pinet}.

A prominent example arises in adversarial robustness. In this setting, adversarial examples are generated by maximizing the loss over a norm-bounded perturbation set, typically an $\ell_p$ ball around a data point \citep{Goodfellow2014Explaining, madry2017towards}. The constraint set encodes the allowable perturbations, and robust training then amounts to solving a constrained min--max optimization problem, where the outer minimization seeks model parameters that perform well under worst-case perturbations \citep{madry2017towards}.

In sensitive applications, constraints are used to enforce fairness, safety, and reliability requirements \citep{zafar2017fairness, agarwal2018reductions}. For instance, one may impose constraints on prediction disparities across demographic groups, or enforce monotonicity and boundedness conditions in high-stakes decision systems. These constraints encode requirements that cannot be captured by the loss function alone and are increasingly important in real-world ML deployments.

Many ML models require parameters to lie in structured domains, leading naturally to constrained optimization formulations. A prominent example is the enforcement of orthogonality constraints on weight matrices, which can be formulated as optimization over the Stiefel manifold and has been used to stabilize training and improve generalization in deep networks \citep{AbsMahSep2008, bansal2018can}. 

Another important class of constraints arises from enforcing Lipschitz continuity \citep{gouk2021regularisation}, which is closely tied to robustness and stability. This leads to optimization problems with spectral or norm constraints on network weights, such as bounding operator norms \citep{gouk2020distance, rosca2020case}. 

In many of the settings above, the constraint set $\cX$ is structured but complex: projections may be computationally expensive, while simpler primitives (e.g., linear minimization over $\cX$) can be efficient. As a result, there is strong practical motivation for optimization methods that can exploit such structure without requiring costly projections.
This motivates the study of alternative oracle models for constrained optimization that better align with the computational structure of ML applications. In particular, methods that avoid global projections while retaining strong convergence guarantees are of significant interest.

\subsection{Projected Gradient Descent}

Projection-based first-order methods are among the most classical and widely used approaches for solving constrained optimization problems of the form \eqref{eq:main}. They combine a gradient step with a projection onto the feasible set $\cX$, typically defined via a Bregman divergence $D_\phi$, i.e.,
\begin{align*}
    \Proj_{\cX}^{\phi}(x) \in \argmin_{z \in \cX} D_{\phi}(z, x).
\end{align*}
In the Euclidean setting $\phi(x)=\tfrac{1}{2}\|x\|^2$, this reduces to the standard projection \[\Proj_{\cX}(x)=\arg\min_{z\in\cX}\|x-z\|.\] This leads to the well-known Projected Gradient Descent (\algname{ProjGD}) iteration
\begin{align*}
    x_{k+1} = \Proj_{\cX}\bigl(x_k - \gamma_k \nabla f(x_k)\bigr),
\end{align*}
where $\gamma_k>0$ is a stepsize, typically chosen as $\gamma_k \equiv \nicefrac{1}{L}$ for $L$--smooth convex objectives.

At a high level, \algname{ProjGD} performs an unconstrained gradient step followed by a correction that enforces feasibility. This simple structure has made it a foundational method in constrained optimization, with extensive theoretical and algorithmic development; see, e.g., the classical references \citet{Nocedal2006Numerical, Beck2017First, nesterov2018lectures}.

\algname{ProjGD} enjoys convergence guarantees that match those of unconstrained gradient methods: $\cO(\nicefrac{1}{\varepsilon})$ in the smooth convex case, $\cO(\log(\nicefrac{1}{\varepsilon}))$ in the smooth strongly convex case, $\cO(\nicefrac{1}{\varepsilon^2})$ in the bounded-gradient convex case, and $\cO(\nicefrac{1}{\varepsilon^2})$ for finding an $\varepsilon$-stationary point in the smooth non-convex case, measured by the projected gradient mapping $G_\gamma$. {These are standard guarantees for unaccelerated projected methods; the objective-error and stationarity criteria are distinct.}

A key implication is that constraints do not degrade the oracle complexity of the method: \algname{ProjGD} achieves these guarantees without additional cost beyond projection operations. This makes it one of the few methods that cleanly separates optimization difficulty from feasibility constraints in the oracle complexity sense.

Despite its strong theoretical properties, the practical efficiency of \algname{ProjGD} critically depends on the cost of computing the projection step. For simple constraint sets--such as norm balls, boxes, simplices, or other polyhedral structures--projections admit closed forms or can be computed efficiently \citep{Duchi2008efficient, yu2012efficient, wang2013projection, rutkowski2017closed}.
However, for many modern applications, the projection itself requires solving a nontrivial optimization problem over $\cX$. In such cases, each \algname{ProjGD} iteration may be as expensive as solving a full auxiliary problem, making the projection step the computational bottleneck. This limitation is particularly pronounced when $\cX$ is high-dimensional.

\subsection{Frank--Wolfe}\label{sec:fw_rev}

The main limitation of \algname{ProjGD} is not its convergence rate, but its reliance on efficient projections. {Its per-iteration cost can be prohibitive when projections are expensive.} This has motivated a broad line of work on projection-free methods, which replace projections with different oracles. {Their per-iteration cost depends on the relative difficulty of linear minimization and projection, and their convergence guarantees depend on the chosen variant and assumptions.}

The Frank--Wolfe (\algname{FW}) method \citep{FrankWolfe1956, levitin1966constrained}, also known as the conditional gradient method \citep{braun2022conditional}, is a canonical example of such a projection-free first-order algorithm for constrained convex optimization. Unlike projected gradient methods, which require computing projections onto the feasible set $\cX$, \algname{FW} replaces this step with a single call per iteration to a Linear Minimization Oracle (LMO). Specifically, at iteration $k$, it computes
\begin{align*}
s_k \in {\argmin_{z\in\cX}\langle\nabla f(x_k),z\rangle}
\end{align*}
and updates
\begin{align*}
x_{k+1} = x_k + \gamma_k (s_k - x_k),
\end{align*}
where $\gamma_k \in [0,1]$ is a stepsize.
The LMO solves a linear problem over $\cX$, which is often significantly cheaper than projection. As a result, \algname{FW} is particularly well suited for large-scale problems where projections can be computationally prohibitive.

Due to this favorable structure, \algname{FW} has found numerous applications across machine learning and optimization, including large-scale learning problems \citep{lacoste2013block, negiar2020stochastic}, pruning large language models \citep{roux2025don}, matrix completion \citep{garber2016faster}, optimal transport \citep{luise2019sinkhorn}, optimal experiment design \citep{hendrych2023solving}, image processing \citep{joulin2014efficient}, submodular function maximization \citep{feldman2011unified, bach2019submodular}, and quantum physics \citep{designolle2023improved}.

From a theoretical perspective, \algname{FW} is known to achieve a sublinear convergence rate in general. {For smooth convex objectives on compact convex domains, the standard guarantee is an $\mathcal{O}(L{\cD^2_{\cX}}/K)$ objective error after $K$ iterations \citep{Jaggi2013}. Oracle lower bounds concern worst-case problem classes and specified oracle models, rather than every instance or every variant \citep{lan2013complexity}.} The smoothness assumption is often expressed through the \algname{FW} curvature constant ${C_{\algnamemath{FW}}}$ \citep{Jaggi2013}, defined as
\begin{align}\label{eq:curvature}
    {C_{\algnamemath{FW}}} \eqdef \sup_{\substack{x, s \in \cX, t\in[0,1], \\ y = x + t(s-x)}} \frac{2}{t^2} \parens{f(y) - f(x) - \inp{\nabla f(x)}{y-x}},
\end{align}
rather than a global Lipschitz constant for the gradient. 
The same projection-free structure also extends to smooth non-convex objectives, but with a stationarity rather than primal-gap guarantee: over a compact convex set $\cX$, the minimum \algname{FW} gap satisfies $\bar{g}_K=\cO\parens{\nicefrac{1}{\sqrt{K}}}$ \citep{lacoste2016convergence}.

More recently, \citet{vyguzov2025frank} studied \algname{FW} under the $(L_0,L_1)$--smoothness condition and proposed an $(L_0,L_1)$-aware short-step variant. For compact convex feasible sets, they proved an $\mathcal{O}(1/K)$ primal-gap guarantee in the general convex case, together with faster rates under additional assumptions. This extends \algname{FW} analysis beyond the classical globally smooth setting. However, the result relies on a modified short-step rule adapted to the $(L_0,L_1)$ geometry, rather than the standard curvature-based framework.

Faster convergence can only be achieved under additional structural assumptions. In particular, when the constraint set $\cX$ is strongly convex, linear convergence is possible if the solution lies in the interior of $\cX$ and the objective is strongly convex, or if the unconstrained minimizer lies outside $\cX$ \citep{levitin1966constrained, wolfe1970integer}. These results highlight that the location of the optimizer plays a crucial role in the behavior of \algname{FW}.
It remained unclear whether the strong convexity of $\cX$ alone can yield faster convergence rates without additional assumptions on the optimizer's location. This question was partially resolved by \citet{GarberHazan2015}, who showed that when both the objective function and the feasible set are strongly convex, \algname{FW} achieves a rate of $\cO(\nicefrac{1}{\sqrt{\varepsilon}})$.

Recently, \citet{halbey2026lower} proved that this rate is in fact optimal: for smooth and strongly convex objectives over smooth and strongly convex sets, \algname{FW} with exact line search or short-step variants requires at least $\Omega(\nicefrac{1}{\sqrt{\varepsilon}})$ iterations. This result establishes that the uniform convergence rate of $\Omega(\nicefrac{1}{\sqrt{\varepsilon}})$ is tight with respect to $\varepsilon$, and further shows that the dependence on the optimizer's location is intrinsic rather than an artifact of earlier analyses. {The lower bound applies to the stated \algname{FW} variants and problem class; it does not rule out improvements by other algorithms or under additional assumptions.}
A concurrent work by \citet{grimmer2026lower} derived a similar lower bound of $\Omega(\nicefrac{1}{\sqrt{\varepsilon}})$ for a broader class of LMO-based methods using a zero-chain construction. Their result covers a broader algorithm class, but only applies in the high-dimensional regime and its extension to smooth constraint sets is limited to more restrictive settings where the smoothness parameter scales with the target accuracy.

Beyond strong convexity, other structural assumptions can also lead to improved rates. For example, when $\cX$ is a polytope, several \algname{FW} variants, such as away-step, pairwise, and fully corrective methods, achieve faster convergence rates under suitable conditions. We refer to \citet{braun2022conditional} for a comprehensive overview of these variants and their theoretical guarantees.

For a modern perspective on lower bounds and their historical development, see \citet{halbey2026lower, grimmer2026lower}.

\paragraph{Frank--Wolfe curvature.}
We provide a simple example showing that bounded gradients on $\cX$ do not guarantee finite Frank--Wolfe curvature.

\begin{example}[Bounded gradients do not imply finite Frank--Wolfe curvature]\label{ex:fw-curvature-infinite}
    Consider the convex function $f:\cX\to\R$ defined over $\cX = [0,1]$ by
    \begin{align*}
        f(x) = x^{3/2}.
    \end{align*}
    Then $f$ is differentiable on $\cX$ with gradient $\nabla f(x) = \nicefrac{3 \sqrt{x}}{2}$. Hence, for all $x\in\cX$,
    \begin{align*}
        \|\nabla f(x)\| \le \frac{3}{2},
    \end{align*}
    so the bounded gradient condition \eqref{eq:grad-bounded} holds with $G = \nicefrac{3}{2}$.
    
    We now show that the Frank--Wolfe curvature constant \eqref{eq:curvature} is infinite. Fix $x=0$ and $s=1$, and let $y = x + t(s-x) = t$ for $t\in[0,1]$. Since the right derivative $\nabla f(0) = 0$, we have
    \begin{align*}
        f(y) - f(x) - \inp{\nabla f(x)}{y-x}
        = t^{3/2}.
    \end{align*}
    Substituting into \eqref{eq:curvature}, we obtain
    \begin{align*}
        \frac{2}{t^2} \parens{f(y) - f(x) - \inp{\nabla f(x)}{y-x}}
        = 2 t^{-1/2}.
    \end{align*}
    Taking $t \to 0$ yields divergence, and therefore ${C_{\algnamemath{FW}}} = \infty$, showing that bounded gradients alone do not control the second-order behavior captured by the Frank--Wolfe curvature constant.
\end{example}

\subsection{Interior-point methods}\label{sec:interior_pt}

When additional structure of $\cX$ is available, one can exploit it to design more specialized algorithms. A common setting is when $\cX$ admits an explicit representation of the form \[\cX = \{x\in \R^d \;:\; \psi_i(x) \leq 0, \; i=1,\dots,m, Ax=b\}.\]
Interior-point methods approach such problems by solving a sequence of approximations to the original constrained problem, typically using Newton-type methods.

A canonical example is the barrier method \citep{nesterov1994interior, boyd2004convex}. The key idea is to replace the inequality constraints with a penalty term in the objective, yielding a problem to which Newton's method can be applied.
To motivate this, consider rewriting the problem by incorporating the constraints via indicator functions
\begin{align*}
    I_-(z) =
    \begin{cases}
        0 & z \leq 0, \\
        \infty & z > 0.
    \end{cases}
\end{align*}
This leads to the equivalent formulation
\begin{align*}
    \min_{x: Ax=b} \brac{f(x) + \sum_{i=1}^m I_-\parens{\psi_i(x)}}.
\end{align*}
While this removes the explicit constraints, the resulting objective is not differentiable, preventing the direct application of Newton’s method.
The barrier method replaces $I_-$ with the approximation
\begin{align*}
    \hat{I}_-(z) = \begin{cases}
        -\frac{1}{t} \log(-z) & z < 0, \\
        \infty & z \geq 0,
    \end{cases}
\end{align*}
where $t > 0$ controls the accuracy of the approximation. The extended-real function $\hat{I}_-$ is convex, nondecreasing and closed, and is differentiable on its effective domain $(-\infty,0)$. As $t \to \infty$, it converges pointwise to $I_-$ for $z\ne0$; at $z=0$ it remains infinite, whereas $I_-(0)=0$.
This leads to the \emph{logarithmic barrier function}
\[\phi(x) \eqdef \sum_{i=1}^m - \log (-\psi_i(x)),\]
and the associated problem
\[\min_{x: Ax=b} \brac{t f(x) + \phi(x)}. \]
Barrier methods approach \eqref{eq:main} by approximately solving a sequence of such problems for an increasing sequence of scalars $t$, typically using second-order methods. This requires access to the first- and second-order derivatives of the functions $\psi_i$ representing the constraints.

Barrier methods are only one class within interior-point methods. Another important class is primal–dual interior-point methods \citep{wright1997primal}, which update primal and dual variables by solving perturbed KKT systems. For a comprehensive overview, see \citet{boyd2004convex}.

\subsection{Splitting and primal-dual approaches}
\label{sec:splitting_rev}

Another closely related line of work studies projection-free methods for problems whose feasible region is described by several simpler constraints.
In our case, even when linear minimization over $\cX$ and over the ball $\mathbb{B}(x_k,t_k)$ are individually tractable, linear minimization over their intersection can be substantially harder. 
Splitting methods address this difficulty by decomposing a complicated feasible structure into simpler components.

In the Frank--Wolfe setting, this idea has been developed through augmented-Lagrangian and primal-dual formulations. 
For example, \citet{gidel2018frank} consider problems in which the feasible region is represented through multiple convex sets coupled by a linear constraint.
Instead of requiring a single LMO over the full intersection, their method calls LMOs over the separate component sets and enforces consistency through an augmented Lagrangian mechanism. 
Related primal-dual conditional-gradient methods, such as those of \citet{yurtsever2019conditional} and \citet{silveti2020generalized}, similarly combine LMO-based primal updates with multiplier or proximal mechanisms for handling affine and composite constraints.

These methods are different from the {\algname{BroxGD}} approach studied here. 
Splitting and primal-dual approaches avoid a difficult joint oracle and recover feasibility through consistency enforcing mechanisms, whereas our {\algname{BroxGD}} keeps the intersection $\cX\cap \mathbb{B}(x_k,t_k)$ inside the oracle and therefore enforces local feasibility directly.

\subsection{Penalty approaches}
\label{sec:penalty_rev}

Penalty methods provide a related but conceptually distinct comparison point. 
Rather than decomposing the feasible region into simpler oracle calls, penalty methods replace hard constraints by additional objective terms that measure constraint violation \citep{Nocedal2006Numerical,bertsekas1999nonlinear}. 
For example, for equality constraints $h(x)=0$ and inequality constraints $g_i(x)\le 0$, a quadratic penalty formulation considers a sequence of problems of the form
\begin{align*}
    \min_{x \in \R^d} \left\{  f(x) + \frac{\rho}{2}\norm{h(x)}^2 + \frac{\rho}{2}\sum_i \parens{\max\{0,g_i(x)\}}^2 \right\},
\end{align*}
where $\rho>0$ is a penalty parameter. 
As $\rho$ increases, infeasible points become increasingly expensive, and solutions of the penalized problems are driven toward the feasible region. 
Exact penalty methods instead use non-smooth penalties, such as $\ell_1$-type constraint violation terms, and can recover constrained solutions for sufficiently large finite penalty parameters under suitable assumptions.

From the perspective of {\algname{BroxGD}} methods, a penalty based analogue would replace the exact local feasible subproblem by a softened model. Instead of solving
\begin{align*}
    \min_{z\in \cX\cap \mathbb{B}(x_k,t_k)} \inp{\nabla f(x_k)}{z},
\end{align*}
one could consider a penalized surrogate of the form
\begin{align*}
    \min_{z \in \R^d} \left\{  \inp{\nabla f(x_k)}{z} + \rho \dist^2(z,\cX) + \rho \rbrac{\max\{0,\norm{z-x_k}-t_k\}}^2 \right\}.
\end{align*}
Such a surrogate may be easier to optimize or decompose, but it does not enforce feasibility exactly for finite $\rho$. 
Thus, while penalty methods control feasibility through violation terms, {\algname{BroxGD}} treats both $z\in\cX$ and $z\in\mathbb{B}(x_k,t_k)$ as hard constraints of the linearized subproblem.

\subsection{Local linear optimization oracles and exact BroxGD}\label{sec:lloo}
{\color{black}\mdfsetup{fontcolor=black}
\paragraph{Shared oracle and projection connection.}
Our motivation is the \algname{BPM} derivation in Section~\ref{sec:bpm-motivation} and Appendix~\ref{sec:fb_brox}. Here we place the resulting method alongside its antecedents. \citet{ferris1996linear} study linear minimization over a closed convex set intersected with a norm ball; their Euclidean subproblem is exactly \eqref{eq:method}. Their Theorem~2.10 relates its solutions to projected-gradient displacements. Thus the exact oracle and its projection connection are established ingredients.

\paragraph{Different algorithms: radius search and relaxed oracles.}
Ferris and Zavriev accept full local steps after an Armijo-type radius search using function-decrease tests, potentially solving several trial subproblems. Our analysis considers explicit radius schedules and, separately, observable backtracking. These algorithms share the local subproblem but use different radius rules and acceptance criteria.

\citet{garber2016linearly} introduced LLOO and a polytope implementation using one global LMO call plus a maintained convex decomposition. Their work also treats online, nonsmooth, and stochastic optimization; those settings are not exclusive to the present paper.
For a center $x\in\cX$, radius $t>0$, and direction $g$, the abstraction requires an output $p$ satisfying
\begin{equation}\label{eq:lloo-axioms}
 p\in\cX,\qquad \|p-x\|\le\rho t,\qquad
 \langle g,p\rangle\le\langle g,z\rangle\quad\forall z\in\cX\cap\mathbb{B}(x,t),
 \qquad \rho\ge1.
\end{equation}
When $\rho=1$, these conditions are equivalent to exact local minimization:
\begin{equation}\label{eq:llmo}
 p\in\argmin_{z\in\cX\cap\mathbb{B}(x,t)}\langle g,z\rangle.
\end{equation}
For $\rho>1$, \eqref{eq:lloo-axioms} does not require minimizing over the enlarged ball; it specifies only feasibility, a distance bound, and comparison against the original local ball.
Their strongly convex scheme uses damping $\alpha=\mu/(2L\rho^2)$ and obtains an objective bound $C_0\exp(-\mu K/(4L\rho^2))$. Their construction has $\rho=\sqrt d\,\kappa_{\cP}$, where $\kappa_{\cP}$ denotes their polytope parameter $\mu(\cP)$, distinct from objective strong convexity. One LMO call does not eliminate the decomposition bookkeeping cost.

\paragraph{Different guarantees: asymptotic convergence and explicit radius bounds.}
Under a projected-gradient residual error bound and finiteness of stationary objective values on sublevel sets, Ferris and Zavriev's Theorem~4.5 gives asymptotic Q-linear objective convergence and convergence to a stationary point, allowing nonconvex objectives. Its iteration index counts accepted updates. Our results give explicit radius-dependent guarantees in the regimes summarized in Table~\ref{tbl:main}; backtracking trial costs are accounted for separately in Appendix~\ref{sec:practical-backtracking}. We do not claim that exact local linear minimization or linear convergence with adaptive radii originates here.

A central geometric statement in our analysis of \eqref{eq:method} is the following: under the Type-I or Type-II admissibility conditions in \Cref{thm:descent}, it satisfies
\begin{equation}\label{eq:wdojch}
 \|x_{k+1}-x_\star\|^2
 \overset{\eqref{eq:thm_squared_distance_decrease}}{\le}
 \|x_k-x_\star\|^2-t_k^2.
\end{equation}
This decrease is relative to the selected minimizer for which admissibility holds and requires no polyhedral geometry. Exact local solves exist even on unbounded closed convex sets, whereas a global LMO can be unbounded below. The resulting convergence statements concern this stronger oracle model; they are not a comparison of wall-clock cost. The practical backtracking analysis uses objective descent and endpoint certificates, so we do not claim that every result in the paper follows from \eqref{eq:wdojch}.

\paragraph{Why the relaxed axioms do not imply the same distance decrease.}
The following construction isolates the missing geometric implication, already on a compact rectangle and a smooth strongly convex quadratic.
\begin{proposition}[A relaxed LLOO step can increase distance under any damping]\label{prop:lloo-counterexample}
For every $\rho>1$, there exist a compact polytope $\cX$, a smooth strongly convex $f$ with unique minimizer $x_\star$, and a Type-I/II admissible query $(x,t,\nabla f(x))$ such that an output satisfying \eqref{eq:lloo-axioms} yields
$\|(1-\alpha)x+\alpha p-x_\star\|>\|x-x_\star\|$ for every $0<\alpha\le1$.
\end{proposition}
\begin{proof}
Fix $\rho>1$, put $s=\sqrt{\rho^2-1}$, and choose $b>1/s$. Define
\begin{align}\label{eq:lloo-example-data}
 &\cX=[-1,1]\times[-s,b],\quad x=(0,0),\quad t=1,\quad x_\star=(-1,b),\notag\\
 &f(z)=\tfrac12(z-x_\star)^\top M(z-x_\star),\qquad
 M=\begin{pmatrix}1+b^2&b\\b&1\end{pmatrix}.
\end{align}
The leading principal minors of $M$ are $1+b^2>0$ and $1$, so $M$ is positive definite. Thus $f$ is globally smooth and strongly convex, with unique minimizer $x_\star\in\cX$. Moreover,
\begin{equation}\label{eq:lloo-example-gradient}
 \nabla f(x)\overset{\eqref{eq:lloo-example-data}}{=}(1,0),\qquad
 \nabla f(x_\star)=0,\qquad
 \frac{\langle\nabla f(x),x-x_\star\rangle}{\|\nabla f(x)\|}
 \overset{\eqref{eq:lloo-example-data}}{=}1=t.
\end{equation}
Hence both admissibility conditions hold. The exact solution is $z_{\rm ex}=(-1,0)$, whose squared distance decreases by $1$. In contrast, $p=(-1,-s)\in\cX$ obeys
\begin{equation}\label{eq:lloo-example-output}
 \|p-x\|=\sqrt{1+s^2}=\rho t,\qquad
 \langle\nabla f(x),p\rangle
 \overset{\eqref{eq:lloo-example-gradient}}{=}-1
 \le\langle\nabla f(x),z\rangle\quad(z\in\cX\cap\mathbb{B}(x,1)).
\end{equation}
Thus it satisfies \eqref{eq:lloo-axioms}. For $x^+=(1-\alpha)x+\alpha p$, direct expansion gives
\begin{align}\label{eq:lloo-example-increase}
 \|x^+-x_\star\|^2-\|x-x_\star\|^2
 \alignedrel{\eqref{eq:lloo-example-data},\eqref{eq:lloo-example-output}}{=}
 (1-\alpha)^2+(b+\alpha s)^2-(1+b^2)\notag\\
 &=2\alpha(bs-1)+\alpha^2\rho^2>0.
\end{align}
For $0<\alpha\le1$, the damped point is feasible by convexity. A valid oracle can return this $p$ at the specified query and exact local minimizers at other queries; the example therefore respects the oracle abstraction.
\end{proof}
The proposition rules out a universal positive distance-decrease guarantee from the relaxed axioms alone, including one obtained merely by changing constants as a function of $\rho$. It does not rule out objective convergence for damped relaxed oracles, the Garber--Hazan guarantees, or stronger conclusions for particular oracle selections. The example's conditioning and domain depend on $\rho$; it is not a fixed-condition-number complexity lower bound.

\paragraph{Damping an exact local solution is different.}
Let $z$ be the exact oracle point at an admissible query $(x,t)$, and let $x^+=(1-\alpha)x+\alpha z$, $0<\alpha\le1$. The squared-norm identity and \Cref{thm:descent} give
\begin{align}\label{eq:exact-lmoo-damping}
 \|x^+-x_\star\|^2
 &=(1-\alpha)\|x-x_\star\|^2+\alpha\|z-x_\star\|^2
   -\alpha(1-\alpha)\|z-x\|^2\notag\\
 \alignedrel{\eqref{eq:thm_squared_distance_decrease},\,\text{Thm.~\ref{thm:descent}(ii)}}{\le}
 \|x-x_\star\|^2-\alpha(2-\alpha)t^2.
\end{align}
Thus damping preserves the exact-oracle distance mechanism with a reduced coefficient. What fails in the counterexample is the geometric control supplied by an exact point inside the prescribed ball, not damping itself.
}

\subsection{Modern motivation: LMO methods in deep learning}\label{sec:muon}

Although the projection-versus-LMO dichotomy has been known for decades, recent developments in large-scale machine learning have renewed interest in LMO-based optimization. Several modern optimizers for deep learning explicitly rely on LMOs, including the \algname{Muon} optimizer \citep{jordan2024muon} and related methods such as \algname{Scion} \citep{pethick2025training} and \algname{Gluon} \citep{riabinin2025gluon}. 
\algname{Scion} is closely related to \algname{FW} through its use of LMO over a norm ball. 
In the constrained variant, the update takes the form
\begin{align*}
    {s_k\in\argmin_{z\in\cX}\langle d_k,z\rangle,\qquad x_{k+1}=(1-\gamma_k)x_k+\gamma_k s_k,}
\end{align*}
which coincides with a stochastic or momentum version of \algname{FW} over the feasible set $\cX$. In particular, when $d_k=\nabla f(x_k)$, this reduces exactly to one step of the classical \algname{FW} update over $\cX$. 
These developments have sparked a resurgence of interest in \algname{FW}-type and projection-free methods more broadly. This renewed practical relevance motivates revisiting the theoretical foundations of LMO-based methods and addressing the classical limitations of the \algname{FW} framework.

\clearpage
\section{{Shared geometric tools}}
\label{sec:shared-geometry}

\subsection{Well-posedness of the algorithm}

We now establish well-posedness of the {\algname{BroxGD}} iteration.

\begin{restatable}[Well-posedness]{theorem}{WELLPOSED}\label{thm:well-posedness}
    Let \Cref{ass:main} hold and let $f$ be differentiable. For every $x_k\in \cX$ and every $t_k\ge 0$, the set
    \[
        \cX\cap \mathbb{B}(x_k,t_k)
    \]
    is nonempty, closed, and bounded. Consequently, the optimization problem
    \[
        \min_{z\in \cX\cap \mathbb{B}(x_k,t_k)} \langle \nabla f(x_k),z\rangle
    \]
    admits at least one solution.
\end{restatable}
\begin{proof}
    Since $x_k\in \cX$ and $x_k\in \mathbb{B}(x_k,t_k)$, the feasible set is nonempty. It is closed as the intersection of two closed sets, and bounded because it is contained in the ball $\mathbb{B}(x_k,t_k)$. Since the objective is continuous and the feasible set is nonempty and compact, a minimizer exists.
\end{proof}

\subsection{Examples of feasible sets admitting a closed-form broximal gradient step} 
\label{sec:closed_form}

In this section, we list several examples of convex sets $\cX\subseteq \R^d$ for which the local linear minimization problem
\begin{equation}
    \label{eq:local-lmo}
    \min_{z\in \cX\cap \mathbb{B}(x,t)} \langle g,z\rangle
\end{equation}
admits a closed-form solution, assuming throughout that $x\in \cX$, $t>0$, and $g\neq 0$. 
Let 
$
u \eqdef \nicefrac{g}{\|g\|}.
$
Geometrically, problem \eqref{eq:local-lmo} asks for the feasible point within distance at most $t$ from $x$ whose displacement from $x$ has the largest component in the direction $-u$.
The most transparent families are affine sets, one-dimensional convex sets, Euclidean balls, and orthogonal products of these. 
These examples illustrate that the local linear minimization oracle required by \Cref{eq:method} is not merely an abstract object, but can be evaluated explicitly in a number of relevant cases.

\paragraph{1. Whole space.}
Let $\cX=\R^d$.
Then the unique solution is simply the point on the boundary of the ball in direction $-u$:
\[
    z^\star = x-tu.
\]

\paragraph{2. Singleton.}
Let $\cX=\{c\}$. 
Since $x\in \cX$, we necessarily have $x=c$, and therefore 
$
z^\star=c.
$

\paragraph{3. Affine subspace.}
Let $\cX=a+\cV$, where $\cV \subseteq \R^d$ is a linear subspace, and let $\Proj_{\cV}$ denote the orthogonal projector onto $\cV$. 
Since feasible displacements from $x$ must lie in $\cV$, the problem reduces to minimizing the linear form over the Euclidean ball in $\cV$.
Hence
\[
    z^\star= \begin{cases}
        x - t\dfrac{\Proj_{\cV} g}{\|\Proj_{\cV}g\|}, & \Proj_{\cV} g\neq 0,\\[3mm]
        x, & \Proj_{\cV} g=0.
    \end{cases}
\]
(see \Cref{sec:GD-subspace}).

\paragraph{4. Hyperplane.}
Let $\cX=\{z\in \R^d:\ a^\top z=b\}$, $a\neq 0$.
This is a special case of an affine subspace. 
If we define the component of the gradient tangent to the hyperplane by
$
    g_\top \eqdef g-\frac{a^\top g}{\|a\|^2}a,
$
then
\[
    z^\star= \begin{cases}
        x-t\,\dfrac{g_\top}{\|g_\top\|}, & g_\top\neq 0,\\[3mm]
        x, & g_\top=0.
    \end{cases}
\]

\paragraph{5. Affine line.}
Let $ \cX = a + \operatorname{span}\{v\}$, $v\neq 0$.
Denote $\hat v\eqdef v/\|v\|$. Since feasible displacements are one-dimensional, the solution is obtained by moving to one endpoint of the feasible segment:
\[
    z^\star= \begin{cases}
        x-t\,\hat v, & \langle g,\hat v\rangle>0,\\[1mm]
        x+t\,\hat v, & \langle g,\hat v\rangle<0,\\[1mm]
        \text{any point in }\cX\cap \mathbb{B}(x,t), & \langle g,\hat v\rangle=0.
    \end{cases}
\]

\paragraph{6. Ray.}
Let $\cX=a+\{\alpha v:\ \alpha\ge 0\}$, $\|v\|=1$.
Write $x=a+\alpha_x v$, $\alpha_x\ge 0$.
Then feasible points in $\cX\cap \mathbb{B}(x,t)$ are of the form $a+\alpha v$, where $\alpha\in \left[\max\{0,\alpha_x-t\},\,\alpha_x+t\right]$.
Hence
\[
    z^\star=a+\alpha^\star v,
\]
where
\[
    \alpha^\star= \begin{cases}
        \max\{0,\alpha_x-t\}, & \langle g,v\rangle>0,\\[1mm]
        \alpha_x+t, & \langle g,v\rangle<0,\\[1mm]
        \text{any feasible }\alpha, & \langle g,v\rangle=0.
    \end{cases}
\]

\paragraph{7. Line segment.}
Let $\cX=[a,b]\eqdef \{(1-\lambda)a+\lambda b:\ \lambda\in[0,1]\}$. If $a=b$, the unique feasible point and oracle solution is $z^\star=a$. Otherwise, suppose $x=(1-\lambda_x)a+\lambda_x b$, with $\lambda_x\in[0,1]$.

Then feasible points are parameterized by
$
    \lambda\in [0,1]\cap \left[\lambda_x-\frac{t}{\|b-a\|},\,\lambda_x+\frac{t}{\|b-a\|}\right].
$
Thus
\[
    z^\star=(1-\lambda^\star)a+\lambda^\star b,
\]
where $\lambda^\star$ is the left or right endpoint of this interval depending on the sign of $\langle g, b-a\rangle$:
\[
    \lambda^\star= \begin{cases}
        \text{left endpoint}, & \langle g,b-a\rangle>0,\\[1mm]
        \text{right endpoint}, & \langle g,b-a\rangle<0,\\[1mm]
        \text{any feasible }\lambda, & \langle g,b-a\rangle=0.
    \end{cases}
\]

\paragraph{8. Euclidean ball.}
Let $\cX=\mathbb{B}(c,R)\eqdef \{z\in\R^d:\ \|z-c\|\le R\}$.
The local problem is
\[
    \min_{z\in \mathbb{B}(c,R)\cap \mathbb{B}(x,t)} \langle g,z\rangle, \qquad u\eqdef \frac{g}{\|g\|}.
\]
There are three cases:
\begin{itemize}
    \item If the minimizer over \(\mathbb{B}(x,t)\) remains in \(\cX\), namely if $\|x- tu -c\|\le R$,
    then
    \[z^\star=x-tu.\]

    \item If \(x-tu\notin \cX\), but the minimizer over \(\cX\) lies in \(\mathbb{B}(x,t)\), namely if $\|c-Ru-x\|\le t$, then \[z^\star=c-Ru.\]

    \item In the remaining case, neither of the two single-ball minimizers is feasible for the other constraint. Hence, both constraints are active at the optimum, and \(z^\star\) lies on the intersection $\|z-x\|=t$ and $\|z-c\|=R$.
    In this case $x\neq c$, so the following quantities are well defined:
    \[d\eqdef c-x, \quad \rho\eqdef \|d\|, \quad e_1\eqdef \frac{d}{\rho}, \quad \alpha\eqdef \frac{\rho^2+t^2-R^2}{2\rho}, \quad s\eqdef \sqrt{t^2-\alpha^2}.\]
    Let $p\eqdef u-(u^\top e_1)e_1$ be the component of $u$ orthogonal to $e_1$. If $p\neq 0$, then
    \[z^\star=x+\alpha e_1-s\,\frac{p}{\|p\|}.\]
    If $p=0$, then every point in the sphere intersection is optimal. In the degenerate case $s=0$, this point is unique.
\end{itemize}

\paragraph{9. Slab between two parallel hyperplanes.}
{\color{black}
Let $\cX=\{z\in\R^d:\ell\le a^\top z\le r\}$, where $a\ne0$ and $\ell\le r$, and suppose $x\in\cX$.
If $g=0$ or $t=0$, choose $z^\star=x$.
Otherwise set $u=g/\|g\|$ and $z_{\rm ball}=x-tu$.
If $z_{\rm ball}\in\cX$, it is the solution. If it violates the slab, define the violated boundary by
\begin{equation}\label{eq:slab-boundary}
 b=\begin{cases}
 \ell,&a^\top z_{\rm ball}<\ell,\\
 r,&a^\top z_{\rm ball}>r.
 \end{cases}
\end{equation}
The intersection of the original ball with this boundary hyperplane has center $c$ and radius $s$, where
\begin{equation}\label{eq:slab-cross-section}
 c=x+\frac{b-a^\top x}{\|a\|^2}a,\qquad
 s=\sqrt{t^2-\frac{(b-a^\top x)^2}{\|a\|^2}},\qquad
 g_\top=g-\frac{a^\top g}{\|a\|^2}a.
\end{equation}
The square root is real because the segment from $x$ to $z_{\rm ball}$ crosses the hyperplane inside the ball. An exact oracle selection is
\begin{equation}\label{eq:slab-oracle}
 z^\star=\begin{cases}
 c-s\dfrac{g_\top}{\|g_\top\|},&g_\top\ne0,\\[2mm]
 c,&g_\top=0.
 \end{cases}
\end{equation}
To justify the active boundary, any slab-feasible point off that boundary can be moved toward $z_{\rm ball}$ until it reaches the boundary, staying in the ball and slab and strictly decreasing the linear objective. Here strict decrease follows because $z_{\rm ball}$ is the unique minimizer over the ball when $g\ne0$ and $t>0$.
For a point $z=c+v$ on the boundary, $a^\top v=0$, and orthogonality gives
\begin{equation}\label{eq:slab-orthogonality}
 \|z-x\|^2\overset{\eqref{eq:slab-cross-section}}{=}
 \frac{(b-a^\top x)^2}{\|a\|^2}+\|v\|^2,\qquad
 \langle g,z\rangle\overset{\eqref{eq:slab-cross-section}}{=}
 \langle g,c\rangle+\langle g_\top,v\rangle.
\end{equation}
Thus feasibility is equivalent to $\|v\|\le s$, and Cauchy--Schwarz proves \eqref{eq:slab-oracle}. If $g_\top=0$, every point of this cross-section is optimal; the formula selects its center. It also covers $\ell=r$ and a zero cross-section radius.
}

\subsection{Generic distance decomposition}

The following simple lemma plays an important role in the convergence proof.

\begin{lemma}[Generic decomposition]
    \label{lem:generic-decomposition}
    For any three points $x_{+}, x, z\in\R^d$,
    \[
        \|x_{+} - z\|^2 = \| x - z\|^2 - 2\langle x - x_{+}, x_{+} - z\rangle -\|x - x_{+}\|^2.
    \]
    In particular,
    \begin{equation}
        \label{eq:gen_decomp}
        \|x_{k+1}-x_{\star}\|^2 = \|x_k - x_{\star}\|^2 - 2\langle x_k-x_{k+1},\,x_{k+1}-x_{\star}\rangle - \|x_k-x_{k+1}\|^2.
    \end{equation}
\end{lemma}

\begin{proof}
Observe that $ x-z=(x-x_{+})+(x_{+}-z). $
Hence, by expanding the squared norm, we get
\[
    \|x-z\|^2 = \|(x-x_{+})+(x_{+}-z)\|^2 = \|x-x_{+}\|^2+\|x_{+}-z\|^2+2\langle x-x_{+}, x_{+}-z\rangle.
\]
Rearranging this identity yields
\[
    \|x_{+}-z\|^2 = \|x-z\|^2 - 2\langle x - x_{+}, x_{+} - z\rangle -\|x - x_{+}\|^2,
\]
which proves the first claim. The second identity follows immediately by substituting
\[
    x_{+}=x_{k+1}, \qquad x=x_k, \qquad z=x_{\star}. \qedhere
\]
\end{proof}

\subsection{Strict convexity}

Let $f:\R^d \to \R \cup \{+\infty\}$ be a proper function. 
Recall that $f$ is \emph{strictly convex} if for every pair of distinct points $x, y\in \dom f$ and every $\lambda \in (0,1)$, we have
\[
    f(\lambda x + (1-\lambda)y) < \lambda f(x) + (1-\lambda)f(y).
\]

\begin{lemma}[Strict convexity implies strict monotonicity of the gradient]
\label{lem:strict_convexity}
    Let $f:\R^d \to \R$ be differentiable and strictly convex on a convex set $\cX \subseteq \R^d$. Then 
    \[
        \langle \nabla f(x)-\nabla f(y),\,x-y\rangle>0, \qquad \nabla f(x)\neq \nabla f(y)
    \]
    for any two distinct points $x ,y\in \cX$.
\end{lemma}

\begin{proof}
    Since $f$ is convex and differentiable, the first-order inequality holds at every point:
    \begin{equation}
        \label{eq:foi-x}
        f(y) \ge f(x) + \langle \nabla f(x),\,y-x\rangle,
    \end{equation}
    and
    \begin{equation*}
        f(x) \ge f(y) + \langle \nabla f(y),\,x-y\rangle.
    \end{equation*}

    We claim that if $x \neq y$, then both inequalities are in fact strict:
    \begin{equation}
        \label{eq:strict-foi-x}
        f(y)> f(x)+\langle \nabla f(x),\,y-x\rangle,
    \end{equation}
    \begin{equation}
        \label{eq:strict-foi-y}
        f(x)> f(y)+\langle \nabla f(y),\,x-y\rangle.
    \end{equation}

    Indeed, suppose for contradiction that equality holds in \eqref{eq:foi-x}, i.e., $f(y)=f(x)+\inner{\nabla f(x)}{y - x}$.
    Let $z_t\eqdef (1-t)x+ty$, $t\in[0, 1]$
    By convexity of $f$,
    \[
        f(z_t) \le (1 - t)f(x) + t f(y).
    \]
    Using the assumed equality, the right-hand side becomes
    \[
        (1-t) f(x) + t\bigl(f(x) + \langle \nabla f(x), y - x \rangle\bigr) = f(x) + t\langle \nabla f(x), y-x\rangle.
    \]
    On the other hand, applying \eqref{eq:foi-x} with $z_t$ in place of $y$, we get
    \[
        f(z_t) \ge f(x) + \langle \nabla f(x), z_t - x\rangle = f(x) + t\langle \nabla f(x), y - x\rangle.
    \]
    Hence
    \[
        f(z_t) = f(x) + t\langle \nabla f(x), y - x\rangle = (1 - t)f(x) + t f(y) \qquad \forall t\in[0,1].
    \]
    Thus $f$ is affine on the segment $[x,y]$, which contradicts strict convexity since $x\neq y$. Therefore \eqref{eq:strict-foi-x} holds. 
    The proof of \eqref{eq:strict-foi-y} is identical. Now add \eqref{eq:strict-foi-x} and \eqref{eq:strict-foi-y}:
    \[
        f(y) + f(x) > f(x)+\langle \nabla f(x),y - x\rangle + f(y)+\langle \nabla f(y),x-y\rangle.
    \]
    Canceling $f(x)+f(y)$ from both sides yields
    \[
        0 > \langle \nabla f(x), y - x\rangle + \langle \nabla f(y), x-y\rangle.
    \]
    Rearranging gives $\langle \nabla f(x)-\nabla f(y), x - y \rangle > 0$, which proves the claim.
\end{proof}

\subsection{Reverse Cauchy--Schwarz inequality}

The Cauchy--Schwarz inequality says that 
\[\langle u, v \rangle \leq \|u\| \|v\|\]
for any vectors $u,v\in \R^d$. In general, Cauchy--Schwarz does not admit a reverse form. However, when $u$ and $v$ are special, a reverse form of the Cauchy--Schwarz inequality holds. In particular, such a reverse inequality holds for \[v = x - y, \qquad u = \nabla f(x) - \nabla f(y),\]
where $f$ is $L$--smooth and $\mu$--strongly convex, stated and proved below. It plays an important role in our convergence proof of {\algname{BroxGD}} in the smooth and strongly convex regime.

{\color{black}\mdfsetup{fontcolor=black}
\begin{lemma}
    \label{lem:strong-smooth-interpolation}
    Let $U\subseteq\R^d$ be open and convex. If $f:U\to\R$ is differentiable, $L$--smooth, and $\mu$--strongly convex, with $0<\mu\le L$, then
    \begin{equation}
        \label{eq:reverse-cs}
        \theta\,\|\nabla f(x)-\nabla f(y)\|\,\|x-y\|
        \le \langle\nabla f(x)-\nabla f(y),x-y\rangle,
        \qquad x,y\in U,
    \end{equation}
    where $\theta=2\sqrt{\mu L}/(L+\mu)$.
\end{lemma}
\begin{proof}
Fix $x,y\in U$ and put $d=x-y$. The case $d=0$ is immediate.
The compact segment $[y,x]$ has a convex tubular neighborhood whose closure lies in $U$.
Convolve $f$ there with a nonnegative smooth unit-mass kernel of sufficiently small support, obtaining $f_\varepsilon$.
Averaging translates preserves $\mu$-strong convexity and $L$-smoothness, so
\begin{equation}\label{eq:stoch-domain-hessian}
 \mu I\preceq A_\varepsilon:=\int_0^1\nabla^2f_\varepsilon(y+sd)\,ds\preceq LI,
 \qquad h_\varepsilon:=\nabla f_\varepsilon(x)-\nabla f_\varepsilon(y)=A_\varepsilon d.
\end{equation}
The spectral theorem gives
\begin{align*}
 (L+\mu)\langle h_\varepsilon,d\rangle
 -\|h_\varepsilon\|^2-\mu L\|d\|^2
 \alignedrel{\eqref{eq:stoch-domain-hessian}}{=}
 d^\top(A_\varepsilon-\mu I)(LI-A_\varepsilon)d
 \overset{\eqref{eq:stoch-domain-hessian}}{\ge}0.
\end{align*}
Since $\nabla f$ is continuous, the mollified gradients converge to it as $\varepsilon\downarrow0$. Thus
\begin{equation}\label{eq:interp-strong-smooth}
 \langle\nabla f(x)-\nabla f(y),x-y\rangle
 \ge\frac{\|\nabla f(x)-\nabla f(y)\|^2+\mu L\|x-y\|^2}{L+\mu}.
\end{equation}
Writing $h=\nabla f(x)-\nabla f(y)$, the arithmetic--geometric mean inequality yields
\begin{align*}
 \langle h,d\rangle
 \alignedrel{\eqref{eq:interp-strong-smooth}}{\ge}
 \frac{\|h\|^2+\mu L\|d\|^2}{L+\mu}
 \ge\frac{2\sqrt{\mu L}}{L+\mu}\|h\|\,\|d\|.
\end{align*}
This proves \eqref{eq:reverse-cs}, including $L=\mu$.
\end{proof}
}

\input{appendix-shared-geometry}

\input{appendix-table-proofs}

\clearpage
\section{Convergence theory of Projected Gradient Descent}
\label{sec:pgd_positive}

\subsection{Smooth convex case}
We first consider the smooth convex regime.

\begin{proposition}[Convergence for \algname{ProjGD}]
    \label{prop:pgd_rates}
    Let $\cX \subseteq \R^d$ be a nonempty closed convex set.
    Let $f:\R^d\to\R$ be convex and $L$--smooth on an open set containing $\cX$.  {Let $L>0$ and $x_0\in\cX$.} Consider the \algname{ProjGD} update
    \[
        x_{k+1}=\Proj_{\cX}\left(x_k-\gamma \nabla f(x_k)\right), \qquad 0<\gamma\le \frac1L.
    \]
    Define the projected gradient mapping by
    \[
        G_\gamma(x)\eqdef \frac1\gamma\Bigl(x-\Proj_{\cX}(x-\gamma \nabla f(x))\Bigr),
    \]
and let $x_\star \in \cX_\star\eqdef \argmin_{x\in \cX} f(x)$.    {The one-step statements (i)--(ii) hold for $k\ge0$; the subsequent rate bounds hold for integers $k\ge1$.}

    (i) \textbf{One-step descent:}
    \begin{align}
        \label{eq:non-inc-0001}
        f(x_{k+1}) \le f(x_k)-\frac{\gamma}{2}\|G_\gamma(x_k)\|^2.
    \end{align}

    (ii) \textbf{Fej\'er-type inequality:}
    \[
        \|x_{k+1}-x_\star\|^2 \le \|x_k-x_\star\|^2-2\gamma\bigl(f(x_{k+1})-f(x_\star)\bigr).
    \]

    (iii) \textbf{Function value rate:}
    \[
        f(x_k)-f(x_\star) \le \frac{\|x_0-x_\star\|^2}{2\gamma k}.
    \]

    (iv) \textbf{Best-iterate projected-gradient rate:}
    \[
        \min_{0\le i\le k-1}\|G_\gamma(x_i)\|^2 \le \frac{2\bigl(f(x_0)-f(x_\star)\bigr)}{\gamma k}.
    \]

    (v) \textbf{Distance-based projected-gradient rate for strict stepsizes:}
    if $0<\gamma< \nicefrac1L$, then
    \[
        \min_{0\le i\le k-1}\|G_\gamma(x_i)\|^2\le \frac{\|x_0-x_\star\|^2}{\gamma^2(1-\gamma L)\,k}.
    \]

    (vi) \textbf{Ergodic function value rate:}
    for the average ${\bar x_k^{\algnamemath{ProjGD}}} \eqdef \frac1k\sum_{i=1}^{k}x_i$, one has
    \[
        f({\bar x_k^{\algnamemath{ProjGD}}})-f(x_\star) \le \frac{\|x_0-x_\star\|^2}{2\gamma k}.
    \]

    (vii) \textbf{Gradient-difference bound:}
    \[
        \min_{0\le i\le k-1}\|\nabla f(x_i)-\nabla f(x_\star)\|^2 \le \frac{\|x_0-x_\star\|^2}{\gamma^2 k}.
    \]

  \end{proposition}

\begin{remark}  In the special case $\gamma=\nicefrac{1}{L}$, these bounds become
    \begin{align}
        f(x_k)-f(x_\star) &\leq \frac{L\|x_0-x_\star\|^2}{2k},\nonumber\\      
        \min_{0\le i\le k-1}\|G_{1/L}(x_i)\|^2 & \leq \frac{2L\bigl(f(x_0)-f(x_\star)\bigr)}{k}, \nonumber\\
        f({\bar x_k^{\algnamemath{ProjGD}}})-f(x_\star) &\leq \frac{L\|x_0-x_\star\|^2}{2k}, \nonumber\\
        \min_{0\le i\le k-1}\|\nabla f(x_i)-\nabla f(x_\star)\|^2 &\leq \frac{L^2\|x_0-x_\star\|^2}{k}.        \label{eq:PGD-important}
    \end{align}
\end{remark}

\begin{proof}
    Let $z_k\eqdef \Proj_{\cX}(x_k-\gamma \nabla f(x_k))$, so that $z_k=x_{k+1}$ and $G_\gamma(x_k)=\frac{1}{\gamma}(x_k-z_k)$.

    (i). By the $L$--smoothness of $f$,
    \[
        f(z_k)\le f(x_k)+\langle \nabla f(x_k),z_k-x_k\rangle+\frac{L}{2}\|z_k-x_k\|^2.
    \]
    Since $\gamma\le \nicefrac{1}{L}$, we have $\nicefrac{L}{2}\le \nicefrac{1}{2\gamma}$, hence
    \begin{align}
        \label{eq:q0000001}
        f(z_k)\le f(x_k)+\langle \nabla f(x_k),z_k-x_k\rangle+\frac{1}{2\gamma}\|z_k-x_k\|^2.
    \end{align}
    On the other hand, the optimality condition for projection gives
    \begin{align}
        \label{eq:q0000002}
        \left\langle x_k-\gamma \nabla f(x_k)-z_k, x-z_k\right\rangle\le 0,
        \qquad \forall x\in\cX.
    \end{align}
    Since $x_k\in\cX$, setting $x=x_k$ yields
    \[
        \left\langle x_k-\gamma \nabla f(x_k)-z_k,x_k-z_k\right\rangle\le 0.
    \]
    Expanding this inequality, we obtain $
        \langle \nabla f(x_k),z_k-x_k\rangle\le -\frac{1}{\gamma}\|z_k-x_k\|^2.
    $
    Substituting this into \eqref{eq:q0000001} gives
    \[
        f(z_k)\le f(x_k)-\frac{1}{\gamma}\|z_k-x_k\|^2+\frac{1}{2\gamma}\|z_k-x_k\|^2
        =f(x_k)-\frac{1}{2\gamma}\|z_k-x_k\|^2.
    \]
    Therefore,
    \[
        f(x_{k+1})=f(z_k)\le f(x_k)-\frac{1}{2\gamma}\|z_k-x_k\|^2
        =f(x_k)-\frac{\gamma}{2}\|G_\gamma(x_k)\|^2.
    \]

    (ii). 
    Setting $x=x_\star$ in \eqref{eq:q0000002}, we obtain
    \[
        \langle x_k-x_{k+1},x_\star-x_{k+1}\rangle
        \le
        \gamma \langle \nabla f(x_k),x_\star-x_{k+1}\rangle.
    \]
    Using the identity
    \[
        2\langle a-b,c-b\rangle = \|a-b\|^2+\|c-b\|^2-\|a-c\|^2
    \]
    with $a=x_k$, $b=x_{k+1}$, and $c=x_\star$, we get
    \begin{align}
        \label{eq:q0000003}
        \frac{1}{2\gamma}\Bigl(\|x_k-x_\star\|^2-\|x_{k+1}-x_\star\|^2-\|x_{k+1}-x_k\|^2\Bigr) \ge \langle \nabla f(x_k),x_{k+1}-x_\star\rangle.
    \end{align}
    By convexity, $
        f(x_\star)\ge f(x_k)+\langle \nabla f(x_k),x_\star-x_k\rangle,
    $
    hence
    \begin{align}
        \label{eq:q0000004}
        f(x_k)-f(x_\star)\le \langle \nabla f(x_k),x_k-x_\star\rangle.
    \end{align}
    Also, by $L$--smoothness,
    \begin{align}
        \label{eq:q0000005}
        f(x_{k+1})\le f(x_k)+\langle \nabla f(x_k),x_{k+1}-x_k\rangle+\frac{L}{2}\|x_{k+1}-x_k\|^2.
    \end{align}
        
    Subtracting $f(x_\star)$ from both sideds of \eqref{eq:q0000005} and combining if with \eqref{eq:q0000004},
    \begin{align}
        \label{eq:q00000041}
        f(x_{k+1})-f(x_\star) \le \langle \nabla f(x_k),x_{k+1}-x_\star\rangle+\frac{L}{2}\|x_{k+1}-x_k\|^2.
    \end{align}
    Since $\gamma\le \nicefrac{1}{L}$, we have $\frac{L}{2}\le \frac{1}{2\gamma}$, so
    \[
        f(x_{k+1})-f(x_\star) \le \langle \nabla f(x_k),x_{k+1}-x_\star\rangle+\frac{1}{2\gamma}\|x_{k+1}-x_k\|^2.
    \]
    Therefore, using \eqref{eq:q0000003}
    \[
        f(x_{k+1})-f(x_\star) \le \frac{1}{2\gamma}\Bigl(\|x_k-x_\star\|^2-\|x_{k+1}-x_\star\|^2\Bigr),
    \]
    which is equivalent to
    \begin{align}
        \label{eq:q0000007}
        \|x_{k+1}-x_\star\|^2 \le \|x_k-x_\star\|^2-2\gamma\bigl(f(x_{k+1})-f(x_\star)\bigr).
    \end{align}

    (iii). Summing \eqref{eq:q0000007} from $i=0$ to $k-1$, we get
    \[
        2\gamma \sum_{i=0}^{k-1}\bigl(f(x_{i+1})-f(x_\star)\bigr) \le \|x_0-x_\star\|^2.
    \]
    Since $f(x_i)$ is nonincreasing by \eqref{eq:non-inc-0001}, we have
    \[
        k\bigl(f(x_k)-f(x_\star)\bigr) \le \sum_{i=0}^{k-1}\bigl(f(x_{i+1})-f(x_\star)\bigr).
    \]
    Therefore, $f(x_k)-f(x_\star) \le\frac{\|x_0-x_\star\|^2}{2\gamma k}$.

    (iv). Summing the descent inequality \eqref{eq:non-inc-0001} for $i=0,\dots,k-1$, we obtain
    \[
        \frac{\gamma}{2}\sum_{i=0}^{k-1}\|G_\gamma(x_i)\|^2 \le f(x_0)-f(x_k) \le f(x_0)-f(x_\star).
    \]
    Hence \[\sum_{i=0}^{k-1}\|G_\gamma(x_i)\|^2 \le
    \frac{2\bigl(f(x_0)-f(x_\star)\bigr)}{\gamma}.\] Dividing by $k$ and noticing that the minimum is upper bounded by the average yields the desired result

    (v). 
    Combining \eqref{eq:q0000003} with \eqref{eq:q00000041}, 
    \[
        f(x_{k+1})-f(x_\star) \le
        \frac{1}{2\gamma}\Bigl(\|x_k-x_\star\|^2-\|x_{k+1}-x_\star\|^2\Bigr) - \Bigl(\frac{1}{2\gamma}-\frac{L}{2}\Bigr)\|x_{k+1}-x_k\|^2.
    \]
    Equivalently,
    \[
        \|x_{k+1}-x_\star\|^2 \le \|x_k-x_\star\|^2 - 2\gamma\bigl(f(x_{k+1})-f(x_\star)\bigr) - (1-\gamma L)\|x_{k+1}-x_k\|^2.
    \]
    If $0< \gamma< \nicefrac1L$, then $1- \gamma L>0$, and therefore
    \[
        \|x_{k+1}-x_\star\|^2 \le \|x_k-x_\star\|^2-(1-\gamma L)\|x_{k+1}-x_k\|^2.
    \]
    Since $x_{k+1}-x_k=-\gamma G_\gamma(x_k)$, we get
    \[
        \|x_{k+1}-x_\star\|^2 \le \|x_k-x_\star\|^2-(1-\gamma L)\gamma^2\|G_\gamma(x_k)\|^2.
    \]
    Summing for $i=0,\dots,k-1$, we obtain
    \[
        (1-\gamma L)\gamma^2\sum_{i=0}^{k-1}\|G_\gamma(x_i)\|^2 \le \|x_0-x_\star\|^2,
    \]
    and dividing by $k$ yields
    \[
        \min_{0\le i\le k-1}\|G_\gamma(x_i)\|^2 \le \frac{\|x_0-x_\star\|^2}{\gamma^2(1-\gamma L)k}.
    \]

    (vi). By convexity, $f({\bar x_k^{\algnamemath{ProjGD}}})\le \frac1k\sum_{i=1}^{k}f(x_i)$.
    Hence
    \[
        f({\bar x_k^{\algnamemath{ProjGD}}})-f(x_\star) \le \frac1k\sum_{i=1}^{k}\bigl(f(x_i)-f(x_\star)\bigr) = \frac1k\sum_{i=0}^{k-1}\bigl(f(x_{i+1})-f(x_\star)\bigr) \le \frac{\|x_0-x_\star\|^2}{2\gamma k}.
    \]

    (vii). Since $x_\star$ minimizes $f$ over the closed convex set $\cX$, the first-order optimality condition implies
    \[
        \langle \nabla f(x_\star),x-x_\star\rangle \ge 0,
        \qquad \forall x\in\cX.
    \]
    Equivalently, $x_\star=\Proj_{\cX}\bigl(x_\star-\gamma \nabla f(x_\star)\bigr)$.
    Using the nonexpansiveness of the projection, we get
    \[
        \|x_{i+1}-x_\star\|^2 \le \|x_i-x_\star-\gamma(\nabla f(x_i)-\nabla f(x_\star))\|^2.
    \]
    Expanding,
    \[
        \|x_{i+1}-x_\star\|^2 \le \|x_i-x_\star\|^2 -2\gamma \langle \nabla f(x_i)-\nabla f(x_\star),x_i-x_\star\rangle +\gamma^2\|\nabla f(x_i)-\nabla f(x_\star)\|^2.
    \]
    Since $f$ is convex and $L$--smooth, 
    \[
        \frac1L\|\nabla f(x)-\nabla f(y)\|^2 \le \langle \nabla f(x)-\nabla f(y),x-y\rangle, \qquad {\forall x,y\in\cX}.
    \]
    Because $\gamma \le \nicefrac1L$, it follows that
    \[
        \gamma\|\nabla f(x_i)-\nabla f(x_\star)\|^2 \le \langle \nabla f(x_i)-\nabla f(x_\star),x_i-x_\star\rangle.
    \]
    Hence,
    \[
        \|x_{i+1}-x_\star\|^2 \le \|x_i-x_\star\|^2-\gamma^2\|\nabla f(x_i)-\nabla f(x_\star)\|^2.
    \]
    Summing for $i=0,\dots,k-1$, and dividing by $k$,
    \[
        \min_{0\le i\le k-1}\|\nabla f(x_i)-\nabla f(x_\star)\|^2\le\frac{\|x_0-x_\star\|^2}{\gamma^2 k}. \qedhere
    \]
\end{proof}

\begin{remark}
    Most of the results above are known and can be found in prior literature \citep{rosen1960gradient,levitin1966constrained,bertsekas1999nonlinear,Beck2017First}.
    Among these bounds, the most standard headline results are the $\cO(\nicefrac{1}{K})$ function-value estimate in part {\em (iii)} and the $\cO(\nicefrac{1}{K})$ projected-gradient estimate in part {\em (iv)}. In constrained optimization, $\|G_\gamma(x)\|$ is the canonical stationarity measure; it vanishes if and only if $x$ is a solution of the variational inequality $0\in \nabla f(x)+N_{\cX}(x)$, or equivalently, if and only if $x$ is a minimizer of $f$ over $\cX$.
\end{remark}

\subsection{A negative result for Projected Gradient Descent}
\label{sec:pgd_negative}

We now show that the quadratic dependence on \(L\) in the gradient-difference bound
\[
\min_{0\le i\le k-1}\|\nabla f(x_i)-\nabla f(x_\star)\|^2
\le
\frac{L^2\|x_0-x_\star\|^2}{k}
\]
cannot, in general, be improved to a linear dependence on \(L\).
More precisely, we prove that there is no universal constant \(C>0\) such that, for every \(L\)--smooth convex constrained problem, the iterates of Projected Gradient Descent satisfy
\begin{equation}
\label{eq:pgd_lin}
\min_{0\le i\le k-1}\|\nabla f(x_i)-\nabla f(x_\star)\|^2
\le
\frac{C L\|x_0-x_\star\|^2}{k}.
\end{equation}

\paragraph{Construction.}
Let the feasible set be $\cX \eqdef \R\times \R_+$, and consider the family of convex quadratic functions
\[
f_\alpha(u,v)
\eqdef
\frac{\alpha}{4}u^2+\alpha uv+\alpha v^2+v
= \frac{\alpha}{4}(u+2v)^2+v,
\qquad \alpha>0.
\]
The gradient and Hessian are
\[
\nabla f_\alpha(u,v)=
\begin{pmatrix}
\frac{\alpha}{2}u+\alpha v\\[0.4em]
\alpha u+2\alpha v+1
\end{pmatrix},
\qquad
\nabla^2 f_\alpha=
\begin{pmatrix}
\alpha/2 & \alpha\\
\alpha & 2\alpha
\end{pmatrix}.
\]
The Hessian has eigenvalues \(0\) and \(\frac{5\alpha}{2}\). Hence \(f_\alpha\) is convex and \(L\)--smooth with
\[
L=\frac{5\alpha}{2}.
\]

We claim that the constrained minimizer is $x_\star=(0,0)$.
Indeed,
\[
f_\alpha(u,v)=\frac{\alpha}{4}(u+2v)^2+v.
\]
Since \(v\ge 0\) for all \((u,v)\in \cX\), and the quadratic term is always nonnegative, we have
\[
f_\alpha(u,v)\ge 0
\qquad \forall (u,v)\in\cX.
\]
Moreover, $f_\alpha(0,0)=0$, so \((0,0)\) is a constrained minimizer.

\paragraph{ProjGD with the standard stepsize.}
Let $x_0=(1,0)\in \cX$ and run Projected Gradient Descent with the standard stepsize
\[
\gamma=\frac{1}{L}=\frac{2}{5\alpha}.
\]
We first compute the gradient at \(x_0\):
\[
\nabla f_\alpha(x_0)
=
\nabla f_\alpha(1,0)
=
\begin{pmatrix}
\alpha/2\\
\alpha+1
\end{pmatrix}.
\]
Hence the unconstrained gradient step is
\[
x_0-\gamma \nabla f_\alpha(x_0)
=
\left(
1-\frac{\gamma\alpha}{2},
-\gamma(\alpha+1)
\right).
\]
Since the second component is negative, projection onto \(\cX=\R\times \R_+\) sets it to zero, and therefore
\[
x_1
=
\Proj_{\cX}\bigl(x_0-\gamma \nabla f_\alpha(x_0)\bigr)
=
\left(1-\frac{\gamma\alpha}{2},\,0\right).
\]
Using \(\gamma=\frac{2}{5\alpha}\), we get
\[
\frac{\gamma\alpha}{2}=\frac{1}{5},
\qquad\text{so}\qquad
x_1=\left(\frac{4}{5},0\right).
\]

\paragraph{Gradient differences.}
We next compute the gradients at \(x_\star\), \(x_0\), and \(x_1\).
At the constrained minimizer,
\[
\nabla f_\alpha(x_\star)
=
\nabla f_\alpha(0,0)
=
\begin{pmatrix}
0\\
1
\end{pmatrix}.
\]
Therefore
\[
\nabla f_\alpha(x_0)-\nabla f_\alpha(x_\star)
=
\begin{pmatrix}
\alpha/2\\
\alpha
\end{pmatrix},
\]
and thus
\[
\|\nabla f_\alpha(x_0)-\nabla f_\alpha(x_\star)\|^2
=
\left(\frac{\alpha}{2}\right)^2+\alpha^2
=
\frac{5\alpha^2}{4}.
\]

Now evaluate the gradient at \(x_1=(4/5,0)\):
\[
\nabla f_\alpha(x_1)
=
\begin{pmatrix}
\frac{\alpha}{2}\cdot \frac{4}{5}\\[0.3em]
\alpha\cdot \frac{4}{5}+1
\end{pmatrix}
=
\begin{pmatrix}
2\alpha/5\\[0.3em]
4\alpha/5+1
\end{pmatrix}.
\]
Hence
\[
\nabla f_\alpha(x_1)-\nabla f_\alpha(x_\star)
=
\begin{pmatrix}
2\alpha/5\\[0.3em]
4\alpha/5
\end{pmatrix},
\]
and so
\[
\|\nabla f_\alpha(x_1)-\nabla f_\alpha(x_\star)\|^2
=
\left(\frac{2\alpha}{5}\right)^2+\left(\frac{4\alpha}{5}\right)^2
=
\frac{4\alpha^2}{25}+\frac{16\alpha^2}{25}
=
\frac{4\alpha^2}{5}.
\]

Therefore, for \(k=2\),
\[
\min_{0\le i\le 1}\|\nabla f_\alpha(x_i)-\nabla f_\alpha(x_\star)\|^2
=
\frac{4\alpha^2}{5}.
\]

\paragraph{Contradiction with a hypothetical linear-in-\(L\) bound.}
Now suppose that there exists a universal constant \(C>0\) such that \eqref{eq:pgd_lin} holds for every \(L\)--smooth convex constrained problem.
Applying \eqref{eq:pgd_lin} to the present example with \(k=2\), we would obtain
\[
\min_{0\le i\le 1}\|\nabla f_\alpha(x_i)-\nabla f_\alpha(x_\star)\|^2
\le
\frac{C L\|x_0-x_\star\|^2}{2}.
\]
Since $L=\nicefrac{5\alpha}{2}$ and $\|x_0-x_\star\|^2=\|(1,0)-(0,0)\|^2=1$, the right-hand side becomes
\[
\frac{C L\|x_0-x_\star\|^2}{2}
=
\frac{C}{2}\cdot \frac{5\alpha}{2}
=
\frac{5C\alpha}{4}.
\]
But we already computed that
\[
\min_{0\le i\le 1}\|\nabla f_\alpha(x_i)-\nabla f_\alpha(x_\star)\|^2
=
\frac{4\alpha^2}{5}.
\]
Hence \eqref{eq:pgd_lin} would force
\[
\frac{4\alpha^2}{5}\le \frac{5C\alpha}{4}.
\]
Equivalently,
\[
\alpha \le \frac{25}{16}C.
\]
Therefore, if we choose
\(
\alpha>\frac{25}{16}C,
\)
then \eqref{eq:pgd_lin} fails.

Since \(C\) was arbitrary, no universal constant \(C\) can make \eqref{eq:pgd_lin} valid for all \(L\)--smooth convex constrained problems.
We have thus shown that there is no universal constant \(C>0\) such that
\[
\min_{0\le i\le k-1}\|\nabla f(x_i)-\nabla f(x_\star)\|^2
\le
\frac{C\,L\|x_0-x_\star\|^2}{k}
\]
holds for all \(L\)--smooth convex constrained problems.
Thus, the quadratic dependence on \(L\) in the bound
\[
\min_{0\le i\le k-1}\|\nabla f(x_i)-\nabla f(x_\star)\|^2
\le
\frac{L^2\|x_0-x_\star\|^2}{k}
\]
is, in general, unavoidable.

\subsection{\texorpdfstring{$(L_0,L_1)$}{(L0,L1)}--smooth convex case}

\begin{proposition}[\algname{ProjGD} under asymmetric $(L_0,L_1)$--smoothness]
    \label{cor:pgd_asym_L0L1}
    Let $\cX \subseteq \R^d$ be a nonempty closed convex set. 
    {Let $f:\R^d\to\R$ be convex on $\R^d$ and globally asymmetrically $(L_0,L_1)$--smooth as in \Cref{asp:l0l1}. Fix $x_\star\in\argmin_{x\in\cX}f(x)$, assumed to exist.}
    Let $\{x_k\}_{k\ge 0}$ be generated by the \algname{ProjGD} iteration
    \begin{align*}
        x_{k+1} = \Proj_{\cX}\left(x_k - \gamma_k \nabla f(x_k)\right),
    \end{align*}
    where $\gamma_k > 0$.
    If
    \begin{align}\label{eq:oiwnfd}
        \gamma_k \le \frac{1}{2}\left( \frac{1}{L_0 + L_1 \|\nabla f(x_k)\|} + \frac{1}{L_0 + L_1 \|\nabla f(x_\star)\|},
        \right)
    \end{align}
    then for every $K \ge 1$
    \begin{align*}
        \min_{0 \le k \le K-1} \|\nabla f (x_k) - \nabla f(x_\star)\|^2 \le \frac{\|x_0 - x_\star\|^2}{\sum_{k=0}^{K-1}\gamma_k^2}.
    \end{align*}
    In particular, for
    \begin{align}\label{eq:baosfbnO}
        \gamma_k = \frac{1}{2}\left( \frac{1}{L_0 + L_1 \|\nabla f(x_k)\|} + \frac{1}{L_0 + L_1 \|\nabla f(x_\star)\|} \right),
    \end{align}
    the iterates satisfy
    \begin{align*}
        \min_{0 \le k \le K-1} \left( \frac{\|\nabla f(x_k) - \nabla f(x_\star)\|}{L_0 + L_1 \|\nabla f(x_k)\|} + \frac{\|\nabla f(x_k) - \nabla f(x_\star)\|}{L_0 + L_1 \|\nabla f(x_\star)\|}
        \right)^2 \le \frac{4 \|x_0 - x_\star\|^2}{K}.
    \end{align*}
    Moreover, if $\sqrt{K} > L_1 \|x_0 - x_\star\|$, then
    \begin{align*}
        \min_{0 \le k \le K-1} \|\nabla f(x_k) - \nabla f(x_\star)\| \le \frac{\left(L_0 + L_1 \|\nabla f(x_\star)\|\right)\|x_0 - x_\star\|}{\sqrt{K} - L_1 \|x_0 - x_\star\|},
    \end{align*}
    and, without this restriction, for every $K\ge1$,
    \begin{align*}
        \min_{0 \le k \le K-1} \|\nabla f(x_k) - \nabla f(x_\star)\|^2 \le \frac{4\left(L_0 + L_1 \|\nabla f(x_\star)\|\right)^2 \|x_0 - x_\star\|^2}{K}.
    \end{align*}
\end{proposition}

\begin{proof}
    Denote ${{v_k}} := \nabla f(x_k) - \nabla f(x_\star)$.
    Since $x_\star$ is a minimizer of $f$ over $\cX$, for every $\gamma_k > 0$ we have $x_\star = \Proj_{\cX}\left(x_\star - \gamma_k \nabla f(x_\star)\right)$ (Refer to the proof of \Cref{prop:pgd_rates} (vii)).
    Therefore, by the nonexpansiveness of the projection,
    \begin{align*}
        \|x_{k+1} - x_\star\|^2 &= \left\| \Proj_{\cX}\left(x_k - \gamma_k \nabla f(x_k)\right) - \Proj_{\cX}\left(x_\star - \gamma_k \nabla f(x_\star)\right) \right\|^2 \\ 
        &\le \left\| x_k - x_\star - \gamma_k \left(\nabla f(x_k) - \nabla f(x_\star)\right) \right\|^2 \\
        & =\|x_k - x_\star\|^2 - 2\gamma_k \langle {{v_k}}, x_k - x_\star \rangle + \gamma_k^2 \|{{v_k}}\|^2.
    \end{align*}
    Next, by \Cref{lemma:l0l1}
    \begin{align*}
        &\frac{1}{2} \left(\frac{\|{{v_k}}\|^2}{L_0 + L_1 \|\nabla f(x_k)\|} + \frac{\|{{v_k}}\|^2}{L_0 + L_1 \|\nabla f(x_\star)\|} \right) \le \langle {{v_k}}, x_k - x_\star \rangle.
    \end{align*}
    Hence, under the choice \eqref{eq:oiwnfd},
    \begin{align*}
        \gamma_k \|{{v_k}}\|^2 \le \frac{1}{2}\left(\frac{\|{{v_k}}\|^2}{L_0 + L_1 \|\nabla f(x_k)\|}+\frac{\|{{v_k}}\|^2}{L_0 + L_1 \|\nabla f(x_\star)\|}\right)\le \langle {{v_k}}, x_k - x_\star \rangle.
    \end{align*}
    Substituting this into the previous bound yields
    \begin{align*}
        \|x_{k+1} - x_\star\|^2 \le \|x_k - x_\star\|^2 - \gamma_k^2 \|{{v_k}}\|^2.
    \end{align*}
    Summing for $k=0,1,\dots,K-1$ and dropping the nonnegative term $\|x_K - x_\star\|^2$ gives
    \begin{align*}
        \sum_{k=0}^{K-1} \gamma_k^2 \|\nabla f(x_k) - \nabla f(x_\star)\|^2
        \le \|x_0 - x_\star\|^2,
    \end{align*}
    and hence
    \begin{align*}
        &\left(\min_{0 \le k \le K-1} \|\nabla f(x_k) - \nabla f(x_\star)\|^2\right)\sum_{k=0}^{K-1}\gamma_k^2 \le \sum_{k=0}^{K-1}\gamma_k^2 \|\nabla f(x_k) - \nabla f(x_\star)\|^2 \le \|x_0 - x_\star\|^2.
    \end{align*}
    Therefore
    \begin{align*}
        &\min_{0 \le k \le K-1} \|\nabla f(x_k) - \nabla f(x_\star)\|^2 \le \frac{\|x_0 - x_\star\|^2}{\sum_{k=0}^{K-1}\gamma_k^2},
    \end{align*}
    proving the first part of the statement.
    Now, choosing $\gamma_k$ as in \eqref{eq:baosfbnO}, the bound above becomes
    \begin{align*}
        &\sum_{k=0}^{K-1} \frac{1}{4} \left( \frac{\|\nabla f(x_k) - \nabla f(x_\star)\|}{L_0 + L_1 \|\nabla f(x_k)\|} + \frac{\|\nabla f(x_k) - \nabla f(x_\star)\|}{L_0 + L_1 \|\nabla f(x_\star)\|} \right)^2 \le \|x_0 - x_\star\|^2.
    \end{align*}
    For the unrestricted residual bound, put $c_\star=L_0+L_1\|\nabla f(x_\star)\|$. The prescribed stepsize satisfies $\gamma_k\overset{\eqref{eq:baosfbnO}}{\ge}1/(2c_\star)$. Substitution into the first bound of this proposition gives $\min_{0\le k<K}\|\nabla f(x_k)-\nabla f(x_\star)\|^2\le4c_\star^2\|x_0-x_\star\|^2/K$ for every $K\ge1$. The refined denominator bound follows as in the proof of \Cref{thm:asym_L0L1_rate}.
\end{proof}

\subsection{Non-convex case}
\label{sec:pgd_nonconvex}

Projected gradient descent also admits convergence guarantees in the non-convex setting. 
The following result is obtained by specializing the proximal-gradient convergence theorem of \citet[Theorem~10.15]{Beck2017First}.
\begin{proposition}[\algname{ProjGD} in the smooth non-convex case]
    \label{prop:pgd_nonconvex}
    Let $\cX \subseteq \R^d$ be a nonempty closed convex set. 
    Let $f:\R^d \to \R$ be $L$--smooth on an open set containing $\cX$, and assume that $f_\star \eqdef \min_{x \in \cX} f(x)$ is finite. 
    {Let $L>0$ and $x_0\in\cX$.} Let $\{x_k\}_{k \ge 0}$ be generated by the \algname{ProjGD} iteration
    \begin{align*}
        x_{k+1} = \Proj_{\cX}\rbrac{x_k - \gamma \nabla f(x_k)},
    \end{align*}
    where $0 < \gamma \le \nicefrac{1}{L}$, and define the projected gradient mapping by
    \begin{align*}
        G_{\gamma}(x) = \frac{1}{\gamma}\rbrac{x - \Proj_{\cX}\rbrac{x - \gamma \nabla f(x)}}.
    \end{align*}
    Then, for every $K \ge 1$,
    \begin{align*}
        \min_{0 \le k \le K-1}\norm{G_{\gamma}(x_k)}^2 \le \frac{f(x_0)-f_\star}{\gamma\rbrac{1-\frac{L\gamma}{2}}\,K}.
    \end{align*}
    In particular, for $\gamma = \frac{1}{L}$,
    \begin{align*}
        \min_{0 \le k \le K-1}\norm{G_{1/L}(x_k)}^2 \le \frac{2L\rbrac{f(x_0)-f_\star}}{K}.
    \end{align*}
    Moreover, $G_{\gamma}(x_k) \to 0$ as $k \to \infty$. Consequently, every accumulation point of $\{x_k\}$ is a first-order stationary point.
\end{proposition}

\begin{proof}
    The result is an immediate specialization of \citet[Theorem~10.15]{Beck2017First} to the composite problem
    \begin{align*}
        F(x) = f(x) + \iota_{\cX}(x),
    \end{align*}
    where $\iota_{\cX}$ denotes the indicator function of $\cX$. In \cite{Beck2017First}'s notation, the non-smooth term is $g(x) \equiv \iota_{\cX}(x)$,
    and the proximal-gradient update with constant parameter $\bar L \ge L$ is
    \begin{align*}
        x_{k+1} = \prox_{\iota_{\cX}/\bar L}\rbrac{x_k - \frac{1}{\bar L}\nabla f(x_k)}
        = \Proj_{\cX}\rbrac{x_k - \frac{1}{\bar L}\nabla f(x_k)}.
    \end{align*}
    Choosing $\bar L = \nicefrac{1}{\gamma}$, we recover exactly the \algname{ProjGD} iteration.
    Notice that gradient mapping in \cite{Beck2017First} is
    \begin{align*}
        G_{\bar L}(x) = \bar L \rbrac{x - \Proj_{\cX}\rbrac{x - \frac{1}{\bar L}\nabla f(x)}} = \frac{1}{\gamma}\rbrac{x - \Proj_{\cX}\rbrac{x - \gamma \nabla f(x)}} = G_\gamma(x).
    \end{align*}
    Therefore, \cite[Theorem~10.15]{Beck2017First} yields $G_\gamma(x_k) \to 0$, every accumulation point of $\{x_k\}$ is stationary, and
    \begin{align*}
        \min_{0 \le k \le K-1}\norm{G_\gamma(x_k)}^2 \le \frac{f(x_0)-f_\star}{M K},
    \end{align*}
    where, in the constant-stepsize case,
    $
        M = \frac{\bar L - \nicefrac{L}{2}}{\bar L^2}.
    $
    Substituting $\bar L = \nicefrac{1}{\gamma}$ gives $M = \gamma\rbrac{1-\nicefrac{L\gamma}{2}}$,
    and hence
    \begin{align*}
        \min_{0 \le k \le K-1}\norm{G_{\gamma}(x_k)}^2 \le \frac{f(x_0)-f_\star}{\gamma\rbrac{1-\frac{L\gamma}{2}}\,K}.
    \end{align*}
    Setting $\gamma = \nicefrac{1}{L}$ yields
    \begin{align*}
        \min_{0 \le k \le K-1}\norm{G_{1/L}(x_k)}^2 \le \frac{2L\rbrac{f(x_0)-f_\star}}{K}.
    \end{align*}
\end{proof}

\clearpage
\section{{Stochastic BroxGD: componentwise geometry and convergence}}
\label{sec:stochastic}

We now extend the algorithm to the stochastic setting. In what follows, we assume that $f$ admits a finite-sum structure of the form
\begin{align*}
    f(x) \eqdef \frac{1}{n}\sum_{i=1}^n f_i(x),
\end{align*}
{where $n\ge1$. Throughout this section, $\cX\subseteq\R^d$ is nonempty, closed, and convex, and each $f_i$ is convex and differentiable on a common open convex neighborhood $U$ of $\cX$. Assume that $f$ and every $f_i$ attain their minima over $\cX$. Fix a deterministic $x_0\in\cX$, a minimizer $x_\star$ of $f$, and component minimizers $x_{\star,i}$; write ${H_\star}=\max_i\|x_{\star,i}-x_\star\|$.}

\paragraph{{Sampling and oracle conventions.}}\label{sec:stochastic-setting}
{Let $\mathcal F_k$ contain all randomness revealed before drawing $i_k$, including previous oracle selections, so that $x_k$ is $\mathcal F_k$-measurable. Require $\mathbb P(i_k=i\mid\mathcal F_k)=1/n$ for every $i$. Radii and exact oracle outputs are measurable selections based on the available information; any additional selection randomness is included in $\mathcal F_{k+1}$. If $t_k=0$, the local feasible set is the singleton $\{x_k\}$, so $x_{k+1}=x_k$. Sampling continues at subsequent iterations: a zero step for one component does not terminate the algorithm. These are the common setting and conventions for all results below.}
{Write $g_{k,i}=\nabla f_i(x_k)$ for the gradient of component $i$ at iteration $k$.} The resulting method is formalized in \Cref{alg:stochastic_new}.

\begin{algorithm}[!th]
\begin{algorithmic}[1]
\STATE \textbf{Input:} starting point $x_0\in \cX$
\FOR{$k=0,1,2,\ldots$}
    \STATE {Sample $i_k$ uniformly from $\{1,\dots,n\}$ conditionally on $\mathcal F_k$}
    \STATE Compute the gradient ${g_{k,i_k}} \eqdef \nabla f_{i_k}(x_k)$
    \STATE Choose a radius $t_k\geq0$
    \STATE Compute the next iterate as the solution of the local linear minimization problem
    \begin{align*}
        &x_{k+1}\in \argmin_{z\in \cX\cap \mathbb{B}(x_k,t_k)} \langle {g_{k,i_k}},z\rangle
    \end{align*}
\ENDFOR
\end{algorithmic}
\caption{Stochastic {\algname{BroxGD}}}
\label{alg:stochastic_new}
\end{algorithm}

{For each radius policy analyzed below, positive-radius admissibility implies $t_k\le\|x_k-x_{\star,i_k}\|$, and the same bound holds when $t_k=0$. Consequently, $\|x_{k+1}-x_\star\|\le2\|x_k-x_\star\|+{H_\star}$. Starting from deterministic $x_0$, this gives a deterministic bound at every finite horizon. Continuity on the resulting compact feasible sets ensures that the expectations used below are finite.}

\subsection{Convex objective with bounded gradients}

We first present the theory for the convex case with bounded gradients.
In the subsequent analysis, we assume that each component function $f_i$ is differentiable and admits at least one constrained minimizer.
We introduce the following notation. 
For each $i\in\{1,\dots,n\}$, let
\begin{align*}
    x_{\star,i}\in \argmin_{x\in \cX} f_i(x), \qquad f_i^\star := f_i(x_{\star,i}) = \min_{x\in \cX} f_i(x).
\end{align*}

As in the proof of \Cref{thm:bounded_gradients}, we begin with a stochastic version of the Type-I one-step descent property established in \Cref{thm:descent}.
\begin{theorem}[One-step behavior w.r.t. the sampled component minimizer: Type-I case]
    \label{thm:stochastic_component_type_I}
    {Use the common setting of \Cref{sec:stochastic-setting}.}
    Fix $k\ge 0$, and condition on one realization of $i_k$. Assume that
    \begin{align*}
        &{g_{k,i_k}}\neq 0, \qquad 0<t_k\le \frac{\inner{{g_{k,i_k}}}{x_k-x_{\star,i_k}}}{\norm{{g_{k,i_k}}}}.
    \end{align*}
    Let
    \begin{align*}
        &x_{k+1}\in \argmin_{z\in \cX\cap \mathbb{B}(x_k,t_k)} \inner{{g_{k,i_k}}}{z}.
    \end{align*}
    Then
    \begin{align*}
        &t_k\le \norm{x_k-x_{\star,i_k}}, \qquad \norm{x_{k+1}-x_k} = t_k, \qquad \norm{x_{k+1}-x_{\star,i_k}}^2 \le \norm{x_k-x_{\star,i_k}}^2 - t_k^2.
    \end{align*}
\end{theorem}

\begin{proof}
    Since $x_{\star,i_k}\in \cX$ and the admissibility condition is stated with respect to $x_{\star,i_k}$, the proof is identical to {the deterministic Type-I argument of \Cref{thm:descent}}, with $x_{\star,i_k}$ in place of $x_\star$. Indeed, the argument uses only the facts that $x_{\star,i_k}\in \cX$ and
    \begin{align*}
        &0<t_k\le \frac{\inner{{g_{k,i_k}}}{x_k-x_{\star,i_k}}}{\norm{{g_{k,i_k}}}}.
    \end{align*}
    Applying the same reasoning therefore yields
    \begin{align*}
        &t_k\le \norm{x_k-x_{\star,i_k}}, \qquad \norm{x_{k+1}-x_k} = t_k, \qquad \norm{x_{k+1}-x_{\star,i_k}}^2 \le \norm{x_k-x_{\star,i_k}}^2 - t_k^2.
    \end{align*}
\end{proof}

The following result provides a stochastic extension of \Cref{thm:bounded_gradients}.
\begin{theorem}[{Component Polyak radii with bounded gradients}]
    \label{thm:stochastic_component_rate}
    {Use the common setting of \Cref{sec:stochastic-setting}.}
    Define ${H_\star} := \max_{1\le i\le n} \norm{x_{\star,i} - x_\star}$ and let $\{x_k\}_{k\ge 0}$ be generated by \Cref{alg:stochastic_new} with
    \begin{align*}
        &t_k := \begin{cases}
            \dfrac{f_{i_k}(x_k)-f_{i_k}^\star}{\norm{{g_{k,i_k}}}}, & {g_{k,i_k}}\neq 0,\\[1.2ex]
            0, & {g_{k,i_k}}=0.
        \end{cases}
    \end{align*}
    {Assume $G>0$ and $\|\nabla f_i(x)\|\le G$ for all $i$ and $x\in\cX$. This radius requires the component optimal values $f_i^\star$.} 
    Then, for every $\eta\in \rbr{0,1}$ and every $K\ge 1$,
    \begin{align*}
        &\frac{1}{K}\sum_{k=0}^{K-1}\mathbb{E}\sbr{\rbr{f(x_k)-f(x_\star)}^2} \le \frac{G^2\norm{x_0-x_\star}^2}{\rbr{1-\eta}K} + \frac{G^2{H_\star}^2}{\eta\rbr{1-\eta}}.
    \end{align*}
    Moreover, for the averaged iterate ${\bar x_K} := \frac{1}{K}\sum_{k=0}^{K-1}x_k$,
    we have
    \begin{align*}
        &\mathbb{E}\sbr{f({\bar x_K})-f(x_\star)} \le \sqrt{\frac{G^2\norm{x_0-x_\star}^2}{\rbr{1-\eta}K} + \frac{G^2{H_\star}^2}{\eta\rbr{1-\eta}}}.
    \end{align*}
\end{theorem}

\begin{proof}
    Fix $k\ge 0$. We first verify that the chosen radius is Type-I admissible with respect to $x_{\star,i_k}$ whenever it is positive.

    {If ${g_{k,i_k}}=0$, convexity gives $f_{i_k}(x_k)=f_{i_k}^\star$; the radius and the step are zero, so the component distance-decrease inequality below holds with equality. For a nonzero sampled gradient, proceed as follows.}

    Assume now that ${g_{k,i_k}}\neq 0$. Since $x_{\star,i_k}\in \argmin_{x\in \cX} f_{i_k}(x)$, we have $f_{i_k}(x_{\star,i_k}) = f_{i_k}^\star$.
    Hence
    \begin{align*}
        &t_k = \frac{f_{i_k}(x_k)-f_{i_k}(x_{\star,i_k})}{\norm{{g_{k,i_k}}}}.
    \end{align*}
    By convexity of $f_{i_k}$,
    \begin{align*}
        &f_{i_k}(x_k)-f_{i_k}(x_{\star,i_k}) \le \inner{\nabla f_{i_k}(x_k)}{x_k-x_{\star,i_k}} = \inner{{g_{k,i_k}}}{x_k-x_{\star,i_k}}.
    \end{align*}
    Therefore
    \begin{align*}
        &0\le t_k \le \frac{\inner{{g_{k,i_k}}}{x_k-x_{\star,i_k}}}{\norm{{g_{k,i_k}}}}.
    \end{align*}
    If $t_k>0$, then \Cref{thm:stochastic_component_type_I} applies and yields
    \begin{align*}
        &\norm{x_{k+1}-x_{\star,i_k}}^2 \le \norm{x_k-x_{\star,i_k}}^2 - t_k^2.
    \end{align*}
    If $t_k=0$, then $x_{k+1}=x_k$, and the same inequality still holds. Thus, in all cases,
    \begin{align*}
        &\norm{x_{k+1}-x_{\star,i_k}}^2 \le \norm{x_k-x_{\star,i_k}}^2 - t_k^2.
    \end{align*}

    We now rewrite this inequality with respect to the global minimizer $x_\star$. Expanding both sides gives
    \begin{align*}
        &\norm{x_{k+1}-x_\star}^2 - 2\inner{x_{k+1}-x_\star}{x_{\star,i_k}-x_\star} + \norm{x_{\star,i_k}-x_\star}^2 \\
        &\qquad \le \norm{x_k-x_\star}^2 - 2\inner{x_k-x_\star}{x_{\star,i_k}-x_\star} + \norm{x_{\star,i_k}-x_\star}^2 - t_k^2.
    \end{align*}
    After cancellation, we obtain
    \begin{align*}
        &\norm{x_{k+1}-x_\star}^2 \le \norm{x_k-x_\star}^2 - t_k^2 + 2\inner{x_{k+1}-x_k}{x_{\star,i_k}-x_\star}.
    \end{align*}
    Since $\norm{x_{k+1}-x_k} = t_k$ when $t_k>0$, and trivially also when $t_k=0$, we have $\norm{x_{k+1}-x_k} = t_k$ in all cases. Therefore, by Cauchy-Schwarz and the definition of ${H_\star}$,
    \begin{align*}
        &\norm{x_{k+1}-x_\star}^2 \le \norm{x_k-x_\star}^2 - t_k^2 + 2{H_\star}t_k.
    \end{align*}
    Fix any $\eta\in \rbr{0,1}$. By Young's inequality,
    \begin{align*}
        &2{H_\star}t_k \le \eta t_k^2 + \frac{{H_\star}^2}{\eta}.
    \end{align*}
    Hence
    \begin{align*}
        &\norm{x_{k+1}-x_\star}^2 \le \norm{x_k-x_\star}^2 - \rbr{1-\eta}t_k^2 + \frac{{H_\star}^2}{\eta}.
    \end{align*}

    {Taking conditional expectation with respect to the filtration $\mathcal F_k$ specified in \Cref{sec:stochastic-setting}, we get}
    \begin{align*}
        &\mathbb{E}\sbr{\norm{x_{k+1}-x_\star}^2 \mid \mathcal{F}_k} \le \norm{x_k-x_\star}^2 - \rbr{1-\eta}\mathbb{E}\sbr{t_k^2 \mid \mathcal{F}_k} + \frac{{H_\star}^2}{\eta}.
    \end{align*}

    We next lower bound $\mathbb{E}\sbr{t_k^2 \mid \mathcal{F}_k}$. If ${g_{k,i_k}}\neq 0$, then by the definition of $t_k$ and the gradient bound,
    \begin{align*}
        &t_k^2 = \frac{\rbr{f_{i_k}(x_k)-f_{i_k}^\star}^2}{\norm{{g_{k,i_k}}}^2} \ge \frac{\rbr{f_{i_k}(x_k)-f_{i_k}^\star}^2}{G^2}.
    \end{align*}
    If ${g_{k,i_k}}=0$, then $\nabla f_{i_k}(x_k)=0$. Since $f_{i_k}$ is convex and differentiable, $x_k$ is a global minimizer of $f_{i_k}$ over $\cX$, and hence
    \begin{align*}
        &f_{i_k}(x_k)=f_{i_k}^\star.
    \end{align*}
    Therefore, in this case as well,
    \begin{align*}
        &t_k^2 = 0 = \frac{\rbr{f_{i_k}(x_k)-f_{i_k}^\star}^2}{G^2}.
    \end{align*}
    Thus, in all cases,
    \begin{align*}
        &t_k^2 \ge \frac{\rbr{f_{i_k}(x_k)-f_{i_k}^\star}^2}{G^2}.
    \end{align*}
    Taking conditional expectation yields
    \begin{align*}
        &\mathbb{E}\sbr{t_k^2 \mid \mathcal{F}_k} \ge \frac{\mathbb{E}\sbr{\rbr{f_{i_k}(x_k)-f_{i_k}^\star}^2 \mid \mathcal{F}_k}}{G^2}.
    \end{align*}

    Since $i_k$ is sampled uniformly and independently of the past,
    \begin{align*}
        &\mathbb{E}\sbr{f_{i_k}(x_k)-f_{i_k}^\star \mid \mathcal{F}_k} = \frac{1}{n}\sum_{i=1}^n \rbr{f_i(x_k)-f_i^\star}.
    \end{align*}
    By Jensen's inequality for the convex map $u\mapsto u^2$, we obtain
    \begin{align*}
        &\mathbb{E}\sbr{\rbr{f_{i_k}(x_k)-f_{i_k}^\star}^2 \mid \mathcal{F}_k} \ge \rbr{\frac{1}{n}\sum_{i=1}^n \rbr{f_i(x_k)-f_i^\star}}^2.
    \end{align*}
    Since $f_i^\star \le f_i(x_\star)$ for every $i$, we have
    \begin{align*}
        &\frac{1}{n}\sum_{i=1}^n \rbr{f_i(x_k)-f_i^\star} \ge \frac{1}{n}\sum_{i=1}^n \rbr{f_i(x_k)-f_i(x_\star)} = f(x_k)-f(x_\star).
    \end{align*}
    Moreover, since $x_\star$ minimizes $f$ over $\cX$, we have $f(x_k)-f(x_\star)\ge 0$. Therefore
    \begin{align*}
        &\mathbb{E}\sbr{\rbr{f_{i_k}(x_k)-f_{i_k}^\star}^2 \mid \mathcal{F}_k} \ge \rbr{f(x_k)-f(x_\star)}^2.
    \end{align*}
    Consequently,
    \begin{align*}
        &\mathbb{E}\sbr{t_k^2 \mid \mathcal{F}_k} \ge \frac{\rbr{f(x_k)-f(x_\star)}^2}{G^2}.
    \end{align*}

    Substituting this into the previous conditional inequality yields
    \begin{align*}
        &\mathbb{E}\sbr{\norm{x_{k+1}-x_\star}^2 \mid \mathcal{F}_k} \le \norm{x_k-x_\star}^2 - \frac{1-\eta}{G^2}\rbr{f(x_k)-f(x_\star)}^2 + \frac{{H_\star}^2}{\eta}.
    \end{align*}
    Taking expectation again and using the tower property, we get
    \begin{align*}
        &\mathbb{E}\sbr{\norm{x_{k+1}-x_\star}^2} \le \mathbb{E}\sbr{\norm{x_k-x_\star}^2} - \frac{1-\eta}{G^2}\mathbb{E}\sbr{\rbr{f(x_k)-f(x_\star)}^2} + \frac{{H_\star}^2}{\eta}.
    \end{align*}

    Summing from $k=0$ to $K-1$, we obtain
    \begin{align*}
        &\mathbb{E}\sbr{\norm{x_K-x_\star}^2} \le \norm{x_0-x_\star}^2 - \frac{1-\eta}{G^2}\sum_{k=0}^{K-1}\mathbb{E}\sbr{\rbr{f(x_k)-f(x_\star)}^2} + \frac{K{H_\star}^2}{\eta}.
    \end{align*}
    Since the left-hand side is nonnegative,
    \begin{align*}
        &\frac{1-\eta}{G^2}\sum_{k=0}^{K-1}\mathbb{E}\sbr{\rbr{f(x_k)-f(x_\star)}^2} \le \norm{x_0-x_\star}^2 + \frac{K{H_\star}^2}{\eta}.
    \end{align*}
    Dividing by $K$ yields
    \begin{align}
        \label{eq:taggg}
        &\frac{1}{K}\sum_{k=0}^{K-1}\mathbb{E}\sbr{\rbr{f(x_k)-f(x_\star)}^2} \le \frac{G^2\norm{x_0-x_\star}^2}{\rbr{1-\eta}K} + \frac{G^2{H_\star}^2}{\eta\rbr{1-\eta}}.
    \end{align}

    Now define ${\bar x_K} := \frac{1}{K}\sum_{k=0}^{K-1}x_k$.
    By convexity of $f$,
    \begin{align*}
        &f({\bar x_K})-f(x_\star) \le \frac{1}{K}\sum_{k=0}^{K-1}\rbr{f(x_k)-f(x_\star)}.
    \end{align*}
    Applying Jensen's inequality to the convex map $u\mapsto u^2$, we obtain
    \begin{align*}
        &\rbr{f({\bar x_K})-f(x_\star)}^2 \le \frac{1}{K}\sum_{k=0}^{K-1}\rbr{f(x_k)-f(x_\star)}^2.
    \end{align*}
    Taking expectation and using \eqref{eq:taggg}, we conclude that
    \begin{align*}
        &\mathbb{E}\sbr{\rbr{f({\bar x_K})-f(x_\star)}^2} {\overset{\eqref{eq:taggg}}{\le}} \frac{G^2\norm{x_0-x_\star}^2}{\rbr{1-\eta}K} + \frac{G^2{H_\star}^2}{\eta\rbr{1-\eta}}.
    \end{align*}
    Finally, by Cauchy-Schwarz,
    \begin{align*}
        &\mathbb{E}\sbr{f({\bar x_K})-f(x_\star)} \le \sqrt{\mathbb{E}\sbr{\rbr{f({\bar x_K})-f(x_\star)}^2}} \le \sqrt{\frac{G^2\norm{x_0-x_\star}^2}{\rbr{1-\eta}K} + \frac{G^2{H_\star}^2}{\eta\rbr{1-\eta}}}.
    \end{align*}
    This completes the proof.
\end{proof}

\begin{remark}[Sanity check]
    In the single-client setting, we have $x_{\star,i}=x_\star$ and hence ${H_\star}=0$. Therefore, the additional neighborhood term vanishes, and we recover the result of \Cref{thm:bounded_gradients} up to a constant factor, depending on the choice of $\eta$ when applying Young's inequality. 
    
    The same conclusion holds in the interpolation regime, where there exists a common constrained minimizer $x_\star$ of all functions $f_i$. In this case, we may take $x_{\star,i}=x_\star$ for all $i$, and hence ${H_\star}=0$ again. Thus, the additional neighborhood term disappears, and the stochastic result reduces to the deterministic bounded-gradient result up to the same constant-factor difference.
\end{remark}

\subsection{Smooth objectives and Type-II admissibility}

For smooth objectives, we will rely on the stochastic extension of Type II admissibility condition in \Cref{thm:descent}.
\begin{theorem}[One-step behavior of stochastic {\algname{BroxGD}}: Type-II case]
    \label{thm:stochastic_component_typeII}
    {Use the common setting of \Cref{sec:stochastic-setting}.}
    Let $\{x_k\}_{k\ge 0}\subseteq \cX$ be generated by
    \begin{align*}
        &x_{k+1} \in \argmin_{z\in \cX\cap{\mathbb{B}}(x_k,t_k)} \inner{\nabla f_{i_k}(x_k)}{z},
    \end{align*}
    where $i_k\in\{1,\dots,n\}$. Assume that
    \begin{align*}
        &\nabla f_{i_k}(x_k)\neq \nabla f_{i_k}(x_{\star,i_k}), \quad 0 < t_k \le \frac{\inner{\nabla f_{i_k}(x_k)-\nabla f_{i_k}(x_{\star,i_k})}{x_k-x_{\star,i_k}}}{\norm{\nabla f_{i_k}(x_k)-\nabla f_{i_k}(x_{\star,i_k})}}.
    \end{align*}
    Then\footnote{Consequently,
    \begin{align*}
        &\norm{x_{k+1}-x_\star}^2 \le \norm{x_k-x_\star}^2 - t_k^2 + 2{H_\star} t_k,
    \end{align*}
    where ${H_\star} = \max_{1\le i\le n} \norm{x_{\star,i}-x_\star}$}
    \begin{align*}
        &\norm{x_{k+1}-x_{\star,i_k}}^2 \le \norm{x_k-x_{\star,i_k}}^2 - t_k^2.
    \end{align*}
\end{theorem}

\begin{proof}
    Fix $k\ge 0$. Since $x_{\star, i_k}\in\argmin_{x\in\cX} f_{i_k}(x)$, {\Cref{thm:descent}(iii)} applied to the component objective $f_{i_k}$ yields
    \begin{align*}
        &\norm{x_{k+1}-x_{\star,i_k}}^2 \le \norm{x_k-x_{\star,i_k }}^2 - t_k^2.
    \end{align*}
    This proves the main claim. For the consequence,
    \begin{align*}
        \norm{x_{k+1}-x_\star}^2
        &= \norm{x_{k+1}-x_{\star,i_k} + x_{\star,i_k}-x_\star}^2 \\
        &= \norm{x_{k+1}-x_{\star,i_k}}^2 + 2\inner{x_{k+1}-x_{\star,i_k}}{x_{\star,i_k}-x_\star} + \norm{x_{\star,i_k}-x_\star}^2 \\
        &\le \norm{x_k-x_{\star,i_k}}^2 - t_k^2 + 2\inner{x_{k+1}-x_{\star,i_k}}{x_{\star,i_k}-x_\star} + \norm{x_{\star,i_k}-x_\star}^2 \\
        &= \norm{x_k-x_\star}^2 - t_k^2 + 2\inner{x_{k+1}-x_k}{x_{\star,i_k}-x_\star}.
    \end{align*}
    Since $x_{k+1} \in \mathbb{B}(x_k,t_k)$, we have $\norm{x_{k+1}-x_k}\le t_k$.
    Therefore,
    \begin{align*}
        &2\inner{x_{k+1}-x_k}{x_{\star,i_k}-x_\star} \le 2\norm{x_{k+1}-x_k}\norm{x_{\star,i_k}-x_\star} \le 2{H_\star} t_k.
    \end{align*}
    Substituting this into the previous inequality gives
    \begin{align*}
        &\norm{x_{k+1}-x_\star}^2 \le \norm{x_k-x_\star}^2 - t_k^2 + 2{H_\star} t_k.
    \end{align*}
\end{proof}

\paragraph{Strongly convex case.}
Under the additional assumption of strong convexity, \Cref{thm:stochastic_component_typeII} yields the following result.
\begin{theorem}[{Linear convergence to a neighborhood}]
    \label{thm:stochastic_component_smooth_strongly_convex}
    {Use the common setting of \Cref{sec:stochastic-setting}, and let the iterates be generated by \Cref{alg:stochastic_new}.}
    Assume for each $i\in\{1,\dots,n\}$, the function $f_i$ is $\mu_i$-strongly convex and $L_i$-smooth on an open convex set containing $\cX$, with $0<\mu_i\le L_i$.
    Define
    \begin{align*}
        &\theta_i := \frac{2\sqrt{\mu_i L_i}}{L_i+\mu_i}, \qquad \rho_i := \frac{L_i-\mu_i}{L_i+\mu_i} = \sqrt{1-\theta_i^2}, \qquad \bar\rho := \frac{1}{n}\sum_{i=1}^n \rho_i.
    \end{align*}
    Choose $t_k := \theta_{i_k}\norm{x_k-x_{\star,i_k}}$. {At $k=0$, interpret $\bar\rho^0=1$, including when $\bar\rho=0$.}
    Then, for every $k\ge 0$,
    \begin{align}
        \label{eq:first-descent}
        \norm{x_{k+1}-x_{\star,i_k}}^2 \le \rho_{i_k}^2 \norm{x_k-x_{\star,i_k}}^2.
    \end{align}
    Consequently, for every $k\ge 0$,
    \begin{align*}
        &\mathbb{E}\sbr{\norm{x_k-x_\star}} \le \bar\rho^k \norm{x_0-x_\star} + \frac{1+\bar\rho}{1-\bar\rho}{H_\star}.
    \end{align*}
\end{theorem}

\begin{proof}
    Fix $k\ge 0$.
    If $x_k = x_{\star,i_k}$, then $t_k=0$, and \eqref{eq:first-descent} is trivial. So assume $x_k\neq x_{\star,i_k}$.
    {Strong convexity ensures that the gradient difference is nonzero. Applying \Cref{lem:strong-smooth-interpolation} on the open convex domain of $f_{i_k}$ yields}
    \begin{align*}
        &\theta_{i_k}\norm{x_k-x_{\star,i_k}} {\overset{\eqref{eq:reverse-cs}}{\le}} \frac{\inner{\nabla f_{i_k}(x_k)-\nabla f_{i_k}(x_{\star,i_k})}{x_k-x_{\star,i_k}}}{\norm{\nabla f_{i_k}(x_k)-\nabla f_{i_k}(x_{\star,i_k})}}.
    \end{align*}
    Since the assumptions of \Cref{thm:stochastic_component_typeII} are satisfied, we obtain
    \begin{align*}
        \norm{x_{k+1}-x_{\star,i_k}}^2 &\le \norm{x_k-x_{\star,i_k}}^2 - \theta_{i_k}^2 \norm{x_k-x_{\star,i_k}}^2 \\
        &= \rbrac{1-\theta_{i_k}^2}\norm{x_k-x_{\star,i_k}}^2 \\
        &= \rho_{i_k}^2 \norm{x_k-x_{\star,i_k}}^2.
    \end{align*}
    This proves \eqref{eq:first-descent}. Taking square roots, we get $\norm{x_{k+1}-x_{\star,i_k}} \le \rho_{i_k}\norm{x_k-x_{\star,i_k}}$.
    Using the triangle inequality twice,
    \begin{align*}
        \norm{x_{k+1}-x_\star} &\le \norm{x_{k+1}-x_{\star,i_k}} + \norm{x_{\star,i_k}-x_\star} \\
        \alignedrel{\eqref{eq:first-descent}}{\le} \rho_{i_k}\norm{x_k-x_{\star,i_k}} + \norm{x_{\star,i_k}-x_\star} \\
        &\le \rho_{i_k}\rbrac{\norm{x_k-x_\star} + \norm{x_\star-x_{\star,i_k}}} + \norm{x_{\star,i_k}-x_\star} \\
        &= \rho_{i_k}\norm{x_k-x_\star} + \rbrac{1+\rho_{i_k}}\norm{x_{\star,i_k}-x_\star}.
    \end{align*}
    Since $i_k$ is sampled uniformly from $\{1,\dots,n\}$, taking conditional expectation with respect to ${\mathcal F_k}$ yields
    \begin{align*}
        \mathbb{E}\sbr{\norm{x_{k+1}-x_\star}\mid \mathcal{F}_k} &\le \mathbb{E}\sbr{\rho_{i_k}\norm{x_k-x_\star} + \rbrac{1+\rho_{i_k}}\norm{x_{\star,i_k}-x_\star}\mid \mathcal{F}_k} \\
        &= \bar\rho \norm{x_k-x_\star} + \frac{1}{n}\sum_{i=1}^n \rbrac{1+\rho_i}\norm{x_{\star,i}-x_\star}.
    \end{align*}
    Taking expectation again and using tower property, we obtain
    \begin{align*}
        &\mathbb{E}\sbr{\norm{x_{k+1}-x_\star}} \le \bar\rho\,\mathbb{E}\sbr{\norm{x_k-x_\star}} + \frac{1}{n}\sum_{i=1}^n \rbrac{1+\rho_i}\norm{x_{\star,i}-x_\star}.
    \end{align*}
    Iterating this recursion yields
    \begin{align*}
        &\mathbb{E}\sbr{\norm{x_k-x_\star}} \le \bar\rho^k\norm{x_0-x_\star} + \frac{1-\bar\rho^k}{1-\bar\rho}\frac{1}{n}\sum_{i=1}^n \rbrac{1+\rho_i}\norm{x_{\star,i}-x_\star}.
    \end{align*}
    Finally, since $\norm{x_{\star,i}-x_\star}\le {H_\star}$ for every $i$,
    \begin{align*}
        \frac{1}{n}\sum_{i=1}^n \rbrac{1+\rho_i}\norm{x_{\star,i}-x_\star} &\le \frac{1}{n}\sum_{i=1}^n \rbrac{1+\rho_i}{H_\star} = \rbrac{1+\bar\rho}{H_\star},
    \end{align*}
    and therefore
    \begin{align*}
        &\mathbb{E}\sbr{\norm{x_k-x_\star}} \le \bar\rho^k\norm{x_0-x_\star} + \frac{1+\bar\rho}{1-\bar\rho}{H_\star}.
    \end{align*}
\end{proof}

\paragraph{Convex case.}
In the absence of strong convexity, the following result applies.
\begin{theorem}
    \label{thm:stochastic_component_smooth_convex_global}
    {Use the common setting of \Cref{sec:stochastic-setting}, and let the iterates be generated by \Cref{alg:stochastic_new}.} 
    Moreover, assume that for each $i\in\{1,\dots,n\}$, the function $f_i$ is convex and $L_i$-smooth on an open convex set containing $\cX$, {with $L_i>0$}. Define
    \begin{align*}
        &L_{\max} := \max_{1\le i\le n} L_i, \qquad \bar L_2^2 := \frac{1}{n}\sum_{i=1}^n L_i^2.
    \end{align*}
    Suppose that $i_k$ is sampled uniformly from $\{1,\dots,n\}$ and choose
    \begin{align*}
        &t_k := \begin{cases}
            \dfrac{\norm{\nabla f_{i_k}(x_k)-\nabla f_{i_k}(x_{\star,i_k})}}{L_{i_k}}, & \nabla f_{i_k}(x_k)\neq \nabla f_{i_k}(x_{\star,i_k}), \\[1.2ex]
            0, & \nabla f_{i_k}(x_k)=\nabla f_{i_k}(x_{\star,i_k}).
        \end{cases}
    \end{align*}
    Then, for every $K\ge 1$,\footnote{In the common-minimizer regime $x_{\star,i}=x_\star$ for all $i$, we have ${H_\star}=0$, and therefore
    \begin{align*}
        &\min_{0\le k\le K-1}\mathbb{E}\norm{\nabla f(x_k)-\nabla f(x_\star)} \le \frac{2L_{\max}\norm{x_0-x_\star}}{\sqrt{K}}.
    \end{align*}}
    \begin{align*}
        &\min_{0\le k\le K-1}\mathbb{E}\norm{\nabla f(x_k)-\nabla f(x_\star)} \le \sqrt{\frac{4L_{\max}^2\norm{x_0-x_\star}^2}{K} + \rbrac{8L_{\max}^2 + 2\bar L_2^2}{H_\star}^2}.
    \end{align*}
\end{theorem}

\begin{proof}
    Fix $k\ge 0$.
    {If $\nabla f_{i_k}(x_k)=\nabla f_{i_k}(x_{\star,i_k})$, then $t_k=0$ and $x_{k+1}=x_k$. The first distance recurrence below therefore holds with equality on this event. It remains to establish that recurrence for a nonzero sampled residual.} 
    Assume now that $\nabla f_{i_k}(x_k)\neq \nabla f_{i_k}(x_{\star,i_k})$.
    Since $f_{i_k}$ is convex and $L_{i_k}$-smooth, cocoercivity gives
    \begin{align*}
        &\frac{1}{L_{i_k}}\norm{\nabla f_{i_k}(x_k)-\nabla f_{i_k}(x_{\star,i_k})}^2 \le \inner{\nabla f_{i_k}(x_k)-\nabla f_{i_k}(x_{\star,i_k})}{x_k-x_{\star,i_k}}.
    \end{align*}
    Hence
    \begin{align*}
        &t_k = \frac{\norm{\nabla f_{i_k}(x_k)-\nabla f_{i_k}(x_{\star,i_k})}}{L_{i_k}} \le \frac{\inner{\nabla f_{i_k}(x_k)-\nabla f_{i_k}(x_{\star,i_k})}{x_k-x_{\star,i_k}}}{\norm{\nabla f_{i_k}(x_k)-\nabla f_{i_k}(x_{\star,i_k})}}.
    \end{align*}
    Therefore the assumptions of \Cref{thm:stochastic_component_typeII} are satisfied, and we obtain
    \begin{align*}
        &\norm{x_{k+1}-x_\star}^2 \le \norm{x_k-x_\star}^2 - \frac{\norm{\nabla f_{i_k}(x_k)-\nabla f_{i_k}(x_{\star,i_k})}^2}{L_{i_k}^2} + \frac{2{H_\star}}{L_{i_k}}\norm{\nabla f_{i_k}(x_k)-\nabla f_{i_k}(x_{\star,i_k})}.
    \end{align*}
    {The preceding recurrence now holds for every sampled component, including those with zero residual.} By Young's inequality,
    \begin{align*}
        &\frac{2{H_\star}}{L_{i_k}}\norm{\nabla f_{i_k}(x_k)-\nabla f_{i_k}(x_{\star,i_k})} \le \frac{1}{2}\frac{\norm{\nabla f_{i_k}(x_k)-\nabla f_{i_k}(x_{\star,i_k})}^2}{L_{i_k}^2} + 2{H_\star}^2.
    \end{align*}
    Hence
    \begin{align*}
        &\norm{x_{k+1}-x_\star}^2 \le \norm{x_k-x_\star}^2 - \frac{1}{2}\frac{\norm{\nabla f_{i_k}(x_k)-\nabla f_{i_k}(x_{\star,i_k})}^2}{L_{i_k}^2} + 2{H_\star}^2.
    \end{align*}
    Taking conditional expectation with respect to ${\mathcal F_k}$ yields
    \begin{align*}
        &\mathbb{E}\sbrac{\norm{x_{k+1}-x_\star}^2 \mid \mathcal{F}_k} \le \norm{x_k-x_\star}^2 - \frac{1}{2n}\sum_{i=1}^n \frac{\norm{\nabla f_i(x_k)-\nabla f_i(x_{\star,i})}^2}{L_i^2} + 2{H_\star}^2.
    \end{align*}
    Taking expectation again, using the tower property and summing from $k=0$ to $K-1$, we get
    \begin{align*}
        &\mathbb{E}\sbrac{\norm{x_K-x_\star}^2} \le \norm{x_0-x_\star}^2 - \frac{1}{2}\sum_{k=0}^{K-1}\mathbb{E}\sbrac{\frac{1}{n}\sum_{i=1}^n \frac{\norm{\nabla f_i(x_k)-\nabla f_i(x_{\star,i})}^2}{L_i^2}} + 2K{H_\star}^2.
    \end{align*}
    Since $\mathbb{E}\sbrac{\norm{x_K-x_\star}^2} \ge 0$, it follows that
    \begin{align}
        \label{eq:taggg222}
        &\frac{1}{K}\sum_{k=0}^{K-1}\mathbb{E}\sbrac{\frac{1}{n}\sum_{i=1}^n \frac{\norm{\nabla f_i(x_k)-\nabla f_i(x_{\star,i})}^2}{L_i^2}} \le \frac{2\norm{x_0-x_\star}^2}{K} + 4{H_\star}^2.
    \end{align}

    {\color{black}
    For each component, $L_i$-smoothness and the definition of ${H_\star}$ give
    \begin{equation}\label{eq:stoch-component-bias}
        \|\nabla f_i(x_{\star,i})-\nabla f_i(x_\star)\|
        \le L_i\|x_{\star,i}-x_\star\|\le L_i H_\star.
    \end{equation}
    Splitting the gradient difference at $x_{\star,i}$ and using $\|a+b\|^2\le2\|a\|^2+2\|b\|^2$ yields
    \begin{align}
        \|\nabla f_i(x_k)-\nabla f_i(x_\star)\|^2
        \alignedrel{\eqref{eq:stoch-component-bias}}{\le}
        2\|\nabla f_i(x_k)-\nabla f_i(x_{\star,i})\|^2+2L_i^2{H_\star}^2\notag\\
        &\le 2L_{\max}^2\frac{\|\nabla f_i(x_k)-\nabla f_i(x_{\star,i})\|^2}{L_i^2}
           +2L_i^2{H_\star}^2.\label{eq:stoch-component-split}
    \end{align}
    Crucially, retain the individual $L_i^2$ in the second term until averaging. Since $\nabla f$ is the average component gradient, Jensen's inequality gives
    \begin{align}
        \|\nabla f(x_k)-\nabla f(x_\star)\|^2
        &\le\frac1n\sum_{i=1}^n\|\nabla f_i(x_k)-\nabla f_i(x_\star)\|^2\notag\\
        \alignedrel{\eqref{eq:stoch-component-split}}{\le}
        \frac{2L_{\max}^2}{n}\sum_{i=1}^n
        \frac{\|\nabla f_i(x_k)-\nabla f_i(x_{\star,i})\|^2}{L_i^2}
        +2\bar L_2^2{H_\star}^2.\label{eq:stoch-averaged-split}
    \end{align}
    Taking expectations and averaging over $k=0,\ldots,K-1$ now gives
    \begin{align}
        \frac1K\sum_{k=0}^{K-1}\mathbb E\|\nabla f(x_k)-\nabla f(x_\star)\|^2
        \alignedrel{\eqref{eq:stoch-averaged-split},\eqref{eq:taggg222}}{\le}
        2L_{\max}^2\left(\frac{2\|x_0-x_\star\|^2}{K}+4{H_\star}^2\right)
        +2\bar L_2^2{H_\star}^2\notag\\
        &=\frac{4L_{\max}^2\|x_0-x_\star\|^2}{K}
          +(8L_{\max}^2+2\bar L_2^2){H_\star}^2.\label{eq:stoch-corrected-average}
    \end{align}
    Finally, choose an index minimizing the expected squared residual and apply Cauchy--Schwarz:
    \begin{align*}
        \min_{0\le k<K}\mathbb E\|\nabla f(x_k)-\nabla f(x_\star)\|
        &\le\sqrt{\min_{0\le k<K}\mathbb E\|\nabla f(x_k)-\nabla f(x_\star)\|^2}\\
        \alignedrel{\eqref{eq:stoch-corrected-average}}{\le}
        \sqrt{\frac{4L_{\max}^2\|x_0-x_\star\|^2}{K}
        +(8L_{\max}^2+2\bar L_2^2){H_\star}^2}.
    \end{align*}
    This is the stated bound.
    }
\end{proof}

\begin{remark}[Sanity check]
    {When all components share a constrained minimizer, choose $x_{\star,i}=x_\star$, so ${H_\star}=0$ and both neighborhood terms vanish. For a single component, \Cref{thm:stochastic_component_smooth_strongly_convex} gives $\|x_k-x_\star\|\le\rho^k\|x_0-x_\star\|$, the square-root form of the squared-distance guarantee in \Cref{thm:smooth-strong}. In the same single-component setting, \Cref{thm:stochastic_component_smooth_convex_global} gives a best gradient-residual norm bounded by $2L\|x_0-x_\star\|/\sqrt K$, whereas taking square roots in \Cref{thm:L-smooth} gives $L\|x_0-x_\star\|/\sqrt K$. Thus these norm bounds have the same order, with a factor of two in the smooth-convex comparison. For multiple components with ${H_\star}=0$, the stated guarantees retain their stochastic expectation and component-dependent constants.}
\end{remark}

\input{appendix-experiments}

\mdfsetup{nobreak=true}
\input{appendix-potential-applications}

\input{appendix-notation}

\input{appendix-concepts}

\end{document}

%% file: macros-common.tex
\usepackage{multirow}
\usepackage[utf8]{inputenc} %
\usepackage[T1]{fontenc} %
\usepackage{calligra}
\usepackage{layout}

\usepackage{amsmath}
\usepackage{amssymb}
\usepackage{amsfonts} %
\usepackage{amsthm}

\usepackage{mathtools}
\usepackage{thmtools} 
\usepackage{thm-restate}

\usepackage{subcaption}

\usepackage{latexsym}

\usepackage{xcolor} %
\usepackage{graphicx}

\usepackage{algorithmic}
\usepackage{algorithm}

\usepackage{url}           %
\usepackage{booktabs} %
\usepackage{microtype} %

\usepackage{cleveref}

\usepackage{bbm}
\usepackage{bm} %

\usepackage{nicefrac} %
\usepackage{adjustbox}
\usepackage{threeparttable}

\usepackage{import, ifthen}

\usepackage{tcolorbox}
\usepackage{pifont}
\definecolor{mydarkred}{RGB}{192,25,25}
\definecolor{mydarkgreen}{RGB}{25,192,25}
\definecolor{mydarkblue}{RGB}{25,25,192}
\definecolor{darkscarlet}{rgb}{0.34, 0.01, 0.1}
\definecolor{yaleblue}{rgb}{0.06, 0.3, 0.57}
\definecolor{darkpowderblue}{rgb}{0.0, 0.2, 0.6}
\definecolor{midnightblue}{HTML}{0059b3}
\definecolor{noonblue}{HTML}{e5eef7}
\definecolor{chromered}{HTML}{f14233}
\definecolor{olivedrab}{HTML}{6b8e23}

\definecolor{RedOrange}{cmyk}{0,0.77,0.87,0}
\definecolor{Orchid}{cmyk}{0.32,0.64,0,0}

\newcommand{\red}{\color{RedOrange}}

\newcommand{\green}{\color[rgb]{0.13,0.55,0.13}} %

\newcommand{\cmark}{{\green\ding{51}}}%
\newcommand{\xmark}{{\red\ding{55}}}%

\newcommand{\algnamemath}[1]{\text{\normalfont\sffamily #1}}
\usepackage{xspace}

\newcommand{\algname}[1]{{\small\sf#1}\xspace}
\newcommand{\algnamesmall}[1]{{\scriptsize\sf#1}\xspace}
\newcommand{\algnameintable}[1]{\ifmmode\text{\scalebox{0.90}{\normalfont\sffamily #1}}\else\scalebox{0.90}{\sffamily #1}\xspace\fi}

\newcommand{\squeeze}{}
\newcommand{\norm}[1]{{\left\| #1 \right\|}}

\newcommand{\rbr}[1]{\left(#1\right)} %
\newcommand{\sbr}[1]{\left[#1\right]} %
\newcommand{\inp}[2]{\left\langle#1,#2\right\rangle} %
\newcommand{\anglebr}[1]{\left\langle#1\right\rangle} %

\newcommand{\parens}[1]{\left( #1 \right)}
\newcommand{\brac}[1]{\left\{ #1 \right\}}

\newcommand{\inner}[2]{\anglebr{#1, #2}}

\newcommand{\cC}{\mathcal{C}}
\newcommand{\cD}{\mathcal{D}}

\newcommand{\cO}{\mathcal{O}}
\newcommand{\cP}{\mathcal{P}}

\newcommand{\cV}{\mathcal{V}}
\newcommand{\cX}{\mathcal{X}}

\newcommand{\R}{\mathbb{R}}  %

\newcommand{\mQ}{\bm{Q}}
\newcommand{\mR}{\bm{R}}

\newcommand{\del}[1]{}

\newcommand{\eqdef}{\coloneqq}

\DeclareMathOperator{\diam}{diam}        %
\DeclareMathOperator{\dist}{dist}            %

\DeclareMathOperator{\prox}{prox}

\newcommand{\Proj}{{\rm Proj}} %

\DeclareMathOperator{\brox}{brox}

\newcommand{\pr}[1][]{
  \ifthenelse { \equal{#1}{} }
  { \ensuremath{\mathrm{P}} }
  { \ensuremath{\mathrm{P}\left(#1\right)} }
}

\DeclareMathOperator{\dom}{dom}         %

\newcommand{\argmin}{\mathop{\arg\!\min}}

\usepackage[colorinlistoftodos,bordercolor=orange,backgroundcolor=orange!20,linecolor=orange,textsize=scriptsize]{todonotes}

\newcommand{\ppeter}[1]{}

\newcommand{\hidesolutions}[1]{} %
\ExplSyntaxOn
\box_new:N \l__brox_relation_box
\box_new:N \l__brox_annotation_box
\NewExpandableDocumentCommand{\alignedrel}{mm}{
  \group_begin:
  \hbox_set:Nn \l__brox_relation_box {$\displaystyle {}#2{}$}
  \hbox_set:Nn \l__brox_annotation_box {$\displaystyle {}\overset{#1}{#2}{}$}
  \hspace{\dim_eval:n{(\box_wd:N\l__brox_annotation_box-\box_wd:N\l__brox_relation_box)/2}}
  \exp_args:NNNx \group_end: {&}
  \hspace{-\dim_eval:n{(\box_wd:N\l__brox_annotation_box-\box_wd:N\l__brox_relation_box)/2}}
  \overset{#1}{#2}
}
\ExplSyntaxOff

%% file: macros-theorems-article.tex
\usepackage{mdframed}

\definecolor{thmbackground}{rgb}{240, 240, 255}

\mdfdefinestyle{theoremstyle}{
    backgroundcolor=lightgray!12, %
    linecolor=lightgray!12, %
    innertopmargin=4pt, %
    innerbottommargin=4pt,
    innerrightmargin=10pt,
    innerleftmargin=10pt
}

\declaretheoremstyle[
    headfont=\bfseries,
    bodyfont=\normalfont,
    mdframed={
        style=theoremstyle
    }
]{graytheoremstyle}

\declaretheorem[
    name=Theorem,
    style=graytheoremstyle,
    numberwithin=section
]{theorem}

\declaretheorem[
    name=Lemma,
    style=graytheoremstyle,
    numberwithin=section
]{lemma}

\declaretheorem[
    name=Assumption,
    style=graytheoremstyle,
    numberwithin=section
]{assumption}

\declaretheorem[
    name=Proposition,
    style=graytheoremstyle,
    numberwithin=section
]{proposition}

\declaretheorem[
    name=Corollary,
    style=graytheoremstyle,
    numberwithin=section
]{corollary}

\theoremstyle{plain}
\numberwithin{claim}{section}
\numberwithin{fact}{section}
 \numberwithin{exercise}{section}
\newtheorem{example}{Example}\numberwithin{example}{section}

\newtheorem{remark}[theorem]{Remark}

\theoremstyle{definition}
\numberwithin{definition}{section}

%% file: main-theory.tex
\color{black}
\mdfsetup{fontcolor=black}
\section{Theory}\label{sec:theory}

Throughout this section, we impose the following assumption. Differentiability, convexity, and smoothness of $f$ are specified separately in each result.

\begin{assumption}\label{ass:main}
The set $\cX\subseteq\R^d$ is nonempty, closed, and convex, and the function $f:\R^d\to\R$ has a nonempty set of constrained minimizers $\cX_\star \eqdef \argmin_{x\in\cX}f(x)$.
\end{assumption}

\subsection{The geometry of localization}\label{sec:master_thm}

We first show that a suitable radius makes a single \algname{BroxGD} iteration decrease the distance to an optimal solution. This statement does {\em not} require  $f$ to be convex.

\begin{restatable}[One-step geometry]{theorem}{ONESTEP}\label{thm:descent}
Let \Cref{ass:main} hold and let $f$ be differentiable. Fix $x_k\in\cX$ and $x_\star\in\cX_\star$. Suppose either of the following radius admissibility conditions holds, with its displayed denominator nonzero:
\begin{align}
 \text{Type I:}\quad&
 0<t_k\le\frac{\langle\nabla f(x_k),x_k-x_\star\rangle}{\|\nabla f(x_k)\|},\label{eq:radius-condition-1}\\
 \text{Type II:}\quad&
 0<t_k\le\frac{\langle\nabla f(x_k)-\nabla f(x_\star),x_k-x_\star\rangle}{\|\nabla f(x_k)-\nabla f(x_\star)\|}.\label{eq:radius-condition-2}
\end{align}
Then the update \eqref{eq:method} satisfies:
\begin{enumerate}
\item[(i)] $x_\star$ does not belong to the interior of the ball $\mathbb{B}(x_k,t_k)$, i.e.,
\begin{equation}\label{eq:thm_radius_bound}
 t_k\le\|x_k-x_\star\|;
\end{equation}
\item[(ii)] $x_{k+1}$ lies on the boundary of the ball $\mathbb{B}(x_k,t_k)$, i.e., $\|x_{k+1}-x_k\|=t_k$;
\item[(iii)] the squared distance to $x_\star$ decreases by at least the square of the radius, i.e.,
\begin{equation}\label{eq:thm_squared_distance_decrease}
 \|x_{k+1}-x_\star\|^2\le\|x_k-x_\star\|^2-t_k^2.
\end{equation}
\end{enumerate}
\end{restatable}

We now offer a quick insight into why the distance decreases. Full proofs and denominator-degeneracy discussion appear in \Cref{sec:proof-I,sec:proof-II}.

\phantomsection\label{concept:normal}For a radius satisfying Type I admissibility, let us write $g=\nabla f(x_k)$ and $y=x_{k+1}$. Admissibility ensures $\langle g,y-x_\star\rangle\ge0$: even the largest linear decrease available inside the ball cannot cross the hyperplane passing through $x_\star$ orthogonal to $g$. Optimality conditions then supply the relation $g+n+\lambda(y-x_k)=0$, with $n\in N_{\cX}(y)$ (the normal cone of $\cX$ at $y$) and $\lambda\ge0$. Since $\langle n,y-x_\star\rangle\ge0$, a positive multiplier gives $\langle y-x_k,y-x_\star\rangle\le0$ and a full-radius step. If $\lambda=0$, equality in the admissibility and Cauchy--Schwarz bounds gives the same conclusions. The two cases establish these two geometric facts:
\begin{equation}\label{eq:main-geometric-facts}
\|y-x_k\|=t_k,\qquad\langle y-x_k,y-x_\star\rangle\le0.
\end{equation}
Expanding the squared distance now gives
\begin{align*}
 \|y-x_\star\|^2
 &= \|x_k-x_\star\|^2-\|y-x_k\|^2
       +2\langle y-x_k,y-x_\star\rangle\\
 \alignedrel{\eqref{eq:main-geometric-facts}}{\le}\|x_k-x_\star\|^2-t_k^2.
\end{align*}
Type II admissibility case uses the gradient difference together with the normal-cone condition at $x_\star$; the full argument follows similar steps.

The key conclusion of \Cref{thm:descent} is the squared-distance decrease: a radius can be justified through separation from a solution, without performing a projection!

\subsection{Practical radii with a fixed oracle budget}

The result we present here uses the geometry established in \Cref{thm:descent} only while the requested objective accuracy has not yet been reached; the full argument is presented in \Cref{sec:practical-radii}.

\begin{restatable}[Practical fixed-horizon objective guarantees]{theorem}{MAINHORIZON}\label{thm:main-horizon}
Let \Cref{ass:main} hold, let $f$ be convex and differentiable, fix $x_\star \in \cX_\star$, and let $R\ge R_0$ be a known positive upper bound  on $R_0 \eqdef \|x_0-x_{\star}\|$. Starting from $x_0$, run \algname{BroxGD} for $K \geq 1$ iterations, and return the iterate ${x_K^{\mathrm{best}}}$ with the smallest objective value among $x_0,\ldots,x_K$.

\begin{itemize}
\item [(i)] \textbf{Bounded gradients.} If $f$ has bounded gradients, i.e. $\|\nabla f(x)\|\le G$ for all $x\in\cX$, and we use the radius rule
\begin{equation}\label{eq:main-horizon-bounded-radius}
 t_k=\frac{GR}{\sqrt K\,\|\nabla f(x_k)\|},
\end{equation}
then
\begin{equation}\label{eq:main-horizon-bounded}
 f({x_K^{\mathrm{best}}})-f_\star\le\frac{GR}{\sqrt K}.
\end{equation}
If $\nabla f(x_k)=0$ for some $k$, we stop and return $x_k$; the same guarantee holds.
\item  [(ii)] \textbf{Smooth convex objectives.} If $f$ has an $L$-Lipschitz gradient on an open convex neighborhood of $\cX$, and we use the radius rule
\begin{equation}\label{eq:main-horizon-smooth-radius}
 t_k=\frac R{\sqrt K},
\end{equation}
then
\begin{equation}\label{eq:main-horizon-smooth}
 f({x_K^{\mathrm{best}}})-f_\star\le\frac{LR^2}{2K}.
\end{equation}
\end{itemize}
\end{restatable}

Note that in the smooth case, the method does not need to know $L$ to choose the radius. The two results are proved in \Cref{cor:practical-horizon,cor:practical-smooth-horizon}. They count local-oracle calls directly and require no boundedness of $\cX$.  These guarantees control objective error, whereas the solution-dependent rule in \Cref{thm:L-smooth} covering the $L$-smooth case controls a gradient-difference residual. 

\subsection{\texorpdfstring{Linear convergence via predetermined geometric radii}{Linear convergence via predetermined geometric radii}}\label{sec:main-geometric}

Under smoothness and strong convexity, a geometrically decreasing radius gives geometric bounds on both objective error and squared distance. Unlike the fixed-horizon schedules, this rule does not require choosing a final horizon. It uses certified smoothness and strong-convexity bounds, together with a certified bound on the initial objective gap.

\begin{restatable}[Geometric radii from certified bounds]{theorem}{GEOMETRICRADII}\label{thm:practical-geometric}
Let \Cref{ass:main} hold, and let $f$ be convex and differentiable on an open convex set containing $\cX$. Fix $x_\star\in\cX_\star$ and $x_0\in\cX$ and write $f_\star=f(x_\star)$ and $\Delta_k=f(x_k)-f_\star$. Suppose $f$ is $L$-smooth and $\mu$-strongly convex on $\cX$, with known valid bounds $0<\mu\le L$; conservative upper and lower bounds may be used as $L$ and $\mu$, respectively. Let ${C_0}>0$ satisfy $\Delta_0\le {C_0}$. Set
\begin{equation}\label{eq:practical-geometric-rule}
 \alpha=\frac{\mu}{2L},\qquad q=1-\frac{\mu}{4L},\qquad
 t_k=\alpha\sqrt{\frac{2{C_0}}{\mu}}\,q^{k/2}.
\end{equation}
Then the iterates of \eqref{eq:method} satisfy, for every $k\ge0$,
\begin{equation}\label{eq:practical-geometric-rate}
 \Delta_k\le {C_0}q^k,\qquad \|x_k-x_\star\|^2\le\frac{2{C_0}}{\mu}q^k.
\end{equation}
Thus, for $0<\varepsilon<{C_0}$, it suffices to take
\begin{equation}\label{eq:practical-geometric-budget}
 K\ge\left\lceil\frac{4L}{\mu}\log\frac{C_0}{\varepsilon}\right\rceil
\end{equation}
to ensure $f(x_K)-f_\star\le\varepsilon$.
\end{restatable}

The schedule is fixed in advance: it requires neither backtracking nor the optimal value $f_\star$ itself, provided a valid upper bound $C_0$ on the initial gap is available. Each iteration uses one exact local-oracle solve. The proof is in \Cref{sec:geometric-radius-proof}.

\subsection{Polyak radii: objective-gap adaptation}\label{sec:main-polyak}

When the optimal objective value $f_\star$ is known, the radius can adapt to the current objective gap without a prescribed horizon or an initial-distance bound. For a convex objective, choose $g_k\in\partial f(x_k)$ (so $g_k=\nabla f(x_k)$ in the differentiable case) and use
\begin{equation}\label{eq:polyak_radius_smooth_convex}
 t_k=\frac{f(x_k)-f_\star}{\|g_k\|},\qquad f(x_k)>f_\star.
\end{equation}
At a zero-gap iterate, set $t_k=0$ and freeze the sequence. Convexity ensures that $g_k\ne0$ whenever the gap is positive. The corresponding local step is
\begin{equation}\label{eq:local_lmo_nondiff}
 x_{k+1}\in\argmin_{z\in\cX\cap\mathbb{B}(x_k,t_k)}\langle g_k,z\rangle.
\end{equation}
This extends \algname{BroxGD} to nonsmooth convex objectives. The rule requires neither $G$ nor $L$, but it does require the exact value $f_\star$; hence Table~\ref{tbl:main} does not classify it as practical in general.

\begin{restatable}[Rate under bounded subgradients]{theorem}{POLYAKSUBGRADIENT}
\label{thm:bounded_subgradients}
    Let \Cref{ass:main} hold, let $f$ be convex, and let $\{x_k\}_{k \ge 0}$ be generated by the modified {\algname{BroxGD}} method in \eqref{eq:local_lmo_nondiff}, with $x_0 \in \cX$.
    Assume that the subgradients of $f$ are bounded on $\cX$, i.e.,
    \begin{align*}
        \norm{g} \le G \qquad \forall x \in \cX,\ \forall g \in \partial f(x), 
    \end{align*}
    for some constant $G > 0$.
    Fix $x_\star\in\cX_\star$ and use the Polyak radius \eqref{eq:polyak_radius_smooth_convex}, with $g_k\in\partial f(x_k)$, freezing the sequence whenever $f(x_k)=f(x_\star)$.
    Then, for every $K \ge 1$,
    \begin{align}
        \frac{1}{K}\sum_{k=0}^{K-1} \rbrac{f(x_k)-f(x_\star)}^2 \le \frac{G^2 \norm{x_0-x_\star}^2}{K}. 
        \label{eq:square-sum-bound-subgradient}
    \end{align}
    Moreover, the average iterate $ {\bar x_K} \eqdef \frac{1}{K}\sum_{k=0}^{K-1} x_k$ satisfies
    \begin{align}
        f({\bar x_K})-f(x_\star) \le \frac{G \norm{x_0-x_\star}}{\sqrt{K}}. 
        \label{eq:avg_iter_G_subgradient}
    \end{align}
\end{restatable}

\begin{restatable}[Polyak radius in the \(L\)-smooth convex regime]{theorem}{POLYAKSMOOTH}
\label{thm:polyak_radius_smooth_convex}
Let \Cref{ass:main} hold, and assume that \(f\) is convex, differentiable, and \(L\)--smooth on an open set containing \(\cX\), {with $L>0$}.
Let \(\{x_k\}_{k\ge 0}\) be the iterates generated by {\algname{BroxGD}}, where \(x_0\in \cX\), and let
\(x_\star\in \cX_\star\).
Use the Polyak radius \eqref{eq:polyak_radius_smooth_convex} with $g_k=\nabla f(x_k)$, freezing the sequence whenever $f(x_k)=f(x_\star)$.

Define
\[
{R_0} \eqdef \|x_0-x_\star\|,
\qquad
{M_0} \eqdef \|\nabla f(x_\star)\|+L {R_0}.
\]
Then, for every \(K\ge 1\),
\begin{equation}
\label{eq:polyak_radius_smooth_convex_square_sum}
    \frac{1}{K}\sum_{k=0}^{K-1} \rbrac{f(x_k)-f(x_\star)}^2
    \le
    \frac{{M_0}^2 {R_0}^2}{K}.
\end{equation}
Moreover, the average iterate
\[
{\bar x_K} \eqdef \frac{1}{K}\sum_{k=0}^{K-1}x_k
\]
satisfies
\begin{equation}
\label{eq:polyak_radius_smooth_convex_average}
    f({\bar x_K})-f(x_\star)
    \le
    \frac{{M_0} {R_0}}{\sqrt{K}}.
\end{equation}

If, additionally, $f$ is convex and $L$-smooth on all of $\R^d$ and $\nabla f(x_\star)=0$, then the same iterates satisfy the sharper bound
\begin{equation}\label{eq:polyak-smooth-stationary-average}
 f(\bar x_K)-f_\star\le\frac{2LR_0^2}{K},\qquad K\ge1.
\end{equation}

More generally, suppose $\nabla f(x_\star)=0$ and $f$ is convex and $L$-smooth on an open convex neighborhood $U$ of $\cX$. Set $S=\cX\cap\overline{\mathbb{B}}(x_\star,R_0)$ and choose $r_U>0$ such that $S+\overline{\mathbb{B}}(0,r_U)\subset U$. Such a radius exists because $S$ is compact. Then
\begin{equation}\label{eq:polyak-neighborhood-average}
 f(\bar x_K)-f_\star\le\frac{2LR_0^2}{K}\max\left\{1,\frac{R_0}{r_U}\right\},\qquad K\ge1.
\end{equation}
The neighborhood radius $r_U$ enters only the guarantee, not the algorithm.
\end{restatable}

The first result requires bounded subgradients on $\cX$. The second replaces this global bound by smoothness: the iterates stay within distance $R_0$ of $x_\star$, which bounds their gradient norms by $M_0$. In general, both guarantees give $\mathcal O(1/\sqrt K)$ objective error for the average iterate. If $\nabla f(x_\star)=0$, \eqref{eq:polyak-smooth-stationary-average} and \eqref{eq:polyak-neighborhood-average} improve this to $\mathcal O(1/K)$ under global or neighborhood smoothness, respectively. In the unconstrained case, this is the classical Polyak-stepsize \algname{GD} argument; the local oracle allows the same guarantee with constraints. The smooth fixed-horizon rule in \Cref{thm:main-horizon} instead guarantees $\mathcal O(1/K)$ best-iterate error, but requires a certified initial-distance bound and a prescribed horizon. These are guarantees of the available analyses, not a claim that the Polyak rule cannot converge faster. Full proofs appear in \Cref{sec:nondiff,sec:polyak-smooth-proof}.

\subsection{Radius backtracking}

We now revisit the smooth convex regime already covered by part (ii) of \Cref{thm:main-horizon}. However, we seek to establish a result which does {\em not} require the final iteration horizon $K$ to be known in advance. Moreover, we seek a method which does {\em not} need to know an upper bound on $R_0$. We solve this problem via  a {\em radius-backtracking} strategy for \algname{BroxGD}. Its acceptance test certifies objective decrease; the convergence proof uses endpoint optimality conditions rather than the radius admissibility conditions of \Cref{thm:descent}. This modification is formalized as \Cref{alg:radius-backtracking}.

In this new method, however, we do need $L$ as an input. As a byproduct, we prove that, under strong convexity, \Cref{alg:radius-backtracking} converges linearly with the same order of condition-number dependence as \algname{ProjGD}, for fixed $\xi$; the contraction factors differ.

\color{black}
\begin{algorithm}[H]
\caption{{\algname{BroxGD}} with radius backtracking}\label{alg:radius-backtracking}
\begin{algorithmic}[1]
\STATE \textbf{Input:} $x_0\in\cX$, known $L>0$, {$\xi\in(0,1)$}, horizon $K$, exact local linear oracle
\FOR{$k=0,\ldots,K-1$}
 \STATE Compute $a_k=\nabla f(x_k)$; if $a_k=0$, return $x_k$
 \STATE Set $t=\|a_k\|/L$
 \LOOP
  \STATE Compute
  \[
   y\in\argmin_{z\in\cX\cap\mathbb{B}(x_k,t)}\langle a_k,z-x_k\rangle.
  \]
  \STATE Set $s=y-x_k$ and ${\mathcal D}=-\langle a_k,s\rangle$. If ${\mathcal D}=0$, return $x_k$.
  \IF{${\mathcal D}\ge L\|s\|^2$}
   \STATE Set $x_{k+1}=y$ and exit the inner loop
  \ENDIF
  \STATE {Set $t=\xi t$}
 \ENDLOOP
\ENDFOR
\STATE \textbf{Output:} $x_K$
\end{algorithmic}
\end{algorithm}
\color{black}
\begin{restatable}[Practical backtracking and last-iterate rates]{theorem}{MAINBACKTRACKING}\label{thm:main-backtracking}
Let \Cref{ass:main} hold and suppose $f$ is convex and differentiable with a globally $L$-Lipschitz gradient, where $L>0$ is known. Fix $\xi\in(0,1)$. Then \Cref{alg:radius-backtracking} makes finitely many local LMO calls at each ball center $x_k$ and either stops at a minimizer or accepts an objective-decreasing step. Freeze the sequence after optimal termination, whether the stopping test is a zero gradient or $\mathcal D=0$.

If, for some $R_{\rm lev}>0$, the accepted centers satisfy
\begin{equation}\label{eq:main-bt-distance}
 \operatorname{dist}(x_k,\cX_\star)\le R_{\rm lev}\qquad(k\ge0),
\end{equation}
then, for every $K\ge0$,
\begin{equation}\label{eq:main-bt-convex}
 {\Delta_K}\le\frac{\max\{{\Delta_0},4LR_{\rm lev}^2/\xi^2\}}{K+1}.
\end{equation}
If instead $f$ is $\mu$-strongly convex on $\cX$, $\mu>0$, then without \eqref{eq:main-bt-distance},
\begin{equation}\label{eq:main-bt-strong}
 {\Delta_K}\le\left(1+\frac{\xi^2\mu}{L}\right)^{-K}{\Delta_0}.
\end{equation}
\end{restatable}

The method needs neither $R_{\rm lev}$ nor ${\Delta_0}$ nor $\mu$. Here $K$ counts accepted steps; rejected trials incur additional local-oracle calls.

At $x_k$, we reuse $a_k=\nabla f(x_k)$ while trying radii starting from $\|a_k\|/L$ and shrinking by a fixed  factor $\xi\in(0,1)$, which is a hyper-parameter of the method. We accept a local step $s$ when its linear-model decrease dominates the smoothness remainder:
\begin{equation}\label{eq:main-bt-test}
 -\langle a_k,s\rangle\ge L\|s\|^2.
\end{equation}
The test uses only  the knowledge of $L$, the current gradient, and the displacement returned by the local oracle. It does not require objective evaluations or a gradient at the trial point.

\Cref{thm:bt-certificates,thm:bt-rates,thm:bt-strong-rates} prove this result. The proof bounds an analysis multiplier between $L$ and $L/\xi^2$, and combines objective descent with a selected endpoint subgradient bound. It is a separate argument from \Cref{thm:descent}. The convex distance condition concerns the entire accepted trajectory, not just $R_0$; for example, bounded distance to the solution set on the initial sublevel set suffices. Finite backtracking does not by itself give a uniform bound on the total number of trial calls.

\subsection{A sharp solution-dependent benchmark}\label{sec:smooth-strong}

Practical radius selection and the sharpest distance contraction answer different questions. The following result identifies a contraction matching optimally tuned \algname{ProjGD}, but its radius uses the unknown current solution distance.

\begin{restatable}[Sharp strongly convex contraction]{theorem}{SMOOTHSTRONG}\label{thm:smooth-strong}
Let \Cref{ass:main} hold and suppose $f$ is differentiable, $L$-smooth, and $\mu$-strongly convex on $\cX$, where $0<\mu\le L$. For a fixed $x_\star\in\cX_\star$, use
\begin{equation}\label{eq:radius_smooth_strong}
 t_k=\theta\|x_k-x_\star\|,\qquad \theta=\frac{2\sqrt{\mu L}}{L+\mu}.
\end{equation}
Freeze the sequence at $x_\star$. Then, for every $K\ge0$,
\begin{equation}\label{eq:linear-improved}
 \|x_K-x_\star\|^2\le\left(\frac{L-\mu}{L+\mu}\right)^{2K}R_0^2.
\end{equation}
At $K=0$, the factor in \eqref{eq:linear-improved} is understood to be $1$, including when $L=\mu$.
\end{restatable}
\Cref{sec:proof-sharp-strong} gives a short proof under global assumptions and an affine-hull proof under the stated assumptions on $\cX$. The factor in \eqref{eq:linear-improved} matches the squared-distance contraction of \algname{ProjGD} with stepsize $2/(L+\mu)$ \citep{schmidt2011convergence}. This is a benchmark for the local-oracle framework; \Cref{thm:main-backtracking} supplies an implementable linear-rate alternative, and \Cref{thm:practical-geometric} gives predetermined geometric radii from certified initial bounds.

\subsection{Further results and computational scope}\label{sec:further-results}
The appendix develops the remaining guarantees in \Cref{tbl:main}, following the same four groups.

\paragraph{Practical schedules and convex benchmarks (A--B).}
Predetermined geometric radii give a geometric objective-error bound from certified initial bounds (\Cref{thm:practical-geometric}). The three Polyak-radius results cover convex objectives with bounded gradients (\Cref{thm:bounded_gradients}), nonsmooth convex objectives with bounded subgradients (\Cref{thm:bounded_subgradients}), and smooth convex objectives (\Cref{thm:polyak_radius_smooth_convex}). They give $O(K^{-1/2})$ objective guarantees, for the best iterate in the first case and the averaged iterate in the latter two. For smooth convex objectives with $\nabla f(x_\star)=0$, the averaged-iterate guarantee improves to $O(K^{-1})$ under the global or neighborhood assumptions of \Cref{thm:polyak_radius_smooth_convex}. Their radii use the optimal objective value. Target-accuracy refinements of the main fixed-horizon guarantees also appear in \Cref{sec:practical-radii}.

\paragraph{Gradient residuals and nonconvex objectives (C--D).}
For smooth convex objectives, gradient-difference radii give an $O(K^{-1})$ bound on the best squared residual $Z_k^2$ (\Cref{thm:L-smooth}); a constant radius gives a residual hitting-time guarantee and its fixed-horizon counterpart (\Cref{thm:practical-residual,cor:practical-residual-horizon}). Global $(L_0,L_1)$-smoothness admits a corresponding squared-residual rate (\Cref{thm:asym_L0L1_rate}). These control gradient differences, rather than objective error. Beyond convexity, gradient-mapping radii yield an $O(K^{-1})$ bound on the minimum squared norm of the projected gradient mapping (\Cref{thm:local_lmo_nonconvex_smooth}) and linear objective convergence under a projected-P\L\ condition (\Cref{thm:local_lmo_projected_pl}). These radii require a projection, as indicated in \Cref{tbl:main}.

\paragraph{Oracle structure and further extensions.}
The table's affine-subspace result gives an explicit local oracle and recovers a gradient step when $\cX=\mathbb R^d$ (\Cref{prop:affine-subspace-explicit}). Stochastic guarantees are developed separately in \Cref{sec:stochastic}. Practicality of a radius does not imply that the exact local oracle is cheap. Oracle implementations and potential applications are discussed in \Cref{sec:potential-applications}; \Cref{sec:main-experiment} gives a computational illustration, with full protocols and further experiments in \Cref{sec:experiments}.

\color{black}
\mdfsetup{fontcolor=black}

%% file: main-experiment.tex
{\color{black}
\section{Correlation-matrix experiments}\label{sec:main-experiment}
We test weighted quadratic losses over correlation matrices $X\succeq0$, $\operatorname{diag}(X)=\mathbf1$, at orders $128$, $256$, and $512$, with ten seeds per size and three timing repetitions. \algname{BroxGD} uses \Cref{thm:main-horizon}'s fixed radius, $R=\sqrt{n(n-1)}$ and $K_{\rm rad}=2000n/128$, reporting its best iterate. Its affine-ball solve is exact when a Cholesky check confirms feasibility. \algname{ProjGD} and \algname{ProjGD-BB} share this shortcut, but can require costly Dykstra projections. Timings include all inner work; \Cref{sec:projection-gap} gives the full protocol.

At relative objective error $10^{-2}$, \algname{BroxGD} beats both projected baselines on all 30 instances. Median paired speedups over the faster baseline are $6.0\times$, $6.2\times$, and $7.0\times$ with projection tolerance $10^{-10}$, and $1.9\times$, $2.0\times$, and $2.2\times$ at $10^{-4}$. Initial projections cost about $50\times$ and $14\times$ a local solve, respectively.

\begin{figure}[ht]
\centering
\includegraphics[width=\textwidth]{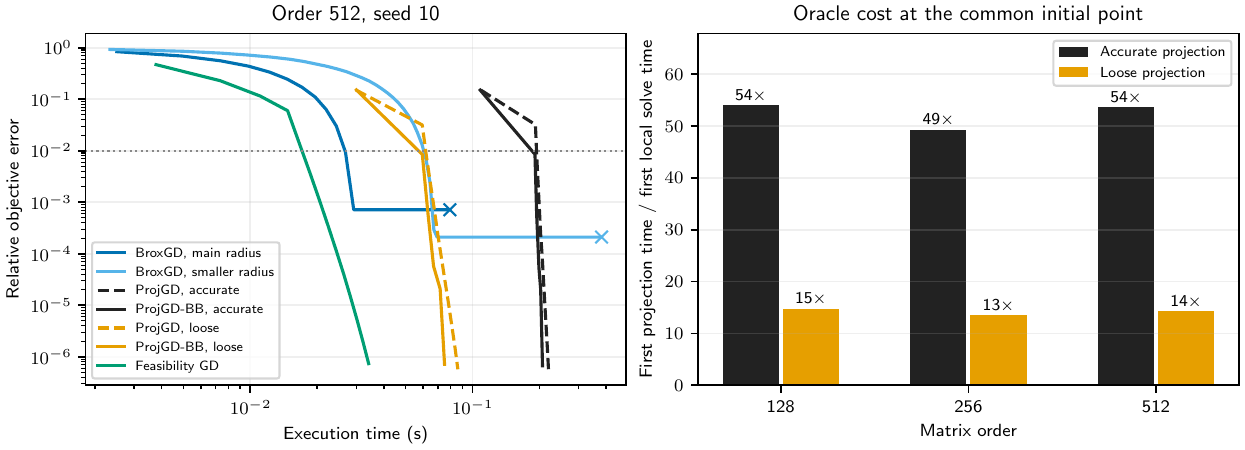}
\caption{Left: order 512, seed 10; median timings over three repeats. Dashed/solid projected curves show \algname{ProjGD}/\algname{ProjGD-BB}. Crosses mark failed continuation of the affine-ball shortcut; the dotted line is the $10^{-2}$ target. Right: median initial projection/local-solve cost ratios over ten seeds.}\label{fig:projection-gap}
\end{figure}

This advantage is specific to low accuracy and the tested projected implementations. To test whether feasibility checks alone suffice, we include \emph{Feasibility \algname{GD}}: take the gradient with its diagonal removed, start with stepsize $1/L$, and halve it until a Cholesky check succeeds. This projection-free control is faster on 29 instances but stalls on one. The primary fixed radius reaches $10^{-3}$ error on 29 instances but $10^{-4}$ on none within the tested budget or before the shortcut fails. Further comparisons, including \algname{FW} and unfavorable regimes, appear in \Cref{sec:experiments}.
}

%% file: appendix-shared-geometry.tex
\subsection{{One-step geometry and admissibility}}

{\color{black}\mdfsetup{fontcolor=black}\ONESTEP*}

\paragraph{{Admissibility and interpretation.}}
\phantomsection\label{concept:fejer}The key inequality \eqref{eq:thm_squared_distance_decrease} shows that each step yields a \emph{Fej\'er-type decrease} with respect to the solution set~$\cX_\star$. 
Note that \eqref{eq:thm_radius_bound} follows from \eqref{eq:thm_squared_distance_decrease}. However, while \eqref{eq:thm_radius_bound} is a simple consequence of Cauchy--Schwarz, the proof of \eqref{eq:thm_squared_distance_decrease} is more involved. So, while not necessary from a mathematical perspective, we decided to state and prove \eqref{eq:thm_radius_bound} separately to make it clearer and intuitive why \eqref{eq:thm_radius_bound} holds.

A subtle point concerns the admissibility conditions in \Cref{thm:descent}, which involve ratios of the form \eqref{eq:radius-condition-1}-\eqref{eq:radius-condition-2}. These expressions require that the corresponding gradients (or gradient differences) are nonzero. When the denominator in \eqref{eq:radius-condition-1} vanishes, the situation can be treated separately: if $\nabla f(x_k)=0$, then $x_k$ is a first-order stationary point, {which implies optimality when $f$ is convex}. 
This highlights an interesting structural feature of the framework. When $\nabla f(x_\star)=0$, the constraint set plays no active role at the solution, since the unconstrained minimizer already lies in $\cX$. In this case, the problem effectively reduces to unconstrained optimization, and one could in principle discard the constraint and run vanilla \algname{GD} instead. Nevertheless, even when the constraint is inactive at optimality, it may still provide algorithmic benefit: prior knowledge that $x_\star \in \cX$ can be exploited to improve performance.
More generally, our method can leverage such structural information by inheriting the improved conditioning of the restricted problem. For instance, when $\cX$ is an affine subspace of $\R^d$, \Cref{alg:new} reduces to \algname{GD} on that subspace, and its convergence is governed by the corresponding subspace condition number rather than the ambient one. This leads to acceleration whenever the objective is better conditioned along the constraint set.
{For an affine constraint, projected gradient steps also reduce to gradient steps on the subspace; this restricted conditioning is therefore shared by the two methods.}

The Type-II admissibility condition \eqref{eq:radius-condition-2} is more delicate. Since it involves gradient differences, the denominator can vanish without implying optimality, and hence such a condition cannot be used as a stopping criterion without additional assumptions. To rule out such degeneracies, one must assume certain identifiability assumptions (e.g., strict convexity; see \Cref{lem:strict_convexity}), which guarantee that equality of gradients implies equality of points.

\subsubsection{Type-I admissibility case} \label{sec:proof-I}

\begin{proof}
    By the radius condition \eqref{eq:radius-condition-1}, and the Cauchy--Schwarz inequality, we have
    \[
        t_k \overset{\eqref{eq:radius-condition-1}}{\le} \frac{\langle\nabla f(x_k),x_k-x_\star\rangle}{\|\nabla f(x_k)\|} \leq \frac{\|\nabla f(x_k)\|\,\|x_k-x_\star\|}{\|\nabla f(x_k)\|} = \|x_k-x_\star\|,
    \]
    establishing claim {\em (i)}. 
    Let us now establish claims {\em (ii)} and {\em (iii)}. 
    Fix $k\ge 0$, and for the purposes of this proof, let us use the shorter notation $\mathbb{B}_k \eqdef \mathbb{B}(x_k,t_k)$.

    \paragraph{Step 1: A favorable half-space inequality.}
    Using the trivial decomposition
    \[
        x_{k+1}-x_\star=(x_k-x_\star)+(x_{k+1}-x_k),
    \]
    we can write
    \begin{equation}\label{eq:bgd-old-step1-decomp}
        \langle \nabla f(x_k),x_{k+1}-x_\star\rangle
        =
        \langle \nabla f(x_k),x_k-x_\star\rangle
        +
        \langle \nabla f(x_k),x_{k+1}-x_k\rangle.
    \end{equation}

    Since $x_{k+1}\in \mathbb{B}_k$, we have $\|x_{k+1}-x_k\|\le t_k$. 
    Using the Cauchy--Schwarz inequality, we can therefore bound
    \begin{eqnarray}
        \langle \nabla f(x_k),x_{k+1}-x_k\rangle &\ge& -\|\nabla f(x_k)\|\,\|x_{k+1}-x_k\| \notag\\ 
        &\ge& -\|\nabla f(x_k)\|\,t_k.
        \label{eq:bgd-old-step1-cs}
    \end{eqnarray}
    Plugging this bound back into \eqref{eq:bgd-old-step1-decomp}, and using the radius bound \eqref{eq:radius-condition-1}, we get
    \begin{eqnarray}
        \langle \nabla f(x_k),x_{k+1}-x_\star\rangle
        \alignedrel{\eqref{eq:bgd-old-step1-decomp}+\eqref{eq:bgd-old-step1-cs}}{\ge}&
        \langle \nabla f(x_k),x_k-x_\star\rangle-\|\nabla f(x_k)\|\,t_k \notag\\
        \alignedrel{\eqref{eq:radius-condition-1}}{\ge}&
        0.
        \label{eq:bgd-old-gk-halfspace}
    \end{eqnarray}

    \paragraph{Step 2: KKT conditions for a constrained linear optimization problem.}
    Since $x_{k+1}$ minimizes the linear functional $z \mapsto \langle \nabla f(x_k),z\rangle$ over the nonempty closed convex set $\cX\cap \mathbb{B}_k$, we have
    \[
        0\in \nabla f(x_k)+N_{\cX}(x_{k+1})+N_{\mathbb{B}_k}(x_{k+1}),
    \]
    where $N_{\cC}(z)$ denotes the normal cone of a set $\cC$ at a point $z$.
    Hence there exist
    \[
        n_k\in N_{\cX}(x_{k+1}), \qquad \lambda_k\ge 0
    \]
    such that
    \begin{equation}
        \label{eq:bgd-old-kkt}
        \nabla f(x_k) + n_k + \lambda_k(x_{k+1}-x_k)= 0.
    \end{equation}

    \paragraph{Step 3: Deductions from the KKT conditions.}
    Taking the inner product of the zero vector from~\eqref{eq:bgd-old-kkt} with $x_{k+1} - x_\star$ gives the identity
    \begin{equation}
        \label{eq:bgd-old-three-term}
        \langle \nabla f(x_k),x_{k+1} - x_\star\rangle + \langle n_k, x_{k+1} - x_\star\rangle + \lambda_k\langle x_{k+1}-x_k,x_{k+1}- x_\star\rangle =0.
    \end{equation}
    Note that due to \eqref{eq:bgd-old-gk-halfspace}, the first term in \eqref{eq:bgd-old-three-term} is nonnegative. 
    We shall now argue that the second term is nonnegative, too. 
    Indeed, since $n_k\in N_{\cX}(x_{k+1})$ and $x_\star\in \cX$, we must have
    \begin{equation}
        \label{eq:bgd-old-normalcone}
        \langle n_k,x_{k+1}-x_\star\rangle\ge 0.
    \end{equation}
    Since the sum of the three terms is zero, the third term in \eqref{eq:bgd-old-three-term} must be nonpositive, which can be equivalently written as
    \begin{equation}\label{eq:bgd-old-lambda-inner}
        \lambda_k\langle x_k-x_{k+1},x_{k+1}- x_\star\rangle\ge 0.
    \end{equation}

    \paragraph{Step 4: Two cases.}
    We now split the argument into two cases.

    \paragraph{Case 1: $\lambda_k > 0$.}
    Then $x_{k+1}$ lies on the boundary of the ball $\mathbb{B}_k$, and thus
    \begin{equation}
        \label{eq:bgd-old-fullstep-case1}
        \|x_{k+1}-x_k\|=t_k.
    \end{equation}
    So, identity {\em (ii)} holds in this case.
    Moreover, \eqref{eq:bgd-old-lambda-inner} implies that
    \begin{equation}
        \label{eq:bgd-old-innerprod-case1}
        \langle x_k-x_{k+1},x_{k+1}-x_\star\rangle\ge 0.
    \end{equation}
    Therefore
    \begin{eqnarray*}
        \|x_{k+1}-x_\star\|^2 &=& \|x_k-x_\star-(x_k-x_{k+1})\|^2 \\
        \alignedrel{\eqref{eq:gen_decomp}}{=}&
        \|x_k-x_\star\|^2 - 2\langle x_k-x_{k+1},x_{k+1}-x_\star\rangle -\|x_k-x_{k+1}\|^2\\
        \alignedrel{\eqref{eq:bgd-old-fullstep-case1}+\eqref{eq:bgd-old-innerprod-case1}}{\le}& \|x_k-x_\star\|^2-t_k^2.
    \end{eqnarray*}
    So, inequality {\em (iii)} holds in this case.

    \paragraph{Case 2: $\lambda_k=0$.}
    Then \eqref{eq:bgd-old-kkt} becomes $ \nabla f(x_k)+n_k=0 $, which means that
    \[
        0\in \nabla f(x_k)+N_{\cX}(x_{k+1}).
    \]
    Therefore, $x_{k+1}$ minimizes the linear function $z\mapsto \langle \nabla f(x_k),z\rangle$ over all of $\cX$. 
    In particular, it must be the case that
    $
        \langle \nabla f(x_k),x_{k+1}-x_\star\rangle\le 0.
    $
    Combined with \eqref{eq:bgd-old-gk-halfspace}, this yields
    \begin{equation}
        \label{eq:bgd-old-gk-zero}
        \langle \nabla f(x_k),x_{k+1}-x_\star\rangle=0.
    \end{equation}

    Likewise, from $\nabla f(x_k)+n_k=0$ and the normal-cone inequality
    $
        \langle n_k,x_{k+1}-x_\star\rangle\ge 0
    $
    (recall \eqref{eq:bgd-old-normalcone}), we also get
    \begin{equation*}
        \langle n_k,x_{k+1}-x_\star\rangle=0.
    \end{equation*}

    Further, expanding \eqref{eq:bgd-old-gk-zero}, we obtain
    \[
        0 = \langle \nabla f(x_k),x_k-x_\star\rangle + \langle \nabla f(x_k),x_{k+1}-x_k\rangle.
    \]
    Using \eqref{eq:radius-condition-1} and the Cauchy--Schwarz inequality, we arrive at
    \begin{eqnarray*}
        0 &=& \langle \nabla f(x_k),x_k-x_\star\rangle + \langle \nabla f(x_k),x_{k+1}-x_k\rangle\\
        \alignedrel{\eqref{eq:radius-condition-1}}{\ge}&
        \|\nabla f(x_k)\|\,t_k-\|\nabla f(x_k)\|\,\|x_{k+1}-x_k\|.
    \end{eqnarray*}
    Hence
    $
        \|x_{k+1}-x_k\| \ge t_k.
    $
    However, since $x_{k+1}\in \mathbb{B}_k$, we also have $\|x_{k+1}-x_k\|\le t_k$. Therefore,
    \begin{equation}
        \label{eq:bgd-old-fullstep-case2}
        \|x_{k+1}-x_k\|=t_k,
    \end{equation}
    and hence identity {\em (ii)} holds in this case also.

    Moreover, equality in Cauchy--Schwarz implies that $x_{k+1}-x_k$ is a multiple of $-\nabla f(x_k)$. Since $\|x_{k+1}-x_k\|=t_k$, we must therefore have
    \[
        x_{k+1}-x_k=-\,t_k\,\frac{\nabla f(x_k)}{\|\nabla f(x_k)\|}.
    \]
    Since $\langle \nabla f(x_k),x_{k+1}-x_\star\rangle=0$ (see \eqref{eq:bgd-old-gk-zero}), it follows that
    \begin{equation}
        \label{eq:bgd-old-innerprod-case2}
        \langle x_k-x_{k+1},x_{k+1}-x_\star\rangle=0.
    \end{equation}
    Therefore,
    \begin{eqnarray*}
        \|x_{k+1}-x_\star\|^2 \alignedrel{\eqref{eq:gen_decomp}}{=}& \|x_k-x_\star\|^2 - 2\langle x_k-x_{k+1},x_{k+1}-x_\star\rangle -\|x_k-x_{k+1}\|^2\\
        \alignedrel{\eqref{eq:bgd-old-fullstep-case2}+\eqref{eq:bgd-old-innerprod-case2}}{=}& \|x_k-x_\star\|^2-t_k^2.
    \end{eqnarray*}
    So, inequality {\em (iii)} holds in this case also.
\end{proof}

\subsubsection{Type-II admissibility case} \label{sec:proof-II}

\begin{proof}
    By the radius condition \eqref{eq:radius-condition-2}, and the Cauchy--Schwarz inequality, we have
    \[
        t_k \overset{\eqref{eq:radius-condition-2}}{\le} \frac{\langle\nabla f(x_k)- \nabla f(x_\star),x_k-x_\star\rangle}{\|\nabla f(x_k)- \nabla f(x_\star)\|} \leq \frac{\|\nabla f(x_k)- \nabla f(x_\star)\|\,\|x_k-x_\star\|}{\|\nabla f(x_k)- \nabla f(x_\star)\|} = \|x_k-x_\star\|,
    \]
    establishing claim {\em (i)}. 
    Let us now establish claims {\em (ii)} and {\em (iii)}. 
    Fix $k\ge 0$, and for the purposes of this proof, let us use the shorter notation $\mathbb{B}_k \eqdef \mathbb{B}(x_k,t_k)$.

    \paragraph{Step 1: Perturbed monotonicity.}
    Using the trivial decomposition 
    \[
        x_{k+1}-x_{\star}=(x_k-x_{\star})+(x_{k+1}-x_k),
    \]
    we can write
    \begin{eqnarray} 
        \label{eq:-9-fyd8fdg}
        &&\hspace{-8mm} \langle \nabla f(x_{k}) - \nabla f(x_{\star}),x_{k+1}-x_{\star}\rangle \nonumber \\
        &=& \langle \nabla f(x_{k}) - \nabla f(x_{\star}),x_k-x_{\star}\rangle + \langle \nabla f(x_{k}) - \nabla f(x_{\star}),x_{k+1}-x_k\rangle.
    \end{eqnarray}

    Since $x_{k+1} \in \mathbb{B}_k$, we have $\|x_{k+1}-x_k\|\le t_k$. 
    Using the Cauchy--Schwarz inequality, we can therefore bound
    \begin{eqnarray} 
        \langle \nabla f(x_{k}) - \nabla f(x_{\star}),x_{k+1}-x_k\rangle &\geq& -\|\nabla f(x_{k}) - \nabla f(x_{\star})\|\,\|x_{k+1}-x_k\| \notag \\
        & \geq &  -\|\nabla f(x_{k}) - \nabla f(x_{\star})\|\,t_k.\label{eq:B_k}
    \end{eqnarray}
    Plugging this bound back into \eqref{eq:-9-fyd8fdg}, and using the radius bound \eqref{eq:radius-condition-2}, we get
    \begin{eqnarray}
        \langle \nabla f(x_{k}) - \nabla f(x_{\star}),x_{k+1}-x_{\star}\rangle
        \alignedrel{\eqref{eq:-9-fyd8fdg},\eqref{eq:B_k}}{\geq} &
        \langle \nabla f(x_{k}) - \nabla f(x_{\star}),x_k-x_{\star}\rangle \notag \\
        && {}-\|\nabla f(x_{k}) - \nabla f(x_{\star})\|t_k \notag \\
        \alignedrel{\eqref{eq:radius-condition-2}}{\geq} & 0.\label{eq:r-halfspace-fixed}
    \end{eqnarray}

    \paragraph{Step 2: Applying the first-order optimality condition.}
    Since $x_{\star}$ minimizes $f$ over $\cX$, the first-order optimality condition for \eqref{eq:main} gives
    $
        \langle \nabla f(x_{\star}),z-x_{\star}\rangle \ge 0,
    $
    for all $z\in \cX$.
    Since $x_{k+1}\in \cX$, this implies
    \begin{equation}
        \label{eq:gstar-halfspace-fixed}
        \langle \nabla f(x_{\star}),x_{k+1}-x_{\star}\rangle\ge 0.
    \end{equation}

    \paragraph{Step 3: Combining Steps 1 and 2.}
    Adding \eqref{eq:r-halfspace-fixed} and \eqref{eq:gstar-halfspace-fixed}, we obtain
    \begin{equation}
        \label{eq:gk-halfspace-fixed}
        \langle \nabla f(x_{k}),x_{k+1}-x_{\star}\rangle = \langle \nabla f(x_{k}) - \nabla f(x_{\star}),x_{k+1}-x_{\star}\rangle + \langle \nabla f(x_{\star}),x_{k+1}-x_{\star}\rangle \ge 0.
    \end{equation}

    \paragraph{Step 4: KKT conditions for a constrained linear optimization problem.}
    Since $x_{k+1}$ minimizes the linear functional $z \mapsto \langle \nabla f(x_{k}), z\rangle $ over the nonempty closed convex set $\cX\cap \mathbb{B}_k$, we have
    \[
        0\in \nabla f(x_{k})+N_{\cX}(x_{k+1})+N_{\mathbb{B}_k}(x_{k+1}),
    \]
    where $N_{\cC}(z)$ denotes the normal cone of set $\cC$ at point $z$ .
    Hence there exist
    \[
        n_k\in N_{\cX}(x_{k+1}), \qquad \lambda_k\ge 0
    \]
    such that
    \begin{equation}
        \label{eq:kkt-fixed}
        \nabla f(x_{k})+n_k+\lambda_k(x_{k+1}-x_k)=0.
    \end{equation}

    \paragraph{Step 5: Deductions from the KKT conditions.}
    Taking the inner product of the zero vector from \eqref{eq:kkt-fixed} with $x_{k+1}-x_{\star}$ gives the identity
    \begin{equation}
        \langle \nabla f(x_{k}),x_{k+1}-x_{\star}\rangle + \langle n_k,x_{k+1}-x_{\star}\rangle + \lambda_k\langle x_{k+1}-x_k,x_{k+1}-x_{\star}\rangle =0. 
        \label{eq:3-term=0-9u0f9df}
    \end{equation}
    Note that due to \eqref{eq:gk-halfspace-fixed}, the first term in \eqref{eq:3-term=0-9u0f9df} is nonnegative. 
    We shall now argue that the second term is nonnegative, too. 
    Indeed, since $n_k\in N_{\cX}(x_{k+1})$ and $x_{\star}\in \cX$, we must have
    \begin{equation} 
        \label{eq:nci_098fu980df} 
        \langle n_k,x_{k+1}-x_{\star}\rangle\ge 0.
    \end{equation}
    Since the sum of the three terms is zero, the third term in \eqref{eq:3-term=0-9u0f9df} must be nonpositive, which can be equivalently written as
    \begin{equation}
        \lambda_k\langle x_k-x_{k+1},x_{k+1}-x_{\star}\rangle \ge 0. 
        \label{eq:lambda_inner_nonnegative_98f8df}
    \end{equation}

    \paragraph{Step 6: Two cases.}
    We now split the argument into two cases.

    \paragraph{Case 1: $\lambda_k>0$.}
    Then $x_{k+1}$ lies on the boundary of the ball $\mathbb{B}_k$, and thus
    \begin{equation} 
        \label{eq:08fdu0fdg}
        \|x_{k+1}-x_k\|=t_k.
    \end{equation}
    So, identity {\em (ii)} holds in this case.
    Moreover, \eqref{eq:lambda_inner_nonnegative_98f8df} implies that
    \begin{equation}
        \langle x_k-x_{k+1},x_{k+1}-x_{\star}\rangle\ge 0. 
        \label{eq:o=f-98ygf08gfdd}
    \end{equation}
    Therefore
    \begin{eqnarray*}
        \|x_{k+1}-x_{\star}\|^2 &=& \|x_k-x_{\star}-(x_k-x_{k+1})\|^2 \\
        \alignedrel{\eqref{eq:gen_decomp}}{=}& \|x_k-x_{\star}\|^2 - 2\langle x_k-x_{k+1},x_{k+1}-x_{\star}\rangle - \|x_k-x_{k+1}\|^2\\
        \alignedrel{\eqref{eq:08fdu0fdg} + \eqref{eq:o=f-98ygf08gfdd}}{\le} & \|x_k-x_{\star}\|^2-t_k^2.
    \end{eqnarray*}
    So, inequality {\em (iii)} holds in this case.

    \paragraph{Case 2: $\lambda_k=0$.}
    Then \eqref{eq:kkt-fixed} becomes $\nabla f(x_{k}) + n_k=0$, which means that 
    \[
        0\in \nabla f(x_{k})+N_{\cX}(x_{k+1}).
    \]
    Therefore, $x_{k+1}$ minimizes the linear function $z\mapsto \langle \nabla f(x_{k}),z\rangle$ over all of $\cX$. In particular, it must be the case that
    $
        \langle \nabla f(x_{k}),x_{k+1}-x_{\star}\rangle \le 0.
    $
    Combined with \eqref{eq:gk-halfspace-fixed}, this yields
    \begin{equation}
        \label{eq:gk_09ufd89dfg_09gf}
        \langle \nabla f(x_{k}),x_{k+1}-x_{\star}\rangle = 0.
    \end{equation}

    Likewise, from $\nabla f(x_{k})+n_k=0$ and the normal-cone inequality 
    $
        \langle n_k,x_{k+1}-x_{\star}\rangle \ge 0
    $
    (recall \eqref{eq:nci_098fu980df}),
    we also get
    $
        \langle n_k,x_{k+1}-x_{\star}\rangle = 0.
    $
    Further,
    \[
        0\overset{\eqref{eq:gk_09ufd89dfg_09gf}}{=}\langle \nabla f(x_{k}),x_{k+1}-x_{\star}\rangle = \langle \nabla f(x_{k}) - \nabla f(x_{\star}), x_{k + 1}-x_{\star}\rangle + \langle \nabla f(x_{\star}),x_{k+1}-x_{\star}\rangle.
    \]
    Both terms on the right are nonnegative by \eqref{eq:r-halfspace-fixed} and \eqref{eq:gstar-halfspace-fixed}, so each must vanish. In particular, 
    \begin{equation} 
        \label{eq:=-9=dug(*&*908f}
        \langle \nabla f(x_{k}) - \nabla f(x_{\star}),x_{k+1}-x_{\star}\rangle=0.
    \end{equation}
    Expanding \eqref{eq:=-9=dug(*&*908f}, using \eqref{eq:radius-condition-2} and the Cauchy--Schwarz inequality, we arrive at
    \begin{eqnarray*}
        0 \alignedrel{\eqref{eq:=-9=dug(*&*908f}}{=}&
        \langle \nabla f(x_{k}) - \nabla f(x_{\star}),x_k-x_{\star}\rangle + \langle \nabla f(x_{k}) - \nabla f(x_{\star}),x_{k+1}-x_k\rangle \\
        \alignedrel{\eqref{eq:radius-condition-2}}{\ge} &
        \|\nabla f(x_{k}) - \nabla f(x_{\star})\|\,t_k-\|\nabla f(x_{k}) - \nabla f(x_{\star})\|\,\|x_{k+1}-x_k\|.
    \end{eqnarray*}
    Since we assume that $\nabla f(x_{k}) \neq \nabla f(x_{\star})$, this implies $\|x_{k+1}-x_k\| \geq t_k$.
    However, since $x_{k+1}\in \mathbb{B}_k$, we also have $\|x_{k+1}-x_k\|\le t_k$. Therefore,
    \begin{equation} 
        \label{eq:(*98y98f_08yf98dfd}
        \|x_{k+1}-x_k\|=t_k,
    \end{equation}
    and hence identity {\em (ii)} holds in this case also.
    Moreover, equality in Cauchy--Schwarz implies that $x_{k+1}-x_{k}$ is a multiple of $\nabla f(x_{\star}) - \nabla f(x_{k})$. Since $\|x_{k+1} - x_k\|=t_k$, we must therefore have
    \[
        x_{k+1}-x_k = -\,t_k\,\frac{\nabla f(x_{k}) - \nabla f(x_{\star})}{\|\nabla f(x_{k}) - \nabla f(x_{\star})\|}.
    \]
    Since $\langle \nabla f(x_{k}) - \nabla f(x_{\star}),x_{k+1}-x_{\star}\rangle=0$ (see \eqref{eq:=-9=dug(*&*908f}), it follows that
    \begin{equation}
        \label{eq:09u0f9u0_0fu0d98ff}
        \langle x_k-x_{k+1},x_{k+1}-x_{\star}\rangle=0.
    \end{equation}
    Therefore,
    \begin{eqnarray*}
        \|x_{k+1}-x_{\star}\|^2 \alignedrel{\eqref{eq:gen_decomp}}{=}& \|x_k-x_{\star}\|^2 - 2\langle x_k-x_{k+1},x_{k+1}-x_{\star}\rangle - \|x_k-x_{k+1}\|^2\\
        \alignedrel{\eqref{eq:(*98y98f_08yf98dfd}+\eqref{eq:09u0f9u0_0fu0d98ff}}{=}& \|x_k-x_{\star}\|^2-t_k^2.
    \end{eqnarray*}
    So, inequality {\em (iii)} holds in this case also.
\end{proof}
\mdfsetup{nobreak=true}
\subsection{{Setting and separation for practical schedules}}
\begin{assumption}[Setting for the practical schedules]\label{ass:practical-setting}
The set $\cX\subseteq\R^d$ is nonempty, closed, and convex. The function $f$ is convex and differentiable on an open convex set containing $\cX$, and has a constrained minimizer $x_\star\in\cX$. All norms in this appendix are Euclidean. Write $f_\star=f(x_\star)$, $\Delta_k=f(x_k)-f_\star$, and $R_0=\|x_0-x_\star\|$, where $x_0\in\cX$. For a specified positive radius $t_k$, use
\begin{equation}\label{eq:practical-update}
 x_{k+1}\in\argmin_{z\in\cX\cap\mathbb{B}(x_k,t_k)}\langle\nabla f(x_k),z\rangle.
\end{equation}
The oracle may return any minimizer. The feasible local set is compact and nonempty, so the update exists.
\end{assumption}

The next lemma isolates the geometry needed for the two accuracy-based rules. Its hypothesis will hold until an accurate iterate has appeared; it need not hold afterwards.
\begin{lemma}[Distance decrease from a separating linear objective]\label{lem:practical-separation}
Let $\cX$ be nonempty, closed, and convex, let $x,u\in\cX$, $t>0$, and
\begin{equation}\label{eq:practical-local-q}
 z\in\argmin_{y\in\cX\cap\mathbb{B}(x,t)}\langle g,y\rangle.
\end{equation}
If $\langle g,u-z\rangle<0$, then
\begin{equation}\label{eq:practical-distance}
 \|z-x\|=t,\qquad \|z-u\|^2\le\|x-u\|^2-t^2.
\end{equation}
\end{lemma}
\begin{proof}
If $\|z-x\|<t$, a sufficiently short segment from $z$ toward $u$ stays in the ball and in $\cX$, and strictly decreases the linear objective, contradicting \eqref{eq:practical-local-q}. Thus $\|z-x\|=t$.
Define the closed convex set $H=\{y\in\cX:\langle g,y-z\rangle\le0\}$. It contains both $u$ and $z$. No $y\in H$ can satisfy $\|y-x\|<t$: if its linear objective is strictly smaller than that of $z$, it contradicts \eqref{eq:practical-local-q} directly; if equal, a short move from $y$ toward $u$ gives that contradiction. Consequently $z$ is a nearest point in $H$ to $x$. For $s\in[0,1]$, the point $z+s(u-z)$ belongs to $H$, so minimality gives
\[
 0\le\|z+s(u-z)-x\|^2-\|z-x\|^2
   =2s\langle z-x,u-z\rangle+s^2\|u-z\|^2.
\]
Divide by $s>0$ and let $s$ decrease to zero. Hence $\langle x-z,u-z\rangle\le0$, and the Euclidean norm identity yields
\[
 \|x-u\|^2
 =\|x-z\|^2+\|z-u\|^2+2\langle x-z,z-u\rangle
 \ge t^2+\|z-u\|^2.
\]
This proves \eqref{eq:practical-distance}. Projection is used only to interpret this proof, not as an algorithmic operation.
\end{proof}

\mdfsetup{nobreak=false}

\subsection{{Affine-subspace oracle and connection to gradient descent}} \label{sec:GD-subspace}

Assume that the feasible set is an affine subspace of $\R^d$, i.e., $\cX = a+\cV$, where $a\in \R^d$ and $\cV\subseteq \R^d$ is a linear subspace. 
We have the following result.

\begin{proposition}[Explicit form of {\algname{BroxGD}} on an affine subspace]
\label{prop:affine-subspace-explicit}
    Assume that $\cX = a+\cV$,
    where $a\in \R^d$ and $\cV\subseteq \R^d$ is a linear subspace. 
    Let $\Proj_{\cV}:\R^d\to \cV$ denote the orthogonal projector onto $\cV$. Consider one step of \Cref{alg:new}, i.e.,
    \begin{equation}
        \label{eq:affine-subspace-subproblem}
        x_{k+1}\in \argmin_{z\in \cX\cap \mathbb{B}(x_k,t_k)} \langle \nabla f(x_k),z\rangle,
    \end{equation}
    where $x_k\in \cX$ and $t_k>0$. {Impose the tie-breaking rule $x_{k+1}=x_k$ whenever $\Proj_{\cV}\nabla f(x_k)=0$.} Then
    \[
        x_{k+1} = \begin{cases}
            x_k - t_k \dfrac{\Proj_{\cV} \nabla f(x_k)}{\|\Proj_{\cV} \nabla f(x_k)\|}, & \text{if } \Proj_{\cV}\nabla f(x_k)\neq 0,\\[3mm]
            x_k, & \text{if } \Proj_{\cV}\nabla f(x_k)=0.
        \end{cases}
    \]
    In particular, if $\Proj_{\cV}\nabla f(x_k)\neq 0$, then the step is unique.
\end{proposition}

\begin{proof}
    Since $x_k \in \cX = a+\cV$, every point {$z\in \cX\cap\mathbb{B}(x_k,t_k)$} can be written uniquely in the form $z=x_k+s$, $s \in \cV$, $\norm{s} \leq t_k$.
    Therefore the subproblem \eqref{eq:affine-subspace-subproblem} can be written as
    \[
        \min_{s\in \cV,\ \norm{s}\le t_k}
        \inner{\nabla f(x_k)}{x_k+s}.
    \]
    Since the term $\inner{\nabla f(x_k)}{x_k}$ does not depend on $s$, this is the same as
    \[
        \min_{s\in \cV,\ \norm{s}\le t_k}
        \inner{\nabla f(x_k)}{s}.
    \]
    Now, for every $s\in \cV$, we have
    \[
        \inner{\nabla f(x_k)}{s} = \inner{\Proj_{\cV}\nabla f(x_k)}{s},
    \]
    because $s\in \cV$ and $\nabla f(x_k)-\Proj_{\cV}\nabla f(x_k)\in \cV^\perp$. 
    Hence the problem further reduces to
    \[
        \min_{s\in \cV,\ \norm{s}\le t_k} \inner{\Proj_{\cV}\nabla f(x_k)}{s}.
    \]
    If $\Proj_{\cV}\nabla f(x_k)=0$, then the objective is identically zero over the feasible set. 
    Hence every $s\in \cV$ satisfying $\norm{s}\le t_k$ is optimal. 
    Equivalently, every point in $ x_k+\{s\in \cV:\norm{s}\le t_k\} $ is a solution of~\eqref{eq:affine-subspace-subproblem}. 
    {The imposed tie-breaking rule selects $s=0$, hence $x_{k+1}=x_k$.}
    Assume now that $\Proj_{\cV}\nabla f(x_k)\neq 0$. 
    Then by the Cauchy--Schwarz inequality,
    \[
        \inner{\Proj_{\cV}\nabla f(x_k)}{s} \ge -\norm{\Proj_{\cV}\nabla f(x_k)}\norm{s} \ge -\norm{\Proj_{\cV}\nabla f(x_k)}t_k
    \]
    for every feasible $s$. 
    Equality holds if and only if
    \[
        \norm{s}=t_k \qquad\text{and}\qquad s = - t_k \frac{\Proj_{\cV}\nabla f(x_k)}{\norm{\Proj_{\cV}\nabla f(x_k)}}.
    \]
    Therefore,
    \[
        x_{k+1} = x_k - t_k \frac{\Proj_{\cV}\nabla f(x_k)}{\norm{\Proj_{\cV}\nabla f(x_k)}}. \qedhere
    \]
\end{proof}

The above proposition shows that, when $\cX$ is an affine subspace, the method in \Cref{alg:new} coincides with normalized \algname{GD} on the restriction of $f$ to that affine subspace whenever the projected gradient is nonzero. 
At points where the projected gradient vanishes, {the imposed tie-breaking rule gives $x_{k+1}=x_k$, agreeing with the stationary \algname{GD} update.}
Indeed, consider the restricted problem
\[
    \min_{x\in a + \cV} f(x).
\]
The tangent space of the feasible set is $\cV$, and the gradient of the restricted objective is precisely the projection of the ambient gradient onto this tangent space:
\[
    \nabla_{\cX} f(x)=\Proj_{\cV}\nabla f(x).
\]
Therefore, whenever $\Proj_{\cV}\nabla f(x_k)\neq 0$, the update from \Cref{prop:affine-subspace-explicit} can be rewritten as
\[
    x_{k+1} = x_k - t_k \frac{\nabla_{\cX} f(x_k)} {\norm{\nabla_{\cX} f(x_k)}}.
\]
Thus the method moves in exactly the same direction as \algname{GD} for minimizing $f$ over the affine subspace $\cX$, namely the negative projected gradient direction, but with step length prescribed directly by the radius $t_k$.
For comparison, ordinary \algname{GD} on the restricted problem would take the form
\[
    x_{k+1}^{\algnamemath{GD}} = x_k-{\gamma_k}\Proj_{\cV}\nabla f(x_k) = x_k - {\gamma_k}\nabla_{\cX} f(x_k),
\]
for some scalar stepsize ${\gamma_k} > 0$. 
Hence, whenever $\Proj_{\cV}\nabla f(x_k)\neq 0$, the two methods differ only in how the step length is parameterized. 
In fact, the update from \Cref{alg:new} is exactly a \algname{GD} step with effective stepsize
\[
    {\gamma_k} = \frac{t_k}{\norm{\Proj_{\cV}\nabla f(x_k)}} = \frac{t_k}{\norm{\nabla_{\cX} f(x_k)}}.
\]
This discussion can be summarized formally as follows.

\begin{remark}[Connection with \algname{GD} on the affine subspace]
\label{rem:affine-subspace-gd}
    When $\cX=a+\cV$ is an affine subspace, {under the tie-breaking rule of \Cref{prop:affine-subspace-explicit},} \Cref{alg:new} coincides with normalized {\small \sf \textit{\algname{GD}}} applied to the restriction of $f$ to $\cX$. 
    More precisely, if $\Proj_{\cV}\nabla f(x_k)\neq 0$, then
    \[
        x_{k+1} = x_k - t_k \frac{\nabla_{\cX} f(x_k)}{\|\nabla_{\cX} f(x_k)\|}, \qquad \nabla_{\cX} f(x_k)=\Proj_{\cV}\nabla f(x_k).
    \]
    Equivalently, the same step can be written as an ordinary \algname{GD} step
    \[
        x_{k+1} = x_k - {\gamma_k}\,\nabla_{\cX} f(x_k), \qquad {\gamma_k}=\frac{t_k}{\|\nabla_{\cX} f(x_k)\|}.
    \]
    Hence, in the affine subspace case, \Cref{alg:new} is exactly \algname{GD} on the restricted problem $f|_{\cX}$, expressed in trust-region form.
\end{remark}

Finally, the unconstrained case appears as the special case $\cV=\R^d$, in which $\Proj_{\cV}= \operatorname{Id}$ is the identity mapping. {When $\nabla f(x_k)\neq0$,} the formula reduces to
\[
    x_{k+1} = x_k - t_k \frac{\nabla f(x_k)}{\|\nabla f(x_k)\|},
\]
which is precisely normalized \algname{GD} in the ambient space. {When $\nabla f(x_k)=0$, the tie-breaking rule keeps $x_{k+1}=x_k$.}

%% file: appendix-table-proofs.tex
\clearpage
\section{{Proofs of the convergence guarantees in Table 1}}
\label{sec:table-proofs}
{The four subsections and their numbered entries follow the corresponding groups and rate rows of \Cref{tbl:main}. Group B places the bounded-gradient, bounded-subgradient, and smooth convex Polyak results together, before the strongly convex distance benchmark. Shared geometric tools appear in \Cref{sec:shared-geometry}. Main-text theorems are restated with their original numbers; additional formulations and supporting proofs are placed beside the corresponding row.}
\mdfsetup{nobreak=true}
\input{appendix-practical-radii}

\mdfsetup{nobreak=false}

\input{appendix-table-benchmarks}

\input{appendix-table-residuals}

\input{appendix-table-nonconvex}

%% file: appendix-practical-radii.tex
\subsection{{A. Practical objective guarantees}}
\label{sec:practical-radii}
\label{sec:table-proof-A}
{This subsection follows group A of \Cref{tbl:main}: fixed schedules first, then radius backtracking. We restate the combined main-text horizon theorem before giving the two cases and their proofs in table order. The shared setting is \Cref{ass:practical-setting}.}
{\color{black}\mdfsetup{fontcolor=black}\MAINHORIZON*}
\subsubsection{{Fixed horizon with bounded gradients}}
\begin{corollary}[Fixed-horizon bounded-gradient schedule]\label{cor:practical-horizon}
Under the assumptions of \Cref{thm:practical-target}, let $R\ge R_0$ be known, $R>0$, and let $K\ge1$ be the number of oracle calls. Set
\begin{equation}\label{eq:practical-horizon-rule}
 {\varepsilon}=\frac{GR}{\sqrt K},\qquad
 t_k=\frac{GR}{\sqrt K\,\|\nabla f(x_k)\|}.
\end{equation}
Return an iterate ${x_K^{\mathrm{best}}}$ of smallest objective value among $x_0,\ldots,x_K$, or the point at which a zero gradient is detected. Then
\begin{equation}\label{eq:practical-horizon-rate}
 f({x_K^{\mathrm{best}}})-f_\star\le\frac{GR}{\sqrt K}.
\end{equation}
\end{corollary}

{We establish the target-accuracy result below, then derive this fixed-horizon guarantee.}

\begin{theorem}[Target accuracy without the optimal value]\label{thm:practical-target}
Let \Cref{ass:practical-setting} hold and suppose
\begin{equation}\label{eq:practical-gradient-bound}
 \|\nabla f(x)\|\le G\qquad(x\in\cX),\qquad G>0.
\end{equation}
Given ${\varepsilon}>0$, use
\begin{equation}\label{eq:practical-target-rule}
 t_k=\frac{{\varepsilon}}{\|\nabla f(x_k)\|}
\end{equation}
whenever the gradient is nonzero; otherwise stop and return $x_k$, which is optimal. At every iteration with $\Delta_k>{\varepsilon}$,
\begin{equation}\label{eq:practical-target-decrease}
 \|x_{k+1}-x_\star\|^2\le\|x_k-x_\star\|^2-\frac{{\varepsilon}^2}{G^2}.
\end{equation}
For every integer $K>G^2R_0^2/{\varepsilon}^2$, the best of $x_0,\ldots,x_{K-1}$ has gap at most ${\varepsilon}$. With the endpoint included, the same conclusion holds for every integer $K\ge G^2R_0^2/{\varepsilon}^2$:
\begin{equation}\label{eq:practical-target-best}
 \min_{0\le k\le K}\Delta_k\le{\varepsilon}.
\end{equation}
Early termination satisfies both conclusions.
\end{theorem}
\begin{proof}
Let $g_k=\nabla f(x_k)$. A zero gradient certifies optimality by convexity. Otherwise, convexity, Cauchy--Schwarz, and feasibility give
\begin{align*}
 \langle g_k,x_\star-x_{k+1}\rangle
 &=\langle g_k,x_\star-x_k\rangle+\langle g_k,x_k-x_{k+1}\rangle\\
 \alignedrel{\eqref{eq:practical-update}}{\le}-\Delta_k+\|g_k\|t_k
 \overset{\eqref{eq:practical-target-rule}}{=}-\Delta_k+{\varepsilon}.
\end{align*}
If $\Delta_k>{\varepsilon}$, this expression is negative, so \Cref{lem:practical-separation} applies with $u=x_\star$:
\[
 \|x_{k+1}-x_\star\|^2
 \overset{\eqref{eq:practical-distance}}{\le}\|x_k-x_\star\|^2-t_k^2
 \overset{\eqref{eq:practical-target-rule},\eqref{eq:practical-gradient-bound}}{\le}
 \|x_k-x_\star\|^2-{\varepsilon}^2/G^2.
\]
If all queried iterates $x_0,\ldots,x_{K-1}$ have gap greater than ${\varepsilon}$, summing gives
\begin{equation}\label{eq:practical-target-sum}
 0\le\|x_K-x_\star\|^2
 \overset{\eqref{eq:practical-target-decrease}}{\le}R_0^2-K{\varepsilon}^2/G^2.
\end{equation}
For $K>G^2R_0^2/{\varepsilon}^2$, this is impossible. At equality, it forces $x_K=x_\star$, which proves the endpoint-inclusive assertion as well. This handles the distinction between strict and non-strict iteration budgets.
\end{proof}

\begin{proof}[{Proof of \Cref{cor:practical-horizon}}]
The choice of ${\varepsilon}$ gives
\[
 \frac{G^2R_0^2}{{\varepsilon}^2}
 \overset{\eqref{eq:practical-horizon-rule}}{=}K\frac{R_0^2}{R^2}\le K.
\]
Apply the endpoint-inclusive part of \Cref{thm:practical-target}; the resulting bound is
\[
 f({x_K^{\mathrm{best}}})-f_\star
 \overset{\eqref{eq:practical-target-best}}{\le}{\varepsilon}
 \overset{\eqref{eq:practical-horizon-rule}}{=}GR/\sqrt K.
\]
If $R=0$ is certified, $x_0=x_\star$ and no call is needed.
\end{proof}

To run the target-accuracy rule one needs only ${\varepsilon}$ and the observed gradient norm. A predetermined budget additionally needs $G$ and a bound on $R_0$; for example, $\lfloor G^2R^2/{\varepsilon}^2\rfloor+1$ calls suffice even when only the queried points are retained. These statements do not assert last-iterate accuracy after the target has first been reached, or convergence to arbitrary precision for a single fixed positive ${\varepsilon}$.
\subsubsection{{Fixed horizon for smooth convex objectives}}
\begin{corollary}[Fixed-horizon smooth convex schedule]\label{cor:practical-smooth-horizon}
Under the assumptions of \Cref{thm:practical-constant}, let $R\ge R_0$, $R>0$, and $K\ge1$. Use $t_k=R/\sqrt K$ for $K$ calls and return the best of $x_0,\ldots,x_K$. Then
\begin{equation}\label{eq:practical-smooth-horizon}
 f({x_K^{\mathrm{best}}})-f_\star\le\frac{LR^2}{2K}.
\end{equation}
\end{corollary}

{We establish the target-accuracy result below, then derive this fixed-horizon guarantee.}

\begin{theorem}[Constant radius and objective accuracy]\label{thm:practical-constant}
Let \Cref{ass:practical-setting} hold. Suppose $f$ has an $L$-Lipschitz gradient on its open convex domain, with known $L>0$. Fix $\varepsilon>0$ and use
\begin{equation}\label{eq:practical-constant-rule}
 t_k\equiv t=\sqrt{\frac{2\varepsilon}{L}}.
\end{equation}
For every integer $K\ge LR_0^2/(2\varepsilon)$, the exact full-step iterates satisfy
\begin{equation}\label{eq:practical-constant-rate}
 \min_{0\le k\le K}\{f(x_k)-f_\star\}\le\varepsilon.
\end{equation}
Equivalently, a radius $t=\sqrt{2/L}\,c$ with $c>0$ achieves best-iterate gap at most $c^2$ after any integer $K\ge LR_0^2/(2c^2)$.
\end{theorem}
\begin{proof}
If $\Delta_0\le\varepsilon$ there is nothing to prove. At any subsequent step, smoothness and convexity imply
\begin{align}
 \langle\nabla f(x_k),x_\star-x_{k+1}\rangle
 &\le f_\star-f(x_k)-\langle\nabla f(x_k),x_{k+1}-x_k\rangle\notag\\
 &\le-\Delta_{k+1}+\frac L2\|x_{k+1}-x_k\|^2\notag\\
 \alignedrel{\eqref{eq:practical-update},\eqref{eq:practical-constant-rule}}{\le}
       -\Delta_{k+1}+\varepsilon.\label{eq:practical-smooth-separation}
\end{align}
If $\Delta_{k+1}>\varepsilon$, \eqref{eq:practical-smooth-separation} is strictly negative, so \Cref{lem:practical-separation} gives
\begin{equation}\label{eq:practical-constant-decrease}
 \|x_{k+1}-x_\star\|^2
 \overset{\eqref{eq:practical-distance}}{\le}\|x_k-x_\star\|^2-t^2.
\end{equation}
Suppose, for contradiction, that every iterate $x_0,\ldots,x_K$ has gap greater than $\varepsilon$. Summing over all $K$ steps yields
\[
 0\le\|x_K-x_\star\|^2
 \overset{\eqref{eq:practical-constant-decrease}}{\le}R_0^2-Kt^2
 \overset{\eqref{eq:practical-constant-rule}}{=}R_0^2-2K\varepsilon/L\le0.
\]
If the right side is negative there is a contradiction; if it is zero, $x_K=x_\star$, again contradicting the assumed gap. The case $K=0$ permitted by the budget has $R_0=0$ and is immediate. This proves \eqref{eq:practical-constant-rate}. Substituting $\varepsilon=c^2$ proves the equivalent formulation.
\end{proof}

\begin{proof}[{Proof of \Cref{cor:practical-smooth-horizon}}]
Set $\varepsilon=LR^2/(2K)$. Then
\[
 \sqrt{2\varepsilon/L}=R/\sqrt K,\qquad
 LR_0^2/(2\varepsilon)=K R_0^2/R^2\le K.
\]
The hypotheses and budget of \Cref{thm:practical-constant} are satisfied, and
\[
 f({x_K^{\mathrm{best}}})-f_\star
 \overset{\eqref{eq:practical-constant-rate}}{\le}\varepsilon=LR^2/(2K).
\]
\end{proof}

The target-accuracy radius needs only $L$ and $\varepsilon$ to run; a certified stopping budget additionally needs a bound on $R_0$. In the fixed-horizon formulation, the radius itself uses $R$ and $K$, while $L$ enters only the accuracy guarantee. No strict convexity is needed. The constant-radius conclusion controls objective values, rather than only the gradient-difference residual. The distance argument is Euclidean, and no arbitrary-norm extension is claimed here. Both accuracy-based schedules allow objective increases after a good point is reached, which is why the output is the best observed iterate.
\subsubsection{{Predetermined geometric radii}}
\label{sec:geometric-radius-proof}
For convenience, we restate \Cref{thm:practical-geometric}.
\GEOMETRICRADII*

\begin{proof}
Constrained first-order optimality gives $\langle\nabla f(x_\star),x-x_\star\rangle\ge0$ for $x\in\cX$. Strong convexity therefore implies
\begin{equation}\label{eq:practical-growth}
 f(x)-f_\star\ge\frac{\mu}{2}\|x-x_\star\|^2\qquad(x\in\cX).
\end{equation}
We prove the objective bound by induction. It holds at $k=0$. Assuming $\Delta_k\le {C_0}q^k$, define $y_k=(1-\alpha)x_k+\alpha x_\star$. Since $0<\alpha\le1/2$, this point belongs to $\cX$, and
\[
 \|y_k-x_k\|=\alpha\|x_k-x_\star\|
 \overset{\eqref{eq:practical-growth}}{\le}\alpha\sqrt{2\Delta_k/\mu}
 \le\alpha\sqrt{2{C_0}q^k/\mu}
 \overset{\eqref{eq:practical-geometric-rule}}{=}t_k.
\]
The comparator $y_k$ is thus feasible for the local oracle. By oracle optimality followed by convexity,
\begin{equation}\label{eq:practical-comparator}
 \langle\nabla f(x_k),x_{k+1}-x_k\rangle
 \overset{\eqref{eq:practical-update}}{\le}
 \alpha\langle\nabla f(x_k),x_\star-x_k\rangle
 \le-\alpha\Delta_k.
\end{equation}
Smoothness and feasibility of the update give
\begin{align*}
 \Delta_{k+1}
 &\le\Delta_k+\langle\nabla f(x_k),x_{k+1}-x_k\rangle
        +\frac L2\|x_{k+1}-x_k\|^2\\
 \alignedrel{\eqref{eq:practical-comparator},\eqref{eq:practical-update}}{\le}
       (1-\alpha)\Delta_k+\frac L2t_k^2\\
 \alignedrel{\eqref{eq:practical-geometric-rule}}{\le}
       \left(1-\frac{\mu}{2L}+\frac{\mu}{4L}\right){C_0}q^k
 \overset{\eqref{eq:practical-geometric-rule}}{=}{C_0}q^{k+1}.
\end{align*}
The penultimate step also uses the induction hypothesis and $1-\alpha\ge0$. This closes the induction. The distance bound follows from \eqref{eq:practical-growth}. Finally, using $1-u\le e^{-u}$,
\[
 \Delta_K\overset{\eqref{eq:practical-geometric-rate}}{\le}{C_0}q^K
 \overset{\eqref{eq:practical-geometric-rule}}{\le}C_0\exp\!\left(-\frac{\mu K}{4L}\right)
 \overset{\eqref{eq:practical-geometric-budget}}{\le}\varepsilon.
\]
\end{proof}

A known lower bound $\ell\le f_\star$ gives the valid initial bound ${C_0}=f(x_0)-\ell$ whenever this is positive. If it is zero, $x_0$ is already optimal and no step is needed. For a nonnegative loss one can use $\ell=0$. The schedule needs $L,\mu,{C_0}$, but no current optimality gap or distance; its guaranteed contraction is weaker than that of the idealized radius in \Cref{thm:smooth-strong}. Even at an optimal iterate an arbitrary exact oracle may move; the induction above still applies.
\paragraph{{Shared ingredients for the backtracking guarantees}}
\label{sec:practical-backtracking}
{The next two table rows share one algorithm and one certificate argument. We restate \Cref{thm:main-backtracking}, establish those ingredients once, then prove its convex and strongly convex cases consecutively.}
{\color{black}\mdfsetup{fontcolor=black}\MAINBACKTRACKING*}

\input{appendix-radius-backtracking}

%% file: appendix-radius-backtracking.tex
{Unlike the fixed schedules above, backtracking selects its radius from an observable test.} Here we instead choose the radius by an observable acceptance test. The method needs only a known smoothness bound and an exact local linear oracle. In particular, neither an initial objective gap nor a solution-distance bound is an algorithm input. The guarantee concerns the last accepted iterate, but a single accepted step may require several local-oracle calls.

\begin{assumption}[Setting for radius backtracking]\label{ass:radius-backtracking}
The set $\cX\subseteq\R^d$ is nonempty, closed, and convex. The function $f:\R^d\to\R$ is convex and differentiable, its minimum over $\cX$ is attained, and a known constant $L>0$ satisfies
\begin{equation}\label{eq:bt-smoothness}
 \|\nabla f(u)-\nabla f(v)\|\le L\|u-v\|
 \qquad(u,v\in\R^d).
\end{equation}
Start at $x_0\in\cX$ and use an exact local linear minimization oracle. Write $\cX_\star=\argmin_{x\in\cX}f(x)$ and $\Delta_k=f(x_k)-f_\star$.
\end{assumption}

\paragraph{Radius policy.}
{Fix an input contraction factor $\xi\in(0,1)$.} At $x_k$, compute $a_k=\nabla f(x_k)$ and stop if $a_k=0$. Otherwise initialize $t_{k,0}=\|a_k\|/L$. For $j=0,1,\ldots$, compute
\begin{equation}\label{eq:bt-trial}
 \begin{gathered}
 t_{k,j}={\xi^j}\frac{\|a_k\|}{L},\qquad
 y_{k,j}\in\argmin_{z\in\cX\cap\mathbb{B}(x_k,t_{k,j})}\langle a_k,z-x_k\rangle,\\
 s_{k,j}=y_{k,j}-x_k,\qquad {\mathcal D}_{k,j}=-\langle a_k,s_{k,j}\rangle\ge0.
 \end{gathered}
\end{equation}
A zero value of ${\mathcal D}_{k,j}$ certifies optimality of $x_k$. Otherwise accept the first trial satisfying
\begin{equation}\label{eq:bt-accept}
 {\mathcal D}_{k,j}\ge L\|s_{k,j}\|^2,
\end{equation}
and set $x_{k+1}=y_{k,j}$. {If the test fails, multiply the radius by $\xi$; halving corresponds to $\xi=1/2$.} The test uses the attained step length, which need not equal the trial radius. It requires no projection, objective evaluation, gradient at the trial point, or multiplier estimate.

{The pseudocode is given in \Cref{alg:radius-backtracking} in the main text. The certificates and rate proofs below establish \Cref{thm:main-backtracking}.}

We first prove the endpoint estimate that keeps the constants sharp. It is a consequence of convex smoothness, independent of the local oracle.
\begin{lemma}[Smooth endpoint estimate]\label{lem:bt-endpoint}
Under \Cref{ass:radius-backtracking}, for $s=y-x$, $v=\nabla f(y)-\nabla f(x)$, and $\tau\ge0$,
\begin{equation}\label{eq:bt-endpoint}
 \|v-\tau s\|\le\max\{\tau,L-\tau\}\|s\|
 \le\max\{L,\tau\}\|s\|.
\end{equation}
\end{lemma}
\begin{proof}
\par\medskip\noindent\textbf{Step 1: Derive cocoercivity.}
Convexity and \eqref{eq:bt-smoothness} imply, respectively,
\begin{equation}\label{eq:bt-models}
 \begin{gathered}
 f(u)\ge f(w)+\langle\nabla f(w),u-w\rangle,\\
 f(u)\le f(w)+\langle\nabla f(w),u-w\rangle+\frac L2\|u-w\|^2
 \qquad(u,w\in\R^d).
 \end{gathered}
\end{equation}
For completeness, the second inequality follows by integrating the gradient along the segment from $w$ to $u$ and applying \eqref{eq:bt-smoothness}. Define $h(z)=f(z)-\langle\nabla f(x),z\rangle$. Its gradient vanishes at $x$, so $x$ is a global minimizer. Applying the upper model to $h$ at $y$ gives
\[
 h(x)\le h(y-v/L)
 \overset{\eqref{eq:bt-models}}{\le}h(y)-\frac{\|v\|^2}{2L}.
\]
Thus $f(y)-f(x)-\langle\nabla f(x),s\rangle\ge\|v\|^2/(2L)$. Interchanging $x,y$ and adding gives
\begin{equation}\label{eq:bt-cocoercivity}
 \langle v,s\rangle\ge\frac{\|v\|^2}{L}.
\end{equation}
\par\medskip\noindent\textbf{Step 2: Center and shift the gradient difference.}
Expanding the square yields
\begin{equation}\label{eq:bt-centered}
 \left\|v-\frac L2s\right\|^2
 =\|v\|^2-L\langle v,s\rangle+\frac{L^2}{4}\|s\|^2
 \overset{\eqref{eq:bt-cocoercivity}}{\le}\frac{L^2}{4}\|s\|^2.
\end{equation}
The triangle inequality therefore gives
\[
 \begin{array}{rcl}
 \|v-\tau s\|&\le&\left\|v-\frac L2s\right\|+\left|\frac L2-\tau\right|\|s\|\\[3pt]
 \alignedrel{\eqref{eq:bt-centered}}{\le}&
 \left(\frac L2+\left|\frac L2-\tau\right|\right)\|s\|\\[3pt]
 &=&\max\{\tau,L-\tau\}\|s\|
 \le\max\{L,\tau\}\|s\|.
 \end{array}
\]
The identity follows by considering $\tau\le L/2$ and $\tau\ge L/2$. This also covers $s=0$.
\end{proof}

\begin{theorem}[Finite backtracking and accepted-step certificates]\label{thm:bt-certificates}
Under \Cref{ass:radius-backtracking}, {fix $\xi\in(0,1)$ in \Cref{alg:radius-backtracking}. The algorithm} uses finitely many oracle calls at each center and either stops at a minimizer or accepts a step. Define $F=f+\iota_{\cX}$, where $\iota_{\cX}$ is the convex indicator. Each accepted step $s_k=x_{k+1}-x_k$ admits $g_{k+1}\in\partial F(x_{k+1})$ and an analysis multiplier $\lambda_k$ such that
\begin{equation}\label{eq:bt-certificates}
 L\le\lambda_k\le{L/\xi^2},\qquad
 f(x_k)-f(x_{k+1})\ge\frac{\lambda_k}{2}\|s_k\|^2,\qquad
 \|g_{k+1}\|\le\lambda_k\|s_k\|.
\end{equation}
Neither $g_{k+1}$ nor $\lambda_k$ is needed to run the algorithm.
\end{theorem}
\begin{proof}
\par\medskip\noindent\textbf{Step 1: Justify stopping and finite backtracking.}
Fix a center $x$, set $a=\nabla f(x)$, and write $y,s,t,{\mathcal D}$ for a trial. Compactness of $\cX\cap\mathbb{B}(x,t)$ guarantees an oracle solution. If $a=0$, convexity certifies optimality. If ${\mathcal D}=0$, then $\langle a,z-x\rangle\ge0$ for all locally feasible $z$. For any $u\in\cX\setminus\{x\}$, choose $\theta=\min\{1,t/\|u-x\|\}>0$. The point $z=x+\theta(u-x)$ is locally feasible. Dividing $\langle a,z-x\rangle\ge0$ by $\theta$ gives $\langle a,u-x\rangle\ge0$. This also holds at $u=x$, so convexity proves optimality over $\cX$.

At a nonoptimal center choose $x_\star\in\cX_\star$ and set $d=\|x-x_\star\|>0$, ${\Delta}=f(x)-f_\star>0$. For $0<t\le d$, the feasible segment comparator $x+(t/d)(x_\star-x)$ gives
\begin{equation}\label{eq:bt-small-radius}
 {\mathcal D}\overset{\eqref{eq:bt-trial}}{\ge}\frac td\langle a,x-x_\star\rangle
 \overset{\eqref{eq:bt-models}}{\ge}\frac{t{\Delta}}{d}.
\end{equation}
If also $t\le{\Delta}/(Ld)$, then
\[
 {\mathcal D}\overset{\eqref{eq:bt-small-radius}}{\ge}\frac{t{\Delta}}{d}
 \ge Lt^2\overset{\eqref{eq:bt-trial}}{\ge}L\|s\|^2.
\]
{Since $\xi^j\|a\|/L\to0$, repeated contraction reaches this positive threshold in finitely many trials.} The unknown $d,{\Delta}$ occur only in this argument.

\par\medskip\noindent\textbf{Step 2: Identify the normal vector and ball multiplier.}
At an accepted nonterminal trial, ${\mathcal D}>0$, hence $r=\|s\|>0$. The ball indicator is continuous at $x\in\cX$, since $t>0$. The convex normal-cone sum rule therefore yields $n\in N_{\cX}(y)$ and $\tau\ge0$ satisfying
\begin{equation}\label{eq:bt-kkt}
 a+n+\tau s=0,\qquad \tau(r^2-t^2)=0.
\end{equation}
Here $N_{\cX}(y)=\{n:\langle n,z-y\rangle\le0\ \text{for all }z\in\cX\}$. Taking $z=x$ shows $\langle n,s\rangle\ge0$, and consequently
\begin{equation}\label{eq:bt-multiplier-decrease}
 {\mathcal D}\overset{\eqref{eq:bt-trial},\eqref{eq:bt-kkt}}{=}
 \tau r^2+\langle n,s\rangle\ge\tau r^2.
\end{equation}
Set
\begin{equation}\label{eq:bt-lambda}
 \lambda=\max\{L,\tau\}.
\end{equation}
This analysis multiplier can differ from the oracle's ball multiplier.

\par\medskip\noindent\textbf{Step 3: Establish the two certificates.}
Acceptance and \eqref{eq:bt-multiplier-decrease} give ${\mathcal D}\ge\lambda r^2$. Thus
\[
 \begin{array}{rcl}
 f(x)-f(y)\alignedrel{\eqref{eq:bt-models},\eqref{eq:bt-trial}}{\ge}&{\mathcal D}-\frac L2r^2\\[3pt]
 \alignedrel{\eqref{eq:bt-accept},\eqref{eq:bt-multiplier-decrease},\eqref{eq:bt-lambda}}{\ge}&
 (\lambda-L/2)r^2
 \overset{\eqref{eq:bt-lambda}}{\ge}\frac\lambda2r^2.
 \end{array}
\]
With $v=\nabla f(y)-\nabla f(x)$, select
\begin{equation}\label{eq:bt-selected-g}
 g=\nabla f(y)+n\overset{\eqref{eq:bt-kkt}}{=}v-\tau s\in\partial F(y).
\end{equation}
The endpoint lemma now gives
\[
 \|g\|\overset{\eqref{eq:bt-selected-g}}{=}\|v-\tau s\|
 \overset{\eqref{eq:bt-endpoint}}{\le}\max\{L,\tau\}r
 \overset{\eqref{eq:bt-lambda}}{=}\lambda r.
\]

\par\medskip\noindent\textbf{Step 4: Bound the analysis multiplier.}
If $\tau=0$, then $\lambda=L$. If the initial radius is accepted and $\tau>0$, complementary slackness gives $r=t=\|a\|/L$. Hence
\[
 \tau r^2\overset{\eqref{eq:bt-multiplier-decrease}}{\le}{\mathcal D}
 \overset{\eqref{eq:bt-trial}}{\le}\|a\|r,
 \qquad \tau\le\frac{\|a\|}{r}=L.
\]
{If instead a previous trial was rejected, its radius was $t/\xi$.} Denote its decrease and step by ${\widetilde{\mathcal D}},\widetilde s$. Nesting of the feasible balls and rejection imply
\begin{equation}\label{eq:bt-rejected}
 {\mathcal D}\le{\widetilde{\mathcal D}}
 \overset{\eqref{eq:bt-accept}}{<}L\|\widetilde s\|^2
 \overset{\eqref{eq:bt-trial}}{\le}{Lt^2/\xi^2}.
\end{equation}
For $\tau>0$, \eqref{eq:bt-kkt} gives $r=t>0$. Combining \eqref{eq:bt-multiplier-decrease} with \eqref{eq:bt-rejected} and dividing by $t^2$ {gives $\tau<L/\xi^2$. All cases yield $L\le\lambda\le L/\xi^2$}, completing the proof.
\end{proof}

\subsubsection{{Convex objective convergence with backtracking}}
\begin{theorem}[{Convex last-iterate objective rate}]\label{thm:bt-rates}
Under \Cref{ass:radius-backtracking}, {use \Cref{alg:radius-backtracking} with fixed $\xi\in(0,1)$ and} freeze the sequence after optimal termination. For the convex rate assume that, for some $R_{\rm lev}>0$, the accepted centers satisfy
\begin{equation}\label{eq:bt-distance-assumption}
 \operatorname{dist}(x_k,\cX_\star)\le R_{\rm lev}\qquad(k\ge0).
\end{equation}
Then, for every $K\ge0$,
\begin{equation}\label{eq:bt-convex-rate}
 \Delta_K\le\frac{\max\{\Delta_0,{4LR_{\rm lev}^2/\xi^2}\}}{K+1}.
\end{equation}
{The algorithm does not need $R_{\rm lev}$ or $\Delta_0$. The index $K$ counts accepted steps.}
\end{theorem}

\begin{proof}
\par\medskip\noindent\textbf{Step 1: Eliminate the step length.}
The certificates and {$\lambda_k\le L/\xi^2$} imply
\begin{equation}\label{eq:bt-progress}
 \Delta_k-\Delta_{k+1}
 \overset{\eqref{eq:bt-certificates}}{\ge}\frac{\|g_{k+1}\|^2}{2\lambda_k}
 \overset{\eqref{eq:bt-certificates}}{\ge}{\frac{\xi^2\|g_{k+1}\|^2}{2L}}.
\end{equation}
\par\medskip\noindent\textbf{Step 2: Obtain the convex recurrence.}
The solution set is nonempty, closed, and convex. Choose its nearest point $z$ to $x_{k+1}$. The subgradient inequality and Cauchy--Schwarz give
\[
 \Delta_{k+1}\le\langle g_{k+1},x_{k+1}-z\rangle
 \overset{\eqref{eq:bt-distance-assumption}}{\le}R_{\rm lev}\|g_{k+1}\|.
\]
Using \eqref{eq:bt-progress}, define {$a=\xi^2/(2LR_{\rm lev}^2)>0$} to obtain
\begin{equation}\label{eq:bt-recurrence}
 \Delta_{k+1}+a\Delta_{k+1}^2\le\Delta_k.
\end{equation}
\par\medskip\noindent\textbf{Step 3: Solve the recurrence explicitly.}
Set $A=\max\{\Delta_0,2/a\}$ and $w_k=A/(k+1)$. Since $A>0$,
\[
 \frac{w_k-w_{k+1}}{w_{k+1}^2}
 =\frac{k+2}{A(k+1)}\le\frac2A\le a.
\]
Thus $w_k\le w_{k+1}+aw_{k+1}^2$. The function $H(u)=u+au^2$ is increasing for $u\ge0$. Starting from $\Delta_0\le w_0$, induction gives
\[
 H(\Delta_{k+1})\overset{\eqref{eq:bt-recurrence}}{\le}\Delta_k
 \le w_k\le H(w_{k+1}),
\]
so $\Delta_{k+1}\le w_{k+1}$. Since $2/a={4LR_{\rm lev}^2/\xi^2}$, this proves \eqref{eq:bt-convex-rate}.

\end{proof}
\subsubsection{{Strongly convex objective convergence with backtracking}}
\begin{theorem}[{Strongly convex last-iterate objective rate}]\label{thm:bt-strong-rates}
{Under \Cref{ass:radius-backtracking}, use \Cref{alg:radius-backtracking} with fixed $\xi\in(0,1)$ and freeze the sequence after optimal termination.} Suppose $f$ is $\mu$-strongly convex on $\cX$, meaning
\begin{equation}\label{eq:bt-strong-assumption}
 f(u)\ge f(y)+\langle\nabla f(y),u-y\rangle+\frac\mu2\|u-y\|^2
 \qquad(u,y\in\cX),\quad\mu>0.
\end{equation}
Without \eqref{eq:bt-distance-assumption},
\begin{equation}\label{eq:bt-strong-rate}
 \Delta_K\le\left(1+{\frac{\xi^2\mu}{L}}\right)^{-K}\Delta_0.
\end{equation}
The algorithm needs neither $R_{\rm lev}$ nor $\mu$ nor $\Delta_0$. The index $K$ counts accepted steps, not all trial oracle calls.
\end{theorem}\begin{proof}
\par\medskip\noindent{\textbf{Step 1: Use strong convexity instead of a distance bound.}}
For $y=x_{k+1}$ and $g=g_{k+1}=\nabla f(y)+n$, normality gives $\langle n,x_\star-y\rangle\le0$. Thus
\[
 \begin{array}{rcl}
 f_\star\alignedrel{\eqref{eq:bt-strong-assumption}}{\ge}&f(y)+\langle g,x_\star-y\rangle+\frac\mu2\|x_\star-y\|^2\\[3pt]
 &\ge&f(y)-\frac{\|g\|^2}{2\mu}.
 \end{array}
\]
The last step completes the square, so $\|g_{k+1}\|^2\ge2\mu\Delta_{k+1}$. Substituting into \eqref{eq:bt-progress} yields
\[
 \Delta_k-\Delta_{k+1}\overset{\eqref{eq:bt-progress}}{\ge}{\frac{\xi^2\mu}{L}}\Delta_{k+1},
 \qquad \Delta_{k+1}\le\left(1+{\frac{\xi^2\mu}{L}}\right)^{-1}\Delta_k.
\]
Multiplying the one-step factors proves \eqref{eq:bt-strong-rate}. All inequalities remain valid for a frozen optimal continuation, whose objective gaps are zero.
\end{proof}

\paragraph{Interpretation and oracle cost.}
The distance assumption in \eqref{eq:bt-distance-assumption} is an analysis condition, not an input. For example, it holds if the initial sublevel set has bounded distance to the solution set, because accepted steps decrease the objective. It must not be silently replaced by the initial distance $R_0$: the proof above establishes objective descent, not Fej\'er monotonicity. The strong-convexity rate needs no such distance assumption and no knowledge of $\mu$ to choose radii. Each outer iteration reuses one gradient, while every rejected trial incurs another exact local-oracle solve. Finite backtracking alone does not assert a uniform bound on their total number. Approximate solves require separate error analysis. In the unconstrained case, the initial trial is the gradient step with stepsize $1/L$ and passes the test with equality.

%% file: appendix-table-benchmarks.tex
\subsection{{B. Solution-dependent convex benchmarks}}
\label{sec:table-proof-B}
{As in group B of \Cref{tbl:main}, we first give Polyak-radius objective guarantees for bounded gradients, bounded subgradients, and smooth convex objectives, then the sharp distance-contraction benchmark. These radii use solution-dependent quantities.}
\subsubsection{{Bounded-gradient Polyak radius}}
\begin{theorem}[Rate under bounded gradients]\label{thm:bounded_gradients}
    Let \Cref{ass:main} hold, let $f$ be convex and differentiable, and assume that  there exists $G>0$ such that \begin{equation}\label{eq:grad-bounded}
    \|\nabla f(x)\|\le G
    \qquad \forall x\in \cX.
    \end{equation}
Let $\{x_k\}_{k\ge 0}$ be the iterates generated by {\algname{BroxGD}}, where $x_0\in \cX$. Further, let $x_\star\in \cX_\star$, and
    {choose the ``Polyak radius'' whenever $f(x_k)>f(x_\star)$:}
    \begin{equation}\label{eq:polyak-radius-bounded}
        \squeeze t_k \eqdef \frac{f(x_k)-f(x_\star)}{\|\nabla f(x_k)\|}.
    \end{equation}
    {If $f(x_k)=f(x_\star)$, set $t_k=0$ and freeze the sequence at $x_k$. This includes every constrained minimizer, not only the selected $x_\star$.} Then for every $K\ge 1$, 
    \begin{equation}\label{eq:square-sum-bound}
        \squeeze \frac{1}{K}\sum_{k=0}^{K-1} (f(x_k)-f(x_\star))^2
        \le \frac{G^2\|x_0-x_\star\|^2}{K}.
    \end{equation}
    Moreover, the average iterate ${\bar x_K} \eqdef \frac{1}{K}\sum_{k=0}^{K-1} x_k$ satisfies
    \begin{equation}\label{eq:avg_iter_G}
        \squeeze f({\bar x_K}) - f(x_\star) \leq \frac{G\|x_0-x_\star\|}{\sqrt{K}}.
    \end{equation}
\end{theorem}

{\color{black}
\begin{proof}
Write $\Delta_k=f(x_k)-f(x_\star)\ge0$. If $\Delta_k=0$, the frozen step has $t_k=0$ and $x_{k+1}=x_k$. If $\Delta_k>0$, convexity implies $\nabla f(x_k)\ne0$ and
\[
 0<t_k\overset{\eqref{eq:polyak-radius-bounded}}{=}
 \frac{\Delta_k}{\|\nabla f(x_k)\|}
 \le\frac{\langle\nabla f(x_k),x_k-x_\star\rangle}{\|\nabla f(x_k)\|}.
\]
Thus the Type-I condition of \Cref{thm:descent} holds at every positive-gap iterate. At such an iterate,
\begin{align*}
 \|x_{k+1}-x_\star\|^2
 \alignedrel{\eqref{eq:thm_squared_distance_decrease}}{\le}\|x_k-x_\star\|^2-t_k^2\\
 \alignedrel{\eqref{eq:polyak-radius-bounded},\eqref{eq:grad-bounded}}{\le}
 \|x_k-x_\star\|^2-\Delta_k^2/G^2.
\end{align*}
The resulting inequality also holds with equality at a frozen zero-gap iterate. Summing over $k=0,\ldots,K-1$ and dropping the nonnegative final squared distance gives
\begin{equation}\label{eq:polyak-bounded-telescope}
 \sum_{k=0}^{K-1}\Delta_k^2\le G^2\|x_0-x_\star\|^2.
\end{equation}
Dividing by $K$ proves \eqref{eq:square-sum-bound}. Convexity and Cauchy--Schwarz then give
\[
 0\le f({\bar x_K})-f(x_\star)
 \le\frac1K\sum_{k=0}^{K-1}\Delta_k
 \le\sqrt{\frac1K\sum_{k=0}^{K-1}\Delta_k^2}
 \overset{\eqref{eq:polyak-bounded-telescope}}{\le}
 \frac{G\|x_0-x_\star\|}{\sqrt K},
\]
which proves \eqref{eq:avg_iter_G}, including trajectories that reach any minimizer early.
\end{proof}
}

\Cref{thm:bounded_gradients} establishes an $\mathcal{O}(\nicefrac{1}{\sqrt{K}})$ convergence rate in function value suboptimality. The adaptive radius in \eqref{eq:polyak-radius-bounded} can be interpreted as a normalized Polyak-type step \citep{polyak1987introduction}, which reduces to the classical Polyak stepsize in the unconstrained setting (the standard Polyak rule $\nicefrac{(f(x_k)-f(x_\star))}{\|\nabla f(x_k)\|^2}$ corresponds exactly to \eqref{eq:polyak-radius-bounded} when $\cX = \R^d$ \citep[Remark F.5]{gruntkowska2025ball}).
The bound \eqref{eq:avg_iter_G} matches the rate achieved by \algname{ProjGD} with Polyak stepsizes \citep[Theorem 8.13]{Beck2017First}.
{Bounded gradients do not ensure finite classical \algname{FW} curvature (\Cref{ex:fw-curvature-infinite}), so they do not imply the standard curvature-based objective rate. Infinite curvature does not preclude \algname{FW} convergence: for example, uniformly continuous gradients on compact convex domains suffice with suitable stepsizes \citep{xu2017convergence}.}
This bounded gradient regime is also the most relevant one for modern deep learning, where global smoothness or strong convexity assumptions are rarely realistic.

\subsubsection{{Bounded-subgradient Polyak radius}}
\label{sec:nondiff}

\Cref{ass:main} specifies the feasible set and existence of a minimizer; differentiability and convexity are imposed separately in the results above.
For convex objectives, we consider three settings: (i) smooth objectives, (ii) strongly convex and smooth objectives, and (iii) objectives with bounded gradients.
The first two explicitly require smoothness, and hence differentiability of $f$.

In this section, we focus on the last setting and remove the differentiability assumption. {Direct substitution of subgradients into vanilla \algname{FW} can fail to converge in this setting} \citep[Example 1]{Nesterov-2017}. Nesterov's counter-example is simple: $$d=2, \quad f(u,v) = \max\{u,v\}, \quad \cX = \{(u,v) \;:\; u^2 + v^2 \leq 1\}.$$ See also \citep[Appendix 3]{FW-YFLC-2018}.

We make a slight modification to \Cref{alg:new}.
At iteration $k$, if $0 \in \partial f(x_k)$, then $x_k$ is already a minimizer of $f$, and the algorithm terminates. Otherwise, the method selects a subgradient $g_k \in \partial f(x_k)$ and defines the next iterate by \eqref{eq:local_lmo_nondiff}.

We now show that the bounded gradient analysis extends directly to this non-smooth setting.

\begin{proposition}[Well-posedness in the non-smooth case]
\label{prop:well-posedness-subgradient}
    Let $\cX \subseteq \R^d$ be nonempty, closed, and convex, and let $f : \R^d \to \R$ be convex.
    For every $x_k \in \cX$, every $t_k \ge 0$, and every choice of subgradient $g_k \in \partial f(x_k)$, the set $\cX \cap \mathbb{B}(x_k,t_k)$ is nonempty, closed, and bounded.
    Consequently, the optimization problem
    \begin{align*}
    \min_{z \in \cX \cap \mathbb{B}(x_k,t_k)} \inner{g_k}{z}
    \end{align*}
    admits at least one solution.
\end{proposition}

\begin{proof}
    The proof is identical to that of \Cref{thm:well-posedness}.
\end{proof}

\begin{proposition}[Type-I admissibility with subgradients]
\label{prop:type-I-subgradient}
    Let $x_k \in \cX$ and $x_\star \in \cX_\star$.
    Assume that
    \begin{align*}
    g_k \in \partial f(x_k),
    \qquad
    0 \notin \partial f(x_k),
    \qquad
    0 < t_k \le \frac{\inner{g_k}{x_k-x_\star}}{\norm{g_k}}.
    \end{align*}
    Define
    \begin{align*}
    x_{k+1} \in \argmin_{z \in \cX \cap \mathbb{B}(x_k,t_k)} \inner{g_k}{z}.
    \end{align*}
    Then the conclusions of the Type-I part of \Cref{thm:descent} remain valid, namely,
    \begin{align*}
    \norm{x_k-x_\star} &\ge t_k, \qquad \norm{x_{k+1}-x_k} = t_k, \qquad \norm{x_{k+1}-x_\star}^2 \le \norm{x_k-x_\star}^2 - t_k^2.
    \end{align*}
\end{proposition}

\begin{proof}
    The proof is the same as that of the Type-I admissibility part in \Cref{thm:descent}, after replacing~$\nabla f(x_k)
    $ by~$g_k$.
    Indeed, the argument uses only the admissibility condition, the optimality conditions of the local linear minimization subproblem, and the geometry of the ball constraint.
    No differentiability of~$f$ is needed.
\end{proof}

For convenience, we restate \Cref{thm:bounded_subgradients}.
\POLYAKSUBGRADIENT*

\begin{proof}
    {Fix $k\ge0$. If $f(x_k)=f(x_\star)$, the sequence is frozen and the distance recurrence below holds with equality. Otherwise the gap is positive, so every $g_k\in\partial f(x_k)$ is nonzero. We establish the recurrence for this positive-gap case.}
    Since $g_k \in \partial f(x_k)$, the subgradient inequality gives $f(x_k)-f(x_\star) \le \inner{g_k}{x_k-x_\star}$.
    Dividing by $\norm{g_k}$, we obtain
    \begin{align*}
        t_k = \frac{f(x_k)-f(x_\star)}{\norm{g_k}} \le \frac{\inner{g_k}{x_k-x_\star}}{\norm{g_k}}.
    \end{align*}
    Hence the Type-I admissibility condition in \Cref{prop:type-I-subgradient} is satisfied.
    Therefore,
    \begin{align*}
        \norm{x_{k+1}-x_\star}^2 \le \norm{x_k-x_\star}^2 - t_k^2.
    \end{align*}
    Substituting the definition of $t_k$ yields
    \begin{align*}
        \norm{x_{k+1}-x_\star}^2 \le \norm{x_k-x_\star}^2 - \frac{\rbrac{f(x_k)-f(x_\star)}^2}{\norm{g_k}^2}.
    \end{align*}
    Using the bounded subgradients assumption, we get
    \begin{align*}
        \norm{x_{k+1}-x_\star}^2 \le \norm{x_k-x_\star}^2 - \frac{\rbrac{f(x_k)-f(x_\star)}^2}{G^2}.
    \end{align*}
    {This last inequality holds also at zero-gap iterates.} Summing this inequality for $k=0,1,\dots,K-1$, we obtain after rearranging,
    \begin{align*}
        \frac{1}{K}\sum_{k=0}^{K-1} \rbrac{f(x_k)-f(x_\star)}^2 \le \frac{G^2 \norm{x_0-x_\star}^2}{K}.
    \end{align*}
    This proves \eqref{eq:square-sum-bound-subgradient}.
    The bound \eqref{eq:avg_iter_G_subgradient} for the average iterate follows exactly as in the proof of \Cref{thm:bounded_gradients}, by Jensen's inequality.
\end{proof}

\subsubsection{{Smooth convex Polyak radius}}

For convenience, we restate \Cref{thm:polyak_radius_smooth_convex}.
\label{sec:polyak-smooth-proof}
\POLYAKSMOOTH*

{\color{black}
\begin{proof}
Write $\Delta_k=f(x_k)-f(x_\star)\ge0$. If $\Delta_k=0$, then $t_k=0$ and $x_{k+1}=x_k$. Otherwise convexity ensures $\nabla f(x_k)\ne0$ and
\[
 0<t_k\overset{\eqref{eq:polyak_radius_smooth_convex}}{=}
 \frac{\Delta_k}{\|\nabla f(x_k)\|}
 \le\frac{\langle\nabla f(x_k),x_k-x_\star\rangle}{\|\nabla f(x_k)\|}.
\]
The Type-I part of \Cref{thm:descent} therefore applies at positive-gap iterates. Together with the frozen zero-gap case, it gives, for every $k\ge0$,
\begin{equation}\label{eq:polyak_radius_smooth_convex_fejer}
 \|x_{k+1}-x_\star\|^2
 \overset{\eqref{eq:thm_squared_distance_decrease}}{\le}
 \|x_k-x_\star\|^2-t_k^2.
\end{equation}
In the zero-gap case this relation is simply equality and requires no invocation of the positive-radius theorem. Iterating yields $\|x_k-x_\star\|\le {R_0}$; hence smoothness implies
\begin{equation}\label{eq:polyak-smooth-gradient-bound}
 \|\nabla f(x_k)\|
 \le\|\nabla f(x_\star)\|+L\|x_k-x_\star\|
 \overset{\eqref{eq:polyak_radius_smooth_convex_fejer}}{\le}
 \|\nabla f(x_\star)\|+L{R_0}={M_0}.
\end{equation}
If ${M_0}=0$, then ${R_0}=0$ because $L>0$, and the sequence starts at a minimizer and remains there. Both claimed bounds are zero. Assume henceforth ${M_0}>0$. At every positive-gap iterate,
\[
 t_k^2\overset{\eqref{eq:polyak_radius_smooth_convex}}{=}
 \frac{\Delta_k^2}{\|\nabla f(x_k)\|^2}
 \overset{\eqref{eq:polyak-smooth-gradient-bound}}{\ge}\frac{\Delta_k^2}{{M_0}^2}.
\]
The inequality $t_k^2\ge\Delta_k^2/{M_0}^2$ also holds at zero-gap iterates, with both sides zero. Summing the distance inequality therefore gives
\[
 \frac1{{M_0}^2}\sum_{k=0}^{K-1}\Delta_k^2
 \le\sum_{k=0}^{K-1}t_k^2
 \overset{\eqref{eq:polyak_radius_smooth_convex_fejer}}{\le}{R_0}^2,
\]
which proves \eqref{eq:polyak_radius_smooth_convex_square_sum}. Finally,
\[
 0\le f({\bar x_K})-f(x_\star)
 \le\frac1K\sum_{k=0}^{K-1}\Delta_k
 \le\sqrt{\frac1K\sum_{k=0}^{K-1}\Delta_k^2}
 \overset{\eqref{eq:polyak_radius_smooth_convex_square_sum}}{\le}
 \frac{{M_0}{R_0}}{\sqrt K},
\]
by convexity and Cauchy--Schwarz. This proves \eqref{eq:polyak_radius_smooth_convex_average}.

For the sharper bound, suppose $f$ is globally convex and $L$-smooth and $\nabla f(x_\star)=0$. Convexity makes $x_\star$ a global minimizer of $f$. The descent lemma therefore gives, for every $x\in\R^d$,
\[
 f_\star\le f\!\left(x-\frac1L\nabla f(x)\right)
 \le f(x)-\frac{\|\nabla f(x)\|^2}{2L}.
\]
Consequently,
\begin{equation}\label{eq:polyak-stationary-self-bound}
 \|\nabla f(x_k)\|^2\le 2L\Delta_k.
\end{equation}
At every positive-gap iterate,
\begin{equation}\label{eq:polyak-stationary-radius-bound}
 t_k^2\overset{\eqref{eq:polyak_radius_smooth_convex}}{=}
 \frac{\Delta_k^2}{\|\nabla f(x_k)\|^2}
 \overset{\eqref{eq:polyak-stationary-self-bound}}{\ge}\frac{\Delta_k}{2L}.
\end{equation}
The same lower bound holds at frozen zero-gap iterates. Hence
\begin{align*}
 \sum_{k=0}^{K-1}\Delta_k
 \alignedrel{\eqref{eq:polyak-stationary-radius-bound}}{\le}
 2L\sum_{k=0}^{K-1}t_k^2\\
 \alignedrel{\eqref{eq:polyak_radius_smooth_convex_fejer}}{\le}
 2L\bigl(R_0^2-\|x_K-x_\star\|^2\bigr)\le2LR_0^2.
\end{align*}
Dividing by $K$ and applying convexity at $\bar x_K$ proves \eqref{eq:polyak-smooth-stationary-average}.

For the neighborhood version, the case $R_0=0$ is immediate. Suppose $R_0>0$. By \eqref{eq:polyak_radius_smooth_convex_fejer}, all iterates lie in the compact set $S$. Since $U$ is open and contains $S$, a radius $r_U>0$ with $S+\overline{\mathbb{B}}(0,r_U)\subset U$ exists. Define
\begin{equation}\label{eq:polyak-neighborhood-step}
 \eta_U=\min\left\{\frac1L,\frac{r_U}{LR_0}\right\}>0.
\end{equation}
Smoothness and stationarity give $\|\nabla f(x_k)\|\le LR_0$, so the whole segment from $x_k$ to $x_k-\eta_U\nabla f(x_k)$ lies in $U$. Convexity and stationarity make $x_\star$ a minimizer on $U$. The descent lemma therefore yields
\begin{align*}
 f_\star&\le f\bigl(x_k-\eta_U\nabla f(x_k)\bigr)\\
 &\le f(x_k)-\eta_U\left(1-\frac{L\eta_U}{2}\right)\|\nabla f(x_k)\|^2\\
 \alignedrel{\eqref{eq:polyak-neighborhood-step}}{\le}
 f(x_k)-\frac{\eta_U}{2}\|\nabla f(x_k)\|^2.
\end{align*}
Thus $\|\nabla f(x_k)\|^2\le2\Delta_k/\eta_U$. For a positive gap, the Polyak rule gives $t_k^2\ge\eta_U\Delta_k/2$; this remains valid at frozen zero-gap iterates. Consequently,
\begin{align*}
 f(\bar x_K)-f_\star
 &\le\frac1K\sum_{k=0}^{K-1}\Delta_k
 \le\frac{2}{\eta_U K}\sum_{k=0}^{K-1}t_k^2\\
 \alignedrel{\eqref{eq:polyak_radius_smooth_convex_fejer}}{\le}
 \frac{2R_0^2}{\eta_U K}
 \overset{\eqref{eq:polyak-neighborhood-step}}{=}
 \frac{2LR_0^2}{K}\max\left\{1,\frac{R_0}{r_U}\right\}.
\end{align*}
This proves \eqref{eq:polyak-neighborhood-average} without requiring global smoothness.

\end{proof}
}

\begin{remark}
{\Cref{thm:polyak_radius_smooth_convex} gives an $\mathcal{O}(1/\sqrt{K})$ objective-error guarantee for the Polyak radius in general, and an $\mathcal{O}(1/K)$ guarantee when $\nabla f(x_\star)=0$, under either global smoothness or the stated neighborhood assumptions (with the additional neighborhood-dependent constant). The fixed-horizon schedule in \Cref{cor:practical-smooth-horizon} instead gives an $\mathcal{O}(1/K)$ best-iterate objective bound. This comparison concerns the available guarantees: the weaker Polyak analysis does not prove that this radius cannot achieve a faster rate. The $\mathcal{O}(1/K)$ bound in \Cref{thm:L-smooth} controls a squared gradient-difference residual and is a different metric.}
\end{remark}

\subsubsection{{Sharp strongly convex distance contraction}}
\label{sec:proof-sharp-strong}
For convenience, we restate \Cref{thm:smooth-strong} here.
{\color{black}\mdfsetup{fontcolor=black}\SMOOTHSTRONG*}
{\color{black}
We give two proofs. The first is shorter but assumes that $f$ is $L$-smooth and $\mu$-strongly convex on all of $\R^d$. The second proves the theorem under its stated assumptions on $\cX$ alone, including lower-dimensional feasible sets.
\begin{proof}[Proof under global smoothness and strong convexity]
At an iterate $x_k\ne x_\star$, set $d_k=x_k-x_\star$ and $w_k=\nabla f(x_k)-\nabla f(x_\star)$. Strong convexity gives $\langle w_k,d_k\rangle\ge\mu\|d_k\|^2>0$, so $w_k\ne0$. The global assumptions allow a direct application of \Cref{lem:strong-smooth-interpolation}:
\[
 t_k\overset{\eqref{eq:radius_smooth_strong}}{=}\theta\|d_k\|
 \overset{\eqref{eq:reverse-cs}}{\le}
 \frac{\langle w_k,d_k\rangle}{\|w_k\|}.
\]
This is the Type-II radius condition of \Cref{thm:descent}. Therefore
\begin{align*}
 \|x_{k+1}-x_\star\|^2
 \alignedrel{\eqref{eq:thm_squared_distance_decrease}}{\le}\|x_k-x_\star\|^2-t_k^2\\
 \alignedrel{\eqref{eq:radius_smooth_strong}}{=}
 \left(\frac{L-\mu}{L+\mu}\right)^2\|x_k-x_\star\|^2.
\end{align*}
The same contraction holds for the frozen continuation at $x_\star$. Iterating proves \eqref{eq:linear-improved}. At $K=0$ the bound is the initial identity; if $L=\mu$, the method reaches $x_\star$ in at most one step.
\end{proof}
}

{\color{black}
\begin{proof}[Proof under the stated assumptions on $\cX$]
\par\medskip\noindent\textbf{Step 1: Restrict gradients to the affine hull.}
Write $\operatorname{aff}(\cX)=a+V$, where $V$ is a linear subspace, and let $P_V$ be its orthogonal projector. If $V=\{0\}$, the feasible set is a singleton and the conclusion is immediate. Otherwise define the intrinsic gradient
\begin{equation}\label{eq:sharp-intrinsic-gradient}
 b(x)=P_V\nabla f(x)\qquad(x\in\cX).
\end{equation}
For $x,y\in\cX$, the displacement $x-y$ lies in $V$. Hence
\begin{equation}\label{eq:sharp-intrinsic-identities}
 \langle b(y),x-y\rangle=\langle\nabla f(y),x-y\rangle,
 \qquad \|b(x)-b(y)\|\le L\|x-y\|.
\end{equation}
The second inequality uses nonexpansiveness of $P_V$ and the assumed smoothness on $\cX$. Thus the restriction of $f$ to $\cX$, viewed in $a+V$, retains the same strong-convexity and smoothness constants. Moreover, replacing $\nabla f(x_k)$ by $b(x_k)$ leaves every local linear minimizer unchanged, since their inner products with feasible displacements coincide by \eqref{eq:sharp-intrinsic-identities}. No projection is required by the algorithm; $P_V$ is used only in the proof.

\par\medskip\noindent\textbf{Step 2: Establish interpolation, including at the relative boundary.}
We claim that, with $d=x-y$ and $w=b(x)-b(y)$,
\begin{equation}\label{eq:sharp-intrinsic-interpolation}
 \langle w,d\rangle\ge\frac{\|w\|^2+\mu L\|d\|^2}{L+\mu}
 \qquad(x,y\in\cX).
\end{equation}
Here it is essential to use intrinsic gradients: the analogous assertion for ambient gradient differences need not hold under assumptions imposed only on a lower-dimensional $\cX$.

First take $x,y$ in the relative interior $U=\operatorname{ri}(\cX)$, which is open and convex in $a+V$. The segment $[x,y]$ has a neighborhood contained in $U$. In orthonormal coordinates on $V$, convolve the restricted function with a nonnegative smooth kernel of integral one and sufficiently small support. Denote the resulting smooth function by $f_\varepsilon$. On a neighborhood of this segment, averaging preserves $\mu$-strong convexity and the $L$-Lipschitz gradient bound in \eqref{eq:sharp-intrinsic-identities}. Therefore
\begin{equation}\label{eq:sharp-mollified-hessian}
 \mu I_V\preceq\nabla_V^2 f_\varepsilon(z)\preceq L I_V,
 \qquad
 H_\varepsilon=\int_0^1\nabla_V^2 f_\varepsilon(y+s d)\,ds,
 \qquad \mu I_V\preceq H_\varepsilon\preceq L I_V.
\end{equation}
The fundamental theorem of calculus gives
\begin{equation}\label{eq:sharp-mollified-difference}
 w_\varepsilon=\nabla_V f_\varepsilon(x)-\nabla_V f_\varepsilon(y)=H_\varepsilon d.
\end{equation}
Every eigenvalue $\lambda$ of $H_\varepsilon$ lies in $[\mu,L]$, so $(\lambda-\mu)(L-\lambda)\ge0$. Consequently,
\begin{align*}
 (L+\mu)\langle w_\varepsilon,d\rangle-\|w_\varepsilon\|^2-\mu L\|d\|^2
 \alignedrel{\eqref{eq:sharp-mollified-difference}}{=}
 \langle d,(H_\varepsilon-\mu I_V)(L I_V-H_\varepsilon)d\rangle\\
 \alignedrel{\eqref{eq:sharp-mollified-hessian}}{\ge}0.
\end{align*}
The intrinsic gradient is continuous, so the mollified gradients converge to it as $\varepsilon\downarrow0$. Passing to the limit proves \eqref{eq:sharp-intrinsic-interpolation} on $U$.

For arbitrary $x,y\in\cX$, fix $c\in U$ and take $x_\eta=(1-\eta)x+\eta c$ and $y_\eta=(1-\eta)y+\eta c$, where $0<\eta<1$. Both points belong to $U$. Apply the inequality just proved to $x_\eta,y_\eta$ and let $\eta\downarrow0$. Continuity from \eqref{eq:sharp-intrinsic-identities} proves \eqref{eq:sharp-intrinsic-interpolation} on all of $\cX$.

\par\medskip\noindent\textbf{Step 3: Apply the one-step geometry intrinsically.}
If $x_k=x_\star$, the prescribed frozen continuation satisfies the claim. Otherwise set $d_k=x_k-x_\star$ and $w_k=b(x_k)-b(x_\star)$. Adding the two strong-convexity inequalities gives
\[
 \langle w_k,d_k\rangle
 \overset{\eqref{eq:sharp-intrinsic-identities}}{=}
 \langle\nabla f(x_k)-\nabla f(x_\star),d_k\rangle
 \ge\mu\|d_k\|^2>0,
\]
so $w_k\ne0$. Arithmetic--geometric mean and the intrinsic interpolation inequality imply
\begin{equation}\label{eq:sharp-intrinsic-admissibility}
 t_k\overset{\eqref{eq:radius_smooth_strong}}{=}\theta\|d_k\|
 \le\frac{\|w_k\|^2+\mu L\|d_k\|^2}{(L+\mu)\|w_k\|}
 \overset{\eqref{eq:sharp-intrinsic-interpolation}}{\le}
 \frac{\langle w_k,d_k\rangle}{\|w_k\|}.
\end{equation}
Identify $a+V$ isometrically with $\R^{\dim V}$. The feasible set remains closed and convex, the local oracle is unchanged by Step 1, and the restricted differentiable objective has minimizer $x_\star$. Thus \eqref{eq:sharp-intrinsic-admissibility} is precisely the Type-II condition of \Cref{thm:descent} in these coordinates. It follows that
\begin{align*}
 \|x_{k+1}-x_\star\|^2
 \alignedrel{\eqref{eq:thm_squared_distance_decrease}}{\le}\|x_k-x_\star\|^2-t_k^2\\
 \alignedrel{\eqref{eq:radius_smooth_strong}}{=}
 \left(\frac{L-\mu}{L+\mu}\right)^2\|x_k-x_\star\|^2.
\end{align*}
Iteration proves \eqref{eq:linear-improved}. The case $K=0$ is the initial identity. When $L=\mu$, the one-step factor is zero, so the method reaches $x_\star$ in at most one step and then stays there; the $K=0$ bound is understood separately.
\end{proof}
}

%% file: appendix-table-residuals.tex
\subsection{{C. Gradient-difference residual guarantees}}
\label{sec:table-proof-C}
{These results control gradient differences rather than objective gaps. The order matches group C of \Cref{tbl:main}.}
\subsubsection{{Smooth convex gradient-difference radius}}
\begin{theorem}[Rate under $L$--smoothness]\label{thm:L-smooth}
    Let \Cref{ass:main} hold, and assume that $f$ is differentiable, $L$--smooth, and convex, {with $L>0$}. If $\{x_k\}_{k\ge 0}$ are the iterates generated by \Cref{alg:new}  with radii
    \begin{equation}\label{eq:smooth-radius}
    t_k=\tfrac{\|\nabla f(x_{k})-\nabla f(x_{\star})\|}{L},
    \end{equation}
     where $x_0\in \cX$ and $x_\star\in \cX_\star$,
    then, for every $K\ge 1$,
    \begin{equation}\label{eq:bgd-smooth-residual-sum}
        \squeeze \min\limits_{0\le k\le K-1}\|\nabla f(x_{k})-\nabla f(x_{\star})\|^2 \leq \frac{1}{K}\sum\limits_{k=0}^{K-1}\|\nabla f(x_{k})-\nabla f(x_{\star})\|^2
        \leq \frac{L^2\|x_{0} - x_\star\|^2}{K}.
    \end{equation}
\end{theorem}

\begin{proof}
    {If $\nabla f(x_k)=\nabla f(x_\star)$, then $t_k=0$ and $x_{k+1}=x_k$, so the distance recurrence below holds with equality. Otherwise the residual is nonzero.} By cocoercivity,
    \[
        \frac1L\|\nabla f(x_{k}) - \nabla f(x_{\star})\|^2 \leq \langle \nabla f(x_{k}) - \nabla f(x_{\star}),x_{k}-x_{\star}\rangle,
    \]
    so the choice $t_k= \frac{\|\nabla f(x_{k}) - \nabla f(x_{\star})\|}{L}$ satisfies radius condition \eqref{eq:radius-condition-2}. 
    Hence \Cref{thm:descent} gives
    \[
        \|x_{k+1}-x_\star\|^2\le \|x_{k}-x_\star\|^2-\frac{\|\nabla f(x_{k}) - \nabla f(x_{\star})\|^2}{L^2}.
    \]
    {The recurrence therefore holds in both cases.} Summing this for $k=0,\dots,K-1$ yields
    \[
        \|x_{K} - x_\star\|^2 \le \|x_{0} - x_\star\|^2-\frac1{L^2}\sum_{k=0}^{K-1}\|\nabla f(x_{k}) - \nabla f(x_{\star})\|^2.
    \]
    Since $\|x_{k}-x_\star\|^2\ge 0$, we obtain
    \[
        \sum_{k=0}^{K-1}\|\nabla f(x_{k}) - \nabla f(x_{\star})\|^2 \le L^2\|x_{0} - x_\star\|^2,
    \]
    which proves \eqref{eq:bgd-smooth-residual-sum}. Dividing by $K$ and taking the minimum gives
    \[
        \min_{0\le k\le K-1}\|\nabla f(x_{k}) - \nabla f(x_{\star})\|^2 \le \frac{1}{K} \sum_{k=0}^{K-1}\|\nabla f(x_{k}) - \nabla f(x_{\star})\|^2 \le \frac{L^2\|x_{0} - x_\star\|^2}{K}.
    \]
\end{proof}

\begin{remark}\label{remark:denum_cond}
    The proof of \Cref{thm:L-smooth} relies on the Type-II admissibility condition in \eqref{eq:radius-condition-2}. Consequently, we must ensure that the denominator in \eqref{eq:radius-condition-2} does not vanish whenever this condition is invoked. {The zero-residual branch is handled directly by the singleton local ball, without invoking that condition. Thus neither strict convexity nor gradient injectivity is required for the residual guarantee. If imposed additionally, strict convexity ensures that a zero residual identifies $x_k=x_\star$; more generally, gradient injectivity on $\cX$ suffices for this identification.}
\end{remark}
\begin{remark}
    If $\nabla f(x_\star)=0$, the radius choice in \eqref{eq:smooth-radius} simplifies to 
    $t_k=\nicefrac{\|\nabla f(x_k)\|}{L}$. When $\cX=\R^d$, this translates into the standard \algname{GD} stepsize $\gamma_k=\nicefrac{t_k}{\norm{\nabla f(x_k)}} = \nicefrac{1}{L}$ (see~\eqref{eq:gd} and \citet[Theorem 3.2]{gruntkowska2025ball}).
\end{remark}
{\Cref{thm:L-smooth} bounds the best squared gradient-difference residual by $L^2R_0^2/K$. This matches the residual upper bound proved for \algname{ProjGD} in \Cref{sec:pgd_positive}; it is not an objective-error bound. The counterexample in \Cref{sec:pgd_negative} rules out replacing $L^2$ by $L$ uniformly in this residual estimate, but does not establish optimality of the dependence on $K$.}

{For an objective-error comparison, the fixed-horizon {\algname{BroxGD}} schedule in \Cref{cor:practical-smooth-horizon} gives a best-iterate bound of order $LR_0^2/K$ when its certified distance bound equals $R_0$. Classical \algname{FW} on a compact domain gives a last-iterate objective bound of order $L{\cD^2_{\cX}}/K$ \citep{Jaggi2013}. These concern the same objective metric, but use different iterates, inputs, and oracles: the local oracle optimizes over a ball intersection, whereas \algname{FW} uses a global LMO. Neither comparison alone establishes an oracle-complexity lower bound or a computational advantage.}

\subsubsection{{Constant-radius residual guarantee}}
\label{sec:practical-residual}
The same constant-radius algorithm also admits a guarantee for the gradient-difference residual. This complements the objective guarantee in \Cref{thm:practical-constant}; the two conclusions concern different quantities. The radius is implementable, but the residual itself involves the unknown gradient at a minimizer and is generally not an observable stopping criterion.

\begin{theorem}[Constant-radius residual guarantee]\label{thm:practical-residual}
Let \Cref{ass:practical-setting} hold, and suppose that $f$ has an $L$-Lipschitz gradient on its open convex domain, with $L>0$. Fix a constant radius
\begin{equation}\label{eq:const_0909fd}
 t_k\equiv t>0.
\end{equation}
Write ${v_k}=\nabla f(x_k)-\nabla f(x_\star)$. If $\|{v_k}\|\ge Lt$, the radius is Type-II admissible in \eqref{eq:radius-condition-2}. At every such iteration (and, more generally, whenever that admissibility condition holds),
\begin{equation}\label{eq:constant_radius_smooth_fejer}
 \|x_{k+1}-x_\star\|^2\le\|x_k-x_\star\|^2-t^2.
\end{equation}
For every integer $K>R_0^2/t^2$,
\begin{equation}\label{eq:constant_radius_smooth_hit}
 \min_{0\le k\le K-1}\|\nabla f(x_k)-\nabla f(x_\star)\|<Lt.
\end{equation}
With the endpoint included, every integer $K\ge R_0^2/t^2$ satisfies
\begin{equation}\label{eq:practical-residual-endpoint}
 \min_{0\le k\le K}\|\nabla f(x_k)-\nabla f(x_\star)\|<Lt.
\end{equation}
\end{theorem}
\begin{proof}
Convexity and smoothness imply cocoercivity:
\begin{equation}\label{eq:practical-residual-cocoercivity}
 \frac1L\|{v_k}\|^2\le\langle {v_k},x_k-x_\star\rangle.
\end{equation}
If $\|{v_k}\|\ge Lt$, then $\|{v_k}\|>0$, and
\[
 t\le\frac{\|{v_k}\|}{L}
 \overset{\eqref{eq:practical-residual-cocoercivity}}{\le}
 \frac{\langle {v_k},x_k-x_\star\rangle}{\|{v_k}\|}.
\]
This verifies Type-II admissibility. The master theorem then gives
\[
 \|x_{k+1}-x_\star\|^2
 \overset{\eqref{eq:thm_squared_distance_decrease}}{\le}
 \|x_k-x_\star\|^2-t_k^2
 \overset{\eqref{eq:const_0909fd}}{=}\|x_k-x_\star\|^2-t^2,
\]
proving \eqref{eq:constant_radius_smooth_fejer}.
If the residual is at least $Lt$ at all queried points $x_0,\ldots,x_{K-1}$, summation yields
\begin{equation}\label{eq:practical-residual-sum}
 0\le\|x_K-x_\star\|^2
 \overset{\eqref{eq:constant_radius_smooth_fejer}}{\le}R_0^2-Kt^2.
\end{equation}
For $K>R_0^2/t^2$ the right side is negative, a contradiction proving \eqref{eq:constant_radius_smooth_hit}. At equality, \eqref{eq:practical-residual-sum} forces $x_K=x_\star$, whose residual is zero and hence strictly less than $Lt$. This proves \eqref{eq:practical-residual-endpoint}; if $K=0$ is permitted by the budget, $R_0=0$ and the assertion is immediate. No strict convexity is needed: a zero residual already meets the stated target, irrespective of whether it identifies a unique minimizer.
\end{proof}

\begin{corollary}[Fixed-horizon residual guarantee]\label{cor:practical-residual-horizon}
Under the hypotheses of \Cref{thm:practical-residual}, let $R\ge R_0$, $R>0$, and $K\ge1$. Using $t=R/\sqrt K$ for $K$ oracle calls gives
\begin{equation}\label{eq:practical-residual-horizon}
 \min_{0\le k\le K}\|\nabla f(x_k)-\nabla f(x_\star)\|
 <\frac{LR}{\sqrt K}.
\end{equation}
\end{corollary}
\begin{proof}
The chosen radius satisfies $R_0^2/t^2=KR_0^2/R^2\le K$. Therefore,
\[
 \min_{0\le k\le K}\|\nabla f(x_k)-\nabla f(x_\star)\|
 \overset{\eqref{eq:practical-residual-endpoint}}{<}Lt=LR/\sqrt K.
\]
\end{proof}

This formalizes the fixed-horizon interpretation while retaining the endpoint required at equality in the budget. If $R=0$ is certified, the starting point is already a minimizer. Although \Cref{cor:practical-smooth-horizon,cor:practical-residual-horizon} use the same radius, their good iterates need not coincide: selecting the smallest observed objective certifies the objective bound, but need not select the point with smallest gradient-difference residual.
\subsubsection{{Generalized-smoothness gradient-difference radius}}
\label{sec:L0L1case}

We now consider the generalized smoothness regime in place of the standard $L$--smoothness assumption. 
{Throughout this section, $f:\R^d\to\R$ is convex and differentiable on $\R^d$, and the following smoothness condition holds globally. The feasible set $\cX$ may still be unbounded or lower-dimensional.}
\begin{assumption}
    \label{asp:l0l1}
    {The differentiable function $f:\R^d\to\R$ is globally asymmetrically $(L_0,L_1)$--smooth: there exist constants $L_0>0$, $L_1\ge0$ such that for all $x,y\in\R^d$,}
    \begin{align*}
        &\norm{\nabla f(x) - \nabla f(y)}
        \le \rbr{L_0 + L_1 \norm{\nabla f(y)}} \norm{x-y}. 
    \end{align*}
\end{assumption}
We refer the readers to \cite{gorbunov2025methods} for further details on this assumption.

{Global convexity and \Cref{asp:l0l1} imply the following cocoercivity inequality. Requiring the smoothness bound only for pairs in $\cX$ would not suffice for this ambient-gradient conclusion.}
\begin{lemma}[Lemma 2.2 of \cite{gorbunov2025methods}]
    \label{lemma:l0l1} {Let $f:\R^d\to\R$ be convex on $\R^d$ and satisfy the global smoothness condition in \Cref{asp:l0l1}.} 
    For any $x, y \in \R^d$, we have 
    \begin{align*}
        \frac{1}{2}\rbrac{\frac{\norm{\nabla f(x) - \nabla f(y)}^2}{L_0 + L_1 \norm{\nabla f(y)}} + \frac{\norm{\nabla f(x) - \nabla f(y)}^2}{L_0 + L_1 \norm{\nabla f(x)}}}  \leq \inner{\nabla f(x) - \nabla f(y)}{x-y}.
    \end{align*}
\end{lemma}

\begin{theorem}[Rate under asymmetric $(L_0,L_1)$--smoothness]
    \label{thm:asym_L0L1_rate}
    {Let Assumptions \ref{ass:main} and \ref{asp:l0l1} hold, with $f$ convex on $\R^d$ and the asymmetric smoothness bound valid for all pairs in $\R^d$.} Let $\{x_k\}_{k \ge 0}$ be the iterates generated by \Cref{alg:new}. Assume that the radii are chosen as
    {
        \begin{align*}
            &t_k :=
            \begin{cases}
                \dfrac{1}{2}\rbr{ \dfrac{\norm{\nabla f(x_k) - \nabla f(x_\star)}}{L_0 + L_1 \norm{\nabla f(x_k)}} + \dfrac{\norm{\nabla f(x_k) - \nabla f(x_\star)}}{L_0 + L_1 \norm{\nabla f(x_\star)}} } , & \nabla f(x_k) \neq \nabla f(x_\star), \\[1.2ex]
                0, & \nabla f(x_k) = \nabla f(x_\star).
            \end{cases}
        \end{align*}
    }
    Then, for every $K \ge 1$,
    \begin{align*}
        &\frac{1}{K}\sum_{k=0}^{K-1} \rbr{ \dfrac{1}{2}\rbr{ \dfrac{\norm{\nabla f(x_k) - \nabla f(x_\star)}}{L_0 + L_1 \norm{\nabla f(x_k)}} + \dfrac{\norm{\nabla f(x_k) - \nabla f(x_\star)}}{L_0 + L_1 \norm{\nabla f(x_\star)}} } }^2 \le \frac{\norm{x_0 - x_\star}^2}{K},
    \end{align*}
    and hence
    \begin{align*}
        &\min_{0 \le k \le K-1} \rbr{ \dfrac{1}{2}\rbr{ \dfrac{\norm{\nabla f(x_k) - \nabla f(x_\star)}}{L_0 + L_1 \norm{\nabla f(x_k)}} + \dfrac{\norm{\nabla f(x_k) - \nabla f(x_\star)}}{L_0 + L_1 \norm{\nabla f(x_\star)}} } }^2 \le \frac{\norm{x_0 - x_\star}^2}{K}.
    \end{align*}
    Moreover, if $\sqrt{K} > L_1 \norm{x_0 - x_\star}$, then
    \begin{align*}
        &\min_{0 \le k \le K-1} \norm{\nabla f(x_k) - \nabla f(x_\star)}
        \le
        \frac{\rbr{L_0 + L_1 \norm{\nabla f(x_\star)}}\norm{x_0 - x_\star}}{\sqrt{K} - L_1 \norm{x_0 - x_\star}}.
    \end{align*}
    Independently of this restriction, for every $K\ge1$,
    \begin{align*}
        &\min_{0 \le k \le K-1} \norm{\nabla f(x_k) - \nabla f(x_\star)}^2 \le \frac{ 4(L_0 + L_1 \norm{\nabla f(x_\star)})^2\norm{x_0 - x_\star}^2}{K}.
    \end{align*}
\end{theorem}

\begin{proof}
Write $Z_k=\|\nabla f(x_k)-\nabla f(x_\star)\|$, $c_k=L_0+L_1\|\nabla f(x_k)\|$, $c_\star=L_0+L_1\|\nabla f(x_\star)\|$, and $R_0=\|x_0-x_\star\|$. The prescribed radius, including its zero-residual branch, satisfies
\begin{equation}\label{eq:asym-radius-lower}
 t_k=\frac{Z_k}{2}\left(\frac1{c_k}+\frac1{c_\star}\right)
 \ge\frac{Z_k}{2c_\star}.
\end{equation}
If $Z_k>0$, \Cref{lemma:l0l1} implies
\[
 t_k\overset{\eqref{eq:asym-radius-lower}}{=}
 \frac{Z_k}{2}\left(\frac1{c_k}+\frac1{c_\star}\right)
 \overset{\text{Lemma~\ref{lemma:l0l1}}}{\le}
 \frac{\langle\nabla f(x_k)-\nabla f(x_\star),x_k-x_\star\rangle}{Z_k}.
\]
Thus Type-II admissibility holds. If $Z_k=0$, the radius is zero and the iterate is unchanged. In both cases, the distance recurrence gives, after summation,
\begin{equation}\label{eq:asym-radius-telescope}
 \|x_K-x_\star\|^2+\sum_{k=0}^{K-1}t_k^2
 \overset{\eqref{eq:thm_squared_distance_decrease}}{\le}R_0^2.
\end{equation}
At zero-radius steps the recurrence is equality and requires no invocation of the positive-radius theorem. Dropping the final squared distance and dividing by $K$ proves the first two claims. Moreover, for every $K\ge1$,
\[
 \min_{0\le k<K}Z_k^2
 \le\frac1K\sum_{k=0}^{K-1}Z_k^2
 \overset{\eqref{eq:asym-radius-lower}}{\le}
 \frac{4c_\star^2}{K}\sum_{k=0}^{K-1}t_k^2
 \overset{\eqref{eq:asym-radius-telescope}}{\le}
 \frac{4c_\star^2R_0^2}{K}.
\]
For the additional denominator bound, the triangle inequality gives $c_k\le c_\star+L_1Z_k$, and hence
\begin{equation}\label{eq:asym-radius-refined}
 \frac{Z_k}{c_\star+L_1Z_k}
 \le\frac{Z_k}{2}\left(\frac1{c_k}+\frac1{c_\star}\right)
 \overset{\eqref{eq:asym-radius-lower}}{=}t_k.
\end{equation}
Choose $j\in\{0,\ldots,K-1\}$ with $t_j\le R_0/\sqrt K$, which exists by \eqref{eq:asym-radius-telescope}. Then
\[
 \frac{Z_j}{c_\star+L_1Z_j}
 \overset{\eqref{eq:asym-radius-refined}}{\le}t_j
 \le\frac{R_0}{\sqrt K}.
\]
When $\sqrt K>L_1R_0$, rearranging yields
\[
 \min_{0\le k<K}Z_k\le Z_j
 \le\frac{c_\star R_0}{\sqrt K-L_1R_0},
\]
which proves the remaining claim.
\end{proof}

\begin{remark}
    \label{rem:asym_L0L1_rate}
    Theorem~\ref{thm:asym_L0L1_rate} is the natural analogue of \Cref{thm:L-smooth} under asymmetric $(L_0,L_1)$--smoothness. 
    In the standard $L$--smooth case, the coefficient in front of $\norm{\nabla f(x_k) - \nabla f(x_\star)}$ is the constant $\nicefrac{1}{L}$, and one recovers the bound
    \begin{align*}
        \min_{0 \le k \le K-1} \norm{\nabla f(x_k) - \nabla f(x_\star)} \le \frac{L_0\norm{x_0 - x_\star}}{\sqrt{K}}.
    \end{align*}
\end{remark}

%% file: appendix-table-nonconvex.tex
\subsection{{D. Beyond convexity}}
\label{sec:non-convex}
{Group D of \Cref{tbl:main} first gives stationarity without convexity, then objective convergence under projected P\L. In both rows, $G_\gamma(x)=\gamma^{-1}(x-\Proj_{\cX}(x-\gamma\nabla f(x)))$ is the projected gradient mapping.}
\subsubsection{{Smooth nonconvex stationarity}}
\begin{theorem}[Sublinear convergence in the smooth non-convex case]
\label{thm:local_lmo_nonconvex_smooth}
    Let \Cref{ass:main} hold and let $f$ be differentiable; convexity of $f$ is not required.
    Let $f$ be $L$--smooth on an open set containing $\cX$.
    Assume that the constrained minimum $f_\star = \min_{x \in \cX} f(x)$ is attained.
    Fix some $\gamma \in (0,\nicefrac{1}{L}]$ and let $\{x_k\}_{k \ge 0}$ be generated by
    \begin{align*}
        x_{k+1} \in \argmin_{z \in \cX \cap \mathbb{B}(x_k,t_k)} \inner{\nabla f(x_k)}{z},
    \end{align*}
    with radii
    \begin{align*}
        t_k := \gamma \norm{G_\gamma(x_k)}.
    \end{align*}
    Then, for every $k \ge 0$,
    \begin{align*}
        f(x_{k+1}) \le f(x_k) - \frac{\gamma}{2}\norm{G_\gamma(x_k)}^2.
    \end{align*}
    Consequently, for every $K \ge 1$,
    \begin{align*}
        \min_{0 \le k \le K-1}\norm{G_\gamma(x_k)}^2 \le \frac{2\rbrac{f(x_0)-f_\star}}{\gamma K}.
    \end{align*}
    In particular, if $\gamma = \nicefrac{1}{L}$, then
    \begin{align*}
        \min_{0 \le k \le K-1}\norm{G_{1/L}(x_k)}^2 \le \frac{2L\rbrac{f(x_0)-f_\star}}{K}.
    \end{align*}
\end{theorem}

\begin{proof}
    Fix $k \ge 0$, and for simplicity, define
    \begin{align*}
        y_k := \Proj_{\cX}\rbrac{x_k - \gamma \nabla f(x_k)}.
    \end{align*}
    Then, by the definition of $G_\gamma$,
    \begin{align*}
        y_k = x_k - \gamma G_\gamma(x_k), \qquad \norm{y_k - x_k} = \gamma \norm{G_\gamma(x_k)} = t_k.
    \end{align*}
    Hence $y_k \in \cX \cap \mathbb{B}(x_k,t_k)$, and by the optimality of $x_{k+1}$ for the local linear minimization problem,
    \begin{align*}
        \inner{\nabla f(x_k)}{x_{k+1}} \le \inner{\nabla f(x_k)}{y_k}.
    \end{align*}
    Equivalently, 
    $
        \inner{\nabla f(x_k)}{x_{k+1} - x_k} \le \inner{\nabla f(x_k)}{y_k - x_k}.
    $
    By the optimality condition of projection onto the closed convex set $\cX$ for $y_k = \Proj_{\cX}\rbrac{x_k - \gamma \nabla f(x_k)}$, we have
    \begin{align*}
        \inner{y_k - \rbrac{x_k - \gamma \nabla f(x_k)}}{z - y_k} \ge 0, \qquad \forall z \in \cX.
    \end{align*}
    Choosing $z = x_k \in \cX$, we obtain
    \begin{align*}
        \inner{y_k - x_k + \gamma \nabla f(x_k)}{x_k - y_k} \ge 0.
    \end{align*}
    Expanding this inequality gives
    \begin{align*}
        \inner{\nabla f(x_k)}{y_k - x_k} \le - \frac{1}{\gamma}\norm{y_k - x_k}^2 = - \gamma \norm{G_\gamma(x_k)}^2.
    \end{align*}
    Therefore,
    \begin{align*}
        \inner{\nabla f(x_k)}{x_{k+1} - x_k} \le - \gamma \norm{G_\gamma(x_k)}^2.
    \end{align*}
    Since $f$ is $L$--smooth, we have
    \begin{align*}
        f(x_{k+1}) \le f(x_k) + \inner{\nabla f(x_k)}{x_{k+1} - x_k} + \frac{L}{2}\norm{x_{k+1} - x_k}^2.
    \end{align*}
    Because $x_{k+1} \in \mathbb{B}(x_k,t_k)$, we also have
    \begin{align*}
        \norm{x_{k+1} - x_k} \le t_k = \gamma \norm{G_\gamma(x_k)}.
    \end{align*}
    Substituting the previous two estimates yields
    \begin{align*}
        f(x_{k+1}) &\le f(x_k) - \gamma \norm{G_\gamma(x_k)}^2 + \frac{L\gamma^2}{2}\norm{G_\gamma(x_k)}^2 \\
        &= f(x_k) - \gamma \rbrac{1 - \frac{L\gamma}{2}} \norm{G_\gamma(x_k)}^2.
    \end{align*}
    Since $\gamma \le \nicefrac{1}{L}$, we have $1 - \nicefrac{L\gamma}{2} \ge \nicefrac{1}{2}$, and thus
    \begin{align*}
        f(x_{k+1}) \le f(x_k) - \frac{\gamma}{2}\norm{G_\gamma(x_k)}^2.
    \end{align*}
    Summing this inequality for $k = 0,1,\dots,K-1$, we obtain
    \begin{align*}
        \frac{\gamma}{2}\sum_{k=0}^{K-1}\norm{G_\gamma(x_k)}^2 \le f(x_0) - f(x_K) \le f(x_0) - f_\star.
    \end{align*}
    Hence
    \begin{align*}
        \sum_{k=0}^{K-1}\norm{G_\gamma(x_k)}^2 \le \frac{2\rbrac{f(x_0)-f_\star}}{\gamma}.
    \end{align*}
    Dividing by $K$, we get
    \begin{align*}
        \min_{0 \le k \le K-1}\norm{G_\gamma(x_k)}^2 \le \frac{1}{K}\sum_{k=0}^{K-1}\norm{G_\gamma(x_k)}^2 \le \frac{2\rbrac{f(x_0)-f_\star}}{\gamma K}.
    \end{align*}
    The last claim follows by setting $\gamma = \nicefrac{1}{L}$.
\end{proof}

\begin{remark}[Interpretation of the non-convex guarantee]
\label{rem:local_lmo_nonconvex_stationary_cluster}
    The conclusion of \Cref{thm:local_lmo_nonconvex_smooth} is a rate to approximate first order stationarity, measured by the projected gradient mapping. The proof of the theorem shows that
  \[
        \sum_{k=0}^{\infty}\norm{G_\gamma(x_k)}^2 < \infty,
\]
    and hence
 \[
        \lim \limits_{k \to \infty}\norm{G_\gamma(x_k)} = 0.
\]
    Therefore, if $x_{k_j} \to \bar{x}$ along some subsequence, then by continuity of $\nabla f$ and of the Euclidean projection onto a closed convex set,
    \begin{align*}
        G_\gamma(\bar{x}) = \lim_{j \to \infty} G_\gamma(x_{k_j}) = 0.
    \end{align*}
    Equivalently,
    \begin{align*}
        \bar{x} = \Proj_{\cX}\rbrac{\bar{x} - \gamma \nabla f(\bar{x})},
    \end{align*}
    and hence
    \begin{align*}
        \inner{\nabla f(\bar{x})}{z - \bar{x}} \ge 0, \qquad \forall z \in \cX,
    \end{align*}
    or, equivalently,
    \begin{align*}
        0 \in \nabla f(\bar{x}) + N_{\cX}(\bar{x}).
    \end{align*}
    Thus, every accumulation point is a first order stationary point of the constrained problem.
\end{remark}
\subsubsection{{Projected-PL objective convergence}}

We next analyze the non-convex case under the Polyak--{\L}ojasiewicz assumption, which is standard in nonconvex optimization.
\paragraph{The problem with standard P\L.}
The usual Polyak--\L ojasiewicz (P\L) inequality is
\begin{align*}
    \frac{1}{2}\norm{\nabla f(x)}^2 \ge \mu \rbrac{f(x) - f(x_\star)}, \quad \forall x \in \cX.
\end{align*}
For unconstrained problems, this is natural, because if $x_\star$ is a minimizer, then necessarily $\nabla f(x_\star) = 0$.
Therefore, the norm of the full gradient is a valid measure of stationarity.
However, for constrained problems, a constrained minimizer $x_\star \in \cX_\star$ does not in general satisfy $\nabla f(x_\star) = 0$.
Instead, what holds is the first-order optimality condition
\begin{align*}
    \inner{\nabla f(x_\star)}{z - x_\star} \ge 0, \qquad \forall z \in \cX.
\end{align*}

A natural first attempt to extend the Polyak--\L ojasiewicz inequality to the constrained setting is to impose
\begin{align*}
    \frac{1}{2}\norm{\nabla f(x)-\nabla f(x_\star)}^2 \ge \mu \rbrac{f(x)-f(x_\star)}, \qquad \forall x \in \cX.
\end{align*}
However, this condition is not the most intrinsic constrained analogue of the P\L\  inequality. 
The quantity $\norm{\nabla f(x)-\nabla f(x_\star)}$ measures whether the ambient gradient at $x$ matches the ambient gradient at one fixed minimizer $x_\star$. 
But constrained stationarity is not characterized by matching $\nabla f(x_\star)$. 
Rather, it is characterized by the condition $0 \in \nabla f(x)+N_{\cX}(x)$, where the normal cone depends on the current point $x$. 
This suggests that a constrained P\L\ condition should be formulated in terms of a residual whose zero set coincides with the constrained first-order condition, rather than in terms of agreement with the ambient gradient at a fixed minimizer.

\paragraph{Projected P\L.}
A more principled way to extend the Polyak--\L ojasiewicz inequality to the constrained setting is to view the problem through the Projected Gradient Descent framework.
For a stepsize $\gamma > 0$, the \algname{ProjGD} update is $x^+ = \Proj_{\cX}\bigl(x - \gamma \nabla f(x)\bigr)$.
This naturally leads to the projected gradient mapping
\begin{align*}
    G_\gamma(x) \eqdef \frac1\gamma\parens{x-\Proj_{\cX}(x-\gamma \nabla f(x))}.
\end{align*}
This object is the constrained analogue of the gradient.
When $\cX = \R^d$, we have
\begin{align*}
    G_\gamma(x) = \frac{1}{\gamma}\rbrac{x - \Proj_{\R^d}(x - \gamma \nabla f(x))} = \nabla f(x).
\end{align*}
Moreover, for a closed convex set $\cX$, the condition $G_\gamma(x) = 0$ is equivalent to
\begin{align*}
    x = \Proj_{\cX}\rbrac{x - \gamma \nabla f(x)}.
\end{align*}
The above identity, from the perspective of projection onto a closed convex set, is equivalent to $\inner{-\gamma \nabla f(x)}{z - x} \le 0$ for all $z \in \cX$, i.e.,
\begin{align*}
    \inner{\nabla f(x)}{z - x} \ge 0, \quad \forall z \in \cX,
\end{align*}
which is precisely the variational inequality form of first-order stationarity for minimizing $f$ over $\cX$.
More explicitly, the above condition is equivalent to $-\nabla f(x) \in N_{\cX}(x)$, or $0 \in \nabla f(x) + N_{\cX}(x)$.

\mdfsetup{nobreak=true}
\begin{assumption}[Projected P\L]
    \label{ass:projected_pl}
    Let \Cref{ass:main} hold and let $f$ be differentiable; convexity of $f$ is not required.
    Furthermore, assume that there exist constants $\mu > 0$ and $\gamma > 0$ such that
    \begin{align}
        \label{eq:projected_pl}
        \frac{1}{2}\norm{G_{\gamma}(x)}^2 \ge \mu \rbrac{f(x)-f(x_\star)}, \qquad \forall x \in \cX,
    \end{align}
    where
    \begin{align*}
        G_{\gamma}(x) = \frac{1}{\gamma}\rbrac{x-\Proj_{\cX}\rbrac{x-\gamma \nabla f(x)}}.
    \end{align*}
\end{assumption}
\mdfsetup{nobreak=false}

\begin{theorem}[Linear convergence under Projected P\L]
\label{thm:local_lmo_projected_pl}
    Let \Cref{ass:projected_pl} hold for some $\gamma \in (0,\nicefrac{1}{L}]$, and let $f$ be $L$--smooth on an open set containing $\cX$.
    Assume that the constrained minimum $f_\star = \min_{x \in \cX} f(x)$ is attained.
    {Let $L>0$ and $x_0\in\cX$.} Let $\{x_k\}_{k \ge 0}$ be generated by
    \begin{align*}
        x_{k+1} \in \argmin_{z \in \cX \cap \mathbb{B}(x_k,t_k)} \inner{\nabla f(x_k)}{z},
    \end{align*}
    with radii
    \begin{equation} \label{eq:-0-08uyh9f8y9gfg9_09uv}
        t_k \eqdef \gamma \norm{G_\gamma(x_k)}.
    \end{equation}
    Then, for every $k \ge 0$,
    \begin{align*}
        f(x_{k+1}) \le f(x_k) - \frac{\gamma}{2}\norm{G_\gamma(x_k)}^2.
    \end{align*}
    Consequently, for any $k \geq 0$, with zeroth powers interpreted as $1$ even when the base is zero,
    \begin{align*}
        f(x_k) - f_\star \le \rbrac{1 -  \gamma \mu}^k \rbrac{f(x_0) - f_\star}.
    \end{align*}
    In particular, if $\gamma = \nicefrac{1}{L}$, then
    \begin{align*}
        f(x_k) - f_\star \le \rbrac{1 - \frac{\mu}{L}}^k \rbrac{f(x_0) - f_\star}.
    \end{align*}
\end{theorem}

\begin{proof}
    Fix $k \ge 0$, and for simplicity, define
    \begin{align*}
        y_k \eqdef \Proj_{\cX}\rbrac{x_k - \gamma \nabla f(x_k)}.
    \end{align*}
    Then, by the definition of $G_\gamma$,
    \begin{align}\label{eq:-iny7dX079fg}
        y_k = x_k - \gamma G_\gamma(x_k), \qquad \norm{y_k - x_k} = \gamma \norm{G_\gamma(x_k)} \overset{\eqref{eq:-0-08uyh9f8y9gfg9_09uv}}{=} t_k.
    \end{align}
    In addition, $y_k \in \cX \cap \mathbb{B}(x_k,t_k)$, and thus
    \begin{align*}
        \inner{\nabla f(x_k)}{x_{k+1}} \le \inner{\nabla f(x_k)}{y_k},
    \end{align*}
    or equivalently,
    \begin{align}\label{eq:n8t6td98b96t9df}
        \inner{\nabla f(x_k)}{x_{k+1} - x_k} \le \inner{\nabla f(x_k)}{y_k - x_k}.
    \end{align}
    By the optimality condition of projection onto a convex set, for $y_k = \Proj_{\cX}\rbrac{x_k - \gamma \nabla f(x_k)}$ we have
    \begin{align*}
        \inner{y_k - \rbrac{x_k - \gamma \nabla f(x_k)}}{z - y_k} \ge 0, \qquad \forall z \in \cX.
    \end{align*}
    Choosing $z = x_k \in \cX$, we obtain
    \begin{align*}
        \inner{y_k - x_k + \gamma \nabla f(x_k)}{x_k - y_k} \ge 0.
    \end{align*}
    Expanding this inequality gives
    \begin{align} \label{eq:n8bv9fy09_98y_-8u89fd}
        \inner{\nabla f(x_k)}{y_k - x_k} \le - \frac{1}{\gamma}\norm{y_k - x_k}^2 \overset{\eqref{eq:-iny7dX079fg}}{=} - \gamma \norm{G_\gamma(x_k)}^2.
    \end{align}
    By combining \eqref{eq:n8t6td98b96t9df} and \eqref{eq:n8bv9fy09_98y_-8u89fd}, we get
    \begin{align}\label{eq:gb9v0Z9ybbv87f}
        \inner{\nabla f(x_k)}{x_{k+1} - x_k} \le - \gamma \norm{G_\gamma(x_k)}^2.
    \end{align}
    Since $f$ is $L$--smooth, we have
    \begin{align}\label{eq:ncvx_L-smooth_9hfd}
        f(x_{k+1}) \le f(x_k) + \inner{\nabla f(x_k)}{x_{k+1} - x_k} + \frac{L}{2}\norm{x_{k+1} - x_k}^2.
    \end{align}
    Because $x_{k+1} \in \mathbb{B}(x_k,t_k)$, we also have
    \begin{align} \label{eq:__in980d980f8g}
        \norm{x_{k+1} - x_k} \le t_k \overset{\eqref{eq:-0-08uyh9f8y9gfg9_09uv}}{=} \gamma \norm{G_\gamma(x_k)}.
    \end{align}
    Substituting \eqref{eq:gb9v0Z9ybbv87f} and \eqref{eq:__in980d980f8g} into \eqref{eq:ncvx_L-smooth_9hfd} yields
    \begin{align*}
        f(x_{k+1}) &\le f(x_k) - \gamma \norm{G_\gamma(x_k)}^2 + \frac{L\gamma^2}{2}\norm{G_\gamma(x_k)}^2 \\
        &= f(x_k) - \gamma \rbrac{1 - \frac{L\gamma}{2}} \norm{G_\gamma(x_k)}^2.
    \end{align*}
    Since $\gamma \le \nicefrac{1}{L}$, we have $1 - \nicefrac{L\gamma}{2} \ge \nicefrac{1}{2}$, and thus
    \begin{align*}
        f(x_{k+1}) \le f(x_k) - \frac{\gamma}{2}\norm{G_\gamma(x_k)}^2.
    \end{align*}
    Finally, by \eqref{eq:projected_pl},
    \begin{align*}
        \frac{1}{2}\norm{G_\gamma(x_k)}^2 \ge \mu \rbrac{f(x_k) - f_\star}.
    \end{align*}
    Combining this with the descent estimate gives
    \begin{align*}
        f(x_{k+1}) \le f(x_k) - \gamma \mu \rbrac{f(x_k) - f_\star}.
    \end{align*}
    Rearranging, we obtain
    \begin{align*}
        f(x_{k+1}) - f_\star \le \rbrac{1 -  \gamma\mu}  \rbrac{f(x_k) - f_\star}.
    \end{align*}
    Iterating this inequality yields the result.
\end{proof}

%% file: appendix-experiments.tex
\clearpage
\section{Experiments and numerical illustrations}
\label{sec:experiments}
\label{sec:num_exp}

The experiments address two questions: how does \algname{BroxGD} behave, and when is its local linear solve cheap enough to give a runtime advantage? These questions need different tests. The two-dimensional examples make the geometry visible; the larger synthetic problems compare the cost of reaching a given objective accuracy.

We begin with an explicit simplex problem where projection is cheap and \algname{ProjGD} wins. We then compare all five application families from Appendix~\ref{sec:potential-applications}, including interior and boundary solutions, and examine how radius selection changes the cost. A selected ranking example illustrates a possible advantage and its limitations. The correlation-matrix study in \Cref{sec:projection-gap} gives the strongest evidence for faster local solves and expands the experiment in the main text. Finally, the two-dimensional examples connect the observed trajectories to the theory. All instances are synthetic; the experiments assess optimization behavior rather than performance on real application data.

\input{appendix-structured-experiment}

\FloatBarrier
\input{appendix-application-study}

\input{appendix-fair-ranking-example}

\FloatBarrier
\input{appendix-projection-gap}

\subsection{A two-dimensional problem for understanding the geometry}\label{sec:exp_setup}

We consider the quadratic objective
\[
    f(x)=\frac{1}{2} x^\top \mQ x, \qquad x\in\R^2,
\]
where $\mQ\in \R^{2\times 2}$ is symmetric positive definite with eigenvalues $\mu=1$ and  $L=100.$
Thus, $f$ is $\mu$--strongly convex and $L$--smooth. We choose $\mQ$ as a rotated diagonal matrix,
\[
    \mQ = \mR
    \begin{pmatrix}
    1 & 0\\
    0 & 100
    \end{pmatrix}
    \mR^\top,
    \quad
    \mR=
    \begin{pmatrix}
    \cos(\pi/6) & -\sin(\pi/6)\\
    \sin(\pi/6) & \cos(\pi/6)
    \end{pmatrix}
\]
so that explicitly
\[
    \mQ=
    \begin{pmatrix}
    25.75 & -42.86825749\\
    -42.86825749 & 75.25
    \end{pmatrix}.
\]
By construction, the unconstrained minimizer of $f$ is the origin $x^\mathrm{unc}_\star=(0,0).$
The feasible set is the box $\cX \eqdef \{x\in \R^2 : \|x-c\|_\infty \le 1\}$, $c=(3,3)$, that is, $\cX=[2,4]\times [2,4]$.
The starting point is $x_0=(4,4)\in \cX$.
In this experiment, the constrained minimizer is $x_\star \approx (3.3295734,2)$, which lies on the lower edge of the box.
At iteration $k$, {\algname{BroxGD}} computes
\[
    x_{k+1}\in \argmin_{z\in \cX \cap \mathbb{B}(x_k,t_k)} \langle \nabla f(x_k),z\rangle,
\]
where the radius is chosen according to the theoretically prescribed rule
\[
t_k\overset{\eqref{eq:radius_smooth_strong}}{=}\theta \|x_k-x_\star\|, \qquad \theta=\frac{2\sqrt{\mu L}}{L+\mu} = \frac{2\sqrt{100}}{101} \approx 0.198020.
\]
\Cref{thm:smooth-strong} predicts
\[
    \|x_{k+1}-x_\star\|^2 \overset{\eqref{eq:linear-improved}}{\le} \left(1-\theta^2\right)\|x_k-x_\star\|^2 \approx 0.960788 \cdot \|x_k-x_\star\|^2.
\]
For the convergence plots we run \Cref{alg:new} for $200$ iterations; the method and radius comparisons below use $100$ iterations. 
Since the problem is two-dimensional, the subproblem can be solved exactly by enumerating the relevant boundary candidates: corners of the box lying inside the ball, intersections of the circle and the box edges, and the unconstrained minimizer of the linear function over the ball whenever it is feasible.

\subsection{Geometry of the iterates}

\Cref{fig:local-lmo-2d} shows the contour lines of the quadratic objective, the box constraint, the unconstrained minimizer, the constrained minimizer, the first few balls $\mathbb{B}(x_k,t_k)$, and the trajectory of the iterates.

\begin{figure}[!htbp]
    \centering
    \includegraphics[width=0.85\textwidth]{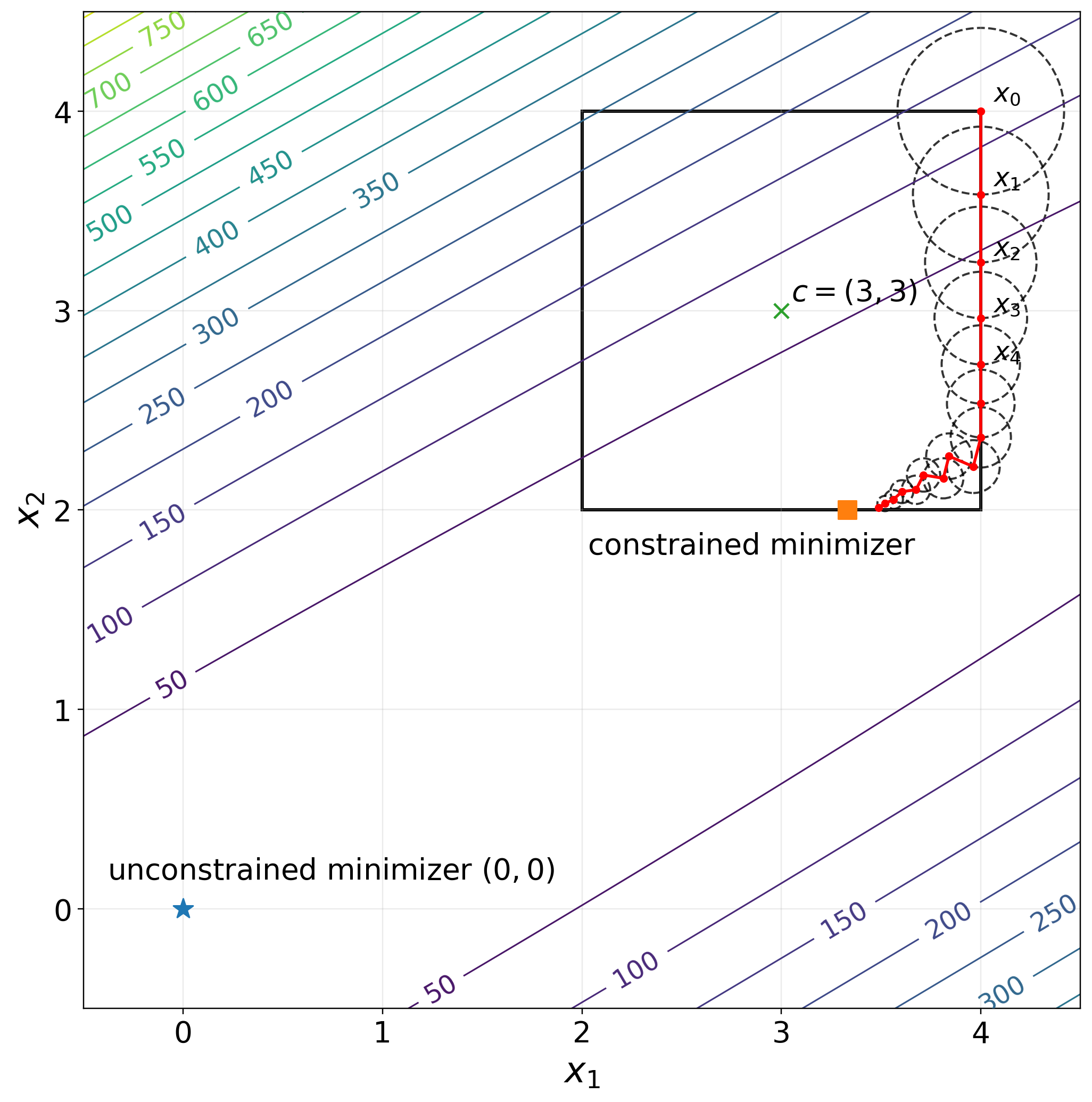}
    \caption{Two-dimensional illustration of {\algname{BroxGD}} with the radius rule $t_k=\theta\|x_k-x_\star\|$. The plot shows the contour lines of the quadratic objective, the feasible box $\cX=[2,4]\times [2,4]$ with center $c=(3,3)$, the unconstrained minimizer $(0,0)$, the constrained minimizer $x_\star$, the first few trust-region balls, and 15 iterates $\{x_k\}$. In this case, $L=100$, $\mu=1$ and $\theta\approx 0.19802$.}
    \label{fig:local-lmo-2d}
\end{figure}

The trajectory shows how the ball and the constraint boundary jointly determine the step. Initially, the method starts at the upper-right corner of the box and initially moves down along the right boundary. This is natural, because the negative gradient points roughly toward the lower-right direction, but the iterates must remain inside the feasible set and inside the local ball. 
Second, once the iterates get close to the lower boundary, they quickly move toward the constrained minimizer $x_\star$.
Third, the shrinking radii $t_k$ make the local feasible regions progressively smaller as the iterates approach the solution.

\FloatBarrier
\subsection{Gradient-difference convergence}

We next examine the quantity $\|\nabla f(x_k)-\nabla f(x_\star)\|^2$.
For its gradient-difference radius rule, \Cref{thm:L-smooth} gives the upper bound
\[
    \min_{0\le k\le K-1}\|\nabla f(x_k)-\nabla f(x_\star)\|^2 \overset{\eqref{eq:bgd-smooth-residual-sum}}{\le} \frac{L^2\|x_0-x_\star\|^2}{K}.
\]
The present run instead uses the strongly convex distance-based radius from \Cref{thm:smooth-strong}. Thus, the $L^2\|x_0-x_\star\|^2/(k+1)$ curve is a reference benchmark from a different radius rule, not a bound supplied by \Cref{thm:L-smooth} for this trajectory. \Cref{fig:local-lmo-grad-corrected} shows the raw gradient-difference sequence, its running minimum, and this reference curve.

\begin{figure}[!htbp]
    \centering
    \includegraphics[width=0.85\textwidth]{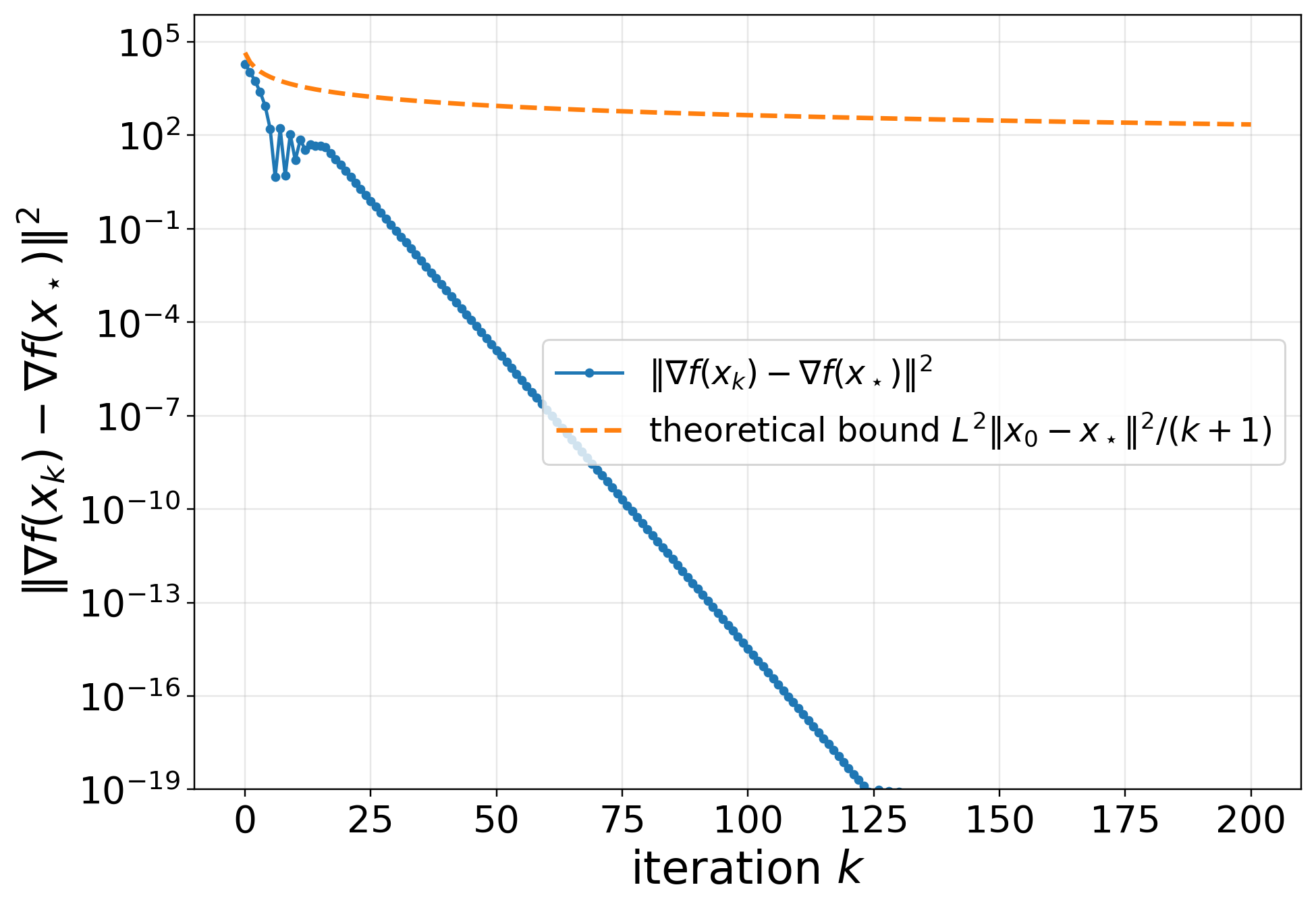}
    \caption{Semi-log plot of the gradient-difference quantity $\|\nabla f(x_k)-\nabla f(x_\star)\|^2$, together with its running minimum and the reference curve $L^2\|x_0-x_\star\|^2/(k+1)$.}
    \label{fig:local-lmo-grad-corrected}
\end{figure}

The behavior is strikingly better than what the $\cO(\nicefrac{1}{k})$ bound alone would suggest. 
After a short transient phase, the quantity $\|\nabla f(x_k)-\nabla f(x_\star)\|^2$ decreases essentially linearly on the logarithmic scale until it reaches machine precision. 
In particular, the running minimum decays much faster than the generic $\nicefrac{1}{k}$ benchmark. 
This is consistent with the fact that the present problem is strongly convex and that the radius rule was designed using the strong convexity and smoothness constants.

\FloatBarrier
\subsection{Distance convergence and linear theory}

Finally, \Cref{fig:local-lmo-distance-corrected} shows the squared distance to the constrained minimizer, together with the linear upper bound predicted by \Cref{thm:smooth-strong}:
\[
 \psi(k) \eqdef   \left(\frac{L-\mu}{L+\mu}\right)^{2k}\|x_0-x_\star\|^2.
\]

\begin{figure}[!htbp]
    \centering
    \includegraphics[width=0.85\textwidth]{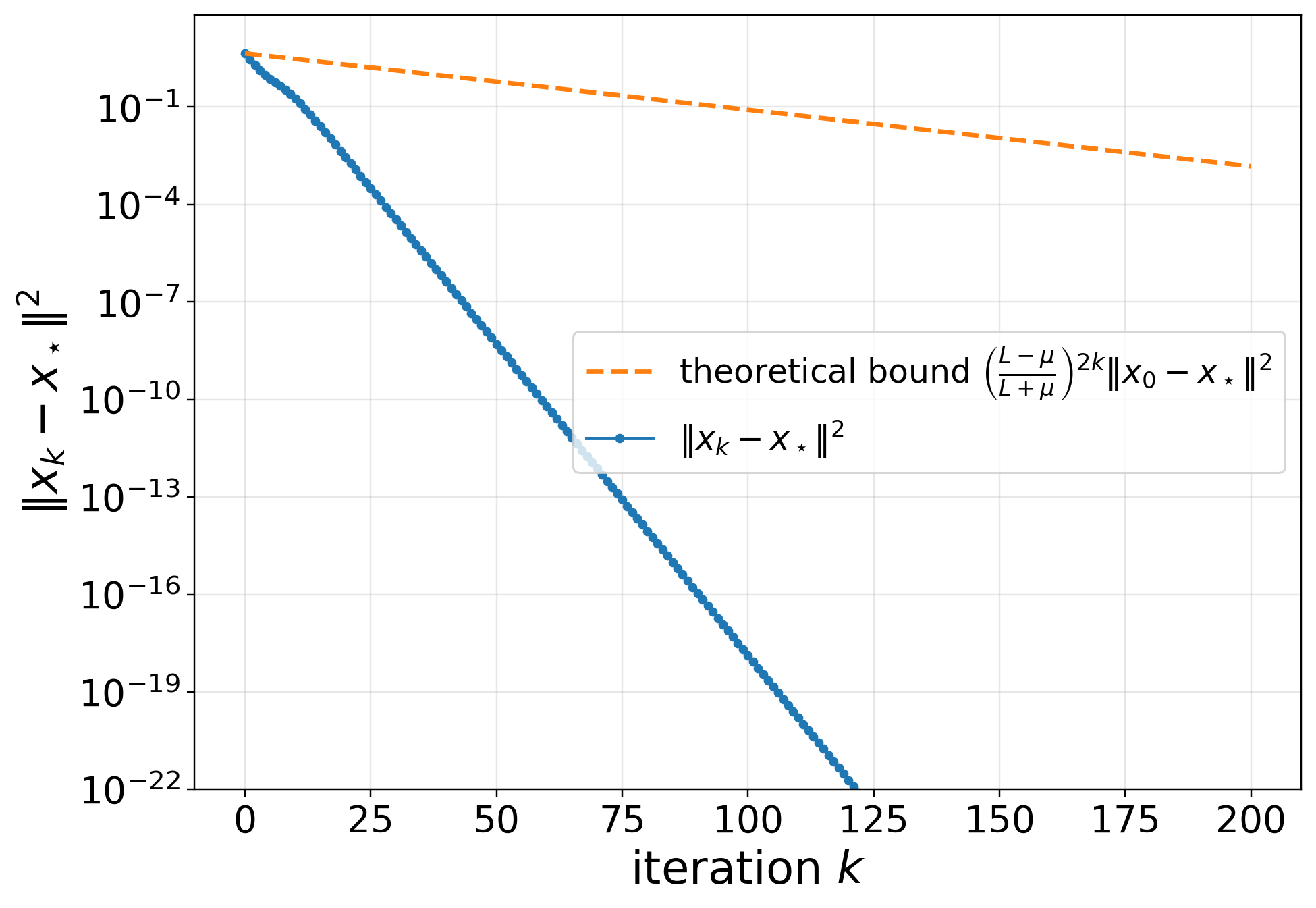}
    \caption{Semi-log plot of the squared distance to the constrained minimizer for {\algname{BroxGD}}, together with the geometric upper bound predicted by \Cref{thm:smooth-strong}.}
    \label{fig:local-lmo-distance-corrected}
\end{figure}

The empirical behavior supports the theorem. 
The observed sequence decays much faster than the theoretical geometric upper bound and reaches numerical precision after roughly $125$ iterations. 
This experiment also highlights an important implementation detail. The radius rule in \Cref{thm:smooth-strong} is $t_k=\theta\|x_k-x_\star\|$,
where $x_\star$ is the \emph{constrained} minimizer. 
If one mistakenly uses the unconstrained minimizer instead, then the predicted linear behavior may fail completely.

\FloatBarrier
\subsection{Comparison with Projected Gradient Descent and Frank--Wolfe}
\label{subsec:comparison-pgd-fw}

We now compare {\algname{BroxGD}} with two standard first-order methods for constrained smooth convex optimization: Projected Gradient Descent and Frank--Wolfe. 
We use exactly the same problem instance as in \Cref{sec:exp_setup}. The three methods are implemented as follows.

\paragraph{BroxGD.}
We use the same method and the same theoretically prescribed radius rule as before:
\[
    x_{k+1}\in \argmin_{z\in \cX\cap \mathbb{B}(x_k,t_k)} \langle \nabla f(x_k),z\rangle, \qquad t_k=\theta \|x_k-x_\star\|.
\]

\paragraph{Projected Gradient Descent.}
We run Projected Gradient Descent with the standard stepsize $\nicefrac{1}{L}$:
\begin{align}\label{eq:pgd}
    x_{k+1} = \Proj_{\cX}\left(x_k-\frac1L \nabla f(x_k)\right). \tag{\algname{ProjGD}}
\end{align}

\paragraph{Frank--Wolfe.}
We run the classical Frank--Wolfe method
\begin{align}\label{eq:fw}
    s_k\in \argmin_{s\in \cX}\langle \nabla f(x_k),s\rangle,
    \qquad
    x_{k+1}=(1-\gamma_k)x_k+\gamma_k s_k \tag{\algname{FW}}, \qquad  \gamma_k=\frac{2}{k+2}.
\end{align}
Since $\cX=[2,4]^2$ is a box, the linear minimization oracle is explicit.
We run all three methods for $100$ iterations. 

\begin{figure}[!htbp]
    \centering
    \begin{overpic}[width=0.85\textwidth]{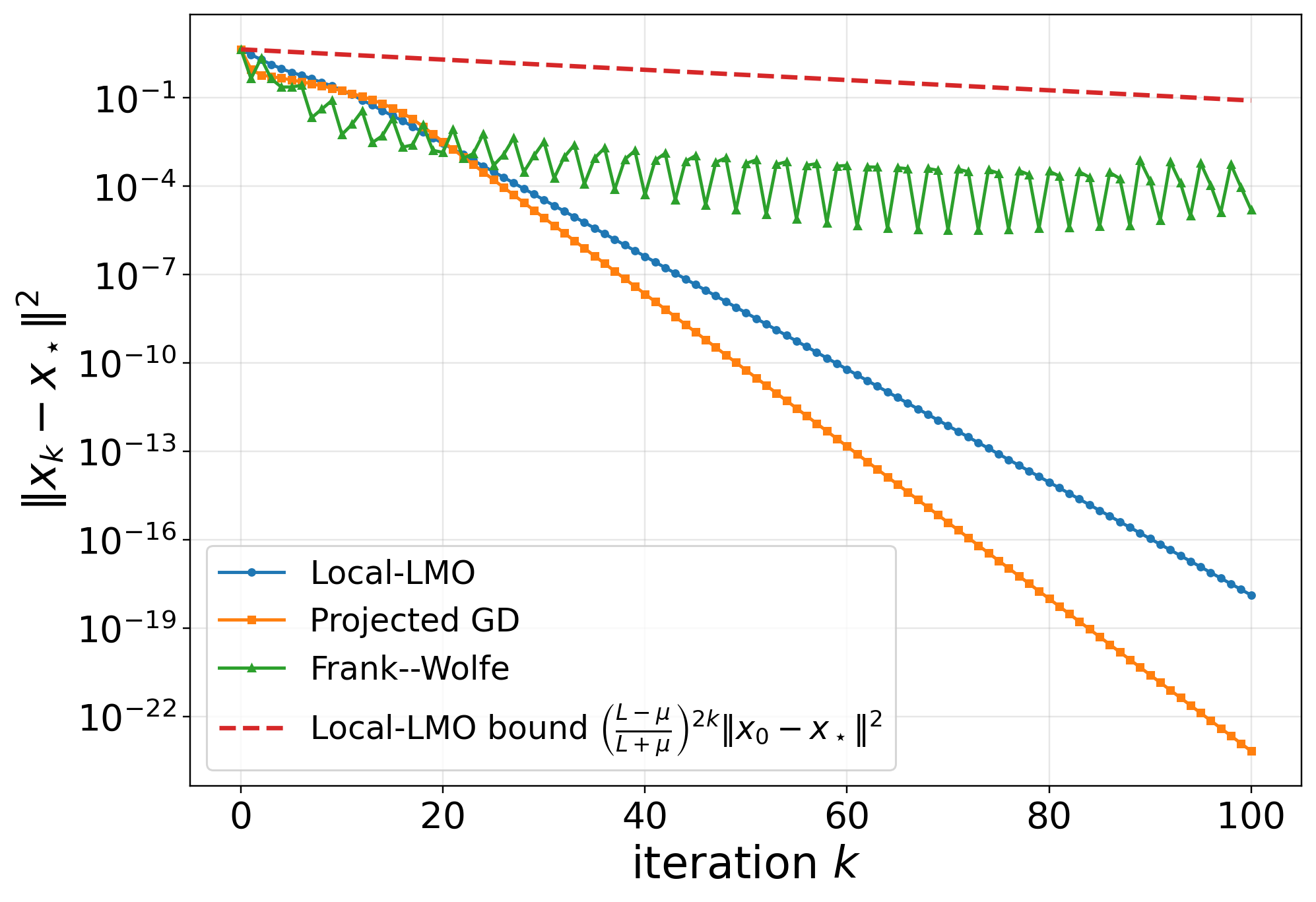}
      \put(23.3,23.9){\color{white}\rule{17\unitlength}{3.1\unitlength}}
      \put(23.6,24.3){{\algname{BroxGD}}}
      \put(23.3,19.9){\color{white}\rule{17\unitlength}{3.4\unitlength}}
      \put(23.6,20.6){\algname{ProjGD}}
      \put(23.3,10.9){\color{white}\rule{43.7\unitlength}{5.2\unitlength}}
      \put(23.6,12.3){{\algname{BroxGD}} {\sffamily\small bound} $\left(\frac{L-\mu}{L+\mu}\right)^{2k}\|x_0-x_\star\|^2$}
    \end{overpic}
    \caption{Semi-log plot of the squared distance to the constrained minimizer for {\algname{BroxGD}}, Projected Gradient Descent, and Frank--Wolfe, all run on the same two-dimensional strongly convex quadratic problem for $100$ iterations. For {\algname{BroxGD}}, we also plot the geometric upper bound $\left(\nicefrac{(L-\mu)}{(L+\mu)}\right)^{2k}\|x_0-x_\star\|^2$.}
    \label{fig:comparison-lmo-pgd-fw-ap}
    \label{fig:comparison-lmo-pgd-fw}
\end{figure}

\Cref{fig:comparison-lmo-pgd-fw-ap} shows the semi-log plot of the squared distance to the constrained minimizer $\|x_k-x_\star\|^2$.
The figure reveals a clear separation between the three methods.
First, \algname{ProjGD} is the fastest method on this example. 
Its convergence is essentially linear on the logarithmic scale and reaches extremely high accuracy already within the $100$-iteration horizon. 
This is not surprising: \algname{ProjGD} has direct access to the Euclidean projection onto the feasible set, and the present problem is a smooth strongly convex quadratic.

Second, {\algname{BroxGD}} also exhibits clear linear convergence, in full agreement with the theory. 
Moreover, the empirical decay is substantially faster than the conservative geometric upper bound.
Thus, at least on this simple problem, the theorem provides a valid but somewhat pessimistic worst-case estimate. 
The observed behavior is significantly better.

Third, Frank--Wolfe is much slower than the other two methods. Its trajectory displays the familiar zig-zagging behavior associated with projection-free methods on constrained problems. 
At the end of the $100$-iteration horizon, the observed squared distances are summarized in \Cref{tab:toy_distances}.
\begin{table}[htbp]
    \centering
    \caption{Observed squared distances after $100$ iterations.}
    \label{tab:toy_distances}
    \begin{tabular}{lccc}
        \toprule
        & {\algnameintable{BroxGD}} & \algnameintable{ProjGD} & \algnameintable{FW} \\
        \midrule
        $\norm{x_{100}-x_\star}^2$
        & $1.32\times 10^{-18}$
        & $6.71\times 10^{-24}$
        & $1.58\times 10^{-5}$ \\
        \bottomrule
    \end{tabular}
\end{table}

On this instance, {\algname{BroxGD}} converges faster than \algname{FW} and slower than \algname{ProjGD}. The comparison counts iterations, not runtime or oracle cost: projection onto this box is cheap, and the local oracle is solved by a specialized two-dimensional enumeration. These observations do not establish the same ordering for other constraint sets.

\FloatBarrier
\subsection{Comparison with geometric radius schedules}
\label{sec:comparison-geometric-radius}

We now compare the theoretically prescribed adaptive radius rule
$
    t_k=\theta\|x_k-x_\star\|,
$
with a family of geometric radius schedules of the form
\[
    t_k=\theta\|x_0-x_\star\|{r}^k,
    \qquad {r}\in(0,1).
\]
{Here $r$ is a radius-decay factor; the certified schedule in \Cref{thm:practical-geometric} uses $r=\sqrt q$ for the objective contraction factor $q=1-\mu/(4L)$.} The linear distance contraction in \eqref{eq:linear-improved} motivates a geometric schedule $t_k=c{r}^k$. Here we set $c=\theta\|x_0-x_\star\|$. The motivation for this comparison is that the adaptive rule depends on the unknown current distance to the solution, whereas the geometric rule uses only the initial distance and a fixed decay parameter~${r}$. 
Hence, it is natural to ask whether a well-tuned geometric schedule can mimic the behavior of the adaptive rule in practice.

We use exactly the same problem instance as before. The constant appearing in the radius rule is $
\theta \approx 0.198020.
$
For the geometric rule, we test $10$ equally spaced values of ${r}$ in the interval $[0.8,0.95]$.
All methods are run for $100$ iterations, and we compare the quantity $\|x_k-x_\star\|$.

\begin{figure}[!htbp]
    \centering
    \includegraphics[width=0.85\textwidth]{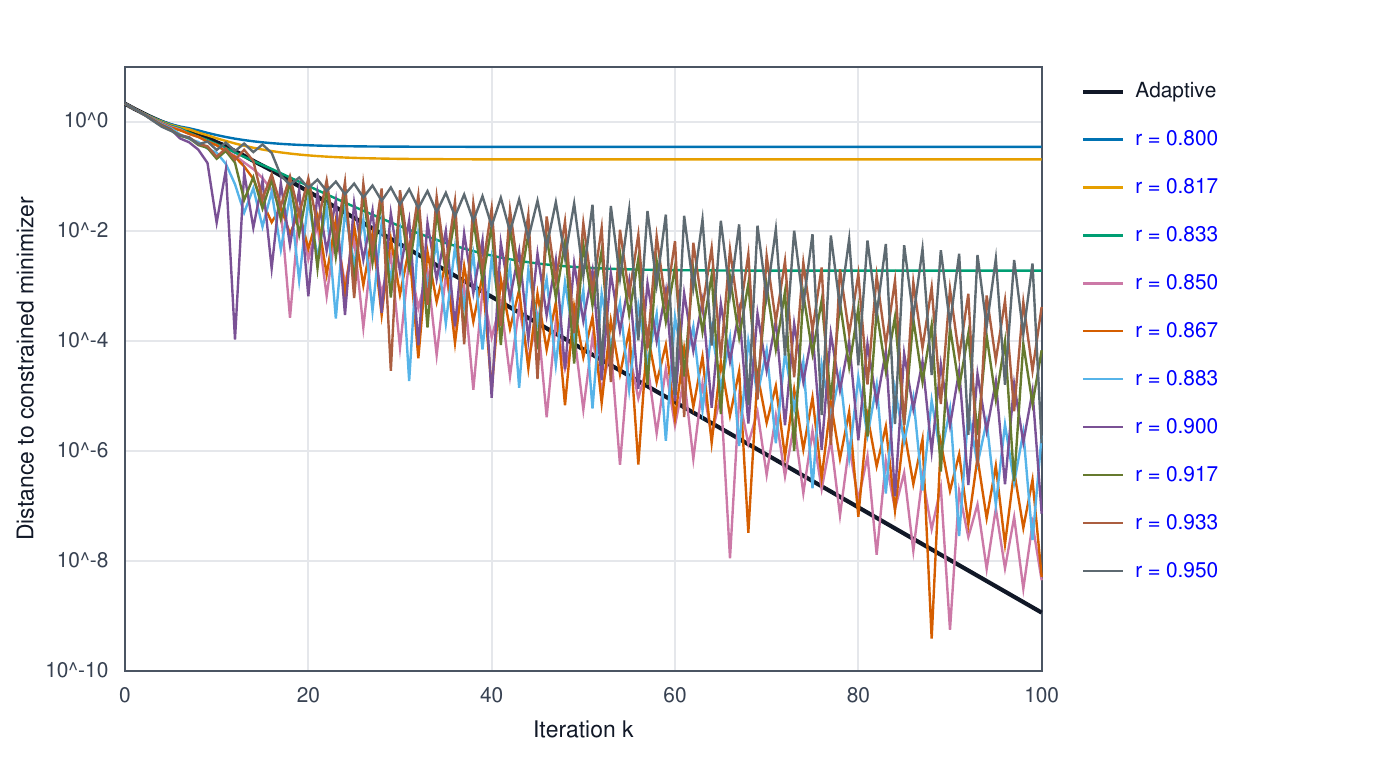}
    \caption{Semi-log plot comparing the adaptive radius rule $t_k=\theta\|x_k-x_\star\|$ with geometric radius schedules $t_k=\theta\|x_0-x_\star\|{r}^k$ for $10$ values of ${r}\in[0.8,0.95]$.}
    \label{fig:comparison-geometric-radius-ap}
    \label{fig:comparison-geometric-radius}
\end{figure}

\Cref{fig:comparison-geometric-radius-ap} reveals a clear pattern. The adaptive rule produces a stable and nearly straight line on the semi-log scale, indicating clean linear convergence. Several geometric schedules also perform well, but their behavior is much more sensitive to the choice of ${r}$.

If ${r}$ is too small, the radius shrinks too quickly. In that case the method becomes overly conservative, and progress stalls long before the iterate gets close to the solution. This phenomenon is clearly visible for ${r}=0.800$ and ${r}=0.817$, where the method stagnates at a relatively large distance from~$x_\star$.
If ${r}$ is too large, the radius shrinks too slowly. Then the local subproblems remain too large for too long, which leads to oscillatory behavior and poorer asymptotic performance. This can be seen for ${r}=0.917$, ${r}=0.933$, and ${r}=0.950$, where the convergence is still present but noticeably less efficient.

The best-performing geometric schedules are clustered around ${r}\approx 0.85 \text{ to } 0.90$.
In particular, the choices ${r}=0.850$ and ${r}=0.867$ almost coincide with the adaptive baseline over the entire horizon, and ${r}=0.900$ also performs very well.
At the end of the $100$-iteration run, the observed distances are summarized in \Cref{tab:q_res}.
\begin{table}[htbp]
    \centering
    \caption{Observed distances after $100$ iterations for different choices of ${r}$.}
    \label{tab:q_res}
    \begin{tabular}{lclclc}
        \toprule
        Setting & $\norm{x_{100}-x_\star}$ 
        & Setting & $\norm{x_{100}-x_\star}$ 
        & Setting & $\norm{x_{100}-x_\star}$ \\
        \midrule
        Adaptive baseline & $1.15\times 10^{-9}$
        & ${r}=0.800$ & $3.47\times 10^{-1}$
        & ${r}=0.817$ & $2.05\times 10^{-1}$ \\
        ${r}=0.833$ & $1.93\times 10^{-3}$
        & ${r}=0.850$ & $4.45\times 10^{-9}$
        & ${r}=0.867$ & $5.15\times 10^{-9}$ \\
        ${r}=0.883$ & $1.91\times 10^{-6}$
        & ${r}=0.900$ & $7.27\times 10^{-8}$
        & ${r}=0.917$ & $6.86\times 10^{-5}$ \\
        ${r}=0.933$ & $4.22\times 10^{-4}$
        & ${r}=0.950$ & $1.39\times 10^{-6}$
        & & \\
        \bottomrule
    \end{tabular}
\end{table}

A well-tuned geometric schedule nearly matches the adaptive rule on this example, while the adaptive rule avoids selecting a decay parameter. Both baselines use the known reference solution: the geometric schedule still requires the initial distance, and the adaptive schedule requires the current distance at every iteration. Thus, this is a controlled heuristic comparison, not a test of a solution-independent implementation. The certified practical schedules in \Cref{sec:practical-radii} provide separate theoretical guarantees.

\FloatBarrier
\subsection{Alternative constraint geometries}

\subsubsection{Shifted boxes}

The two shifted boxes in \Cref{fig:local-lmo-2d-intro} illustrate how the local balls interact with the boundary at different locations relative to the unconstrained minimizer. Each run uses the distance-based radius with respect to its own constrained solution.

\begin{figure}[!htbp]
    \centering
    \includegraphics[width=0.35\textwidth]{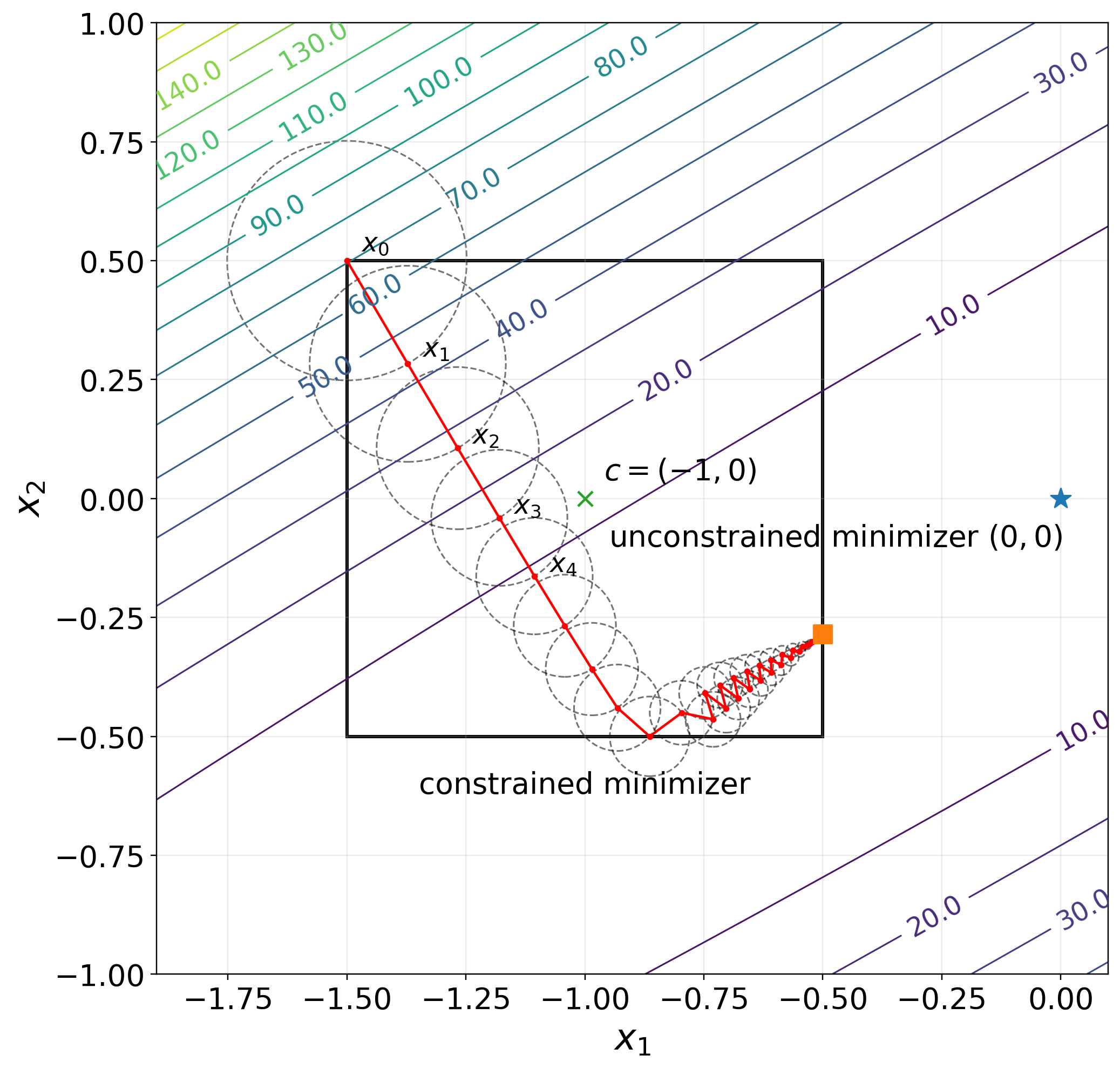}
        \includegraphics[width=0.35\textwidth]{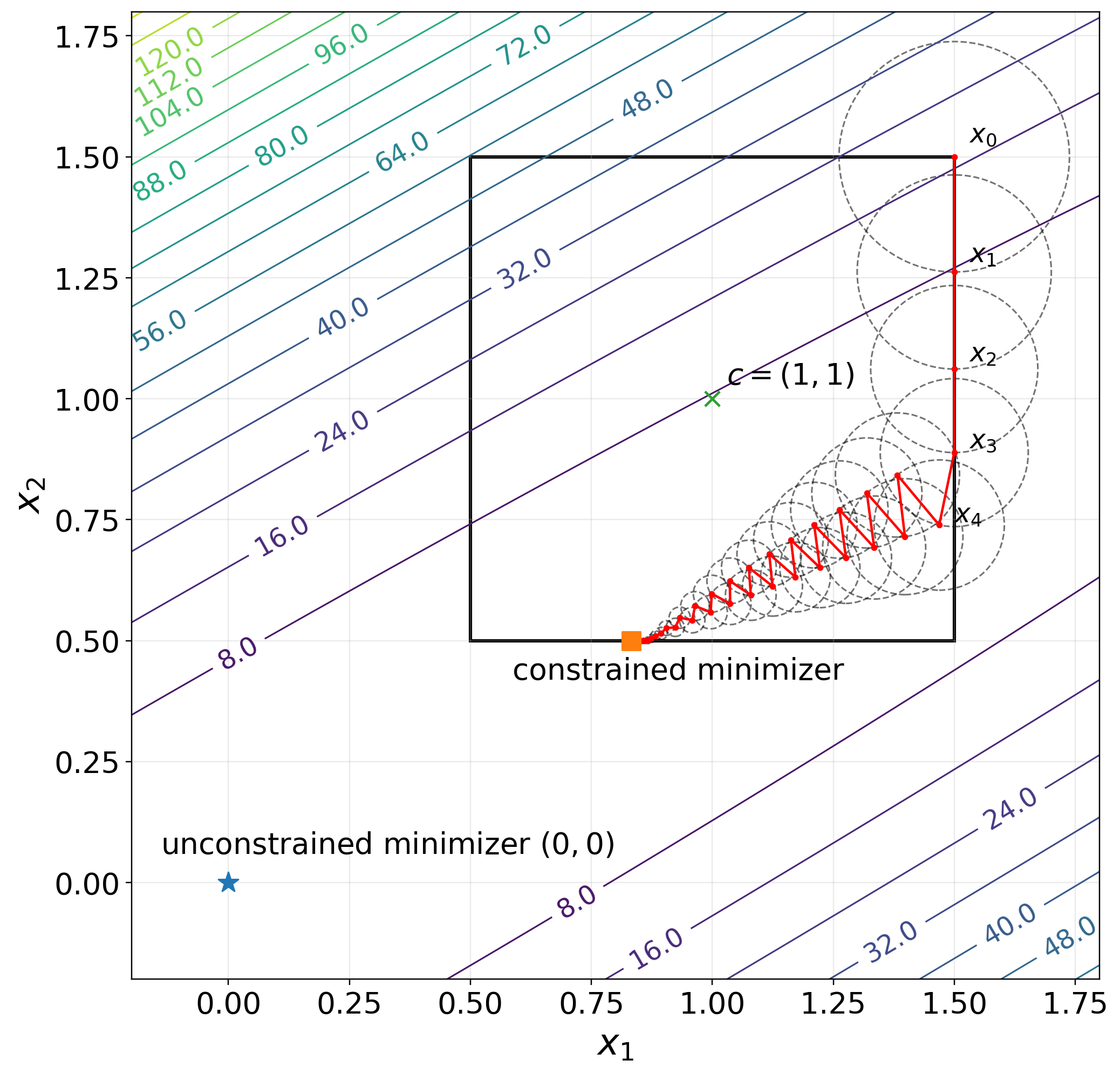}
    \caption{\small
    Illustration of {\algname{BroxGD}} dynamics for an $L$--smooth ($L=100$) and $\mu$--strongly convex ($\mu=1$) quadratic  $f:\R^2\to \R$, and a box constraint of radius $0.5$ centered at $c=(-1,0)$ (left) and $c=(1,1)$ (right), with the  radius rule $t_k=\theta\|x_k-x_\star\|$, where $\theta=2\sqrt{\mu L}/(L+\mu)\approx 0.19802$ (see \Cref{thm:smooth-strong}). Shown: 100 iterates $\{x_k\}$ of {\algname{BroxGD}}  and the corresponding  balls $\mathbb{B}(x_k,t_k)$. Note that $\|x_{k+1}-x_k\|=t_k$ for all $k$ (\Cref{thm:descent}(ii) says that this is always the case).}
    \label{fig:local-lmo-2d-intro}
\end{figure}

\subsubsection{Three ball geometries}
\label{sec:exp-3-geom}

Here we perform an experiment with the same two-dimensional function as before, but with three different constraints: $\ell_1$, $\ell_2$ and $\ell_{\infty}$ balls of radius one centered at various points. \Cref{fig:local-lmo-2d-intro-3-in-1} shows 100 iterates of {\algname{BroxGD}} in each case. The point of this experiment is to showcase the different ways the iterates of our methods interact with the boundary of constraint sets of different geometries. 

\begin{figure}[!htbp]
    \centering
    \begin{overpic}[width=0.75\textwidth]{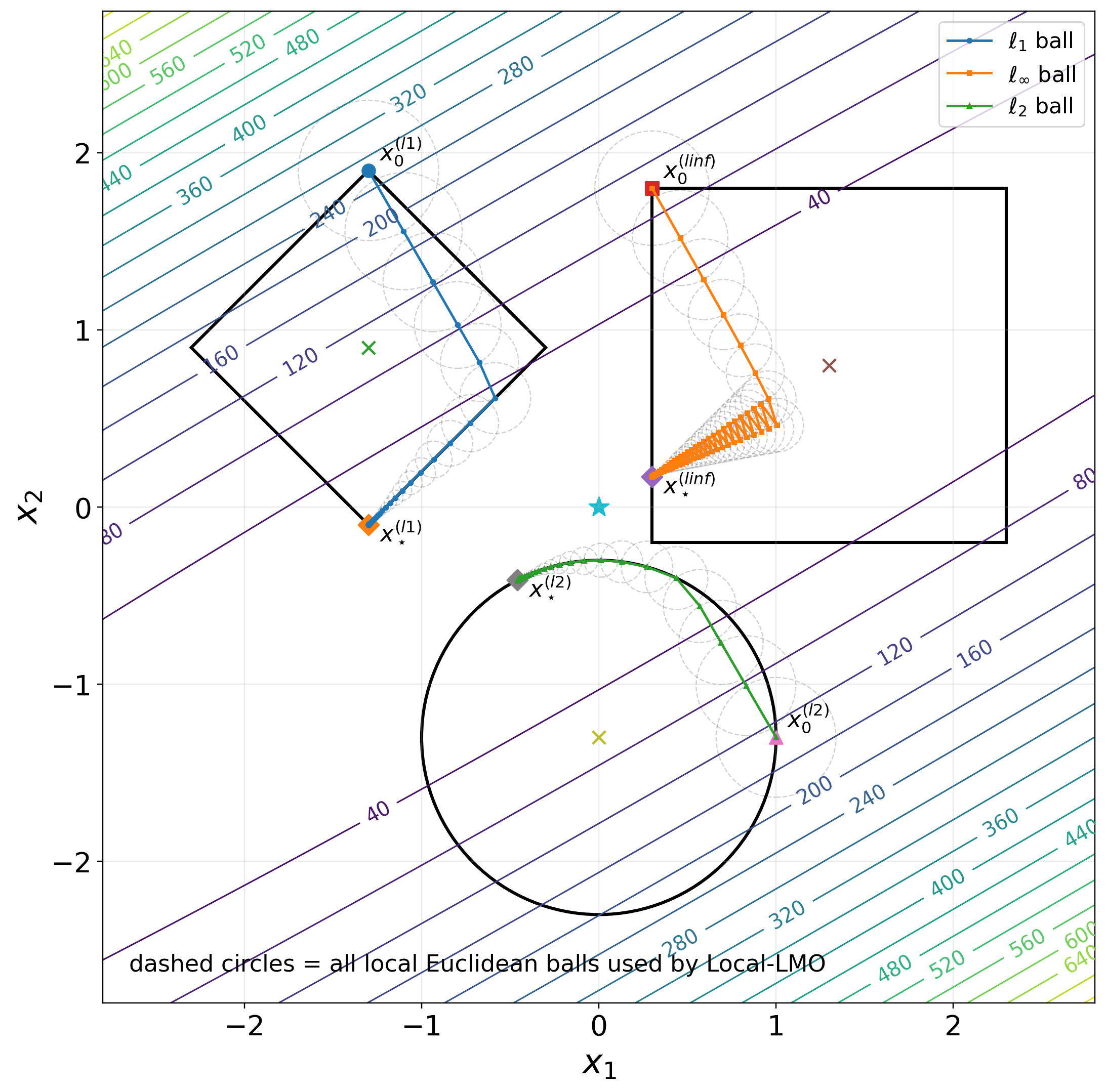}
      \put(63.7,10.6){\color{white}\rule{11.1\unitlength}{2.4\unitlength}}
      \put(64,10.9){{\algnamesmall{BroxGD}}}
    \end{overpic}
    \caption{Illustration of {\algname{BroxGD}} dynamics for an $L$--smooth ($L=100$) and $\mu$--strongly convex ($\mu=1$) quadratic  $f:\R^2\to \R$, and three ball constraints of radius $1$ (with $\ell_1$, $\ell_2$ and $\ell_\infty$ ball geometries), with the {\algname{BroxGD}} radius rule $t_k=\theta\|x_k-x_\star\|$, where $\theta\approx 0.19802$ (see \Cref{thm:smooth-strong}). Shown: 100 iterates $\{x_k\}$ of {\algname{BroxGD}}  and the corresponding  balls $\mathbb{B}(x_k,t_k)$. Note that $\|x_{k+1}-x_k\|=t_k$ for all $k$ (\Cref{thm:descent}(ii) says that this is always the case).}
    \label{fig:local-lmo-2d-intro-3-in-1}
\end{figure}

\subsection{Discussion and scope}

These experiments illustrate the predicted geometry, compare radius rules, and test controlled synthetic instances from the proposed application families. The correlation-matrix study identifies a regime with a wall-clock advantage over projected-gradient baselines, while other tests expose sensitivity to radius tuning and costly boundary solves. These findings distinguish iteration counts from computational cost and do not establish an advantage across all applications. Real-data application performance and large-scale deployment remain untested.
\FloatBarrier

%% file: appendix-structured-experiment.tex
{\color{black}
\subsection{When projection is cheap: an explicit simplex problem}\label{sec:structured-experiment}
Can a useful rate per iteration translate into a fast implementation when projection is already inexpensive? This example shows why it need not. On the tested problem, \algname{ProjGD} reaches relative error $10^{-6}$ in a median of $1.12$\,ms, whereas \algname{BroxGD-BT} takes $216.79$\,ms. The main reason is the local solver: it performs many simplex projections for each accepted step.

We instantiate the regularized quadratic over a product of simplices in \eqref{eq:apps-structured-domain}--\eqref{eq:apps-structured-loss}. The variable consists of eight probability distributions, each over 24 explicitly listed labels ($192$ variables in total). Thus each block lies in a simplex. This small synthetic problem lets us solve and check the local subproblems even when nonnegativity constraints are active. It does not test prediction with exponentially many implicit labels.

\begin{figure}[ht]
\centering
\includegraphics[width=\textwidth]{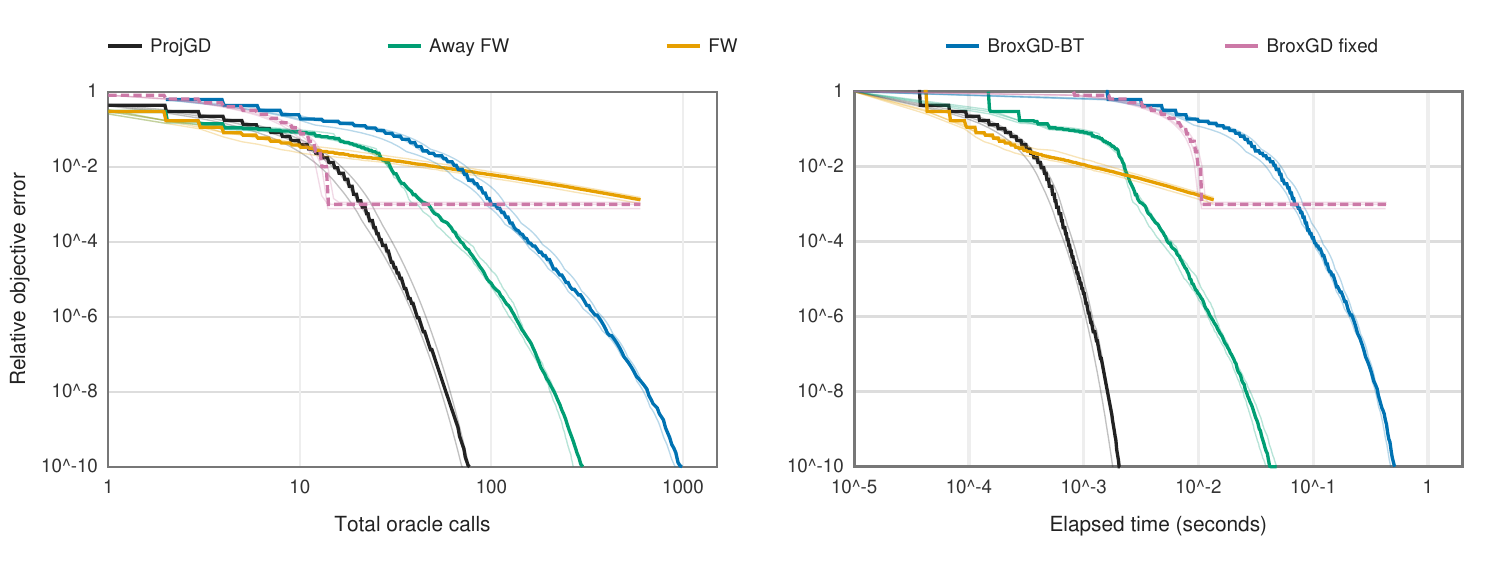}
\caption{{Relative objective error versus total oracle calls and elapsed time. Thick curves are medians over three seeds; thin curves show individual seeds. Times use three repeats per seed. Local-oracle counts include rejected backtracking trials; runtime includes their inner solves. {\algname{BroxGD}} fixed reports best-so-far error. Oracle types differ across methods.}}\label{fig:structured-main}
\end{figure}

\paragraph{Instances and reference optimum.}
For seeds $0,1,2$, a QR factorization of a standard Gaussian matrix supplies an orthogonal matrix $V$. Set $Q=V\operatorname{diag}(q)V^\top$, with $192$ geometrically spaced eigenvalues between $0.011$ and $1.01$. In each block, choose three support coordinates uniformly without replacement and assign Dirichlet$(1,1,1)$ weights to form $\alpha_\star$. Independently draw $c_j$ uniformly from $[0.005,0.03]$ off this support and set $c_j=0$ on it. The tested objective, up to an irrelevant additive constant, is
\begin{equation}\label{eq:exp-structured-objective}
 f(\alpha)=\tfrac12\alpha^\top Q\alpha+(c-Q\alpha_\star)^\top\alpha,
 \qquad \alpha\in(\Delta_{24})^8.
\end{equation}
Since $\nabla f(\alpha_\star)=c$ and $c^\top(\alpha-\alpha_\star)\ge0$ on the domain, $\alpha_\star$ is the unique constrained minimizer. Thus the reference error has no optimization-solver uncertainty. The model is a special case of \eqref{eq:apps-structured-loss}: take $\mu=0.01$, $B=(Q-\mu I)^{1/2}$, and $b=B^{-\top}(Q\alpha_\star-c)$. Both these matrices are nonsingular. The optimum is used for constructing the instances and evaluating error, not for choosing steps or radii.

\paragraph{Methods and budgets.}
All methods start at the uniform distribution in every block. \algname{ProjGD} uses stepsize $1/L$, $L=1.01$, and sorting-based Euclidean simplex projection. \algname{FW} and away-step \algname{FW} use exact quadratic line search; their global oracle chooses a minimum-gradient coordinate in every block. Away-step \algname{FW} starts from the uniform mixture of the $24$ product vertices that choose the same label in all blocks, and maintains a convex decomposition. These are full-vector methods, not block-coordinate variants. \algname{BroxGD-BT} (contraction factor $\xi=1/2$) follows \Cref{alg:radius-backtracking}, including its prescribed restart radius at every accepted center. {\algname{BroxGD}} fixed uses $t=R/\sqrt K$, $R=\sqrt{2N}=4$, and $K=600$, as in \Cref{thm:main-horizon}; $R$ is the diameter bound, not an optimum-dependent input. Every run has at most $600$ accepted updates and is truncated for evaluation when relative objective error reaches $10^{-10}$. For {\algname{BroxGD}} fixed, we report the best objective value seen so far, as required by its theorem; the other methods are monotone on these runs. The reported $10^{-6}$ target is a common evaluation threshold, not an implementable stopping certificate.

\paragraph{Local oracle and numerical certificate.}
For center $x$, gradient $g$, and radius $t$, first test whether a global minimizing vertex lies in the ball. Otherwise bracket and bisect $a>0$ in $z(a)=\Proj_{\cX}(x-ag)$, keeping the endpoint with $\|z(a)-x\|\le t$. Each projection sorts the coordinates of each block. For the returned feasible $z$ and $\lambda=1/a$, put $d=z-x$ and $q=g+\lambda d$. A computable upper bound on local linear suboptimality is
\begin{equation}\label{eq:exp-local-certificate}
 \epsilon_{\rm loc}=\left[\langle q,z\rangle-\sum_{i=1}^N\min_j q_{ij}\right]_+
       +\frac\lambda2\left(t^2-\|d\|^2\right).
\end{equation}
Indeed, for any feasible $w$ in the local ball, expanding $\|w-x\|^2$ gives
\begin{align*}
 \langle g,w-z\rangle
 &=\langle q,w-z\rangle+\tfrac\lambda2
   \left(\|d\|^2-\|w-x\|^2+\|w-z\|^2\right)\\
 \alignedrel{\eqref{eq:exp-local-certificate}}{\ge}-\epsilon_{\rm loc}.
\end{align*}
We stop bisection at $\epsilon_{\rm loc}\le10^{-10}\max\{1,\|g\|t\}$. The largest recorded certificate is below $10^{-10}$, and the maximum simplex-feasibility residual is below $2\cdot10^{-15}$. Tightening the tolerance to $10^{-12}$ changes the relative-error curves by less than $8\cdot10^{-9}$ at matching accepted indices. The certificates are evaluated in floating-point arithmetic, so they check numerical accuracy rather than extend the exact-oracle theorems to finite-tolerance solves. In this implementation the local oracle itself uses projections; their cumulative cost explains the runtime disadvantage.

\paragraph{Timing and interpretation.}
We use NumPy 2.3.5 in a single Python process on macOS 26.6.2 arm64, request one computational thread, warm up the methods, and take three timing repeats per seed. The runtime includes gradients, oracle solves, local certificates, rejected trials, away-step bookkeeping, and best-iterate objective evaluations for {\algname{BroxGD}} fixed; it excludes common instance construction and reference-error logging. Figure~\ref{fig:structured-main} shows thin seed traces and their thick pointwise median; each seed's time trace is the median of three repeats. Oracle counts mean projections for \algname{ProjGD}, global linear solves for \algname{FW} variants, and local linear solves including rejected trials for {\algname{BroxGD}}; these are different operations, so runtime and inner-projection counts are essential companions.

\begin{table}[ht]
\centering\small\color{black}
\caption{{Median work to relative error $10^{-6}$ across three seeds. A dash means the target was not reached within $600$ accepted updates. Times are implementation-specific.}}\label{tab:structured-work}
\begin{tabular}{lrrrr}
\toprule
Method & Accepted & Oracle calls & Inner projections & Time (ms)\\
\midrule
\algnameintable{ProjGD} & 41 & 41 & 41 & 1.12\\
Away-step \algnameintable{FW} & 133 & 133 & 0 & 12.64\\
\algnameintable{FW} & --- & --- & --- & ---\\
\algnameintable{BroxGD-BT} & 58 & 357 & 9682 & 216.79\\
{\algnameintable{BroxGD}} fixed & --- & --- & --- & ---\\
\bottomrule
\end{tabular}
\end{table}
At this target, \algname{BroxGD-BT} takes $58$--$66$ accepted updates but $346$--$388$ local solves and $9435$--$10513$ inner projections. The median terminal relative errors for \algname{FW} and {\algname{BroxGD}} fixed are $1.31\cdot10^{-3}$ and $9.87\cdot10^{-4}$, respectively. These results illustrate the difference between accepted-step convergence and computational efficiency. They favor \algname{ProjGD} on explicit simplices and do not establish a practical advantage for {\algname{BroxGD}} or a scaling benefit for implicit structured prediction. Faster certified local solvers remain an application-specific research question.

\paragraph{Reproducibility.}
The accompanying scripts \nolinkurl{scripts/experiment_structured.py}, \nolinkurl{scripts/run_structured_suite.py}, and \nolinkurl{scripts/summarize_structured.py} generate the instances, all runs, certificates, summaries, and the vector figure. Raw traces and timing repeats are in \nolinkurl{output/experiments/structured-repeats.json}; the summary is generated directly from those traces.
}

%% file: appendix-application-study.tex
{\color{black}
\subsection{A numerical study across the application families}\label{sec:application-study}
We next study all five model families from Appendix~\ref{sec:potential-applications}: ranking/transport, PSD-constrained covariance estimation, occupancy measures, robust low-rank matrix learning, and structured-label distributions. The central comparison is between solutions in the interior, where a simple local step may remain feasible, and solutions on the boundary, where solving the local problem can be expensive. We construct instances with known optima so that all methods can be compared at the same objective accuracy.

The main finding is that cheap interior steps often make \algname{BroxGD} faster than the tested \algname{FW} methods, but do not give a consistent advantage over the better projected-gradient baseline. Near an active boundary, the inner local solves can dominate runtime. These are synthetic optimization tests, not measurements of prediction quality, fairness on new data, or reinforcement-learning performance. The dense-Hessian example in \Cref{sec:structured-experiment} is a separate experiment.

We compare fixed-step \algname{ProjGD}, an adaptive projected-gradient variant \algname{ProjGD-BB}, standard and away-step \algname{FW}, radius-backtracking \algname{BroxGD-BT}, and fixed-horizon \algname{BroxGD}; implementation details follow the findings. A paired speedup compares two methods on the same instance: baseline time divided by \algname{BroxGD} time. We report its median across seeds, using the median of three timing repeats for each method on each seed.

\input{application-study-findings}

\paragraph{Design and coverage.}
The primary grid contains $180$ quadratic instances: five families, two sizes, two types of optimum (interior and boundary), three curvature ranges, and three seeds per configuration. It adds six ranking instances with a fairness equality and six robust Huber matrix instances. The resulting $192$ instances produce $1080$ method runs, since \algname{FW} is unavailable on the unbounded PSD cone. The size pairs are $6\times6$ and $12\times12$ for ranking, $8$ and $16$ states with four actions for occupancy measures, four and eight blocks with 24 labels for structured distributions, and $8\times8$ and $24\times24$ for both matrix families. Curvature ratios are controlled by coefficients geometrically spaced between $1/\kappa$ and $1$, for $\kappa\in\{3,30,300\}$. Matrix symmetrization can improve this nominal condition bound.

Each primary run has a cap of 400 accepted updates or three seconds, checked between updates. We record relative objective targets $10^{-2}$, $10^{-4}$, and $10^{-6}$; reaching $10^{-6}$ ends a run. If a run exhausts its budget, we record that it did not reach the target within that budget; this does not imply that the method fails to converge. To confirm the promising regimes and retain negative controls, we use five new seeds, three timing repeats, and a larger cap of 2000 updates or five seconds on ten configurations, producing another 870 runs. These include ranking matrices of orders 12 and 24 with and without fairness equalities, 32- and 64-state occupancy models, nuclear-norm matrices of orders 32 and 64, and PSD and structured-label controls. The confirmation instances all use $\kappa=30$. Exploratory runs, including simpler target constructions and a generic-LP occupancy baseline, are retained separately and are not pooled into these reported comparisons.

\paragraph{Reference solutions and objectives.}
Write $u=\operatorname{vec}(x-x_\star)$. The quadratic tests use
\begin{equation}\label{eq:study-objective}
 f(x)-f(x_\star)=\tfrac12\sum_j w_j u_j^2+\langle c,u\rangle,
 \qquad w_j>0,
\end{equation}
where $c$ is constructed so that $\langle c,\operatorname{vec}(x-x_\star)\rangle\ge0$ on the feasible set. This explicitly certifies the unique minimizer; it is used for evaluation, not by the algorithms. Such quadratics are weighted least-squares models with a linear term, and can be written as regularized regression objectives by completing the square. All methods start from the same feasible point.

For ranking, $x_\star$ is a random mixture of three permutation matrices; interior instances mix this plan with the uniform plan, while boundary instances retain zero entries. A fairness equality fixes the mass in a block of rows and columns to its value at the constructed optimum. For occupancy measures, transitions are random probability vectors and the discount is $0.9$; the target occupancy comes from either a fully supported policy or a policy supported on two actions per state. Off-support nonnegative entries of $c$ certify the boundary optimum in these two families and in the structured-label tests, whose boundary targets have three labels per block. Interior instances take $c=0$.

PSD instances have a rotated full-rank or rank-deficient covariance target, with the matrix representation of $c$ positive semidefinite and supported on the target's nullspace. The feasible set is the unbounded PSD cone; \algname{FW} therefore has no finite global linear oracle in general and is omitted from this family. Nuclear-norm instances have a low-rank target with nuclear norm $0.4$ or $1$ in a unit nuclear-norm ball. At the boundary, $c=-0.03UV^\top$ for the target's singular subspaces, which is a supporting linear functional. The additional robust tests replace the quadratic term in \eqref{eq:study-objective} by weighted Huber penalties of threshold $0.03$ and add a centered ridge term of coefficient $0.1/\kappa$; the supporting linear term and known optimum are preserved. They test robustness of the optimization loss, not statistical outlier recovery.

\paragraph{Algorithms and fair oracle implementations.}
We compare \algname{ProjGD} with stepsize $1/L$, a safeguarded Barzilai--Borwein projected-gradient variant (\algname{ProjGD-BB}), \algname{FW}, away-step \algname{FW}, \algname{BroxGD-BT} from \Cref{alg:radius-backtracking} with $\xi=1/2$, and the fixed-horizon {\algname{BroxGD}} radius $R/\sqrt K$ from \Cref{thm:main-horizon}. \algname{ProjGD-BB} clips its spectral stepsize to $[1/L,1/\mu]$ and halves it until the monotone projected Armijo test with coefficient $10^{-4}$ holds. \algname{FW} variants use exact line search for quadratics and bounded scalar minimization for Huber objectives. {\algname{BroxGD}} fixed reports its best iterate and includes selection evaluations in its runtime. Radius bounds are domain diameters on compact sets; on the PSD cone we use a certified bound obtained from the unconstrained quadratic minimum and strong convexity, not the reference minimizer.

Ranking uses an assignment oracle without extra fairness constraints and a linear program when those equalities are present. Occupancy \algname{FW} uses policy iteration, independently checked against linear programming, rather than relying on generic-LP overhead. Nuclear \algname{FW} uses a leading singular pair; at order 64 we use a partial singular-value solver. Away-step methods maintain a convex decomposition, including the common feasible initial point as an initial atom. Decomposition work is timed.

Both {\algname{BroxGD}} and \algname{ProjGD} receive applicable feasibility shortcuts. On an affine slice, projection onto the affine space is accepted if it satisfies the inequalities, and the affine-ball linear minimizer is accepted if it is feasible. PSD feasibility is screened by Cholesky factorization, and the nuclear norm has a shared Frobenius-norm sufficient check. Otherwise, ranking and occupancy projection/local subproblems use a quadratic/conic solver; explicit simplices use sorting; matrix projections use eigenvalue or singular-value thresholding. The remaining local matrix and simplex problems use a bracketed projection path and a numerical optimality certificate. These implementations can use projections internally and are not evidence that a local solve is intrinsically cheaper than projection.

\paragraph{Numerical accuracy, timing, and reporting.}
We warm up each family, deterministically randomize method order, and request one computational thread. Timings include gradients, rejection tests, numerical certificates, inner solves, and decomposition bookkeeping; they exclude data generation, common feasible initialization, and reference-error logging. Affine factorizations are timed on first use, except in fairness-constrained ranking, where the common initialization already computes and retains the factorization and projection-solver setup. Thus these are optimization-loop timings, not end-to-end deployment costs. The experiments use NumPy 2.5.3, SciPy 1.18.1, Clarabel 0.11.1, and macOS arm64. Results are implementation-specific. Tables report how many runs reached the target alongside their median time. Two such medians may describe different subsets of instances and therefore cannot always be divided to obtain a meaningful speedup. For the confirmation study, we instead compare the methods on the same instance: we first take each method's median time over three repeats, divide the baseline time by the \algname{BroxGD} time, and then take the median of these ratios across seeds. A ratio above one favors \algname{BroxGD}. Numerical certificates and feasibility residuals are recorded; exact-oracle convergence theorems are not claimed for floating-point solves.
}

\input{application-study-grid-table}
\input{application-study-extra-table}
\input{application-study-confirm-table}

{\color{black}
\begin{figure}[ht]
\centering
\includegraphics[width=\textwidth]{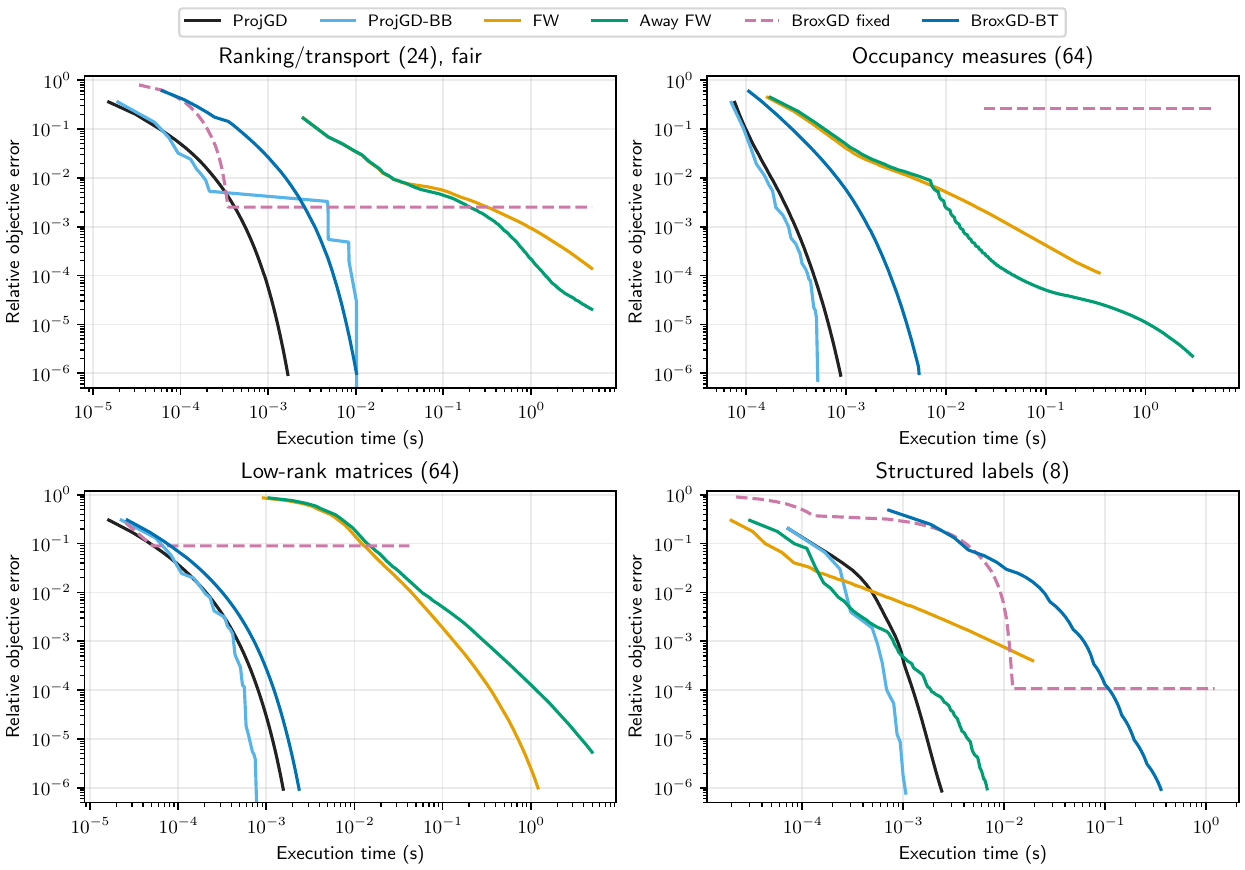}
\caption{{Held-out seed 10, selected before inspecting outcomes, in four confirmation configurations. Execution times are medians of three repeats at matching iteration indices. Curves terminate at the evaluation target or budget; {\algname{BroxGD}} fixed reports best-so-far error. All held-out seeds enter \Cref{tab:apps-confirm}, rather than only the plotted seed.}}\label{fig:apps-confirm}
\end{figure}
\FloatBarrier
\input{appendix-radius-ablation}

\paragraph{Reproducibility and limits.}
The complete design, pilot history, raw traces, statuses, and analysis scripts are provided under \texttt{scripts/study/} and \texttt{output/experiments/application-study/}. Tables and figures are generated from the saved traces. All tested instances, including unsuccessful and unfavorable ones, remain in the archive. The examples establish implementation-dependent regimes within these synthetic models; they neither prove a general separation between oracle models nor establish an advantage on real application datasets. In particular, no positive claim about {\algname{BroxGD}} versus an adaptive projected method follows merely from outperforming vanilla \algname{FW}.
}

%% file: application-study-findings.tex
\paragraph{Where localization helps relative to FW.}
The clearest confirmed regime is an interior solution with inexpensive feasible local steps. On the nuclear-norm instances, \algname{BroxGD-BT} reaches relative error $10^{-4}$ in all five new seeds reserved for confirmation at both sizes. Its median paired speedup over \algname{FW} is $40.9\times$ at order 32 and $282.9\times$ at order 64; over away-step \algname{FW} it is $127.0\times$ and $921.6\times$. The individual \algname{FW} ratios range from $38.9$ to $42.8$ and from $278.3$ to $288.0$, respectively. These large ratios reflect cheap feasible gradient-like steps versus repeated spectral-oracle calls and atomic updates. They are not a scaling theorem: the sizes are modest, and stronger corrective or problem-specific matrix solvers are not tested.

Interior occupancy models also favor \algname{BroxGD-BT} over the tested \algname{FW} variants. Its median paired speedup over away-step \algname{FW} is $13.6\times$ at 32 states and $17.5\times$ at 64 states, with all five seeds reaching the target for both methods. Vanilla \algname{FW} does not reach $10^{-4}$ within the confirmation budgets. The global oracle here uses policy iteration, so the observed difference is not merely the overhead of a generic LP solver. On interior ranking models, \algname{BroxGD-BT} beats away-step \algname{FW} in all new seeds reserved for confirmation, with median ratios $9.6\times$ and $43.9\times$ without fairness constraints, and $162.0\times$ and $263.3\times$ with them. The fairness comparisons include the cost of solving global linear programs; specialized fairness-aware baselines could change these timings. At order 24, vanilla \algname{FW} reaches the target for only one seed in each ranking variant, so its paired ratio is not representative of the whole configuration.

\paragraph{What the projected baselines reveal.}
These \algname{FW} comparisons do not establish a general advantage over projection. In every confirmation configuration, the median paired runtime favors the better of \algname{ProjGD} and \algname{ProjGD-BB} over \algname{BroxGD-BT}. In fact, on the interior nuclear and PSD tests, the first \algname{BroxGD-BT} trial is feasible and equals $x_k-\nabla f(x_k)/L$, exactly the \algname{ProjGD} update when projection is inactive. Their matched-index relative errors agree to below $6\cdot10^{-17}$. \algname{BroxGD-BT} therefore inherits the favorable gradient-descent trajectory but adds overhead; \algname{ProjGD} receives the same feasibility shortcuts. The interior occupancy and ranking tests likewise have cheap affine projections. \algname{ProjGD-BB} occasionally takes costly boundary-crossing trials, allowing \algname{BroxGD-BT} to beat that particular baseline on some instances, but fixed-step \algname{ProjGD} remains faster there. A practical advantage must be assessed against both.

\paragraph{Active boundaries are the main obstacle in these implementations.}
In the primary grid, away-step \algname{FW} is particularly effective on sparse ranking and occupancy optima. For explicit structured labels, \algname{ProjGD-BB} is inexpensive and \algname{BroxGD-BT}'s repeated water-filling solves dominate runtime. At $10^{-4}$ in the structured-label confirmation test, median work is 10 accepted \algname{ProjGD-BB} updates and 12 projections, versus 48 accepted \algname{BroxGD-BT} updates, 276 local-oracle calls, and 3782 inner projections. The PSD boundary similarly requires repeated spectral projections inside the local solve. The robust Huber matrix tests reproduce this limitation: all six \algname{BroxGD-BT} runs reach $10^{-4}$, but their median time is $84.4$ ms, versus $0.59$ ms for \algname{ProjGD-BB}. Vanilla and away-step \algname{FW} do not reach that target within the primary budget. Thus a favorable outer trajectory can coexist with unfavorable total cost.

\paragraph{Radius choice and accuracy matter.}
The fixed-horizon rule is much less reliable at high accuracy with the available diameter bounds. None of the 50 confirmation instances reaches $10^{-4}$ with this rule, although all do with \algname{BroxGD-BT}. This is a finite-budget observation, consistent with a conservative radius and an accuracy floor at a fixed horizon, not a counterexample to its best-iterate theorem. At the coarser $10^{-2}$ target, {\algname{BroxGD}} fixed exceeds the faster projected baseline by a factor above $1.25$ on only two of the twenty held-out ranking instances; each ranking configuration's median still favors a projected method. The isolated exploratory speedups therefore do not establish a reliable advantage for \algname{BroxGD} over \algname{ProjGD}.

\paragraph{Numerical checks and interpretation.}
Across the 1950 reported runs there are no solver exceptions or failed final-feasibility checks. Independent checks verify the constructed first-order optimality conditions on all 242 distinct instances, and ten occupancy-oracle comparisons agree with an independent LP to $1.2\cdot10^{-15}$ in objective value. The largest recorded local/projection optimality certificate is $3.82\cdot10^{-10}$, the largest solver residual is $1.82\cdot10^{-10}$, and the largest local certificate divided by its initial objective gap is $4.18\cdot10^{-8}$. These are floating-point diagnostics, not formal error bounds for the entire trajectory. The study supports interior occupancy, ranking, and low-rank models as useful comparisons with global \algname{FW} oracles; it identifies active-boundary local solvers as a bottleneck and supplies no consistent speed advantage over the strongest projected baseline tested.

%% file: application-study-grid-table.tex
{\color{black}
\begin{table}[ht]\centering\footnotesize\color{black}
\caption{{Primary grid at relative objective error $10^{-4}$. Each cell is median execution time in milliseconds among successful runs, followed by successes out of 18 instances. Medians from different success subsets are not paired speedups. A dash denotes no success; n/a means no global \algnameintable{FW} oracle on the unbounded cone.}}\label{tab:apps-grid}
\setlength{\tabcolsep}{2.5pt}
\begin{tabular}{llrrrrrr}
\toprule
Family & Optimum & \algnameintable{ProjGD} & \algnameintable{ProjGD-BB} & \algnameintable{FW} & Away \algnameintable{FW} & {\algnameintable{BroxGD}} fixed & \algnameintable{BroxGD-BT}\\
\midrule
Ranking & interior & 0.34 (18/18) & 0.26 (18/18) & 0.27 (14/18) & 1.01 (17/18) & --- (0/18) & 1.71 (15/18)\\
Ranking & boundary & 4.02 (18/18) & 1.98 (18/18) & --- (0/18) & 0.16 (18/18) & 2.82 (1/18) & 45.84 (18/18)\\
\addlinespace
PSD & interior & 0.40 (18/18) & 0.28 (18/18) & n/a & n/a & 0.23 (1/18) & 0.59 (18/18)\\
PSD & boundary & 0.86 (18/18) & 0.42 (18/18) & n/a & n/a & 4.87 (1/18) & 94.94 (18/18)\\
\addlinespace
Occupancy & interior & 0.30 (17/18) & 0.25 (18/18) & --- (0/18) & 23.45 (16/18) & --- (0/18) & 1.77 (14/18)\\
Occupancy & boundary & 6.74 (18/18) & 2.72 (18/18) & 3.64 (13/18) & 1.93 (18/18) & --- (0/18) & 145.44 (18/18)\\
\addlinespace
Low-rank & interior & 0.30 (16/18) & 0.24 (18/18) & 3.12 (15/18) & 5.62 (15/18) & --- (0/18) & 0.55 (16/18)\\
Low-rank & boundary & 0.69 (18/18) & 0.40 (18/18) & --- (0/18) & --- (0/18) & --- (0/18) & 42.92 (18/18)\\
\addlinespace
Labels & interior & 0.42 (17/18) & 0.35 (18/18) & --- (0/18) & --- (0/18) & --- (0/18) & 1.67 (12/18)\\
Labels & boundary & 1.47 (18/18) & 0.48 (18/18) & --- (0/18) & 0.78 (18/18) & --- (0/18) & 151.26 (18/18)\\
\addlinespace
\bottomrule
\end{tabular}
\end{table}
}

%% file: application-study-extra-table.tex
{\color{black}
\begin{table}[ht]\centering\footnotesize\color{black}
\caption{{Additional primary cases at relative error $10^{-4}$, in the same format as \Cref{tab:apps-grid}, with six instances per row.}}\label{tab:apps-extra}
\setlength{\tabcolsep}{2.5pt}
\begin{tabular}{lrrrrrr}
\toprule
Variant & \algnameintable{ProjGD} & \algnameintable{ProjGD-BB} & \algnameintable{FW} & Away \algnameintable{FW} & {\algnameintable{BroxGD}} fixed & \algnameintable{BroxGD-BT}\\
\midrule
Fair ranking & 0.85 (6/6) & 1.05 (6/6) & 87.40 (2/6) & 202.20 (5/6) & --- (0/6) & 2.10 (5/6)\\
Huber low-rank & 0.97 (6/6) & 0.59 (6/6) & --- (0/6) & --- (0/6) & --- (0/6) & 84.42 (6/6)\\
\bottomrule
\end{tabular}
\end{table}
}

%% file: application-study-confirm-table.tex
{\color{black}
\begin{table}[ht]\centering\footnotesize\color{black}
\caption{{Held-out confirmation at relative error $10^{-4}$. Ratios are baseline time divided by \algnameintable{BroxGD-BT} time, using per-seed medians of three timing repeats, then the median paired ratio. Values above one favor \algnameintable{BroxGD-BT}. Parentheses count seeds where both methods reached the target out of five; missing cases remain censored; n/a denotes an unavailable oracle.}}\label{tab:apps-confirm}
\setlength{\tabcolsep}{2.5pt}
\begin{tabular}{lrrrr}
\toprule
Family / size & \algnameintable{ProjGD} & \algnameintable{ProjGD-BB} & \algnameintable{FW} & Away \algnameintable{FW}\\
\midrule
Ranking/transport 12 & 0.16 (5/5) & 0.42 (5/5) & 5.61 (5/5) & 9.55 (5/5)\\
Ranking/transport 24 & 0.17 (5/5) & 0.67 (5/5) & 1.88 (1/5) & 43.90 (5/5)\\
Ranking/transport 12 fair & 0.14 (5/5) & 0.11 (5/5) & 573.47 (5/5) & 161.99 (5/5)\\
Ranking/transport 24 fair & 0.16 (5/5) & 0.66 (5/5) & 289.98 (1/5) & 263.26 (5/5)\\
PSD covariance 32 & 0.73 (5/5) & 0.51 (5/5) & n/a & n/a\\
Occupancy measures 32 & 0.15 (5/5) & 1.02 (5/5) & --- (0/5) & 13.55 (5/5)\\
Occupancy measures 64 & 0.16 (5/5) & 3.30 (5/5) & --- (0/5) & 17.50 (5/5)\\
Low-rank matrices 32 & 0.61 (5/5) & 0.46 (5/5) & 40.93 (5/5) & 127.02 (5/5)\\
Low-rank matrices 64 & 0.66 (5/5) & 0.50 (5/5) & 282.90 (5/5) & 921.61 (5/5)\\
Structured labels 8 boundary & $<0.01$ (5/5) & $<0.01$ (5/5) & --- (0/5) & $<0.01$ (5/5)\\
\bottomrule
\end{tabular}
\end{table}
}

%% file: appendix-radius-ablation.tex
\paragraph{Can a less aggressive radius search improve these results?}
\label{sec:radius-ablation}
A smaller reduction in radius may avoid unnecessarily short steps. But finding a larger acceptable radius also costs local solves, so fewer outer iterations need not mean less time. Our tests show both effects: contraction by $0.9$ helps on the tested interior ranking and occupancy problems, but is slower at active boundaries. Bisection reduces the number of accepted steps in several cases without consistently reducing runtime.

We compare halving with contraction factors $0.8$ and $0.9$, and with two or four safeguarded bisection refinements. All variants reset the initial radius to $\|\nabla f(x_k)\|/L$ and use the same acceptance test \eqref{eq:bt-accept}; they change radius selection, not the gradient, local subproblem, or final full-step update. A bisection variant first halves until it obtains an accepted radius $\ell$ and retains the immediately preceding rejected radius $h$. It tests the midpoint $(\ell+h)/2$, replacing the accepted endpoint if the test passes and the rejected endpoint otherwise. After its fixed refinement budget, it takes the retained accepted point. If the initial trial passes, it performs no refinement. Every additional local solve is counted and timed, including accepted trial points later replaced by a larger accepted trial.

\paragraph{Radius-search protocol.}
A pilot uses one seed in each of 20 configurations: the five families, interior and boundary optima, and $\kappa\in\{30,300\}$, with the larger primary-grid size in each family. It compares the five radius variants and both projected baselines, for 140 runs capped at 300 updates or 1.5 seconds. Confirmation uses seeds 10--14 and three timing repeats in six configurations: order-24 ranking with and without fairness, 64-state occupancy, order-32 interior nuclear-norm problems, order-24 boundary PSD problems, and eight-block boundary label problems. All confirmation instances use $\kappa=30$, with caps of 1000 updates or three seconds, producing 630 additional runs. As before, time caps are checked between updates, all methods start from the same feasible point, all methods receive feasibility shortcuts, method order is randomized, and timings include rejected trials and inner solves. These comparisons isolate radius search; the \algname{FW} implementations and earlier measurements are unchanged. The archive is \texttt{output/experiments/radius-search/}, with drivers and analysis scripts in \texttt{scripts/study/}.

\input{application-radius-table}
\input{application-radius-findings}

\paragraph{Why the refinement preserves the certificate.}
Under \Cref{ass:radius-backtracking} with exact local solves and exact acceptance tests, geometric contraction by any fixed $\xi\in(0,1)$ still terminates: its trial radii tend to zero, so the small-radius argument of \Cref{thm:bt-certificates} applies. The accepted-step certificate proof also extends to a retained accepted radius $\ell$ and rejected radius $h>\ell$. Write $\mathcal D(t)$ for the optimal linear decrease at radius $t$. If the accepted ball multiplier $\tau$ is positive, complementarity gives $\|s(\ell)\|=\ell$, and nesting of feasible balls yields
\begin{equation}\label{eq:radius-bracket-certificate}
 \tau\ell^2
 \overset{\eqref{eq:bt-multiplier-decrease}}{\le}\mathcal D(\ell)
 \le\mathcal D(h)
 \overset{\eqref{eq:bt-accept}}{<}L\|s(h)\|^2
 \overset{\eqref{eq:bt-trial}}{\le}Lh^2.
\end{equation}
Thus the analysis multiplier $\lambda=\max\{L,\tau\}$ satisfies $\lambda\le L(h/\ell)^2$; for $\tau=0$ it equals $L$. The objective-decrease and endpoint-subgradient certificates in \eqref{eq:bt-certificates} are unchanged. Contraction by $\xi$ gives $h/\ell=1/\xi$, recovering the generalized multiplier bound $L/\xi^2$ in \eqref{eq:bt-certificates}. After $m$ midpoint refinements of a halving bracket, its width is divided by $2^m$ while its lower endpoint cannot decrease, so $h/\ell\le1+2^{-m}$. The bound is then $L(1+2^{-m})^2$. An accepted initial trial already has $\lambda=L$, by the original proof. This argument needs no monotonicity of the acceptance predicate and makes no claim that the retained point has the globally largest acceptable radius. It explains the potential for better accepted-step progress, not a reduction in total oracle work. The floating-point experiments retain the same $10^{-12}$ additive acceptance tolerance as the preceding study; the exact certificate argument is not applied literally to those solves.

%% file: application-radius-table.tex
{\color{black}
\begin{table}[ht]\centering\footnotesize\color{black}
\caption{{Radius-search confirmation at relative error $10^{-4}$. Entries are median paired speedups over halving (halving time / variant time), after taking the median of three timing repeats per seed. Above one favors the variant. Parentheses give the number of mutually successful seeds out of five. The first numerical column gives halving runtime in milliseconds, median across seeds.}}\label{tab:radius-search}
\setlength{\tabcolsep}{3pt}
\begin{tabular}{lrrrrr}
\toprule
Regime & Halving (ms) & Factor $0.8$ & Factor $0.9$ & Bisection 2 & Bisection 4\\
\midrule
Ranking, interior & 5.63 & 1.52 (5/5) & 1.75 (5/5) & 0.90 (5/5) & 0.67 (5/5)\\
Fair ranking, interior & 5.73 & 1.56 (5/5) & 1.79 (5/5) & 0.89 (5/5) & 0.65 (5/5)\\
Occupancy, interior & 3.45 & 1.56 (5/5) & 1.78 (5/5) & 0.91 (5/5) & 0.66 (5/5)\\
Low-rank, interior & 0.68 & 1.02 (5/5) & 1.02 (5/5) & 1.01 (5/5) & 1.01 (5/5)\\
PSD, boundary & 116.41 & 0.51 (5/5) & 0.27 (5/5) & 0.97 (5/5) & 0.83 (5/5)\\
Labels, boundary & 234.53 & 0.50 (5/5) & 0.26 (5/5) & 0.98 (5/5) & 0.82 (5/5)\\
\bottomrule
\end{tabular}
\end{table}
}

%% file: application-radius-findings.tex
\paragraph{Findings: fewer outer steps need not mean less work.}
At relative error $10^{-4}$, factor $0.9$ gives median paired speedups over halving of $1.75\times$ for ranking, $1.79\times$ for fair ranking, and $1.78\times$ for occupancy measures. Factor $0.8$ gives $1.52$--$1.56\times$ on these three interior configurations. For ranking, the median accepted-step count falls from 126 to 70 with factor $0.9$, and local-oracle calls fall from 252 to 140. Two bisection refinements also reduce the median step count, to 72, but require 288 local solves; four refinements need 67 steps and 402 solves. Their median speedups are therefore only $0.90$ and $0.67$, respectively. The analogous occupancy counts are 103 steps / 206 calls for halving, 57 / 120 for factor $0.9$, and 59 / 236 for two refinements. Gentler contraction helps here because the first rejected radius is already close to an acceptable one.

The result reverses at active boundaries. On PSD problems, factor $0.9$ reduces the median accepted steps from 25 to 18 but increases local solves from 134 to 502, making it about $3.7\times$ slower. On explicit labels, the corresponding counts are 48 / 276 and 35 / 1038, with about a $3.8\times$ slowdown. Two bisection refinements are approximately tied with halving on both boundary configurations (median speedups $0.97$ and $0.98$); four are slower ($0.83$ and $0.82$). The encouraging PSD bisection speedup in the small pilot does not persist in the confirmation median. On the interior nuclear-norm control, every initial trial is accepted, so all five rules give identical iterates and call counts; differences of a few percent in timing are noise rather than a radius-selection improvement.

All 30 confirmation instances reach $10^{-4}$ with every method in all three repeats. The 770 pilot and confirmation runs have no solver exceptions or failed final-feasibility checks, and no recorded objective increases. The largest accepted-test excess is below $10^{-12}$, consistent with the common numerical tolerance. These results support testing a gentler contraction on interior affine problems, but do not support replacing halving universally with either $0.9$ or fixed-budget bisection. None of these refinements overturns the projected-baseline conclusion: fixed-step \algname{ProjGD} remains faster in each confirmation configuration. Reusing radius information across outer iterations, instead of resetting to $\|\nabla f(x_k)\|/L$, is a separate possible improvement not tested here.

%% file: appendix-fair-ranking-example.tex
{\color{black}
\subsubsection{A selected fair-ranking example}\label{sec:fair-ranking-example}
This example illustrates how a cheap local step can give a large initial speedup that disappears as accuracy increases. It uses a synthetic fair-ranking problem: a $24\times24$ doubly stochastic plan with one exposure equality and a weighted quadratic loss of curvature ratio 30. This instance (seed 12) is selected from the broader study in \Cref{sec:application-study}; it is not representative of every seed. {\algname{BroxGD}} uses the fixed-horizon radius $t=\sqrt{48/2000}$ from \Cref{thm:main-horizon}, reporting its best iterate. We compare it with fixed-step \algname{ProjGD}, adaptive \algname{ProjGD-BB}, \algname{FW}, and away-step \algname{FW}. All methods use the same feasible initialization, and timings include their inner solves.

\begin{figure}[ht]
\centering
\includegraphics[width=\textwidth]{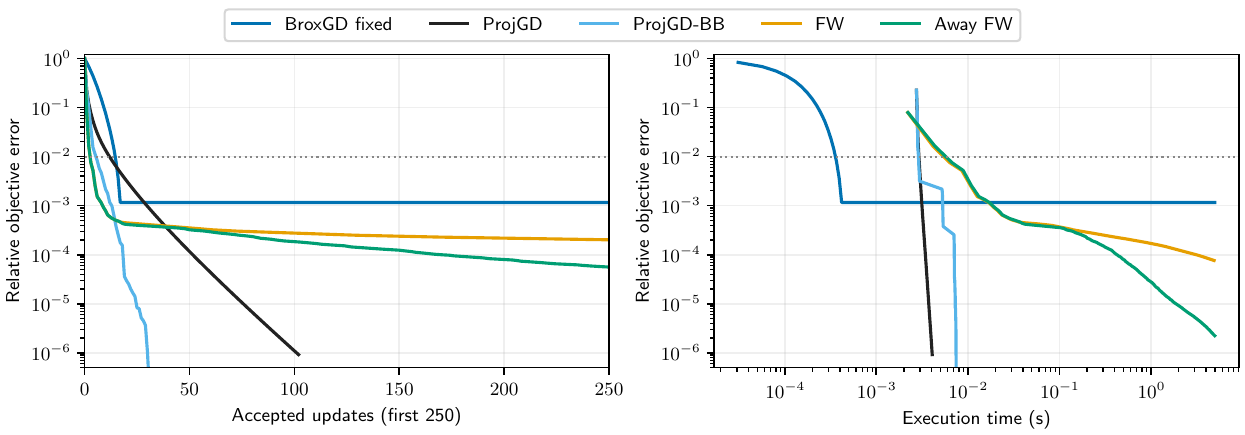}
\caption{{One selected fair-ranking instance: relative objective error against accepted updates and execution time, with a dotted $10^{-2}$ target. Times are medians of three repetitions. The local fixed-radius method reaches this target quickly but later plateaus; the broader seed comparisons and other experiments remain in Appendix~\ref{sec:experiments}.}}\label{fig:fair-ranking-main}
\end{figure}

With the study's generic projection solver, at relative error $10^{-2}$ {\algname{BroxGD}} takes $0.37$\,ms, versus $2.92$\,ms for \algname{ProjGD}, $2.90$\,ms for \algname{ProjGD-BB}, $6.36$\,ms for \algname{FW}, and $6.73$\,ms for away-step \algname{FW}. Thus it is about $8\times$ faster than the projected methods and $17$--$18\times$ faster than the \algname{FW} variants in this case. Local steps avoid a boundary-constrained projection solve at this target. This advantage is accuracy- and instance-dependent: the fixed-radius error plateaus near $1.16\cdot10^{-3}$, and the win does not persist across the tested seeds. The useful lesson is the role of the inner solver: avoiding an expensive projection can help early in the run, but this one instance does not establish a consistent advantage. Section~\ref{sec:projection-gap} tests this explanation with a stronger projection implementation and then examines correlation matrices.
}

%% file: appendix-projection-gap.tex
{\color{black}
\subsection{Searching for a substantial projection-cost advantage}\label{sec:projection-gap}
The previous tests suggest a more specific question: can the local step stay simple when a projected-gradient step needs an expensive projection? Correlation matrices provide such a case. Their diagonal entries must equal one and they must be positive semidefinite. A local step can satisfy both requirements after a cheap affine update and a feasibility check, whereas a larger projected-gradient step may require repeated eigendecompositions.

In the regime tested below, \algname{BroxGD} reaches 1\% relative objective error faster than both projected baselines on all 30 confirmation instances. The median advantage is about $6$--$7\times$ with accurate projections and $1.9$--$2.2\times$ with looser projections. The advantage does not persist to arbitrary accuracy, and a separate projection-avoiding control is faster on most instances. We explain these distinctions below. First, a ranking control shows why the quality of the projection implementation matters. These synthetic experiments measure optimization cost, not statistical estimation performance.

\paragraph{A stronger ranking control.}
On doubly stochastic matrices, the affine projection can be computed by row and column centering, without a dense pseudoinverse. A single block-exposure equality adds a rank-one correction. We implement these formulas for every method and compare a warm-started Dykstra projection onto the affine constraints and the nonnegative orthant with the previous generic quadratic-programming solver. The pilot uses orders 32 and 64, seeds 0 and 1, a curvature ratio of 3, and an 80\% mixture of three permutation matrices with 20\% uniform mass. At order 64, predetermined geometric {\algname{BroxGD}} reaches relative error $10^{-2}$ in about $0.2$\,ms, compared with about $23$\,ms for generic-solver \algname{ProjGD}, but only about $0.4$\,ms for the specialized \algname{ProjGD} projection. Thus a large advantage over the generic solver does not establish a large advantage over projection itself. The selected fair-ranking measurements in \Cref{sec:fair-ranking-example} retain their original implementation; this control limits their interpretation.

\paragraph{Correlation-matrix model and exact local shortcut.}
Consider the weighted least-squares problem
\begin{equation}\label{eq:exp-correlation-model}
 \min_{X\in\mathcal E_n} f(X)=\frac12\sum_{i,j}w_{ij}(X_{ij}-(X_\star)_{ij})^2,
 \qquad \mathcal E_n=\{X=X^\top\succeq0:\operatorname{diag}(X)=\mathbf1\}.
\end{equation}
Here $X_\star=0.6BB^\top+0.4I$, where independent Gaussian rows of $B\in\mathbb R^{n\times r}$ are normalized to unit length. This supplies a known interior optimum. The weights are a shuffled geometric sequence from $1/3$ to 1, symmetrized across the diagonal. We initialize at $I$. For $G=\nabla f(X)$ and $H=G-\operatorname{Diag}(\operatorname{diag}(G))$, the affine-ball minimizer is
\begin{equation}\label{eq:exp-correlation-local}
 X^+=X-t\frac{H}{\|H\|_F}\qquad(H\ne0).
\end{equation}
Whenever this matrix is positive semidefinite, it is also the exact {\algname{BroxGD}} solution by \Cref{prop:affine-subspace-explicit}; a successful Cholesky factorization checks positive definiteness numerically. If $H=0$, we keep $X$ unchanged. This computation does not project onto $\mathcal E_n$. In contrast, when the affine \algname{ProjGD} candidate is indefinite, our Euclidean projector alternates the unit-diagonal and PSD projections with Dykstra corrections, following the correlation-projection approach of \citet{higham2002nearest}. Each PSD projection uses an eigendecomposition. All methods share the same inexpensive feasibility shortcut whenever their own candidate permits it.

\paragraph{Protocol and controls.}
Pilots investigate orders 64--512, two covariance-rank scalings, several fixed radii, and the predetermined geometric rule. Confirmation fixes $(n,r)\in\{(128,5),(256,10),(512,20)\}$ and uses ten new seeds, 10--19, with three timing repetitions per method. {\algname{BroxGD}} uses the fixed-horizon rule in \Cref{thm:main-horizon}, with the certified bound $R=\sqrt{n(n-1)}$ on $\|I-X_\star\|_F$ and horizons $K_{\rm rad}=2000n/128$ or $10000n/128$. The former is the primary rule; both are retained. We measure best-iterate relative objective error. These horizon parameters select the radii; the recorded runs have a separate cap of 200 updates or three seconds, checked between updates. If the local affine-ball candidate becomes indefinite, the run stops with a recorded boundary event: we do not count a target beyond that event as attained, or silently substitute an approximate local oracle.

The projected baselines are \algname{ProjGD} with stepsize $1/L$ and safeguarded spectral \algname{ProjGD-BB}, with $L=\max_{ij}w_{ij}$. Projections are warm-started and stopped by a primal--dual gap at most $\tau\max\{1,\|V-I\|_F^2\}$ for input $V$. We test $\tau=10^{-10}$ and the less accurate $10^{-6}$ and $10^{-4}$ alternatives. The latter distinguish a separation for accurately solved oracles from a practical comparison allowing approximate projections. A further control, \emph{feasibility-backtracked gradient descent}, starts from $X-H/L$ and halves its step until a Cholesky check succeeds. It avoids global projections but can stall near the boundary; no general convergence claim is attached to this control. Timings include feasibility checks, inner iterations, projection certificates, gradient evaluations, and best-iterate selection, with one computational thread and randomized method order. Common instance generation and reference-error logging are excluded.

\input{projection-gap-findings}

\paragraph{Limits and reproducibility.}
The regime depends on covariance rank and requested accuracy. Keeping rank five while increasing the order to 256 or 512 makes the cheap local path hit the PSD boundary before the 1\% target. Increasing the rank to $n/12.8$ in the pilot instead lets \algname{ProjGD} avoid the expensive projection and removes the local advantage. The successful scaling $r=n/25.6$ was chosen from pilots, then checked on fresh seeds. Fixed radii can later plateau or require a boundary local solve; these measurements do not establish a high-accuracy or uniform advantage. Nor do they prove an intrinsic oracle-complexity separation or compare every specialized correlation-matrix solver. All pilots, failed local continuations, tolerance controls, and confirmation traces are retained in \texttt{output/experiments/projection-gap/}; the drivers and independent oracle checks are in \texttt{scripts/study/}. New numerical checks compare small ranking projections against an independent quadratic solver and correlation projections against a variational certificate obtained from an SDP linear oracle.
}

%% file: projection-gap-findings.tex
\paragraph{Confirmed low-accuracy advantage.}
\Cref{tab:projection-gap} reports the final comparison, with the projection gap checked after every Dykstra cycle. Across all 30 size--seed pairs, the primary {\algname{BroxGD}} radius reaches relative error $10^{-2}$ faster than both projected baselines, including their loose-projection variants. The paired median speedup against the faster of \algname{ProjGD} and \algname{ProjGD-BB} is $6.0\times$, $6.2\times$, and $7.0\times$ at orders 128, 256, and 512 with accurate projections; it is $1.9\times$, $2.0\times$, and $2.2\times$ with loose projections. Even the smallest paired loose-projection speedup is $1.72\times$. The reported projected baselines check their projection certificates after every Dykstra cycle, which avoids unnecessary inner iterations.

\input{projection-gap-table}

\paragraph{Where the time is saved.}
At the shared initial matrix $I$, the median accurately solved projection costs approximately $54$, $49$, and $54$ times as much as the first local linear solve at the three sizes. Even the loose projection costs approximately $15$, $13$, and $14$ times as much. These are matched initial-point comparisons, not averages over different later trajectories. The complete local method has a smaller speedup because it takes more outer updates. The smaller-radius variant also reaches the 1\% target on all 30 instances, but its extra iterations erase much of the advantage over loose projections.

\Cref{fig:projection-gap} in the main text illustrates the error curves and these matched oracle costs.

\paragraph{The stronger projection-avoiding control.}
Feasibility-backtracked gradient descent is faster than {\algname{BroxGD}} on all 29 instances where it reaches the 1\% target. It stalls on the remaining order-256 instance within the recorded budget, whereas {\algname{BroxGD}} reaches the target there. Thus the experiment establishes a consistent advantage over the tested \emph{projected} implementations, and a substantial cost difference between their projections and these local solves; it does not establish that {\algname{BroxGD}} is the fastest method available. The advantage is also limited in accuracy: the primary fixed radius reaches $10^{-3}$ on 29 of 30 instances, but reaches $10^{-4}$ on none before a boundary event or the budget. Both projected baselines reach $10^{-6}$ on every instance at every tested projection tolerance.

\paragraph{FW solver check.}
A separate check solves the correlation-matrix global linear oracle as an SDP, using Clarabel on its diagonal dual and an exact quadratic outer line search. At order 128, seeds 10--12, even the first \algname{FW} update takes $34.5$--$41.1$ seconds and leaves relative error between $0.298$ and $0.374$. {\algname{BroxGD}} reaches the 1\% target in about $1.3$\,ms on these instances. These are one-update lower bounds on this SDP-based \algname{FW} implementation's time to target, not completed \algname{FW} convergence runs or a comparison with every approximate or specialized SDP oracle. Away-step \algname{FW} has the same first update from the single initial atom. The expensive SDP check is kept separate from the ten-seed projected-method confirmation.

\paragraph{Run accounting.}
The study contains 450 confirmation runs and 540 additional runs with the improved projection stopping check. The final comparison uses 270 local/control runs and the 540 improved projection runs. An earlier 360-run tolerance check and the superseded projected timings are retained in the experiment archive but are not pooled into the final comparison.

%% file: projection-gap-table.tex
\begin{table}[ht]
\centering\small\color{black}
\setlength{\tabcolsep}{3.5pt}
\caption{{Correlation-matrix confirmation: median time (ms) to relative error $10^{-2}$, first taking the median of three repeats per seed, then across ten seeds. {\algnameintable{BroxGD}} uses $K_{\rm rad}=2000n/128$; {\algnameintable{BroxGD}} (small) uses $10000n/128$. Accurate projections use $\tau=10^{-10}$; loose projections use $10^{-4}$. Every entry succeeds on all ten seeds except the marked feasibility control (nine).}}\label{tab:projection-gap}
\begin{tabular}{rrrrrrrr}
\toprule
$n$ & {\algnameintable{BroxGD}} & {\algnameintable{BroxGD}} (small) & \algnameintable{ProjGD} & \algnameintable{ProjGD-BB} & \algnameintable{ProjGD} loose & \algnameintable{ProjGD-BB} loose & Feas. \algnameintable{GD}\\
\midrule
128 & 1.30 & 2.82 & 7.75 & 7.77 & 2.48 & 2.48 & 0.62\\
256 & 5.82 & 12.96 & 35.97 & 35.55 & 11.57 & 11.27 & 3.57$^{*}$\\
512 & 26.93 & 61.22 & 191.50 & 189.72 & 60.65 & 59.30 & 18.08\\
\bottomrule
\end{tabular}
\end{table}

%% file: appendix-potential-applications.tex
\clearpage
\section{Potential applications and local-oracle implementations}
\label{sec:potential-applications}

This appendix develops five possible applications: ranking and transport, PSD-constrained learning, occupancy-measure optimization, robust low-rank learning, and structured prediction. {These are mathematical models and implementation directions. Correlation-matrix experiments are summarized in \Cref{sec:main-experiment}, a selected fair-ranking instance in \Cref{sec:fair-ranking-example}, an explicit finite-label instance in \Cref{sec:structured-experiment}, and \Cref{sec:application-study} tests controlled synthetic instances from all five families. These experiments distinguish favorable optimization regimes from costly boundary solves; downstream application performance remains untested.} We distinguish cases with a closed-form exact oracle from cases requiring an iterative inner solver. The guarantees of \Cref{sec:practical-radii} apply only when their objective assumptions and exact-oracle requirements hold. A small outer-iteration count alone does not bound wall-clock time.

\subsection{A common template: locally inactive inequalities}
Consider a feasible set with affine equalities and additional convex constraints,
\begin{equation}\label{eq:apps-affine-set}
 \cX=\{x\in\R^d:Ax=b,\ x\in\mathcal K\},\qquad
 P_A=I-A^\top(AA^\top)^\dagger A,
\end{equation}
where $\mathcal K$ is closed and convex and $\dagger$ denotes the Moore--Penrose pseudoinverse. The operator $P_A$ is the orthogonal projector onto $\ker A$. A useful sufficient condition at $x_k\in\cX$ is
\begin{equation}\label{eq:apps-inactive}
 \{x_k+D:AD=0,\ \|D\|\le t_k\}\subseteq\mathcal K.
\end{equation}
This is a local containment condition, not a claim that inequalities are inactive globally.

\begin{proposition}[Exact oracle on a locally affine feasible set]\label{prop:apps-affine}
If \eqref{eq:apps-inactive} holds for $t_k>0$ and $h_k=P_A\nabla f(x_k)\ne0$, the exact {\algname{BroxGD}} solution is
\begin{equation}\label{eq:apps-affine-update}
 x_{k+1}=x_k-t_k\frac{h_k}{\|h_k\|}.
\end{equation}
If $h_k=0$, every local feasible point has the same linear objective and one may choose $x_{k+1}=x_k$. For convex differentiable $f$, this zero projected gradient certifies constrained optimality.
\end{proposition}
\begin{proof}
By \eqref{eq:apps-affine-set} and \eqref{eq:apps-inactive}, the local feasible displacements are exactly $\{D:AD=0,\|D\|\le t_k\}$. For any such displacement,
\[
 \langle\nabla f(x_k),D\rangle
 \overset{\eqref{eq:apps-affine-set}}{=}\langle h_k,D\rangle
 \ge-\|h_k\|\|D\|\ge-t_k\|h_k\|.
\]
The displacement in \eqref{eq:apps-affine-update} attains this bound. When $h_k=0$, the linear objective is constant on the affine space. Every $y\in\cX$ has $y-x_k\in\ker A$, so convexity gives
\[
 f(y)\ge f(x_k)+\langle\nabla f(x_k),y-x_k\rangle=f(x_k).
\]
\end{proof}

Applying $P_A$ may require a linear solve, but a factorization for fixed $A$ can be reused. Checking \eqref{eq:apps-inactive} has its own cost. If the sufficient radius cap is zero, this formula is unavailable; the exact local oracle remains well defined for positive radii and must retain the active inequalities. Moreover, clipping a radius prescribed by \Cref{sec:practical-radii} to a local slack bound changes the algorithm and does not automatically preserve its stated rate.

\subsection{Fair ranking, differentiable matching, and nonlinear transport}
A square stochastic ranking or matching plan can be modeled by
\begin{equation}\label{eq:apps-ranking}
 \cX=\{P\in\R^{n\times n}:P\mathbf1=\mathbf1,\ P^\top\mathbf1=\mathbf1,
                \ P\ge0,\ C\operatorname{vec}(P)=c\}.
\end{equation}
The last equality can encode linear exposure-fairness requirements. Nonlinear transport uses the same construction with prescribed marginal vectors in place of $\mathbf1$. A representative objective is
\begin{equation}\label{eq:apps-ranking-loss}
 f(P)=\ell(P)+\frac{\lambda}{2}\|P-P_{\rm ref}\|_F^2,\qquad\lambda\ge0,
\end{equation}
with a convex differentiable loss $\ell$; when $\ell$ is smooth and $\lambda>0$, the objective is smooth and strongly convex.

For ordinary doubly stochastic constraints without the additional matrix $C$, set
\begin{equation}\label{eq:apps-ranking-tangent}
 J=\frac1n\mathbf1\mathbf1^\top,\qquad G_T=(I-J)\nabla f(P_k)(I-J).
\end{equation}
This is the Frobenius-orthogonal projection onto the space of matrices with zero row and column sums. The observable condition
\begin{equation}\label{eq:apps-ranking-slack}
 0<t_k\le\min_{i,j}(P_k)_{ij}
\end{equation}
ensures that every tangent displacement of Frobenius norm at most $t_k$ preserves nonnegativity, since each entry of a displacement has absolute value at most its Frobenius norm. Hence \Cref{prop:apps-affine} gives
\begin{equation}\label{eq:apps-ranking-update}
 P_{k+1}=P_k-t_k\frac{G_T}{\|G_T\|_F}\qquad(G_T\ne0).
\end{equation}
For a buffered constraint $P_{ij}\ge\delta/n$, the analogous sufficient cap is $\min_{i,j}\{(P_k)_{ij}-\delta/n\}$, when positive. Extra fairness equalities require the projector onto the full equality nullspace, not just the double-centering operation in \eqref{eq:apps-ranking-tangent}.

Computing $G_T$ costs $\mathcal O(n^2)$ beyond gradient evaluation by subtracting row and column means and adding the grand mean. This is the cost of the closed-form oracle under \eqref{eq:apps-ranking-slack}; it is not a uniform bound for the boundary-constrained problem. A global LMO over the unmodified Birkhoff polytope is a linear assignment problem, while Euclidean projection is a constrained quadratic problem. For $n\ge2$, the squared Frobenius diameter of that polytope is $2n$, so the classical smooth vanilla-\algname{FW} bound involves $Ln/K$. Additional fairness constraints can change both global-oracle complexity and diameter.

At zero entries one instead solves the genuine local problem
\begin{equation}\label{eq:apps-ranking-boundary}
 \min_D\{\langle\nabla f(P_k),D\rangle:
       D\mathbf1=0,\ D^\top\mathbf1=0,\ P_k+D\ge0,
       \ C\operatorname{vec}(D)=0,\ \|D\|_F\le t_k\}.
\end{equation}
A zero entry imposes a one-sided condition on its displacement; it does not require $t_k=0$. Efficient active-set methods for \eqref{eq:apps-ranking-boundary}, and whether the interior formula is useful for enough iterations, remain to be studied. Global projections may themselves admit efficient specialized solvers. No practical ranking or transport advantage follows from the per-step interior formula alone.

\paragraph{{Numerical evidence.}}
{The ranking study in \Cref{sec:application-study} includes ordinary doubly stochastic plans and an additional fairness equality. Interior \algname{BroxGD-BT} steps outperform the tested \algname{FW} variants, while cheap affine \algname{ProjGD} steps remain faster; at sparse boundary optima, away-step \algname{FW} is particularly effective.}

\subsection{PSD-constrained metric, covariance, precision, and kernel learning}
For symmetric matrices, consider a shifted PSD cone and a regularized metric-learning objective,
\begin{align}
 \cX_\delta&=\{M\in\mathbb S^d:M\succeq\delta I\},\qquad \delta\ge0,
       \label{eq:apps-psd-set}\\
 f(M)&=\frac1N\sum_{(i,j)\in\mathcal E}
      \ell\bigl(y_{ij},\langle M,(a_i-a_j)(a_i-a_j)^\top\rangle\bigr)
       +\frac\lambda2\|M-M_{\rm ref}\|_F^2,
       \label{eq:apps-psd-loss}
\end{align}
where $\mathcal E$ has $N$ pairs, and $\ell$ is convex in its second argument. Related matrix losses arise in covariance and kernel estimation; precision estimation additionally needs a loss whose domain contains the chosen positive-definite feasible region, for example with $\delta>0$.

Write $\sigma_k=\lambda_{\min}(M_k)-\delta$. If
\begin{equation}\label{eq:apps-psd-slack}
 0<t_k\le\sigma_k,
\end{equation}
then every symmetric $D$ with $\|D\|_F\le t_k$ is feasible, because the eigenvalue perturbation bound gives
\[
 \lambda_{\min}(M_k+D)\ge\lambda_{\min}(M_k)-\|D\|_{\rm op}
 \ge\lambda_{\min}(M_k)-\|D\|_F
 \overset{\eqref{eq:apps-psd-slack}}{\ge}\delta.
\]
Consequently, with the gradient taken in the symmetric matrix space, the exact oracle is
\begin{equation}\label{eq:apps-psd-update}
 M_{k+1}=M_k-t_k\frac{\nabla f(M_k)}{\|\nabla f(M_k)\|_F}
 \qquad(\nabla f(M_k)\ne0).
\end{equation}
This update costs $\mathcal O(d^2)$ beyond the gradient and the slack certificate. Computing a fresh smallest eigenvalue can itself be expensive and is not included in that count.

One sufficient certificate avoids repeated eigenvalue computations. Start from $M_0=\beta I$, with $\beta>\delta$, and choose positive radii satisfying
\begin{equation}\label{eq:apps-psd-total}
 \sum_{k=0}^\infty t_k\le\beta-\delta.
\end{equation}
Every trajectory with $\|M_{j+1}-M_j\|_F\le t_j$ then satisfies
\[
 \lambda_{\min}(M_k)\ge\beta-\|M_k-M_0\|_F
 \ge\beta-\sum_{j<k}t_j
 \overset{\eqref{eq:apps-psd-total}}{\ge}\delta+t_k.
\]
Thus the local ball is feasible at every step. This is only a feasibility certificate: a finite total path length can prevent reaching a distant minimizer. A prescribed convergence schedule must separately meet its theorem hypotheses and this total-radius budget; arbitrary summable radii have no general convergence guarantee.

Euclidean projection onto \eqref{eq:apps-psd-set} is performed by eigenvalue thresholding. A global linear oracle on this unbounded cone can be unbounded below: if a symmetric cost $G$ has a unit eigenvector $v$ with eigenvalue $\gamma<0$, then the feasible matrices $M_s=\delta I+svv^\top$ satisfy
\[
 \langle G,M_s\rangle=\delta\operatorname{tr}(G)+s\gamma\longrightarrow-\infty.
\]
Adding a trace bound can make a global \algname{FW} oracle available but changes the feasible problem. Near the PSD boundary, \eqref{eq:apps-psd-update} need not be feasible and the exact local oracle must include the PSD constraint. The potential benefit is avoiding spectral projection during certified interior steps, not avoiding all spectral work throughout an arbitrary run.

\paragraph{{Numerical evidence.}}
{The PSD covariance tests in \Cref{sec:application-study} retain the unbounded cone. On the interior confirmation instances, \algname{BroxGD-BT} and \algname{ProjGD} follow the same trajectory, with lower \algname{ProjGD} overhead; boundary local solves are substantially more expensive. The correlation-matrix variant in \Cref{sec:projection-gap} adds unit-diagonal constraints and tests a regime where projection is much more expensive than a feasible affine-ball local solve. Metric and precision losses are not separately evaluated.}

\subsection{Occupancy measures for safe, fair, exploratory, or imitation RL}
In a known finite discounted Markov decision process, a state-action occupancy vector satisfies linear flow equations. Safety and fairness can add linear inequalities:
\begin{equation}\label{eq:apps-occupancy-set}
 \cX=\{d:Ad=b,\ d\ge0,\ Cd\le c\}.
\end{equation}
Representative convex coverage or imitation objectives include
\begin{equation}\label{eq:apps-occupancy-loss}
 f(d)=\tfrac12\|\Phi d-\mu_{\rm target}\|^2
          +\tfrac\lambda2\|d-d_{\rm ref}\|^2,\qquad\lambda\ge0.
\end{equation}
Here $A,b$ encode discounted flow, and $\Phi$ maps occupancy measures to features. With $\lambda>0$ this quadratic is strongly convex. Coordinates forced to zero by the flow equations can be eliminated before computing a minimum slack.

For positive free coordinates and strictly slack nonzero inequality rows, define
\begin{equation}\label{eq:apps-occupancy-slack}
 a_k=\min_i d_{k,i},\qquad
 b_k=\min_{j:\|C_{j:}\|>0}\frac{c_j-(Cd_k)_j}{\|C_{j:}\|},\qquad
 0<t_k\le\min\{a_k,b_k\}.
\end{equation}
An empty minimum is $+\infty$; feasible zero rows impose no restriction. For any displacement $D$ with $\|D\|\le t_k$, its coordinates have absolute value at most $t_k$ and $C_{j:}D\le\|C_{j:}\|t_k$. These bounds and \eqref{eq:apps-occupancy-slack} preserve all inequalities. \Cref{prop:apps-affine} therefore yields
\begin{equation}\label{eq:apps-occupancy-update}
 d_{k+1}=d_k-t_k\frac{P_A\nabla f(d_k)}{\|P_A\nabla f(d_k)\|},
\end{equation}
when the denominator is nonzero. A factorization for the fixed flow matrix can be reused; its cost and sparsity depend on the MDP.

The potential benefit is replacing a globally constrained projection or planning subproblem by a flow-nullspace calculation on certified interior iterations. At boundary occupancies, the local problem retains $d_k+D\ge0$ and $C(d_k+D)\le c$ together with $AD=0$ and $\|D\|\le t_k$. A zero minimum-slack bound does not imply that this feasible set is a singleton. Such boundary problems need a different solver, and neither cheap factorization updates nor a small number of inner iterations is guaranteed here. This is a model-based proposal; estimating transition dynamics and optimizing from samples introduce additional errors outside the exact deterministic theory.

\paragraph{{Numerical evidence.}}
{The controlled occupancy tests in \Cref{sec:application-study} use known transitions and weighted quadratic matching losses. Interior \algname{BroxGD-BT} beats away-step \algname{FW} with a policy-iteration oracle, but fixed-step \algname{ProjGD} remains faster. The tests do not include additional safety inequalities or learning the dynamics from samples.}

\subsection{A common certificate for boundary-active atomic problems}
The remaining applications use a different mechanism. Let $\mathcal A$ be compact and $\cX=\operatorname{conv}(\mathcal A)$, let $x_k\in\cX$, and choose a finite active set $S\subset\mathcal A$ whose convex hull contains $x_k$. The restricted local oracle solves
\begin{equation}\label{eq:apps-master}
 \min_{\beta\in\Delta_{|S|}}
 \left\{\left\langle g_k,\sum_{j=1}^{|S|}\beta_j a_j\right\rangle:
       \left\|\sum_{j=1}^{|S|}\beta_j a_j-x_k\right\|\le t_k\right\},
\end{equation}
where $\Delta_s=\{\beta\in\R^s:\beta\ge0,\sum_j\beta_j=1\}$. The master is a finite-dimensional convex second-order cone problem. Let $\bar x$ be its solution and let $\lambda\ge0$ be a multiplier for the constraint $\tfrac12(\|x-x_k\|^2-t_k^2)\le0$. Define
\begin{equation}\label{eq:apps-shifted-cost}
 q=g_k+\lambda(\bar x-x_k),\qquad
 \lambda(\|\bar x-x_k\|^2-t_k^2)=0.
\end{equation}
The factor $1/2$ fixes the multiplier normalization. With $t_k>0$ and $x_k\in\operatorname{conv}(S)$, a point in the relative interior of the restricted hull can be chosen sufficiently close to $x_k$ to make the ball inequality strict, so the usual convex multiplier condition applies relative to that hull.

\begin{proposition}[Global certificate by atomic pricing]\label{prop:apps-pricing}
Suppose $\bar x\in\cX\cap\mathbb{B}(x_k,t_k)$, $\lambda\ge0$, and \eqref{eq:apps-shifted-cost} holds. If
\begin{equation}\label{eq:apps-pricing-test}
 \min_{a\in\mathcal A}\langle q,a\rangle\ge\langle q,\bar x\rangle,
\end{equation}
then $\bar x$ solves the full exact {\algname{BroxGD}} with cost $g_k$.
\end{proposition}
\begin{proof}
By the convex-hull representation, \eqref{eq:apps-pricing-test} implies $\langle q,z-\bar x\rangle\ge0$ for every $z\in\cX$. For any locally feasible $z$, expand the squared norm to obtain
\begin{align*}
 \langle g_k,z-\bar x\rangle
 \alignedrel{\eqref{eq:apps-shifted-cost}}{=}
 \langle q,z-\bar x\rangle+
 \frac\lambda2\bigl(\|\bar x-x_k\|^2-\|z-x_k\|^2+\|z-\bar x\|^2\bigr)\\
 \alignedrel{\eqref{eq:apps-pricing-test},\eqref{eq:apps-shifted-cost}}{\ge}0.
\end{align*}
For the last step, if $\lambda=0$ the quadratic term vanishes; otherwise complementarity gives $\|\bar x-x_k\|=t_k\ge\|z-x_k\|$. Thus no locally feasible point improves the linear objective.
\end{proof}

If pricing finds a violating atom, it can be added and the master solved again. This gives a route to a certified oracle, not a bound on the number of pricing rounds. In particular, an infinite atomic set need not yield finite exact termination. Approximate master solutions or approximate pricing require explicit error certificates and corresponding inexact-oracle analysis before exact-oracle outer rates can be invoked. Modern away-step, pairwise, blended, lazy, and fully corrective \algname{FW} methods also redistribute weights; comparison with vanilla \algname{FW} alone cannot establish an advantage.

\subsection{Robust low-rank matrix learning over a nuclear-norm ball}
Consider matrix completion, multitask learning, or multivariate regression under
\begin{equation}\label{eq:apps-nuclear-model}
 \min_{X\in\R^{m\times n},\ \|X\|_*\le\tau}
 \left\{f(X)=\frac1{|\Omega|}\sum_{(i,j)\in\Omega}\rho(X_{ij}-Y_{ij})
                      +\frac\mu2\|X\|_F^2\right\},\qquad\tau>0,
\end{equation}
where $\Omega$ is nonempty, $\rho$ is a convex robust loss, and $\mu\ge0$. Huber loss gives a smooth example; absolute-value loss gives a nonsmooth one. The norm $\|\cdot\|_*$ is the nuclear norm, whereas the local ball uses the Frobenius norm.

At $\|X_k\|_*=\tau$, the nuclear-norm constraint is genuinely active: the points $(1+s)X_k$ with $s>0$ are outside the feasible set and approach $X_k$ as $s$ tends to zero. Thus the cheap interior-ball mechanism does not apply at such iterates. Instead use the atomic representation
\begin{equation}\label{eq:apps-nuclear-atoms}
 \{X:\|X\|_*\le\tau\}
 =\operatorname{conv}\bigl(\{0\}\cup\{\tau uv^\top:\|u\|=\|v\|=1\}\bigr).
\end{equation}
Maintain a finite convex decomposition of $X_k$ in these atoms, including the zero atom when needed, and solve \eqref{eq:apps-master} in their coefficients, using Frobenius distance and $G_k=\nabla f(X_k)$ or a subgradient for a nonsmooth loss. After obtaining $\bar X$ and a multiplier, form
\begin{equation}\label{eq:apps-nuclear-price}
 Q_k=G_k+\lambda(\bar X-X_k),\qquad
 A_{\rm new}=-\tau u_1v_1^\top,
\end{equation}
where $u_1,v_1$ are a leading left/right singular-vector pair of $Q_k$. If $Q_k=0$, any atom, including zero, minimizes pricing. Otherwise this rank-one atom solves the global atomic linear problem. The certificate is
\begin{equation}\label{eq:apps-nuclear-certificate}
 \langle Q_k,A_{\rm new}\rangle\ge\langle Q_k,\bar X\rangle.
\end{equation}
By \Cref{prop:apps-pricing}, it certifies the exact local solution. Singular-vector computations can exploit structure, but finite power or Lanczos iterations do not by themselves give an exact pricing certificate.

Projection onto a nuclear-norm ball can be obtained by an SVD and projection of the singular values onto an $\ell_1$ ball; specialized partial-SVD implementations may suffice in favorable cases. The local master offers a different way to reuse an existing decomposition, remove weights, and add directions through pricing. A smooth strongly convex loss permits a geometric objective-error bound (linear convergence) in the outer iterations under \Cref{thm:practical-geometric}, provided the oracle is exact and the required bounds are known. A nonsmooth robust loss instead requires the paper's subgradient theory and its own radius hypotheses. Neither a small active set nor a small number of singular-vector calls is proved here. Runtime, rank growth, pricing accuracy, and comparison against corrective \algname{FW} and specialized proximal solvers are necessary parts of a future evaluation.

\paragraph{{Numerical evidence.}}
{The interior quadratic matrix tests in \Cref{sec:application-study} show large speedups over \algname{FW}, but no consistent advantage over projected methods. Boundary Huber tests favor \algname{ProjGD-BB} over \algname{BroxGD-BT}. Their local solver uses a projection path; they do not test the restricted master-and-pricing implementation proposed above or statistical robustness to outliers.}

\subsection{Structured prediction over large products of simplices}
For training examples $i=1,\ldots,N$, let $\mathcal Y_i$ be a finite set of structured labels, possibly of exponential cardinality. A feasible domain appearing in structured-prediction duals is
\begin{equation}\label{eq:apps-structured-domain}
 \cX=\prod_{i=1}^N\Delta(\mathcal Y_i).
\end{equation}
A representative objective is
\begin{equation}\label{eq:apps-structured-loss}
 f(\alpha)=\tfrac12\|B\alpha-b\|^2+\tfrac\mu2\|\alpha\|^2,\qquad\mu>0.
\end{equation}
This explicitly regularized example is strongly convex in the simplex coordinates. A standard structural-SVM dual need not be strongly convex in those coordinates merely because its primal is regularized; a linear-rate conclusion requires checking the actual objective. Nonsmooth hinge-type or robust objectives are separate possibilities.

Sparse block distributions have many zero coordinates. In blocks with at least two labels, small equality-preserving displacements can make a zero coordinate negative, so the nonnegativity boundary cannot generally be discarded. For the gradient $g_k$, the local problem is
\begin{equation}\label{eq:apps-structured-local}
 \min_{\beta_i\in\Delta(\mathcal Y_i)}
 \left\{\sum_i\langle g_{k,i},\beta_i\rangle:
                  \sum_i\|\beta_i-\alpha_{k,i}\|^2\le t_k^2\right\}.
\end{equation}
When a strictly positive ball multiplier $\lambda$ exists, the KKT equations have the water-filling form
\begin{equation}\label{eq:apps-waterfill}
 \beta_{i,y}=\left[\alpha_{k,i,y}-\frac{g_{k,i,y}+\nu_i}{\lambda}\right]_+,
 \qquad \sum_y\beta_{i,y}=1,\qquad
 \sum_i\|\beta_i-\alpha_{k,i}\|^2=t_k^2.
\end{equation}
Here $[a]_+=\max\{a,0\}$ and $\nu_i$ enforces normalization. These equations follow by minimizing the separable Lagrangian with the squared-ball multiplier convention of \eqref{eq:apps-shifted-cost}. An active ball need not always have a positive multiplier; when $\lambda=0$, division in \eqref{eq:apps-waterfill} is invalid and one instead checks a global linear minimizer that meets the radius constraint.

A restricted master retains a finite support $S_i$ containing that of $\alpha_{k,i}$. At a candidate solution $\bar\beta$, form
\begin{equation}\label{eq:apps-structured-cost}
 q_{i,y}=g_{k,i,y}+\lambda(\bar\beta_{i,y}-\alpha_{k,i,y}).
\end{equation}
The blockwise global certificate is
\begin{equation}\label{eq:apps-structured-certificate}
 \min_{y\in\mathcal Y_i}q_{i,y}\ge\sum_{y\in\mathcal Y_i}q_{i,y}\bar\beta_{i,y}
 \qquad\text{for each }i,
\end{equation}
together with feasibility and complementarity. It implies the linear optimality condition over the product domain, so the proof of \Cref{prop:apps-pricing} applies.

For an inactive label outside $S_i$, both $\bar\beta_{i,y}$ and $\alpha_{k,i,y}$ vanish, and hence
\[
 q_{i,y}\overset{\eqref{eq:apps-structured-cost}}{=}g_{k,i,y}\qquad(y\notin S_i).
\]
This motivates loss-augmented MAP inference or decoding as a pricing primitive when the gradient has the requisite structured form. However, one must either minimize the shifted cost over all labels or find the best label outside the current support and compare it with the active labels. An ordinary MAP call can repeatedly return an active label; it does not automatically implement exclusion or arbitrary support-dependent shifts. Exact pricing therefore needs a suitable exclusion, enumeration, or modified-decoding procedure, whose cost is problem dependent.

The possible benefit is storing only active labels and using restricted water-filling or cone-program masters plus certified pricing, instead of explicitly materializing exponentially large vectors. Euclidean simplex projection is inexpensive for an explicit vector, but finding all positive coordinates can be difficult in an implicit representation. Conversely, the local oracle also may need many labels or many decoding calls; sparsity is not guaranteed by a small radius alone. Corrective \algname{FW} methods can likewise reweight multiple structures. A convincing application study must measure master solves, support growth, exact or certified approximate inference, and total runtime under comparable stopping criteria.

\paragraph{{Numerical evidence.}}
{Both \Cref{sec:structured-experiment} and the larger sweep in \Cref{sec:application-study} favor projected methods on explicitly enumerated labels. Repeated inner projections dominate \algname{BroxGD-BT} runtime. Whether restricted masters and structured inference can reverse this comparison for implicit, exponentially large label spaces remains open.}

\subsection{What these examples establish}
The first three models exhibit conditions under which the exact local oracle has a normalized tangent-gradient formula, with explicit sufficient slack certificates. The last two give a master-and-pricing route and an exact optimality certificate while retaining active inequalities. These mechanisms identify concrete questions for further work: how often cheap local formulas apply, how costly the boundary solvers are, and whether their total cost offsets any reduction in outer iterations. They do not establish a universal computational advantage over projection or modern \algname{FW} variants, nor do they change the exact-oracle assumptions of the convergence theorems.

%% file: appendix-notation.tex
\clearpage
\section{{Notation guide}}\label{sec:notation}
{\color{black}
\subsection{Core problem, iterates, and error measures}\label{notation:core}
\begingroup\small\renewcommand{\arraystretch}{1.22}\setlength{\tabcolsep}{5pt}
\begin{longtable}{@{}>{\raggedright\arraybackslash}p{.20\textwidth}>{\raggedright\arraybackslash}p{.61\textwidth}>{\raggedright\arraybackslash}p{.14\textwidth}@{}}
\toprule
\rowcolor{mainresultshade}\textbf{Symbol} & \textbf{Meaning and scope} & \textbf{Reference}\\
\midrule\endfirsthead
\toprule
\rowcolor{mainresultshade}\textbf{Symbol} & \textbf{Meaning and scope (continued)} & \textbf{Reference}\\
\midrule\endhead
\midrule\multicolumn{3}{r}{\footnotesize Continued on the next page}\\\endfoot
\bottomrule\endlastfoot
$f,\ f_i$ & Objective and finite-sum components; $f=n^{-1}\sum_{i=1}^n f_i$ in the stochastic extension. & \hyperref[sec:stochastic]{\S\,\ref*{sec:stochastic}}\newline p.~\pageref{sec:stochastic}\\[3pt]
$\cX,\ \cX_\star$ & Feasible set and its set of minimizers: $\cX_\star=\arg\min_{x\in\cX}f(x)$. & \hyperref[ass:main]{Assump.~\ref*{ass:main}}\newline p.~\pageref{ass:main}\\[3pt]
$x_\star,\ f_\star$ & A fixed constrained minimizer and its value $f_\star=f(x_\star)$. & \hyperref[ass:main]{Assump.~\ref*{ass:main}}\newline p.~\pageref{ass:main}\\[3pt]
$x_k,\ t_k$ & Iterate and local radius at iteration $k$. Zero radius gives $x_{k+1}=x_k$. & \hyperref[alg:new]{Alg.~\ref*{alg:new}}\newline p.~\pageref{alg:new}\\[3pt]
$k,\ K$ & Iteration index and integer horizon. Backtracking counts accepted steps; rejected trials incur extra oracle calls. & \hyperref[thm:main-backtracking]{Thm.~\ref*{thm:main-backtracking}}\newline p.~\pageref{thm:main-backtracking}\\[3pt]
$x_K^{\mathrm{best}}$ & An iterate with smallest objective value among $x_0,\ldots,x_K$ for the fixed-horizon guarantees; return a detected zero-gradient point if the method stops early. & \hyperref[thm:main-horizon]{Thm.~\ref*{thm:main-horizon}}\newline p.~\pageref{thm:main-horizon}\\[3pt]
$\bar x_K$ & {\algnameintable{BroxGD}} average: $K^{-1}\sum_{k=0}^{K-1}x_k$. This includes the initial iterate and excludes $x_K$. & \hyperref[thm:bounded_gradients]{Thm.~\ref*{thm:bounded_gradients}}\newline p.~\pageref{thm:bounded_gradients}\\[3pt]
$\bar x_K^{\algnameintable{ProjGD}}$ & \algnameintable{ProjGD} average: $K^{-1}\sum_{k=1}^{K}x_k$. This excludes the initial iterate and includes $x_K$. & \hyperref[prop:pgd_rates]{Prop.~\ref*{prop:pgd_rates}}\newline p.~\pageref{prop:pgd_rates}\\[3pt]
$\Delta_k$ & Objective gap $f(x_k)-f_\star$. & \hyperref[sec:theory]{\S\,\ref*{sec:theory}}\newline p.~\pageref{sec:theory}\\[3pt]
$v_k,\ Z_k$ & Gradient-difference vector $v_k=\nabla f(x_k)-\nabla f(x_\star)$ and scalar residual $Z_k=\|v_k\|$. Neither is an objective gap. & \hyperref[sec:theory]{\S\,\ref*{sec:theory}}\newline p.~\pageref{sec:theory}\\[3pt]
$R_k,\ R_0$ & Distance $R_k=\|x_k-x_\star\|$; $R_0$ is the initial distance. & \hyperref[tbl:main]{Table~\ref*{tbl:main}}\newline p.~\pageref{tbl:main}\\[3pt]
$G_\gamma(x)$ & Projected gradient mapping $\gamma^{-1}(x-\Proj_{\cX}(x-\gamma\nabla f(x)))$, with $\gamma>0$. & \hyperref[prop:pgd_rates]{Prop.~\ref*{prop:pgd_rates}}\newline p.~\pageref{prop:pgd_rates}\\[3pt]
$g_{\algnameintable{FW}}(x)$ & \algnameintable{FW} gap $\max_{s\in\cX}\langle\nabla f(x),x-s\rangle$ on a compact domain. & \hyperref[tbl:main]{Table~\ref*{tbl:main}}\newline p.~\pageref{tbl:main}\\[3pt]
\end{longtable}
\endgroup

\subsection{Geometry, regularity, and radius parameters}\label{notation:parameters}
\begingroup\small\renewcommand{\arraystretch}{1.22}\setlength{\tabcolsep}{5pt}
\begin{longtable}{@{}>{\raggedright\arraybackslash}p{.20\textwidth}>{\raggedright\arraybackslash}p{.61\textwidth}>{\raggedright\arraybackslash}p{.14\textwidth}@{}}
\toprule
\rowcolor{mainresultshade}\textbf{Symbol} & \textbf{Meaning and scope} & \textbf{Reference}\\
\midrule\endfirsthead
\toprule
\rowcolor{mainresultshade}\textbf{Symbol} & \textbf{Meaning and scope (continued)} & \textbf{Reference}\\
\midrule\endhead
\midrule\multicolumn{3}{r}{\footnotesize Continued on the next page}\\\endfoot
\bottomrule\endlastfoot
$d,\ \langle\cdot,\cdot\rangle,\ \|\cdot\|$ & Ambient dimension, Euclidean inner product and norm. Matrix applications use the Frobenius inner product and norm unless another norm is explicitly specified. & \hyperref[sec:potential-applications]{\S\,\ref*{sec:potential-applications}}\newline p.~\pageref{sec:potential-applications}\\[3pt]
$\mathbb{B}(x,t)$ & Closed Euclidean ball $\{z:\|z-x\|\le t\}$. & \hyperref[thm:well-posedness]{Thm.~\ref*{thm:well-posedness}}\newline p.~\pageref{thm:well-posedness}\\[3pt]
$\Proj_{\cX},\ N_{\cX}$ & Euclidean projection and convex normal cone; $w\in N_{\cX}(x)$ means $\langle w,z-x\rangle\le0$ for all $z\in\cX$. & \hyperref[eq:monotone-inclusion]{Eq.~(\ref*{eq:monotone-inclusion})}\newline p.~\pageref{eq:monotone-inclusion}\\[3pt]
$\iota_{\cX},\ \partial f$ & Convex indicator (zero on $\cX$, infinity outside) and convex subdifferential. & \hyperref[sec:nondiff]{\S\,\ref*{sec:nondiff}}\newline p.~\pageref{sec:nondiff}\\[3pt]
$L,\ \mu$ & Positive smoothness bound and strong-convexity constant. In a projected-PL result, $\mu$ instead denotes its explicitly stated PL constant. & \hyperref[ass:projected_pl]{Assump.~\ref*{ass:projected_pl}}\newline p.~\pageref{ass:projected_pl}\\[3pt]
$L_0,\ L_1$ & Constants in the global asymmetric generalized-smoothness assumption. & \hyperref[asp:l0l1]{Assump.~\ref*{asp:l0l1}}\newline p.~\pageref{asp:l0l1}\\[3pt]
$c_k,\ c_\star$ & $c_k=L_0+L_1\|\nabla f(x_k)\|$ and $c_\star=L_0+L_1\|\nabla f(x_\star)\|$. & \hyperref[tbl:main]{Table~\ref*{tbl:main}}\newline p.~\pageref{tbl:main}\\[3pt]
$G,\ M_0$ & Uniform gradient/subgradient bound $G$; smooth Polyak bound $M_0=\|\nabla f(x_\star)\|+LR_0$. & \hyperref[thm:polyak_radius_smooth_convex]{Thm.~\ref*{thm:polyak_radius_smooth_convex}}\newline p.~\pageref{thm:polyak_radius_smooth_convex}\\[3pt]
$U,\ S,\ r_U$ & Smoothness neighborhood $U$; $S=\cX\cap\overline{\mathbb{B}}(x_\star,R_0)$; radius $r_U>0$ with $S+\overline{\mathbb{B}}(0,r_U)\subset U$, used only in the Polyak rate bound. & \hyperref[eq:polyak-neighborhood-average]{Eq.~\ref*{eq:polyak-neighborhood-average}}\newline p.~\pageref{eq:polyak-neighborhood-average}\\[3pt]
$\eta_U$ & Auxiliary proof stepsize $\min\{1/L,r_U/(LR_0)\}$ when $R_0>0$; not an algorithm input. & \hyperref[eq:polyak-neighborhood-step]{Eq.~\ref*{eq:polyak-neighborhood-step}}\newline p.~\pageref{eq:polyak-neighborhood-step}\\[3pt]
$R$ & Certified input bound $R\ge R_0$, with $R>0$, for fixed-horizon radii. & \hyperref[thm:main-horizon]{Thm.~\ref*{thm:main-horizon}}\newline p.~\pageref{thm:main-horizon}\\[3pt]
$R_{\rm lev}$ & Analysis bound on $\operatorname{dist}(x_k,\cX_\star)$ for accepted backtracking centers; different from $R_0$. & \hyperref[thm:main-backtracking]{Thm.~\ref*{thm:main-backtracking}}\newline p.~\pageref{thm:main-backtracking}\\[3pt]
$C_0,\ \varepsilon$ & Certified initial-gap bound $C_0\ge\Delta_0$, $C_0>0$, and target objective accuracy $\varepsilon>0$. & \hyperref[thm:practical-geometric]{Thm.~\ref*{thm:practical-geometric}}\newline p.~\pageref{thm:practical-geometric}\\[3pt]
$\theta$ & Sharp radius multiplier $2\sqrt{\mu L}/(L+\mu)$. & \hyperref[thm:smooth-strong]{Thm.~\ref*{thm:smooth-strong}}\newline p.~\pageref{thm:smooth-strong}\\[3pt]
$\xi$ & {Fixed radius-contraction factor in $(0,1)$ for backtracking; $\xi=1/2$ gives halving.} & \hyperref[alg:radius-backtracking]{Alg.~\ref*{alg:radius-backtracking}}\newline p.~\pageref{alg:radius-backtracking}\\[3pt]
$q,\ r$ & $q=1-\mu/(4L)$ contracts the practical geometric objective bound; $r\in(0,1)$ contracts experimental radii. The certified schedule uses $r=\sqrt q$. & \hyperref[sec:comparison-geometric-radius]{\S\,\ref*{sec:comparison-geometric-radius}}\newline p.~\pageref{sec:comparison-geometric-radius}\\[3pt]
$\gamma,\ \gamma_k$ & \algnameintable{ProjGD}/\algnameintable{GD} stepsize, distinct from radius $t_k$. The projected-PL table row assumes its inequality at $\gamma=1/L$. & \hyperref[prop:pgd_rates]{Prop.~\ref*{prop:pgd_rates}}\newline p.~\pageref{prop:pgd_rates}\\[3pt]
$\cD_{\cX},\ C_{\algnameintable{FW}}$ & Feasible-set diameter and \algnameintable{FW} curvature constant; distinct from initial-distance and initial-gap bounds. & \hyperref[eq:curvature]{Eq.~(\ref*{eq:curvature})}\newline p.~\pageref{eq:curvature}\\[3pt]
$G_K^{\algnameintable{FW}}$ & Trajectory gradient bound $\max_{0\le k\le K}\|\nabla f(x_k^{\algnameintable{FW}})\|$ in the generalized-smooth \algnameintable{FW} comparison. & \hyperref[tbl:main]{Table~\ref*{tbl:main}}\newline p.~\pageref{tbl:main}\\[3pt]
\end{longtable}
\endgroup

\subsection{Backtracking, stochastic quantities, and application geometry}\label{notation:extensions}
\begingroup\small\renewcommand{\arraystretch}{1.22}\setlength{\tabcolsep}{5pt}
\begin{longtable}{@{}>{\raggedright\arraybackslash}p{.20\textwidth}>{\raggedright\arraybackslash}p{.61\textwidth}>{\raggedright\arraybackslash}p{.14\textwidth}@{}}
\toprule
\rowcolor{mainresultshade}\textbf{Symbol} & \textbf{Meaning and scope} & \textbf{Reference}\\
\midrule\endfirsthead
\toprule
\rowcolor{mainresultshade}\textbf{Symbol} & \textbf{Meaning and scope (continued)} & \textbf{Reference}\\
\midrule\endhead
\midrule\multicolumn{3}{r}{\footnotesize Continued on the next page}\\\endfoot
\bottomrule\endlastfoot
$j,\ y_{k,j},\ s_{k,j}$ & Trial index, trial oracle point, and displacement $s_{k,j}=y_{k,j}-x_k$. & \hyperref[eq:bt-trial]{Eq.~(\ref*{eq:bt-trial})}\newline p.~\pageref{eq:bt-trial}\\[3pt]
$a_k,\ \mathcal D_{k,j}$ & Center gradient $a_k=\nabla f(x_k)$ and predicted decrease $\mathcal D_{k,j}=-\langle a_k,s_{k,j}\rangle$. Subscripts are omitted for a single trial. & \hyperref[eq:bt-trial]{Eq.~(\ref*{eq:bt-trial})}\newline p.~\pageref{eq:bt-trial}\\[3pt]
$F,\ g_{k+1},\ \lambda_k$ & Composite objective $F=f+\iota_{\cX}$, certified subgradient $g_{k+1}\in\partial F(x_{k+1})$, and an analysis multiplier. & \hyperref[thm:main-backtracking]{Thm.~\ref*{thm:main-backtracking}}\newline p.~\pageref{thm:main-backtracking}\\[3pt]
$\mathbb E,\ \mathbb P$ & Expectation and probability over sampling and any oracle-selection randomness; conditioning on $\mathcal F_k$ fixes all information before iteration $k$ samples its component. & \hyperref[sec:stochastic-setting]{\S\,\ref*{sec:stochastic-setting}}\newline p.~\pageref{sec:stochastic-setting}\\[3pt]
$n,\ i_k,\ \mathcal F_k$ & Number of components, sampled component, and pre-sampling information. $\mathbb P(i_k=i\mid\mathcal F_k)=1/n$. & \hyperref[sec:stochastic-setting]{\S\,\ref*{sec:stochastic-setting}}\newline p.~\pageref{sec:stochastic-setting}\\[3pt]
$g_{k,i}$ & Component gradient $\nabla f_i(x_k)$; the sampled gradient is $g_{k,i_k}$. Elsewhere $g_k$ denotes the deterministic gradient or the explicitly chosen subgradient. & \hyperref[sec:stochastic-setting]{\S\,\ref*{sec:stochastic-setting}}\newline p.~\pageref{sec:stochastic-setting}\\[3pt]
$x_{\star,i},\ f_i^\star$ & Fixed minimizer of component $f_i$ over $\cX$, and its minimum value. & \hyperref[sec:stochastic]{\S\,\ref*{sec:stochastic}}\newline p.~\pageref{sec:stochastic}\\[3pt]
$H_\star$ & Component-minimizer heterogeneity $\max_i\|x_{\star,i}-x_\star\|$; zero under the chosen common-minimizer convention. & \hyperref[sec:stochastic]{\S\,\ref*{sec:stochastic}}\newline p.~\pageref{sec:stochastic}\\[3pt]
$L_i,\ \mu_i$ & Component smoothness and strong-convexity constants, as required by the relevant theorem. & \hyperref[thm:stochastic_component_smooth_strongly_convex]{Thm.~\ref*{thm:stochastic_component_smooth_strongly_convex}}\newline p.~\pageref{thm:stochastic_component_smooth_strongly_convex}\\[3pt]
$L_{\max},\ \bar L_2^2$ & $L_{\max}=\max_i L_i$ and $\bar L_2^2=n^{-1}\sum_i L_i^2$. & \hyperref[thm:stochastic_component_smooth_convex_global]{Thm.~\ref*{thm:stochastic_component_smooth_convex_global}}\newline p.~\pageref{thm:stochastic_component_smooth_convex_global}\\[3pt]
$\theta_i,\ \rho_i,\ \bar\rho$ & $\theta_i=2\sqrt{\mu_iL_i}/(L_i+\mu_i)$, $\rho_i=(L_i-\mu_i)/(L_i+\mu_i)$, and $\bar\rho=n^{-1}\sum_i\rho_i$. & \hyperref[thm:stochastic_component_smooth_strongly_convex]{Thm.~\ref*{thm:stochastic_component_smooth_strongly_convex}}\newline p.~\pageref{thm:stochastic_component_smooth_strongly_convex}\\[3pt]
$\rho,\ \kappa_{\cP},\ \alpha$ & LLOO distance inflation, polytope geometry parameter, and damping fraction in the local-oracle comparison. The contraction factors $\rho_i$ are separate stochastic quantities. & \hyperref[sec:lloo]{\S\,\ref*{sec:lloo}}\newline p.~\pageref{sec:lloo}\\[3pt]
$\eta$ & Young-inequality parameter in $(0,1)$ in the stochastic bounded-gradient analysis; local limiting arguments also define their own auxiliary parameters. & \hyperref[thm:stochastic_component_rate]{Thm.~\ref*{thm:stochastic_component_rate}}\newline p.~\pageref{thm:stochastic_component_rate}\\[3pt]
$\cV,\ \Proj_{\cV}$ & Linear tangent space of an affine feasible set and its orthogonal projector. & \hyperref[sec:GD-subspace]{\S\,\ref*{sec:GD-subspace}}\newline p.~\pageref{sec:GD-subspace}\\[3pt]
$A,\ P_A,\ h_k$ & Application equality matrix, projector onto $\ker A$, and tangent gradient $h_k=P_A\nabla f(x_k)$. & \hyperref[eq:apps-affine-set]{Eq.~(\ref*{eq:apps-affine-set})}\newline p.~\pageref{eq:apps-affine-set}\\[3pt]
$\delta,\ \Delta$ & Locally defined feasibility buffer and simplex notation in applications; neither denotes the objective-gap sequence $\Delta_k$. & \hyperref[sec:potential-applications]{\S\,\ref*{sec:potential-applications}}\newline p.~\pageref{sec:potential-applications}\\[3pt]
\end{longtable}
\endgroup
}

%% file: appendix-concepts.tex
\clearpage
\section{Guide to key concepts}\label{sec:concepts}
Definitions are grouped by their role in the paper. The final column links to the first introduction or mention (including the abstract), followed where useful by the formal definition or result. Appendix~\ref{sec:notation} provides the complementary symbol index.

\begingroup
\small
\renewcommand{\arraystretch}{1.16}
\setlength{\tabcolsep}{5pt}
\setlength{\LTpre}{8pt}
\setlength{\LTpost}{8pt}
\begin{longtable}{@{}>{\raggedright\arraybackslash}p{.22\textwidth}>{\raggedright\arraybackslash}p{.56\textwidth}>{\raggedright\arraybackslash}p{\dimexpr.22\textwidth-20pt\relax}@{}}
\caption{Named concepts: definitions and routes to their first appearance.}\label{tab:concepts}\\
\toprule
\rowcolor{mainresultshade}\textbf{Concept} & \textbf{Definition and role} & \textbf{First appearance / definition}\\
\midrule\endfirsthead
\multicolumn{3}{@{}l}{\small\textit{Guide to key concepts (continued)}}\\[4pt]
\toprule
\rowcolor{mainresultshade}\textbf{Concept} & \textbf{Definition and role} & \textbf{First appearance / definition}\\
\midrule\endhead
\midrule\multicolumn{3}{r}{\footnotesize Continued on the next page}\\\endfoot
\bottomrule\endlastfoot
\rowcolor{mainresultshade}\multicolumn{3}{@{}l}{\rule{0pt}{2.5ex}\textbf{1. Basic objects and local models}}\\* [5pt]
\textbf{Objective and feasible set} & The function $f$ is minimized over the allowed points $\cX$. A constrained minimizer $x_\star$ belongs to $\cX$ and satisfies $f(x_\star)\le f(x)$ for every $x\in\cX$. & \hyperref[concept:abstract]{Abstract}, p.~\pageref{concept:abstract}\newline \hyperref[eq:main]{Eq.~\ref*{eq:main}}, p.~\pageref{eq:main}\\[10pt]
\textbf{Gradient} & The vector $\nabla f(x)$ represents the derivative: $f(x+d)=f(x)+\langle\nabla f(x),d\rangle+o(\|d\|)$. Its inner product with a displacement gives the first-order change in the objective. & \hyperref[concept:abstract]{Abstract}, p.~\pageref{concept:abstract}\\[10pt]
\textbf{Local ball and radius} & $\mathbb B(x,t)=\{z:\|z-x\|\le t\}$. The radius $t$ bounds movement from the current point; intersecting this ball with $\cX$ defines the local feasible set. & \hyperref[concept:abstract]{Abstract}, p.~\pageref{concept:abstract}\newline \hyperref[concept:broxgd]{\S\,\ref*{concept:broxgd}}, p.~\pageref{concept:broxgd}\\[10pt]
\textbf{Indicator function} & $\iota_{\cX}(x)=0$ for $x\in\cX$ and $+\infty$ otherwise. Adding it to an objective enforces the constraint exactly. & \hyperref[concept:model]{\S\,\ref*{sec:bpm-motivation}}, p.~\pageref{concept:model}\newline \hyperref[eq:g-indicator]{Eq.~\ref*{eq:g-indicator}}, p.~\pageref{eq:g-indicator}\\[10pt]
\textbf{First-order local model} & $m_k(z)=\langle\nabla f(x_k),z-x_k\rangle+\iota_{\cX}(z)$. The smooth part is linearized; the constraint is retained exactly. & \hyperref[concept:model]{\S\,\ref*{sec:bpm-motivation}}, p.~\pageref{concept:model}\\[10pt]
\textbf{Subgradient} & For convex $f$, a vector $g\in\partial f(x)$ satisfies $f(y)\ge f(x)+\langle g,y-x\rangle$ on its domain. It replaces the gradient in the nonsmooth Polyak rule. & \hyperref[concept:rates]{\S\,\ref*{concept:summary}}, p.~\pageref{concept:rates}\newline \hyperref[thm:bounded_subgradients]{Thm.~\ref*{thm:bounded_subgradients}}, p.~\pageref{thm:bounded_subgradients}\\[10pt]
\rowcolor{mainresultshade}\multicolumn{3}{@{}l}{\rule{0pt}{2.5ex}\textbf{2. Geometry and regularity}}\\* [5pt]
\textbf{Convex set and convex function} & A convex set contains every segment between its points. A function is convex if $f((1-a)x+ay)\le(1-a)f(x)+af(y)$ for $a\in[0,1]$; strict convexity makes this strict for distinct points and $0<a<1$. & \hyperref[concept:projection]{Introduction}, p.~\pageref{concept:projection}\newline \hyperref[ass:main]{Assump.~\ref*{ass:main}}, p.~\pageref{ass:main}\\[10pt]
\textbf{Smoothness} & $L$-smoothness means $\|\nabla f(x)-\nabla f(y)\|\le L\|x-y\|$ on the stated domain. Whether that domain is $\cX$, a neighborhood, or all of $\R^d$ matters in this paper. & \hyperref[sec:bpm-motivation]{\S\,\ref*{sec:bpm-motivation}}, p.~\pageref{sec:bpm-motivation}\newline \hyperref[thm:main-horizon]{Thm.~\ref*{thm:main-horizon}}, p.~\pageref{thm:main-horizon}\\[10pt]
\textbf{Strong convexity} & For differentiable $f$, $\mu$-strong convexity means $f(y)\ge f(x)+\langle\nabla f(x),y-x\rangle+\frac\mu2\|y-x\|^2$ on the specified convex domain, with $\mu>0$. & \hyperref[sec:bpm-motivation]{\S\,\ref*{sec:bpm-motivation}}, p.~\pageref{sec:bpm-motivation}\newline \hyperref[thm:practical-geometric]{Thm.~\ref*{thm:practical-geometric}}, p.~\pageref{thm:practical-geometric}\\[10pt]
\textbf{Bounded gradients} & $\|\nabla f(x)\|\le G$ throughout the stated set. For nonsmooth objectives, the analogous assumption bounds the relevant subgradients. & \hyperref[concept:rates]{\S\,\ref*{concept:summary}}, p.~\pageref{concept:rates}\newline \hyperref[thm:main-horizon]{Thm.~\ref*{thm:main-horizon}}, p.~\pageref{thm:main-horizon}\\[10pt]
\textbf{Normal cone} & $N_{\cX}(x)=\{v:\langle v,z-x\rangle\le0\ \forall z\in\cX\}$ for $x\in\cX$. It describes the outward normals and enters constrained first-order optimality. & \hyperref[concept:normal]{\S\,\ref*{sec:master_thm}}, p.~\pageref{concept:normal}\\[10pt]
\textbf{Fej\'er-type decrease} & The distance to a solution does not increase. Here admissibility gives the stronger decrease $\|x_{k+1}-x_\star\|^2\le\|x_k-x_\star\|^2-t_k^2$ for the selected solution. & \hyperref[concept:fejer]{\S\,\ref*{concept:fejer}}, p.~\pageref{concept:fejer}\\[10pt]
\textbf{Cocoercivity} & For a convex $L$-smooth function on an open convex domain, $\langle\nabla f(x)-\nabla f(y),x-y\rangle\ge L^{-1}\|\nabla f(x)-\nabla f(y)\|^2$. This relates gradient differences to useful displacement. & \hyperref[eq:bt-cocoercivity]{Eq.~\ref*{eq:bt-cocoercivity}}, p.~\pageref{eq:bt-cocoercivity}\\[10pt]
\textbf{Generalized asymmetric smoothness} & The bound $\|\nabla f(x)-\nabla f(y)\|\le (L_0+L_1\|\nabla f(y)\|)\|x-y\|$ allows the smoothness bound to depend on the gradient. Here the assumption is global. & \hyperref[concept:extensions]{\S\,\ref*{concept:summary}}, p.~\pageref{concept:extensions}\newline \hyperref[asp:l0l1]{Assump.~\ref*{asp:l0l1}}, p.~\pageref{asp:l0l1}\\[10pt]
\rowcolor{mainresultshade}\multicolumn{3}{@{}l}{\rule{0pt}{2.5ex}\textbf{3. Operators, oracles, and algorithms}}\\* [5pt]
\textbf{Euclidean projection} & $\Proj_{\cX}(y)=\arg\min_{z\in\cX}\|z-y\|$. For a nonempty closed convex set it is the unique feasible point nearest to $y$. & \hyperref[concept:abstract]{Abstract}, p.~\pageref{concept:abstract}\newline \hyperref[concept:projection]{Introduction}, p.~\pageref{concept:projection}\\[10pt]
\textbf{Bregman divergence} & $D_\phi(z,x)=\phi(z)-\phi(x)-\langle\nabla\phi(x),z-x\rangle$. It measures the error of the affine model of a differentiable strictly convex $\phi$; it need not be symmetric. & \hyperref[concept:projection]{Introduction}, p.~\pageref{concept:projection}\\[10pt]
\textbf{Proximal operator} & $\prox_{\gamma h}(y)=\arg\min_z\{h(z)+\|z-y\|^2/(2\gamma)\}$ for $\gamma>0$. It trades a decrease in $h$ against a quadratic movement penalty. & \hyperref[concept:abstract]{Abstract}, p.~\pageref{concept:abstract}\newline \hyperref[eq:brox-definition-section]{\S\,A.2}, p.~\pageref{eq:brox-definition-section}\\[10pt]
\textbf{Broximal operator} & $\brox_h^t(x)=\arg\min_{\|z-x\|\le t}h(z)$. It minimizes the original function inside a ball, replacing the quadratic penalty by a hard movement constraint; the output can be set-valued. & \hyperref[concept:abstract]{Abstract}, p.~\pageref{concept:abstract}\newline \hyperref[eq:main-brox-definition]{Eq.~\ref*{eq:main-brox-definition}}, p.~\pageref{eq:main-brox-definition}\\[10pt]
\textbf{Forward--backward splitting} & For $f+h$, the classical update is $\prox_{\gamma h}(x-\gamma\nabla f(x))$. Equivalently, minimize a linear model of $f$ plus $h$ and a quadratic penalty. & \hyperref[concept:abstract]{Abstract}, p.~\pageref{concept:abstract}\newline \hyperref[eq:classical-fb]{Eq.~\ref*{eq:classical-fb}}, p.~\pageref{eq:classical-fb}\\[10pt]
\textbf{Ball-Proximal Method} & \algnameintable{BPM} repeatedly minimizes the objective itself over a ball around the current iterate. Unlike \algnameintable{BroxGD}, it does not first linearize the smooth objective. & \hyperref[sec:bpm-motivation]{\S\,\ref*{sec:bpm-motivation}}, p.~\pageref{sec:bpm-motivation}\\[10pt]
\textbf{Projected gradient descent} & \algnameintable{ProjGD} takes $x^+=\Proj_{\cX}(x-\gamma\nabla f(x))$: a gradient step followed by Euclidean projection. & \hyperref[concept:abstract]{Abstract}, p.~\pageref{concept:abstract}\newline \hyperref[sec:related_work]{\S\,\ref*{sec:related_work}}, p.~\pageref{sec:related_work}\\[10pt]
\textbf{Broximal gradient descent} & \algnameintable{BroxGD} minimizes $\langle\nabla f(x),z\rangle$ over $\cX\cap\mathbb B(x,t)$. The outer update requests a local linear oracle rather than a projection. & \hyperref[concept:abstract]{Abstract}, p.~\pageref{concept:abstract}\newline \hyperref[eq:method]{Eq.~\ref*{eq:method}}, p.~\pageref{eq:method}\\[10pt]
\textbf{Global linear minimization oracle} & An LMO returns a minimizer of $\langle g,z\rangle$ over all of $\cX$. Frank--Wolfe uses such a point to form a feasible convex-combination update. & \hyperref[concept:lmo]{Introduction}, p.~\pageref{concept:lmo}\\[10pt]
\textbf{Exact local linear oracle} & Returns a minimizer of $\langle g,z\rangle$ over $\cX\cap\mathbb B(x,t)$. Exactness concerns the subproblem solution; it does not imply that computing it is inexpensive. & \hyperref[concept:abstract]{Abstract}, p.~\pageref{concept:abstract}\newline \hyperref[eq:method]{Eq.~\ref*{eq:method}}, p.~\pageref{eq:method}\\[10pt]
\textbf{Relaxed local linear optimization oracle} & An LLOO returns a linear value no larger than the minimum over $\cX\cap\mathbb B(x,t)$ but may return a point in $\cX$ as far as $\rho t$ from $x$, with $\rho\ge1$. This differs from solving the exact ball-constrained problem. & \hyperref[concept:prior]{\S\,\ref*{concept:prior}}, p.~\pageref{concept:prior}\newline \hyperref[eq:lloo-axioms]{Eq.~\ref*{eq:lloo-axioms}}\\[10pt]
\textbf{Separation oracle} & Given a point, either certifies membership in $\cX$ or returns a hyperplane separating that point from the set. & \hyperref[concept:separation]{Introduction}, p.~\pageref{concept:separation}\\[10pt]
\rowcolor{mainresultshade}\multicolumn{3}{@{}l}{\rule{0pt}{2.5ex}\textbf{4. Radius selection and step certificates}}\\* [5pt]
\textbf{Radius admissibility} & A condition on a positive radius that yields the one-step geometry of Theorem~\ref{thm:descent}. The denominator in the chosen Type-I or Type-II test must be nonzero. & \hyperref[thm:descent]{Thm.~\ref*{thm:descent}}, p.~\pageref{thm:descent}\\[10pt]
\textbf{Type-I admissibility} & $0<t_k\le\langle\nabla f(x_k),x_k-x_\star\rangle/\|\nabla f(x_k)\|$. The local linear decrease cannot pass the supporting hyperplane through the selected solution. & \hyperref[eq:radius-condition-1]{Eq.~\ref*{eq:radius-condition-1}}, p.~\pageref{eq:radius-condition-1}\\[10pt]
\textbf{Type-II admissibility} & With $v_k=\nabla f(x_k)-\nabla f(x_\star)$, require $0<t_k\le\langle v_k,x_k-x_\star\rangle/\|v_k\|$. This uses the gradient difference from a selected solution. & \hyperref[eq:radius-condition-2]{Eq.~\ref*{eq:radius-condition-2}}, p.~\pageref{eq:radius-condition-2}\\[10pt]
\textbf{Fixed-horizon radius} & A radius chosen using a prescribed budget $K$ and certified initial-distance bound $R$. In the smooth convex case, the paper uses the constant radius $R/\sqrt K$. & \hyperref[concept:extensions]{\S\,\ref*{concept:summary}}, p.~\pageref{concept:extensions}\newline \hyperref[thm:main-horizon]{Thm.~\ref*{thm:main-horizon}}, p.~\pageref{thm:main-horizon}\\[10pt]
\textbf{Predetermined geometric radii} & A schedule $t_k=aq^{k/2}$, with fixed $a>0$ and $0<q<1$ chosen from certified bounds. The radii shrink without using an acceptance test at each step. & \hyperref[concept:extensions]{\S\,\ref*{concept:summary}}, p.~\pageref{concept:extensions}\newline \hyperref[thm:practical-geometric]{Thm.~\ref*{thm:practical-geometric}}, p.~\pageref{thm:practical-geometric}\\[10pt]
\textbf{Polyak radius} & $t_k=(f(x_k)-f_\star)/\|g_k\|$ for a positive objective gap, with $g_k$ a gradient or subgradient as specified. It adapts to the gap but requires the optimal value. & \hyperref[concept:extensions]{\S\,\ref*{concept:summary}}, p.~\pageref{concept:extensions}\newline \hyperref[eq:polyak_radius_smooth_convex]{Eq.~\ref*{eq:polyak_radius_smooth_convex}}, p.~\pageref{eq:polyak_radius_smooth_convex}\\[10pt]
\textbf{Radius backtracking} & Solve a trial local problem, accept if $-\langle\nabla f(x),s\rangle\ge L\|s\|^2$, and otherwise contract the radius by $\xi\in(0,1)$. Rejected trials also cost oracle calls. & \hyperref[concept:extensions]{\S\,\ref*{concept:summary}}, p.~\pageref{concept:extensions}\newline \hyperref[alg:radius-backtracking]{Alg.~\ref*{alg:radius-backtracking}}, p.~\pageref{alg:radius-backtracking}\\[10pt]
\rowcolor{mainresultshade}\multicolumn{3}{@{}l}{\rule{0pt}{2.5ex}\textbf{5. Convergence measures and stochastic structure}}\\* [5pt]
\textbf{Objective gap} & $\Delta_k=f(x_k)-f_\star$. It measures objective suboptimality, rather than distance to a solution or first-order stationarity. & \hyperref[sec:related_work]{\S\,\ref*{sec:related_work}}, p.~\pageref{sec:related_work}\newline \hyperref[tbl:main]{Table~\ref*{tbl:main}}, p.~\pageref{tbl:main}\\[10pt]
\textbf{Linear convergence} & An error is bounded by $Cq^k$ for $0<q<1$. The error may be an objective gap, distance, or squared distance; the theorem specifies which one. & \hyperref[sec:bpm-motivation]{\S\,\ref*{sec:bpm-motivation}}, p.~\pageref{sec:bpm-motivation}\\[10pt]
\textbf{Sublinear convergence} & The reported bounds decay polynomially, such as $C/K$ or $C/\sqrt K$. The guarantee also specifies whether it concerns the last, best, or averaged iterate. & \hyperref[sec:related_work]{\S\,\ref*{sec:related_work}}, p.~\pageref{sec:related_work}\\[10pt]
\textbf{Gradient-difference residual} & $Z_k=\|\nabla f(x_k)-\nabla f(x_\star)\|$. This compares two gradients; it is not, in general, an objective gap or a constrained stationarity measure. & \hyperref[concept:extensions]{\S\,\ref*{concept:summary}}, p.~\pageref{concept:extensions}\newline \hyperref[tbl:main]{Table~\ref*{tbl:main}}, p.~\pageref{tbl:main}\\[10pt]
\textbf{Projected gradient mapping} & $G_\gamma(x)=\gamma^{-1}(x-\Proj_{\cX}(x-\gamma\nabla f(x)))$. Its norm measures the size of a projected-gradient step per unit stepsize. & \hyperref[tbl:main]{Table~\ref*{tbl:main}}, p.~\pageref{tbl:main}\\[10pt]
\textbf{Constrained stationarity} & $0\in\nabla f(x)+N_{\cX}(x)$, equivalently $G_\gamma(x)=0$ for $\gamma>0$. The gradient itself need not vanish at a constrained optimum. & \hyperref[concept:extensions]{\S\,\ref*{concept:summary}}, p.~\pageref{concept:extensions}\newline \hyperref[sec:non-convex]{\S\,\ref*{sec:non-convex}}, p.~\pageref{sec:non-convex}\\[10pt]
\textbf{Projected Polyak--\L ojasiewicz inequality} & $\frac12\|G_{1/L}(x)\|^2\ge\mu(f(x)-f_\star)$. This assumption links the constrained stationarity measure to the objective gap and yields a linear objective rate. & \hyperref[concept:extensions]{\S\,\ref*{concept:summary}}, p.~\pageref{concept:extensions}\newline \hyperref[ass:projected_pl]{Assump.~\ref*{ass:projected_pl}}, p.~\pageref{ass:projected_pl}\\[10pt]
\textbf{Frank--Wolfe curvature constant} & $C_{\algnameintable{FW}}$ uniformly bounds the quadratic remainder of $f$ along feasible segments after division by half the squared segment parameter. It controls the classical objective-rate bound. & \hyperref[sec:related_work]{\S\,\ref*{sec:related_work}}, p.~\pageref{sec:related_work}\newline \hyperref[eq:curvature]{Eq.~\ref*{eq:curvature}}\\[10pt]
\textbf{Frank--Wolfe gap} & $g_{\algnameintable{FW}}(x)=\max_{s\in\cX}\langle\nabla f(x),x-s\rangle$ on a compact set. It measures the best available global linear decrease. & \hyperref[tbl:main]{Table~\ref*{tbl:main}, note (e)}, p.~\pageref{tbl:main}\\[10pt]
\textbf{Finite-sum stochastic update} & For $f=n^{-1}\sum_i f_i$, sample a component uniformly conditional on the past and use its gradient in the local oracle. A zero sampled step does not stop later sampling. & \hyperref[concept:extensions]{\S\,\ref*{concept:summary}}, p.~\pageref{concept:extensions}\newline \hyperref[sec:stochastic]{\S\,\ref*{sec:stochastic}}, p.~\pageref{sec:stochastic}\\[10pt]
\textbf{Component heterogeneity} & $H_\star=\max_i\|x_{\star,i}-x_\star\|$ measures separation between selected component minimizers and a minimizer of the average objective. It enters the stochastic error neighborhoods. & \hyperref[sec:stochastic]{\S\,\ref*{sec:stochastic}}, p.~\pageref{sec:stochastic}\\[10pt]
\end{longtable}
\endgroup